\documentclass[11pt]{article}
\usepackage{a4wide}

\usepackage{amssymb}
\usepackage{amsmath}
\usepackage{amsthm}
\usepackage{tikz-cd}
\usepackage{mathrsfs}
\usepackage{graphicx} 
\usepackage[font=small]{subcaption}

\usepackage{mathtools}
\usepackage{hyperref}
\usepackage{enumerate}

\usepackage{thmtools}
\usepackage{thm-restate}
\usepackage{verbatim}
\usepackage{romannum}

\usepackage{latexsym}
\usepackage{amscd}
\usepackage{graphics}
\usepackage{amsmath}
\usepackage{amssymb}
\usepackage{esint}
\usepackage{mathrsfs}
\usepackage{bbm}
\usepackage{color}
\usepackage{accents}

\newtheorem{theorem}{Theorem}[section]
\newtheorem{prop}[theorem]{Proposition}
\newtheorem{definition}[theorem]{Definition}
\newtheorem{cor}[theorem]{Corollary}
\newtheorem{lemma}[theorem]{Lemma}

\newtheorem{remark}[theorem]{Remark}

\renewcommand{\ker}{\operatorname{ker}}

\renewcommand{\dim}{\operatorname{dim}}

\newcommand{\ind}{\operatorname{ind}}

\newcommand{\Crit}{\mathrm{Crit}}
\newcommand{\PD}{\mathrm{PD}}
\newcommand{\ev}{\mathrm{ev}}
\newcommand{\m}{\mathfrak{m}}

\renewcommand{\H}{\mathrm{H}}
\newcommand{\FC}{\mathrm{FC}}
\newcommand{\FH}{\mathrm{FH}}

\newcommand{\SH}{\mathrm{SH}}
\newcommand{\CZ}{\mathrm{CZ}}
\newcommand{\RFH}{\mathrm{RFH}}

\newcommand{\QC}{\mathrm{QC}}
\newcommand{\QH}{\mathrm{QH}}
\newcommand{\RFC}{\mathrm{RFC}}

\newcommand{\R}{\mathbb{R}}
\newcommand{\Z}{\mathbb{Z}}
\newcommand{\D}{\mathbb{D}}
\newcommand{\N}{\mathbb{N}}
\newcommand{\C}{\mathbb{C}}
\newcommand{\Hom}{\mathrm{Hom}}
\newcommand{\RS}{\mathrm{RS}}
\newcommand{\HW}{\mathrm{HW}}

\newcommand{\wind}{\mathrm{wind}}

\newcommand{\Pin}{\mathrm{Pin}^{\pm}}

\newcommand{\w}{\mathfrak{w}}
\newcommand{\Cv}{C([\mathbf{v}])}
\newcommand{\Cw}{C([\mathbf{w}])}
\newcommand{\overbar}[1]{\mkern 1.5mu\overline{\mkern-1.5mu#1\mkern-1.5mu}\mkern 1.5mu}

\DeclareFontFamily{U}{mathx}{\hyphenchar\font45}
\DeclareFontShape{U}{mathx}{m}{n}{
	<5> <6> <7> <8> <9> <10>
	<10.95> <12> <14.4> <17.28> <20.74> <24.88>
	mathx10
}{}
\DeclareSymbolFont{mathx}{U}{mathx}{m}{n}
\DeclareFontSubstitution{U}{mathx}{m}{n}
\DeclareMathAccent{\widecheck}{0}{mathx}{"71}
\DeclareMathAccent{\wideparen}{0}{mathx}{"75}

\numberwithin{equation}{section}
\AtEndDocument{\bigskip{\footnotesize
\noindent  \textsc{Chonnam National University, Department of Mathematics Education, 61186 Gwangju, South Korea} \par  
\noindent  \textit{E-mail address}: \texttt{\href{mailto:hanwool.bae@jnu.ac.kr}{hanwool.bae@jnu.ac.kr}} \par
  \addvspace{\medskipamount}
\noindent\textsc{Seoul National University, Department of Mathematical Sciences, Research Institute in Mathematics, 
	08826 Seoul, South Korea} \par  
 \noindent \textit{E-mail address}: \texttt{\href{mailto:jungsoo.kang@snu.ac.kr}{jungsoo.kang@snu.ac.kr}} \par
   \addvspace{\medskipamount}
\noindent  \textsc{The Institute of Geometry and Physics, University of Science and Technology of China, 96 Jinzhai Road, Hefei Anhui, 230026, China} \par  
\noindent  \textit{E-mail address}: \texttt{\href{mailto:sunghokim@ustc.edu.cn}{sunghokim@ustc.edu.cn}} \par
}}

\title{Rabinowitz Floer homology for Legendrian submanifolds in prequantization bundles}
\author{Hanwool Bae, Jungsoo Kang, Sungho Kim}
\date{}

\begin{document}
	\pagenumbering{arabic}
	
	\maketitle
	\begin{abstract}
	Let $Y$ be a prequantization bundle over an integral symplectic manifold $(\Sigma,\omega)$. 
	Let $L$ be a closed monotone Lagrangian submanifold that admits a Legendrian lift $\mathcal{L}$ in $Y$. Under the assumption that the minimal Maslov number $N_L$ of $L$ is greater than 2, we define the Rabinowitz Floer homology of $\mathcal{L}$. We then establish an isomorphism between the $\Z_d$-equivariant Rabinowitz Floer homology of $\mathcal{L}$ and the quantum homology of $L$, where $d$ is the degree of the covering map $\mathcal{L}\to L$. Under a more restrictive condition on $N_L$, we show that this map is a ring isomorphism. Using this isomorphism, we compute the quantum homology ring of Lagrangian spheres in quadrics and two-step flag manifolds. Furthermore, we investigate the implications of the quantum invertibility of $\omega$ for the vanishing of the quantum homology of $L$ and the obstructions to topologically simple fillings of $\mathcal{L}$. We also show that if $(\Sigma,\omega)$ admits a polarization and $L$ is disjoint from the Lagrangian trace, the quantum homology of $L$ vanishes. 	
	\end{abstract}
	\setcounter{tocdepth}{1}

	\tableofcontents
	\section{Introduction}\label{sec:Intro}
	
	Let $(\Sigma,\omega)$ be a closed, connected symplectic manifold. We assume that the cohomology class $[\omega]\in \H^2(\Sigma;\R)$ admits an integral lift $[\omega]_\Z\in  \H^2(\Sigma;\Z)=\H^2(\Sigma)$. For notational simplicity, we omit the coefficient ring when it is $\Z$. We consider the prequantization bundle $\pi:Y\to \Sigma$ with the Euler class $-[\omega]_\Z$ and a connection $1$-form $\alpha$ satisfying $d\alpha=\pi^*\omega$. Note that $\alpha$ is a contact form and uniquely determined up to strict contactomorphism. 
Let $L$ be a closed, connected Lagrangian submanifold of $(\Sigma,\omega)$. Let $\mu_L\in\H^2(\Sigma,L)$ denote the Maslov class of $L$. We assume that $L$ is monotone, namely, $\mu_L$ and $\omega$ are positively proportional on $\pi_2(\Sigma,L)$. In addition, we impose the following conditions.
\begin{enumerate}[(C1)]
	\item The holonomy representation of $L$ with respect to $\alpha$ equals $\Z_d\subset S^1$ for some $d\in\N$. 
	\item The minimal Maslov number $N_L$ of $L$ satisfies $N_L>2$.
\end{enumerate}
	If, for example, $\omega(\pi_2(\Sigma))\neq 0$ and the composition of the boundary map and the Hurewicz map $\pi_2(\Sigma,L)\to\pi_1(L)\to\H_1(L)$ is surjective, then (C1) holds with $d=\frac{2c_\Sigma}{\gcd(2c_\Sigma,N_Lm_\Sigma)}$, where $c_\Sigma$ is the minimal Chern number of $\Sigma$ and $m_\Sigma$ is the nonnegative generator of $\omega(\pi_2(\Sigma))$. In particular, if $\H_1(L)$ is trivial, then $d=1$. Alternatively, (C1) is satisfied if $L$ is homologically monotone; see Section  \ref{sec:lift}.
	Under condition (C1), $L$ admits a Legendrian lift $\mathcal{L}\subset (Y,\ker\alpha)$ such that $\pi|_{\mathcal{L}}:\mathcal{L}\to L$ is a $d$-fold covering map. 
 
We denote by $L^\sigma$ the Lagrangian $L$ equipped with a relative $\Pin$-structure $\sigma$. We also write $\mathcal{L}^{\tilde\sigma}$ for $\mathcal{L}$ equipped with the corresponding lift $\tilde\sigma$ of $\sigma$; see Remark \ref{rem:pin_lift} for the precise construction. In this paper, we study the (Lagrangian) Rabinowitz Floer homology $\RFH_*(\mathcal{L}^{\tilde{\sigma}})$, whose chain-level generators are contractible generalized Reeb chords on $(Y,\mathcal{L})$ and boundary operators are defined by counting Floer strips in the symplectization $(\R\times Y,\R\times \mathcal{L})$.
It is worth noting that we do not assume the existence of a Liouville filling of $Y$ or a Lagrangian filling of $\mathcal{L}$. Instead, condition (C2) ensures the necessary compactness properties. From the perspective of augmentations, this corresponds to the existence of trivial augmentations for both periodic Reeb orbits and chords, see Remark \ref{rem:inde_nondeg}.
In the presence of fillings of $Y$ and $L$, this condition can be relaxed to $N_L\geq2$, as discussed in Section \ref{sec:filling}.

\subsection{Main results}
\begin{theorem}\label{thm:isom_transfer}
	Let $(\Sigma,L^\sigma)$ and $(Y,\mathcal{L}^{\tilde{\sigma}})$ be as above. 
	\begin{enumerate}[(a)]
		\item The Rabinowitz Floer homology $\RFH_*(\mathcal{L}^{\tilde{\sigma}})$ and its $\Z_{d}$-equivariant counterpart $\RFH_*^{\Z_{d}}(\mathcal{L}^{\tilde{\sigma}})$ are well-defined as $\Z$-modules. Here, the $\Z_d$-action is induced by the deck transformations of the covering map $\pi|_{\mathcal{L}}:\mathcal{L}\to L$.
		\item There exists a $\Z$-module isomorphism between $\RFH_*^{\Z_{d}}(\mathcal{L}^{\tilde{\sigma}})$ and the quantum homology $\QH_*(L^\sigma)$ of $L^\sigma$.
		\item Assume that $N_L>\max\{\frac{1}{2}(\dim L+1),2\}$. Then, $\RFH_*(\mathcal{L}^{\tilde{\sigma}})$ and $\RFH_*^{\Z_{d}}(\mathcal{L}^{\tilde{\sigma}})$ are well-defined as rings, and the latter is isomorphic to $\QH_*(L^\sigma)$ as a ring.
	\end{enumerate}
\end{theorem}

If we work with $\Z_2$-coefficients, Theorem \ref{thm:isom_transfer} holds without assuming the existence of a relative $\Pin$-structure on $L$. If we use coefficients other than $\Z$, this will be indicated in the notation; otherwise, all homology groups are taken with $\Z$-coefficients.

Condition (C2)  can be weakened in some situations.  
Theorem \ref{thm:isom_transfer} remains valid under the weaker condition $N_L\geq 2$, provided that $Y$ and $\mathcal{L}$ admit topologically simple Liouville and exact Lagrangian fillings, respectively, and the lifted relative $\Pin$-structure $\tilde\sigma$ extends to the fillings; see Proposition \ref{prop:filling} and Proposition \ref{prop:RFH_well_defined} below. On the other hand, even when there are such fillings, if $\tilde\sigma$ does not extend to the fillings, then $\QH_*(L^\sigma)$ may not have any reasonable relation with the Rabinowitz Floer homology of $\mathcal{L}$; 
see Remark \ref{remark equator in two sphere_short}. 

As usual, $\RFH_*$ is invariant under changes of the contact form and Legendrian isotopy. Let $\alpha'$ be a contact form supporting the same contact structure as $\alpha$, and let $\mathcal{L}'$ be Legendrian isotopic to $\mathcal{L}$. To define $\RFH_*$ for $(\alpha',\mathcal{L}')$, we employ a pair consisting of a Liouville cobordism and an exact Lagrangian cobordism whose concave end is $(c\alpha,\mathcal{L})$ for a small $c>0$ and whose convex end is $(\alpha',\mathcal{L}')$. We then take the Floer homology for $\vee$-shaped Hamiltonians (see Section \ref{sec:LFC}) on the region associated with $(\alpha',\mathcal{L}')$ cf.~\cite[Section 9.5]{CO18}. Since the concave end is $(c\alpha,\mathcal{L})$, compactness is guaranteed by condition (C2) as mentioned above. 

\begin{remark}\label{rem:om=0}
	Assume that $\omega$ vanishes on $\H_2(\Sigma,L)$ and that the image of the inclusion-induced map $\H_1(L) \to \H_1(\Sigma)$ is torsion-free. Then, by Lemma \ref{lem:Leg_lift}, $L$ admits a Legendrian lift $\mathcal{L}$ with $d=1$, and Theorem \ref{thm:isom_transfer} yields the isomorphisms
	\[
	\RFH_*(\mathcal{L}^{\tilde{\sigma}})\cong \QH_*(L^\sigma)\cong\H_*(L; \mathcal{O})\cong \H^{\dim L-*}(L),
	\]
	where $\mathcal{O}$ denotes the local system of orientations. 
	Note that $L$ may be nonorientable, and we adopt the convention that  the quantum homology of $L$ recovers the singular cohomology of $L$ when there is no $J$-holomorphic disk. In this setting, by Proposition \ref{prop:contractible}, the chain-level generators of $\RFH_*(\mathcal{L}^{\tilde{\sigma}})$ are necessarily constant chords on $\mathcal{L}$.
\end{remark}

	The  homology $\RFH_*(\mathcal{L}^{\tilde{\sigma}})$ carries a $\Z[T,T^{-1}]$-module or  algebra structure, depending on the assumption on $N_L$, see Remark \ref{rem:RFH_laurent}. This structure corresponds to the Seidel representation adapted to the present setting, cf.~\cite{Ueb19}, and is compatible with the $\Z[T,T^{-1}]$-algebra structure on $\QH_*(L^\sigma)$, see Remark \ref{rem:laurent}.

Recently, the Frobenius algebra structure	on Rabinowitz Floer homology has been studied extensively in the seminal works of \cite{CHO25,CO22b}. It is natural to expect that the isomorphism in Theorem \ref{thm:isom_transfer} in fact holds at the level of Frobenius algebras. This question will be investigated elsewhere. We also would like to mention the related work \cite{KPS24}, where correspondences involving Fukaya categories  are investigated.

\subsection{Examples of computations of QH and RFH}
When $\mathcal{L}\to L$ is bijective,  the isomorphism $\RFH_*(\mathcal{L}^{\tilde{\sigma}}) \cong\QH_*(L^\sigma)$ allows one to compute either homology from the other. We present some examples of this computation.

Let $N$ be a Riemannian manifold all of whose geodesics are simple, closed and have the same length. Then the unit cotangent bundle $S^*N$ of $N$ is a prequantization bundle over the space $\Sigma_N$ of oriented closed geodesics, denoted by $\pi_{S^*N}:S^*N\to\Sigma_N$, and a unit cotangent fiber projects to a Lagrangian sphere $L_N:=\pi_{S^*N}(S^*_qN)$. We compute the quantum homology ring of $L_N$ when $N$ is $S^k$ or  $\mathbb{CP}^k$ using Theorem \ref{thm:isom_transfer}. 

In the case of $N=S^k$, the base symplectic manifold $\Sigma_N$ is the smooth quadric
\[
Q^{k-1} := \{ [z_0:\dots :z_{k}] \in \mathbb{CP}^{k} \mid z_0^2 +\dots +z_{k}^2=0\}
\]
equipped with the restriction of the Fubini-Study form $\omega_{\mathrm{FS}}$, and the Lagrangian sphere $L_N$ is 
\begin{equation}
	\label{eq:lagrangian_trace_s^k}
	L_{S^k} =\left\{ [i : x_1:\dots :x_k] \in \mathbb{CP}^{k} \mid  \sum_{j=1}^{k} x_j^2 =1, x_j\in \R\right\} \subset Q^{k-1}.
\end{equation}
We only consider the case $k\geq 3$ since $L_{S^2}$ is a great circle of $Q^1=\mathbb{CP}^1$ whose quantum homology is well-known. 
Since $\H^1(L_{S^k};\Z_2)=0$, there is a unique (absolute) $\mathrm{Spin}$-structure $\sigma_0$. The symplectic form $\omega_{\mathrm{FS}}|_{Q^{k-1}}$ defines a cohomology class of $\H^2(Q^{k-1};\Z_2)$. Let $\sigma_1$ be a unique relative $\mathrm{Spin}$-structure on $L_{S^k}$ with respect to this cohomology class.

We use the degree shift
$\QH_*(L_{S^k}^{\sigma})[k-1] := \QH_{*+k-1}(L_{S^k}^{\sigma})$
so that the product structure on $\QH_*(L_{S^k}^{\sigma})[k-1]$ has degree zero.

\begin{prop}(Cf.~\cite[Remark 36]{KS21a})
	There exists a ring isomorphism 
	\[
	\QH_*(L_{S^k}^{\sigma})[k-1]  \cong \mathbb{Z}[x,x^{-1}],
	\]
	for any $\sigma\in\{\sigma_0,\sigma_1\}$, where $x$ has degree $k-1$. 
\end{prop}

Let $N=\mathbb{CP}^k$ for $k\geq 2$. The base symplectic manifold $\Sigma_N$ is the two-step flag manifold 
\[
F=F_{1,2}(\C^{k+1}):=\{ ([z],[w]) \in \mathbb{CP}^k \times \mathbb{CP}^k \mid  \langle z,w\rangle=0\}
\]
equipped with $\omega_F:=(\omega_\mathrm{FS}\oplus-\omega_\mathrm{FS})|_F$. 	Here $\langle \cdot,\cdot \rangle$ denotes the Hermitian inner product. By the Lefschetz hyperplane theorem, we know that $\H^2(F;\Z_2)\cong\H^2(\mathbb{CP}^k \times \mathbb{CP}^k;\Z_2)$, which has 4 elements. A unit cotangent fiber of $\mathbb{CP}^k$ projects to the Lagrangian sphere
\begin{equation}\label{eq:lagrangian_trace_CP^k}
L_{\mathbb{CP}^k} = \{([z],[w]) \in F \mid z_0=w_0=1,\; z_i=-w_i\;1\leq i\leq k \},	
\end{equation}
where $[z]=[z_0:\cdots:z_k]$ and $[w]=[w_0:\cdots:w_k]$. 
\begin{prop}\label{prop:QH_flag}
	Let $b\in \H^2(F;\Z_2)$ with $b\in\{0,[\omega_F]_{\Z_2}\}$ if $k$ is odd, and $b\notin\{0,[\omega_F]_{\Z_2}\}$ if $k$ is even.  
	For the unique relative $\mathrm{Spin}$-structure $\sigma$ associated to $b$, there exists a ring isomorphism
	\[
	\QH_*(L_{\mathbb{CP}^k}^{\sigma})[2k-1] \cong \Z[x,y,y^{-1}]/(x^2),
	\]
	where $x$ and $y$ have degrees $1$ and $2k$. 
\end{prop}
	The choice of $\sigma$  is crucial as the following proposition demonstrates.

\begin{prop}\label{prop:QH_flag_zero}
	For the Lagrangian 3-sphere $L_{\mathbb{CP}^2}\subset F_{1,2}(\C^3)$ equipped with the unique (absolute) $\mathrm{Spin}$-structure $\sigma_0$, we have
	\[
	\QH_*(L_{\mathbb{CP}^2}^{\sigma_0};\mathbb{K})=0,
	\] 
	where $\mathbb{K}$ is any field of characteristic different from $2$.
\end{prop}
Proposition \ref{prop:QH_flag} and Proposition \ref{prop:QH_flag_zero} are proved in Section \ref{sec:CP^k}.

\medskip
	
	Next, we consider the Rabinowitz Floer homology of $\mathbb{RP}^k$ in $\mathbb{RP}^{2k+1}$. As proved in \cite[Theorem 1.10]{DRS24}, this homology has rank 1 in every degree as a $\Z_2$-module. Our next result extends this computation as follows. 

\begin{prop}
Let $\mathcal{L}$ be $\mathbb{RP}^k$ in the standard contact $\mathbb{RP}^{2k+1}$ for $k\geq2$. Then, 
there exist ring isomorphisms 
\[
\RFH_*(\mathcal{L};\Z_2)[k]\cong\Z_2[x,x^{-1}], \qquad \RFH_*(\mathcal{L}^{\tilde{\sigma}})[k]\cong \Z_2[y,y^{-1}]
\]	
where $x$ has degree 1 and $y$ has degree 2. Here $[k]$ refers to the degree shift as above.
\end{prop}
\begin{proof}
	Since $L=\mathbb{RP}^k$ in $\mathbb{CP}^k$ lifts to $\mathcal{L}$ in $\mathbb{RP}^{2k+1}$ and $N_L=k+1> \max\{\frac{1}{2}(k+1),2\}$,  
	by Theorem \ref{thm:isom_transfer}, we have $\RFH_*(\mathcal{L};\Z_2)[k]\cong \QH_*(L;\Z_2)[k]$, where the latter is isomorphic to $\Z_2[x,x^{-1}]$. 
	
	Let $\tilde\sigma$ be the lift of a relative $\Pin$-structure $\sigma$ on $L$.
	To compute $\RFH_*(\mathcal{L}^{\tilde{\sigma}})$, we recall from \cite[Section 8]{Z1} and \cite[Section 3.6]{KS21b} that $\QH_*(L^\sigma)[k]$ is isomorphic to $\Z_2$ in even degrees and $0$ in odd degrees. The quantum homology $\QH_*(\mathbb{CP}^k,b)[2k]$ twisted by the background class $b\in\H^2(\mathbb{CP}^k;\Z_2)$ of $\sigma$ (see \cite[Proposition 15]{KS21b}) is isomorphic to $\Z[u,u^{-1}]$ with $\deg u=2$. Now $\RFH_*(\mathcal{L}^{\tilde{\sigma}})[k]\cong\QH_*({L}^{{\sigma}})[k]\cong \Z_2[y,y^{-1}]$ follows from the fact that the closed-open map $\QH_*(\mathbb{CP}^k,b)[2k]\to \QH_*(L^\sigma)[k]$ is a nontrivial ring homomorphism. 
\end{proof}

\subsection{Vanishing of QH for certain Lagrangians}
The isomorphism in Theorem \ref{thm:isom_transfer} is obtained by establishing a chain-level isomorphism for a suitable choice of auxiliary data. In fact, the proof implies the following result.
Let $Y_{d}\to \Sigma$ denote the prequantization bundle with Euler class $-d[\omega]_\Z$. Then $L$ admits a Legendrian lift $\mathcal{L}_{d}\subset Y_{d}$ such that $(Y,\mathcal{L})\to (Y_{d},\mathcal{L}_{d})$ is a $d$-fold covering. This yields the isomorphisms  $\RFH_*^{\Z_{d}}( \mathcal{L}^{\tilde\sigma})\cong \RFH_*( \mathcal{L}^{\tilde\sigma}_{d})\cong \QH_*(L^\sigma)$. The following corollary also immediately follows.

\begin{cor}\label{cor:transfer}
	Let $(\Sigma,L^\sigma)$ and $(Y,\mathcal{L}^{\tilde{\sigma}})$ be as in Theorem \ref{thm:isom_transfer}. Then there exist transfer and projection homomorphisms  
	\[
	\tau:\QH_*(L^\sigma)\to \RFH_* (\mathcal{L}^{\tilde\sigma}),\qquad \wp:\RFH_* (\mathcal{L}^{\tilde\sigma}) \to \QH_*(L^\sigma)
	\]
	such that the composition $\wp\circ\tau$ coincides with the scalar multiplication by $d$ on the $\Z$-module $\QH_*(L^\sigma)$. In particular, if $\RFH_*(\mathcal{L}^{\tilde\sigma})=0$, then  $\QH_*(L^\sigma)$ is torsion of order $d$. 
	\end{cor}

\begin{theorem}\label{thm:torsion}
	Let $(\Sigma,L^\sigma)$ and $(Y,\mathcal{L}^{\tilde{\sigma}})$ be as in Theorem \ref{thm:isom_transfer}, where $\sigma$ is a $\Pin$-structure. If $[\omega]_\Z$ is invertible in the quantum cohomology $\QH^*(\Sigma)$ of $\Sigma$, then, $\QH_*(L^{\sigma})$ is torsion of order $d$.  Moreover, for any prime number $p$, we have $\QH_*(L^{\sigma};\Z_p)\cong \QH_{*+1}(L^{\sigma};\Z_p)$ for all degrees.	If $[\omega]_\mathbb{Q}$ induced by $[\omega]_\Z$ via the inclusion $\Z\hookrightarrow\mathbb{Q}$ is invertible in $\QH^*(\Sigma;\mathbb{Q})$, then $\QH_*(L^\sigma;\mathbb{Q})=0$.
\end{theorem}

The invertibility $[\omega]_\Z$ is independent of the choice of an integral lift of $[\omega]$. 
One case where $L$ admits a Legendrian lift $\mathcal{L}$ is when $\omega(\pi_2(\Sigma))\neq0$ and the map $\pi_2(\Sigma,L)\to\H_1(L)$ is surjective as mentioned above. In this setting, one can show without relying on Corollary \ref{cor:transfer} that $\QH_*(L^{\sigma})$ is torsion of order $d=\frac{2c_\Sigma}{\gcd(2c_\Sigma, m_\Sigma N_L)}$.
 Theorem \ref{thm:torsion} applies, for instance, to $\Sigma=\mathbb{CP}^n$ or $\Sigma=\mathbb{CP}^n\times X$, where $X$ is a symplectically aspherical integral symplectic manifold, and building on this, we revisit some insightful results of \cite{Bir06} regarding Lagrangian spheres. We provide a full account of these in Section \ref{subsection:invertibilityofeulerclass}. 

\medskip

A similar vanishing property holds for Lagrangian submanifolds that are disjoint from the Lagrangian traces in polarized symplectic manifolds. The notion of the Lagrangian trace was introduced in \cite{Bir06}, where its role as an obstruction to displaceability was investigated.
To state our next result, which is along these lines, we recall the setting from \cite{Bir06}, see also \cite{Don96,Gir17}. Let $(X,\Omega,\Sigma)$ be a polarized symplectic manifold, meaning that $(X,\Omega)$ is a closed symplectic manifold with an integral lift $[\Omega]_\Z\in\H^2(X)$, $(\Sigma,\omega:=\Omega|_\Sigma)$ is a symplectic hypersurface such that the homology class $[\Sigma]$ is Poincar\'e dual to $k[\Omega]_\Z$ for some $k\in\N$, and there is a Weinstein domain $(W,\lambda,\varphi)$ such that the interior of $W$ is symplectomorphic to $X\setminus\Sigma$. Here $
\lambda$ is a Liouville $1$-form and $\varphi$ is a Morse function such that the Liouville vector field $Z$ is a gradient-like vector field of $\varphi$. The boundary $Y:=\partial W$ is a prequantization bundle over $(\Sigma,k[\Omega]_\Z|_\Sigma)$. The unstable manifold of $Z$ at a critical point $p$ of $\varphi$ of Morse index $\frac{1}{2}\dim W$ is called a Lagrangian cocore disk, denoted by $\Delta(p)$. The boundary $\partial\Delta(p)$ of $\Delta(p)$ is a Legendrian sphere in $Y$, and its projection to $\Sigma$ is a possibly immersed Lagrangian sphere. The Lagrangian trace refers to the union of such Lagrangian spheres in $\Sigma$ for all critical points of $\varphi$ with index $\frac{1}{2}\dim W$. Equivalently, the Lagrangian trace corresponds to the union of the topological boundaries of $\mathrm{int}(\Delta(p))\subset X\setminus\Sigma$ within $X$.
The Lagrangian spheres in \eqref{eq:lagrangian_trace_s^k} and \eqref{eq:lagrangian_trace_CP^k} are examples of Lagrangian traces. 

By the remarkable result in \cite{GPS23,CDRGG24}, Lagrangian cocore disks are generators of the wrapped Fukaya category of $W$. The following theorem shows that their trace also exhibits a rigidity property. We refer to Section \ref{sec:polarized} for the proof.

\begin{theorem}\label{thm:lag_trace}
	Let $(X,\Omega,\Sigma)$ be a polarized symplectic manifold. Let $L$ be a closed monotone Lagrangian submanifold of $(\Sigma,\omega)$ with $N_L\geq 3$. Assume that $L$ admits a $d$-fold Legendrian lift in $Y$. If $L$ is disjoint from the Lagrangian trace, then for any $\Pin$-structure $\sigma$ on $L$, $\QH_*(L^\sigma)$ is torsion of order $d$ and $\QH_*(L^\sigma;\mathbb{Q})=0$. 
\end{theorem}

If $(W,\lambda,\varphi)$ is subcritical,  the Lagrangian trace is empty. It is proved in \cite[Corollary 1.4]{BKK24} that, under some additional assumptions on $W$,  subcriticality implies that $k[\Omega]_\Z|_\Sigma$ is invertible in $\QH^*(\Sigma)$ (with $k=1$). Thus, in this setting, Theorem \ref{thm:lag_trace} recovers Theorem \ref{thm:torsion}.

\begin{remark}
 	A key technical lemma used in the proofs of Theorems \ref{thm:torsion} and  \ref{thm:lag_trace} might be of independent interest. 
	Let $(Y,\alpha)$ be a $(2n+1)$-dimensional closed contact manifold, which is not necessarily a prequantization space. Let $\mathcal{L}$ be a closed Legendrian submanifold of $Y$ equipped with a $\Pin$-structure $\tilde\sigma$. Assume that every contractible periodic Reeb orbit $\gamma$ of $Y$ and every contractible Reeb chord $c$ of $(Y,\mathcal{L})$ is nondegenerate and satisfies 
		$\mu_\CZ(\gamma)\geq 4-n$ and $\mu_\RS(c)\geq 3-\frac{n}{2}$.
		This index condition corresponds precisely to our assumption $N_L\geq3$ in the case of a prequantization bundle. 
		Then $\RFH_*(Y)$ defined via contractible periodic Reeb orbits and $\RFH_*(\mathcal{L}^{\tilde\sigma})$ defined via  contractible Reeb chords are well-defined without any reference to fillings. Now suppose  $\RFH_*(Y)=0$. In the presence of exact (Lagrangian) fillings for $Y$ and $\mathcal{L}^{\tilde\sigma}$, or under a sufficiently strong index condition, the closed-open map would imply $\RFH_*(\mathcal{L}^{\tilde\sigma})=0$. This vanishing does not necessarily hold under the current index condition. However, we show that the distinguished class $1_\RFH\in \RFH_n(\mathcal{L}^{\tilde\sigma})$, which is the image of the unit of $\H^*(\mathcal{L})$ under the continuation map, vanishes. See Appendix \ref{sec:appendix} for details.
\end{remark}

\subsection{Intersections under Legendrian deformations}	
A typical consequence of nonvanishing of $\RFH_*(\mathcal{L};\Z_2)$ is that every compactly supported Hamiltonian diffeomorphism on $\R\times Y$ admits a Lagrangian leafwise intersection point for $\R\times \mathcal{L}$, see  \cite{AF10, Mer14}. Applying this  
to a Hamiltonian diffeomorphism arising as the lift of a contactomorphism on $Y$, we obtain the following Lagrangian intersection property for a class of perturbations more general than Hamiltonian diffeomorphisms, as investigated in \cite{EHS95,Ono96}.
 For the sake of completeness, we provide a sketch of the proof and refer to \cite{AM13,Mer14} for further details.  

\begin{cor}\label{cor:intersection}
Let $(\Sigma,L)$ and $(Y,\mathcal{L})$ be as in Theorem \ref{thm:isom_transfer}. 
	Assume that the covering degree of $\mathcal{L}\to L$ is odd and $\QH_*(L;\Z_2)\neq 0$. Then, for any contactomorphism  $\varphi$ on $Y$ isotopic to the identity,  $\pi\circ\varphi(\mathcal{L})$ intersects $L$.
\end{cor}
\begin{proof}
	We take a path $\{\psi_t\}_{0\leq t\leq 1}$ of contactomorphisms from the identity to $\psi:=\varphi^{-1}$. The associated contact Hamiltonian $h_t:Y\to\R$ is defined by $h_t\circ\psi_t:=\alpha(\frac{d}{dt}\psi_t)$. 
	Then, the symplectization $\tilde{\psi}_t$ of $\psi_t$, i.e.,
	\[
	\tilde{\psi}_t:\R\times Y\to \R\times Y,\qquad \tilde{\psi}_t(r,x):=(r\rho_t(x)^{-1},\psi_t(x)),
	\]
	where $\rho_t:Y\to(0,\infty)$ is a smooth function characterized by $\psi_t^*\alpha=\rho_t\alpha$, is generated by the Hamiltonian $rh_t(x)$ for $(r,x)\in\R\times Y$. We consider a variant of Rabinowitz Floer homology associated with $\tilde{\psi}$, suitably cut off near infinity. Its chain-level generators correspond to chords $\tilde\gamma=(a,\gamma):[0,1]\to \R\times Y$ such that 
	\[
	\tilde\gamma(0)\in \{0\}\times \mathcal{L},\qquad \tilde\gamma(1)\in \R\times \mathcal{L},\qquad  [\tilde\gamma]=1 \;\textrm{in}\;\pi_1(\R\times Y,\R\times \mathcal{L}),
	\]
	and $\tilde\gamma|_{[0,1/2]}$ and $\tilde\gamma|_{[1/2,1]}$ agree with $\{(0,\phi_R^{\tau t}(\gamma(0)))\}_{0\leq t\leq 1}$ for some $\tau\in\R$ and $\{\tilde\psi_t(\tilde\gamma(1/2))\}_{0\leq t\leq 1}$, respectively, up to reparametrization. Here $\phi_R^t$ denotes the Reeb flow on $Y$. This homology, denoted by  $\RFH_*(\mathcal{L},\tilde\psi;\Z_2)$, is defined under a transversality condition and is, as usual in Floer theory, isomorphic to $\RFH_*(\mathcal{L};\Z_2)$.

	Note that if $\tilde\gamma$ is a chain-level generator of $\RFH_*(\mathcal{L},\tilde\psi;\Z_2)$, then $\pi\circ\gamma(0)\in (\pi\circ\varphi(\mathcal{L}))\cap L$. 
	Assume for contradiction that $\pi\circ\varphi(\mathcal{L})$ does not intersect $L$. Then, $\RFH_*(\mathcal{L},\tilde\psi;\Z_2)$ vanishes since its chain group is trivial. This contradicts the fact that $\RFH_*(\mathcal{L};\Z_2)\neq0$, which is a consequence of our assumption and Corollary \ref{cor:transfer}. This finishes the proof.
\end{proof}

\begin{remark}
	The above proof recovers the multiplicity results established in \cite{EHS95,Ono96}: if $L$ admits a bijective Legendrian lift $\mathcal{L}$ and $\pi_2(\Sigma,L)=0$, then $(\pi\circ \psi)(\mathcal{L})$ intersects $L$ at least $\operatorname{rank} \H_*(L;\Z_2)$ times, provided the intersections are transverse. This follows from the fact that $\RFH_*(\mathcal{L},\tilde\psi;\Z_2)\cong \RFH_*(\mathcal{L};\Z_2) \cong \H_*(L;\Z_2)$ as mentioned in Remark \ref{rem:om=0}, and the observation that,  in this setting, the chain-level generators correspond to distinct intersection points. 
\end{remark}

\begin{cor}\label{cor:chord}
Let $(\Sigma,L)$ and $(Y,\mathcal{L})$ be as in Theorem \ref{thm:isom_transfer}. 
	Assume that the covering degree of $\mathcal{L}\to L$ is odd and $\QH_*(L;\Z_2)\neq 0$. Then, for any $\mathcal{L}'$ Legendrian isotopic to $\mathcal{L}$ and for any contact form supporting the same contact structure as $\alpha$, there exists a Reeb chord between $\mathcal{L}$ and $\mathcal{L}'$. 
\end{cor}

The above corollary follows from the proof Corollary  \ref{cor:intersection}, together with the invariance property of Rabinowitz Floer homology. We refer to \cite{AF09,DRS24} for related quantitative results. 
By working with orientations, the conclusions of Corollaries \ref{cor:intersection} and \ref{cor:chord} remain true for Lagrangians $L$ with the degree $\mathcal{L}\to L$ being  arbitrary, provided that they admit relatively $\Pin$-structures $\sigma$ and satisfy $\mathrm{rank}_\Z\,\QH_*(L^\sigma)\neq 0$.

\subsection{Topologically simple exact Lagrangian fillings}\label{sec:filling}

An exact (or Liouville) filling $W$ of $Y$ is said to be topologically simple if the inclusion-induced map $\pi_1(Y)\to\pi_1(W)$ is injective and the first Chern class $c_1^{TW}$ of $TW\to W$ vanishes on $\pi_2(W)$. Similarly, an exact Lagrangian filling $\mathcal{L}_W\subset W$ is said to be topologically simple if the inclusion-induced map $\pi_1(Y,\mathcal{L})\to \pi_1(W,\mathcal{L}_W)$ is injective and the Maslov class $\mu_{\mathcal{L}_W}$ vanishes on $\pi_2(W,\mathcal{L}_W)$.
 In the presence of such fillings, Theorem \ref{thm:isom_transfer} remains valid even when condition (C2) is weakened to $N_L\geq2$. 
 We state its simplest form below and refer to Proposition \ref{prop:RFH_well_defined} for the general version.
 
\begin{prop}\label{prop:filling}
	Let $(\Sigma,L)$ and $(Y,\mathcal{L})$ be as in Theorem \ref{thm:isom_transfer}, with condition (C2) weakened to $N_L\geq2$.  
	Assume also that $Y$ and $\mathcal{L}$ admit topologically simple exact (Lagrangian) fillings. Then, $\RFH_*(\mathcal{L};\Z_2)$ and $\RFH_*^{\Z_d}(\mathcal{L};\Z_2)$ are well-defined as $\Z_2$-modules. Moreover, $\RFH_*^{\Z_d}(\mathcal{L};\Z_2)\cong \QH_*(L;\Z_2)$ holds. 
\end{prop}

The fillings $W$ and $\mathcal{L}_W$ yield an augmentation for the Rabinowitz Floer homology of $\mathcal{L}$, which is trivial due to the topological simplicity condition. Recently, \cite{BCSW24} proved that the augmentation variety of $\mathcal{L}$ corresponds to the zero set of the disk potential of $L$, which settles \cite[Conjecture 9.1]{DRG23}. This illuminating result suggests that our results might be extended to more general augmentations.

We apply Proposition \ref{prop:filling} to derive an obstruction to topologically simple exact Lagrangian fillings, provided  that the Rabinowitz Floer homology of $(W,\partial W=Y)$ generated by contractible periodic orbits vanishes. This vanishing property holds, for instance, under the additional assumption that $[\omega]_{\Z_2}$ is invertible and the minimal Chern number $c_\Sigma$ of $\Sigma$ is at least $2$. The following results are proved in Section \ref{sec:filling_proof}.

\begin{theorem}
	Let $(\Sigma,L)$ and $(Y,\mathcal{L})$ be as in Theorem \ref{thm:isom_transfer}, but assume  $N_L=2$ instead of condition (C2). Assume also that  $\QH_*(L;\Z_2)\neq 0$, $c_\Sigma\geq 2$, and $[\omega]_{\mathbb{Z}_2}$ is invertible in $\QH^*(\Sigma;\mathbb{Z}_2)$. Let $W$ be a topologically simple exact filling of $Y$.  If the covering degree  of $\mathcal{L}\to L$ is odd, then 
	$\mathcal{L}$ does not admit any topologically simple exact Lagrangian filling in $W$.
\end{theorem}

Note that by Theorem \ref{thm:torsion}, if $N_L=2$ is replaced by $N_L>2$ in the theorem, the covering degree is necessarily even. We refer to Section \ref{sec:filling_proof} for an analogous statement concerning the case where the covering degree of $\mathcal{L}\to L$ is even.

\begin{cor}
	\label{cor:nonexistence_filling_CP^2n_intro}
	Let $L$ be a closed monotone Lagrangian submanifold of $\mathbb{CP}^{2n}$ with the Fubini-Study form. Assume that $N_L=2$ and $\QH_*(L;\Z_2)\neq 0$. Then the Legendrian lift $\mathcal{L}\subset S^{4n+1}$ of $L$ does not admit any exact Lagrangian filling with vanishing Maslov class inside the ball $B^{4n+2}$.
\end{cor}

The case of Clifford tori is well-understood. It was established in \cite{DR11,TZ18} that the lift of the Clifford torus in $\mathbb{CP}^2$ admits no exact Lagrangian filling in $B^{6}$. Recently, \cite{BCSW24} generalized this nonexistence result to all $\mathbb{CP}^n$ for $n\geq 2$.
It seems possible to extend the above theorem and corollary  by investigating augmentations induced by Lagrangian fillings that are not necessarily topologically simple. This will be pursued elsewhere.

\paragraph{Acknowledgments} 
J Kang was supported by National Research Foundation of Korea grants NRF-2020R1A5A1016126, RS-2023-00211186, and RS-2025-02217108. S Kim was partially supported by National Key R$\&$D Program of China No.~2023YFA1010500 and NSFC No.~12511540054. J Kang is grateful to Georgios Dimitroglou Rizell and Zhengyi Zhou for valuable comments and discussions.

\section{Lagrangian quantum homology}
Let $(\Sigma,\omega)$ be a closed connected symplectic manifold. Let $L$ be a closed connected monotone Lagrangian submanifold. Then $(\Sigma,\omega)$ is also monotone, and the monotonicity constant $\tau_\Sigma>0$ satisfies
\[
c_1^{T\Sigma} = \tau_\Sigma\omega \;\text { on }\; \pi_2(\Sigma),\qquad \mu_L =2 \tau_\Sigma \omega \;\text { on }\; \pi_2(\Sigma,L),
\]
where $c_1^{T\Sigma}$ and $\mu_L$ denote the first Chern class of $\Sigma$ and the Maslov class of $L$, respectively. In this section, we assume that the minimal Maslov number $N_L$ of $L$ satisfies $N_L\geq 2$, where $N_L \in \N=\{1,2,\dots\}$ is defined by $\mu_L(\pi_2(\Sigma,L))=N_L\Z$ if $\mu_L(\pi_2(\Sigma,L))\neq 0$, and we set $N_L=\infty$ if $\mu_L(\pi_2(\Sigma,L))=0$.

This section recalls the construction of the Lagrangian quantum homology ring of $L$ equipped with a relative $\Pin$-structure $\sigma$, following \cite{BC4,Z1}.

\begin{remark}\label{rem:pin}(Relative $\Pin$-structure)
	The pin groups $\mathrm{Pin}^+(n)$ and $\mathrm{Pin}^-(n)$ of $\R^n$ are two different central extensions of the orthogonal group $O(n)$ by $\Z_2=\{-1,+1\}$,
	\[
		1 \to \Z_2 \to \Pin(n) \stackrel{\wp}{\to} \mathrm{O}(n) \to 1.
	\]
	Let $Q$ be a smooth $n$-dimensional submanifold of a smooth manifold $M$. 
	Let $\mathcal{V}$ be a good cover of $M$, and let $\mathcal{U}$ be a good cover of $Q$, which is a refinement of $\mathcal{V}$ restricted to $Q$. 
	We fix a Riemannian metric on $Q$. The frame bundle of the tangent bundle $TQ$ of $Q$ is an $O(n)$-principal bundle, and this defines a cocycle $h\in \check{\mathrm{C}}^{1}(\mathcal{U};\mathrm{O}(n))$ in the $\mathrm{\check{\mathrm{C}}}$ech cochain complex. 
	It admits a lift $g\in \check{\mathrm{C}}^1(\mathcal{U};\Pin(n))$, i.e.,~$\wp(g)=h$, and $\wp( \delta g)=0$ holds since $h$ is a cocycle, where $\delta$ denotes the coboundary operator. This yields $\delta g\in \check{\mathrm{C}}^{2}( \mathcal{U};\Z_2)$, where $\Z_2\cong \ker\wp$. The cohomology class $w^\pm(Q):=[\delta g]\in\check{\H}^2(Q;\Z_2)$ satisfies 
	\[
	w^+(Q)=w_2^{TQ},\qquad w^-(Q)=w_2^{TQ}+(w_1^{TQ})^2
	\]
	where $w_1^{TQ}$ and $w_2^{TQ}$ denote the first and second Stiefel--Whitney classes of $TQ$, respectively. The submanifold $Q$ is said to be relatively $\Pin$ if $w^\pm(Q)$ is in the image of the map $\H^2(M;\Z_2)\to \H^2(Q;\Z_2)$ induced by the inclusion $Q\hookrightarrow M$.

	A relative $\Pin$-structure on $Q$ is an equivalence class of a pair 
	\[
	(g,\beta)\in \check{\mathrm{C}}^1(\mathcal{U};\Pin(n))\times \check{\mathrm{Z}}^2(\mathcal{V};\Z_2),
	\]
	where $\check{\mathrm{Z}}^2(\mathcal{V};\Z_2)=\ker \delta|_{\check{\mathrm{C}}^2(\mathcal{V};\Z_2)}$, 
	such that $\wp(g)=h$ and $\beta|_Q=\delta g$. Here, two pairs $(g_1,\beta_1)$ and $(g_2,\beta_2)$ are said to be equivalent if there is $(a,b)\in \check{\mathrm{C}}^0(\mathcal{U};\Z_2)\times \check{\mathrm{C}}^1(\mathcal{V};\Z_2)$ satisfying $(\delta a\cdot g_1,\delta b\cdot \beta_1)=(g_2,\beta_2)$. 
	When $(g,\beta)$ represents a relative $\Pin$-structure $\sigma$ on $Q$, we call $[\beta] \in \check{\H}^2(\mathcal{V};\Z_2) \cong \H^2(M;\Z_2)$ the background class of $\sigma$. 
\end{remark}

\subsection{Lagrangian quantum complex}\label{sec:QC}
For a smooth Morse function $f_L$ on $L$, we denote by $\Crit f_L$ the set of critical points of $f_L$.
Let $Z_L$ be a negative gradient-like vector field for $f_L$ on $L$. 
Denoting the time-$t$ flow of $Z_L$ by $\varphi_{Z_L}^{t}$, we define the unstable and stable manifolds of $p\in \Crit f_L$ with respect to $Z_L$ by
\[
	W^u_{Z_L}(p):=\{x \in L \mid \lim_{t\to -\infty} \varphi_{Z_L}^{t}(x)=p\},\qquad 
	W^s_{Z_L}(p):=\{x \in L \mid \lim_{t\to +\infty} \varphi_{Z_L}^{t}(x)=p\},
\]
respectively. 
The Morse index $\ind_{f_L}(p)$ of $p$ equals $\dim W^u_{Z_L}(p)$. We assume that $(f_L,Z_L)$ is a Morse-Smale pair, i.e.,~$W^u_{Z_L}(p)\pitchfork W^s_{Z_L}(q)$ for all $p,q\in \Crit f_L$.

Let $\mathcal{J}_{\Sigma}$ be the space of $\omega$-compatible almost complex structures on $\Sigma$. 
To discuss $J_\Sigma$-holomorphic disks for  $J_\Sigma \in \mathcal{J}_\Sigma$, we write $\D:= \{z \in \C \mid |z| \leq 1\}$. 
\begin{definition}\label{def:pearl_Lag}
	For $N \in \mathbb{N}$, let ${\bf A} = (A_1, \dots, A_N) \in (\pi_2(\Sigma,L))^N$. Let 
	\[
	\mathcal{N}_{N}({\bf A};J_{\Sigma})=\big\{ {\bf w}= (w_1, \dots, w_N)\big\}
	\]
	be the moduli space of $N$-tuples of $J_{\Sigma}$-holomorphic disks $w_i : (\D,\partial \D) \to (\Sigma,L)$  representing $A_i\in \pi_2(\Sigma,L)$ for $1\leq i\leq N$.
	We denote by $\mathcal{N}_{N}^{*}({\bf A}; J_{\Sigma}) $ the subspace of $\mathcal{N}_{N}({\bf A}; J_{\Sigma})$ consisting of simple elements, i.e.,~each $w_i$ is either somewhere injective or constant, and $w_i(\D) \not\subset \bigcup\limits_{j\neq i} w_j(\D)$ for every $1\leq i\leq N$ for which $w_i$ is nonconstant.
\end{definition}

Let us abbreviate $\mathcal{D} = (f_L,Z_L,J_\Sigma)$. We consider the evaluation map 
\[
	\ev : \mathcal{N}_{N}({\bf A};J_{\Sigma}) \to L^{2N} ,\quad 
	\mathbf{w} \mapsto \left(w_1(-1),w_1(1), \dots, w_N(-1),w_N(1)\right).
\]
For $p,q\in \Crit f_L$, we define the moduli space of chains of pearls from $q$ to $p$ by
\begin{equation*}
	\mathcal{N}_{N}(q,p;\mathbf{A};\mathcal{D}):= \left(\ev\right)^{-1}\big(W^{u}_{Z_L}(q) \times \Delta_{Z_L}^{N-1} \times W^{s}_{Z_L}(p)\big),
\end{equation*} 
where
\[\Delta_{Z_L} := \{(x,\varphi_{Z_L}^t(x)) \in L \times L \mid x\in L \setminus \Crit f_L, \, t \in \R_{>0}\}.\] 
For $N=1$, this means $\mathcal{N}_{N=1}(q,p;\mathbf{A};\mathcal{D}):=(\ev)^{-1}(W^{u}_{Z_L}(q) \times W^{s}_{Z_L}(p))$. We further define the subspace of simple chains of pearls from $q$ to $p$ by
\[
\mathcal{N}^{*}_{N}(q,p;\mathbf{A};\mathcal{D}):= \mathcal{N}_{N} (q,p;{\bf A};\mathcal{D}) \, \cap\, \mathcal{N}_{N}^{*}({\bf A};\mathcal{D}).
\]

\begin{remark}
	To define the quantum homology of $L$, it suffices to consider $\mathbf{A}$ with all $A_i$ nonzero. We include in the above definition the case where some of the classes $A_i$ vanish, as this will be relevant in Section \ref{sec:proof_proposition}.
\end{remark}

From now on, we assume $\mathbf{A}\in (\pi_2(\Sigma,L)\setminus\{0\})^N$.
Then, there is a free $\R^{N}$-action on $\mathcal{N}_{N}(q,p;\mathbf{A};\mathcal{D})$ given by reparametrizing each $w_i$ by elements of $\mathrm{Aut}(\D;\pm 1)$, the group of biholomorphisms of $\D$ fixing $\pm 1$.
We denote
\[
		\overline{\mathcal{N}}_{N}(q,p;\mathbf{A};\mathcal{D}):=\mathcal{N}_{N}(q,p;\mathbf{A};\mathcal{D})/\R^{N},\qquad \overline{\mathcal{N}}^{*}_{N}(q,p;\mathbf{A};\mathcal{D}):=\mathcal{N}^{*}_{N}(q,p;\mathbf{A};\mathcal{D})/\R^{N}.	
\]For the case $N=0$, we set 
\[
\mathcal{N}_{N=0}(q,p):=W^{u}_{Z_L} (q) \cap W^{s}_{Z_L}(p),\qquad \overline{\mathcal{N}}_{N=0}(q,p):=\mathcal{N}_{N=0}(q,p)/\R,
\]
where the $\R$-action is given by the flow $\varphi_{Z_L}^{t}$.

\begin{prop}(\cite[Section 3.1.2]{BC4})\label{prop:base_transv}
	There is a residual set $\mathcal{J}_{(\Sigma,L)}^{\mathrm{reg}} \subset \mathcal{J}_{\Sigma}$ such that the following hold for every $J_{\Sigma} \in \mathcal{J}_{(\Sigma,L)}^{\mathrm{reg}}$, $p,q \in \Crit f_L$, $N\in\N$, and $\mathbf{A} \in (\pi_2(\Sigma,L)\setminus\{0\})^N$.
	\begin{enumerate}[(a)]
		\item 
		The moduli space $\overline{\mathcal{N}}^{*}_{N} (q,p;{\bf A}; \mathcal{D})$ is a smooth manifold of dimension
		\[ \mathrm{ind}_{f_L}(q)- \mathrm{ind}_{f_L}(p) -1+\mu_L(\mathbf{A}), \quad \text{where }\;\mu_L(\mathbf{A}):=\sum\limits_{i=1}^{N}\mu_L(A_i).\]
		\item 
		If $\mathrm{ind}_{f_L}(q)- \mathrm{ind}_{f_L}(p) + \mu_L(\mathbf{A}) \leq 2$, then every element in $\overline{\mathcal{N}}_{N}(q,p;{\bf A}; \mathcal{D})$ is simple, i.e., \[\overline{\mathcal{N}}_{N}(q,p;{\bf A}; \mathcal{D})=\overline{\mathcal{N}}^{*}_{N} (q,p;{\bf A}; \mathcal{D}).\] 
		Moreover, if $\mathrm{ind}_{f_L}(q)- \mathrm{ind}_{f_L}(p) +\mu_L(\mathbf{A})=1$, then this moduli space is compact. 
	\end{enumerate}
\end{prop}

Next, following \cite[Section 4 and Section 7]{Z1}, we define   the Lagrangian quantum complex of $L$ using canonical orientations. We assume that $L$ is relatively $\Pin$ and fix a relative $\Pin$-structure $\sigma$ on $L$. 

Let $J_\Sigma \in \mathcal{J}_{(\Sigma,L)}^{\mathrm{reg}}$.
Let $p \in \Crit f_L$ and $w\in C^\infty((\D,\partial \D,{-1}) , (\Sigma,L,p))$. 
For any connection $\nabla$ on $T\Sigma$, we have the pullback bundle  $(w^*T\Sigma, (w|_{\partial \D})^*TL)$ equipped with $J_{w}:=w^* J_\Sigma$ and $\nabla_{w}:=w^* \nabla$.
For $p>2$, we consider the Cauchy--Riemann operator associated with $J_w$:
\begin{equation}\label{eq:CR_base}
	D_w=\nabla^{0,1}_{w}: W^{1,p}((\D,\partial\D),(w^*T\Sigma,w|_{\partial \D}^*TL)) \to L^p(\D,\Omega^{0,1}_{\D}\otimes w^*T\Sigma),
\end{equation}
which is a Fredholm operator.
There is a well-defined evaluation map
\[
\ev_{w,{-1}} : W^{1,p}((\D,\partial\D),(w^*T\Sigma,w|_{\partial \D}^*TL)) \to T_p L, \qquad \xi \mapsto \xi({-1}).
\]
For a subspace $V$ of $T_p L$, we define the restriction 
\[
D_w \# V := D_w |_{\left(\ev_{w,{-1}}\right)^{-1}(V)},
\]
which is again a Fredholm operator. The orientation line $\mathfrak{o}{(p,w)}$ is defined as the free $\Z$-module of rank $1$ generated by the two orientations of $\det(D_w \# T_pW^u_{Z_L}(p))$ modulo the relation that they sum up to zero. 
Furthermore, for any $w, w'\in C^\infty((\D,\partial \D,{-1}) , (\Sigma,L,p))$ with $\mu_L(w) = \mu_L(w')$, we have a canonical isomorphism $\mathfrak{o} (p,w)\cong \mathfrak{o} (p,w')$ determined by the chosen relative $\Pin$-structure $\sigma$. See Remark \ref{rem:ori_iso} below. For $(p,k)\in\Crit f_L\times \Z$, we set 
\begin{equation}\label{eq:generator_base}
	\mathfrak{o}(p,k) := \varinjlim\limits_{w \in \mu_L^{-1}(kN_L)} \mathfrak{o}(p,w),\qquad \deg \mathfrak{o}(p,k) := \dim L - \ind_{f_L}(p)-kN_L.
\end{equation}

\begin{remark}\label{rem:ori_iso}(Isomorphism between orientation lines \cite{Z1})
\begin{enumerate}[(i)]
	\item For $p\in L$ and $A \in \pi_2(\Sigma,L)$, we consider the following connected space:
	\[
	C^\infty_{A}:= \{ w \in C^\infty((\D,\partial \D,{-1}) , (\Sigma,L,p)) \mid  w \text{ represents }{A}\}.
	\]
	Then, we have the families of Fredholm operators
	\[
	D_A := (D_w)_{w\in C^\infty_A},\qquad 
	D_A\#V:= (D_w \# V)_{w \in C^\infty_A},
	\]
	whose determinant line bundles $\det(D_A)$ and $\det(D_A\#V)$ are orientable since $L$ is relatively $\Pin$, see \cite[Lemmas 4.1 and 4.2]{Z1}. Thus, there is a canonical isomorphism $\mathfrak{o} (p,w)\cong \mathfrak{o} (p,w')$ if $w$ and $w'$ represent the same homotopy class in $\pi_2(\Sigma,L)$.
	\item Let $w,w'\in C^\infty((\D,\partial \D,{-1}) , (\Sigma,L,p))$ satisfy $\mu_L(w) = \mu_L(w')$. The glued disk $(-w)\#w'$, where $-w$ denotes $w$ with the reversed orientation, satisfies $\mu_L((-w)\#w')=0$. Therefore, the orientation of $\det(D_{(-w)\#w'}\#0)$ is determined by the relative $\Pin$-structure $\sigma$ as we recall in the proof of Proposition \ref{prop:orientation_commute}; see  \cite[Proposition 7.4]{Z1} for a detailed account. 
	Using the gluing isomorphism together with a deformation of the incidence condition, we obtain an isomorphism 
	\[
	\det(D_w \#T_pW^u_{Z_L}(p)) \otimes \det(D_{(-w)\#w'}\#0)  \cong \det(D_{w'} \#T_pW^u_{Z_L}(p)).
	\]
	Since the orientation of $\det(D_{(-w)\#w'}\#0)$ is determined by $\sigma$, this yields an isomorphism
	\begin{equation}\label{eq:pin_iso}
		\psi^{\sigma}_{(w,w')} : \det(D_w \#T_pW^u_{Z_L}(p)) \longrightarrow \det(D_{w'} \#T_pW^u_{Z_L}(p)),
	\end{equation}
	and hence an isomorphism $\mathfrak{o}(p,w)\to \mathfrak{o}(p,w')$.
	\item The orientation of $\det (D_{(-w)\#w'}\#0)$ determined by $\sigma$ corresponds to an element of $\det (D_{(-w)\# w'}\#0)$ up to a positive multiple. 
	 Accordingly, $\psi^{\sigma}_{(w,w')}$ is  well-defined only up to a positive multiple. We shall henceforth consider all maps and commutative diagrams between determinant lines to be defined and understood up to a positive multiple.
\end{enumerate}
\end{remark}

We consider $\mathbf{w}=(w_1,\dots,w_N) \in \mathcal{N}_{N}(q,p;\mathbf{A};\mathcal{D})$ with $\ind_{f_L} (q)- \ind_{f_L} (p) +\mu_L(\mathbf{A})=1$.
Let $w \in C^\infty((\D,\partial \D,{-1}) , (\Sigma,L,p))$. We glue the disks ${w}_1, \dots, {w}_N, w$ together with the integral curves of $Z_L$ connecting $q,w_1,\dots,w_N,p$ to obtain $v \in C^\infty((\D,\partial \D,{-1}) , (\Sigma,L,q))$. We have 
\[ 
\mu_L([v]) = \mu_L(\mathbf{A})+\mu_L([w]).
\]
As established in \cite[Section 4.2.2]{Z1}, we have an isomorphism 
\[
\det(D_w \#T_pW^u_{Z_L}(p))\cong \det(D_v \# T_qW^u_{Z_L}(q)),
\] 
which is obtained by gluing the linearized operators associated to $w_1, \dots, w_N, w$ and deforming the incidence conditions.  This induces an isomorphism between the corresponding orientation lines $\mathfrak{o}(p,w)\cong \mathfrak{o}(q,v)$. Furthermore, this isomorphism commutes with the canonical isomorphisms in Remark \ref{rem:ori_iso}.(ii), which yields an isomorphism
\begin{equation}\label{eq:pearl_bdry}
	C([\mathbf{w}]) : \mathfrak{o}(p,k) \longrightarrow \mathfrak{o}(q,l),
\end{equation}
where $k=\mu_L(w)/N_L$ and $l =\mu_L(v)/N_L$. 
As indicated by the notation, reparametrizing $\mathbf{w}$ does not affect the isomorphism. 
For the case $N=0$, an element $[z] \in \overline{\mathcal{N}}_{N=0}(q,p)$ with $\ind_{f_L}(q)-\ind_{f_L}(p)=1$ induces an isomorphism
\[
	C([z]) : \mathfrak{o}(p,k) \longrightarrow \mathfrak{o}(q,k),
\]
as in Morse theory, using the orientation of $T_z \mathcal{N}_{N=0}(q,p)$ given by $Z_L$.

For later purposes, we introduce the following space:~for $(q,l),(p,k)\in \Crit f_L\times \Z$,
\begin{equation}\label{eq:moduli_base}
	\overline{\mathcal{N}}^{*}((q,l),(p,k);\mathcal{D}):=  \begin{cases}
		\overline{\mathcal{N}}_{N=0}(q,p) & \text{if }k=l, \\[1ex]
		\displaystyle\bigcup_{N\in\N}\bigcup_{\mathbf{A}}\, \overline{\mathcal{N}}^{*}_{N}(q,p;\mathbf{A};\mathcal{D}) & \text{if }k\neq l, 
	\end{cases}	
\end{equation}
where the second union ranges over all $\mathbf{A}\in (\pi_2(\Sigma,L)\setminus\{0\})^N$ with $\mu_L(\mathbf{A})=(l-k)N_L$. When $\deg \mathfrak{o}(p,k)-\deg \mathfrak{o}(q,l)=1$, each element $[\mathbf{w}]$ in this moduli space induces an isomorphism $C([\mathbf{w}]):\mathfrak{o}(p,k) \to \mathfrak{o}(q,l)$. Henceforth, we also write $[\mathbf{w}]$ for $[z]$ in the case of $k=l$.

\medskip

Finally, we define the quantum complex associated with $\mathcal{D}=(f_L,Z_L,J_\Sigma)$. To indicate  dependence of the chosen relative $\Pin$-structure $\sigma$, we write  $L^{\sigma}$ for the Lagrangian submanifold $L$ equipped with $\sigma$. The chain module is defined by
\[
\QC_*(L^{\sigma};\mathcal{D}) := \bigoplus_{(p,k)} \mathfrak{o}(p,k),
\] 
where the direct sum ranges over all $(p,k)\in\mathrm{Crit} f_L\times\Z$. The boundary map is defined by 
\begin{equation}\label{eq:differential_base}
	\partial_{\mathcal{D}} : \QC_*(L^{\sigma};\mathcal{D}) \to \QC_{*-1}(L^{\sigma};\mathcal{D}),\qquad \partial_{\mathcal{D}} := \bigoplus_{(p,q,k,l)} \sum_{[\mathbf{w}]} C([\mathbf{w}]).
\end{equation}
The direct sum ranges over all $(p,k),(q,l)\in \Crit f_L\times \Z$  with $\deg \mathfrak{o}(p,k)- \deg \mathfrak{o}(q,l)=1$. The sum runs over all $[\mathbf{w}] \in \overline{\mathcal{N}}^{*}((q,l),(p,k);\mathcal{D})$. The boundary map $\partial_{\mathcal{D}}$ squares to zero, see \cite[Section 4.2.2 and Section 7.3]{Z1}.
The quantum homology of $L^\sigma$ is defined by 
\[ 
\QH_i(L^\sigma):=\H_i(\QC_*(L^{\sigma};\mathcal{D}),\partial_{\mathcal{D}}),
\]
and it is independent of the choice of $\mathcal{D}$. 

\begin{remark}\label{rem:laurent}(Module structure)
	Let us assume that $N_L <\infty$.
	When $N_L$ is even, the chain complex $\QC_*(L^{\sigma})$ admits a module structure over the Laurent polynomial ring $\Z[T,T^{-1}]$, where the formal parameter $T$ has degree $N_L$.
	For any $A,B \in \pi_2(\Sigma,L)$ with $\mu_L(B)=-N_L$, we have an isomorphism 
	\[
	\det(D_A \#T_pW^u_{Z_L}(p)) {\longrightarrow} \det(D_{A\cdot B} \#T_pW^u_{Z_L}(p)) 
	\]
	defined in the same way as the one in \eqref{eq:pin_iso}, which exists under the condition $N_L\in 2\Z$, see \cite[Proposition 7.4]{Z1}. This commutes with the isomorphism in \eqref{eq:pin_iso}, and therefore we obtain an isomorphism $
	T:\mathfrak{o}(p,k) \to \mathfrak{o}(p,k-1)$
	for every $(p,k)\in\Crit f_L\times \Z$. It commutes with the boundary map $\partial_{\mathcal{D}}$, thus induces a $\Z[T,T^{-1}]$-module structure on $\QH_*(L^\sigma)$. 
	
	If $L$ is nonorientable and $N_L$ is odd, there is still a $\Z[T,T^{-1}]$-module structure, but with $\deg T=2N_L$ in general. We refer to \cite[Lemma 14]{KS21b} for a detailed account. 
\end{remark}

\subsection{Product structure of Lagrangian quantum homology}\label{sec:product_QH}

In this section, we provide a brief overview of the quantum product  on $\QH_*(L^\sigma)$; for details, see \cite[Section 5.2]{BC4} and \cite[Section 4.2.3]{Z1}.
Let $\mathcal{D}^P = (f_L,Z^1_L,Z^2_L, Z^3_L,J_{\Sigma})$ be a tuple consisting of a smooth Morse function $f_L$ on $L$, negative gradient-like vector fields $Z^1_L,Z^2_L, Z^3_L$ for $f_L$, each forming  a Morse-Smale pair with $f_L$, and  $J_{\Sigma} \in \mathcal{J}_{\Sigma}$.
We denote
\begin{align}\label{eq:auxiliary_Lag}
	\begin{split}
		&\mathbf{N}=(N_1,N_2,N_3)\in (\N\cup\{0\})^3,\qquad N=N_1+N_2+N_3,  \\[.5ex] 
		&\mathbf{A}=(A_1,\dots, A_{N_1}, A, A_{N_1+1},\dots, A_N) \in (\pi_2(\Sigma,L))^{N+1},\quad
		A_i\neq 0 \;\; \forall i\in\{1,\dots,N\}.
	\end{split}
\end{align}
For $\mathbf{w}=(w_1,\dots,w_{N_1}, w, w_{N_1+1},\dots, w_{N})\in \mathcal{N}_{N+1} (\mathbf{A}; J_\Sigma)$,
we define 
\begin{align*}
	&\ev_P : \mathcal{N}_{N+1} (\mathbf{A}; J_\Sigma) \longrightarrow L^{2N+3} \\[.5ex]
	&\ev_P(\mathbf{w}) := \big( w_1(-1), w_1 (1), \dots, w_{N_1} (-1), w_{N_1} (1), w(-1),\\[.5ex] 
	&\hspace{2.2cm} w(e^{-\frac{\pi i}{3}}), w_{N_1+1}(-1), w_{N_1+1}(1),\dots, w_{N_1+N_2}(-1), w_{N_1+N_2} (1), \\[.5ex]
	& \hspace{2.2cm} w(e^{\frac{\pi i}{3}}), w_{N_1+N_2+1}(-1), w_{N_1+N_2+1}(1),\dots , w_{N}(-1), w_{N}(1) \big). 
\end{align*}  
\begin{definition}\label{def:pearl_Sigma}
	For $p, q, r \in \Crit f_L$, 
	we define the moduli space
	\begin{equation*}
		\mathcal{N}_{\mathbf{N}}(r,p,q ; \mathbf{A} ; \mathcal{D}^P) := (\ev_P)^{-1} \Big(  W^u_{Z^1_L}(r) \times  \Delta_{Z^1_L} ^{N_1}\times \Delta_{Z^2_L}^{ N_2}\times W^s_{Z^2_L}(p) \times  \Delta_{Z^3_L}^{N_3} \times W^s_{Z^3_L}(q) \Big) ,
	\end{equation*}
	where 
	\[
	\Delta_{Z^i_L}:= \big \{ (x,\varphi_{Z^i_L}^t(x)) \in L \times L \mid x \in (L \setminus \Crit f_L), t\in \R_{>0}  \big \}\quad  \text{for }\, i=1,2,3.
	\]
	We also consider the subspace $\mathcal{N}^{*}_{\mathbf{N}}(r,p,q ; \mathbf{A} ; \mathcal{D}^P)$ consisting of $\mathbf{w}$ such that $(w_1,\dots,w_{N_1},w)$,
		$(w, w_{N_1+1},\dots,w_{N_1+N_2})$, and $(w, w_{N_1+N_2+1},\dots,w_{N})$ are simple in the sense of Definition \ref{def:pearl_Lag}. 
\end{definition}

There is a free $\R^{N}$-action on $\mathcal{N}_{\mathbf{N}}(r,p,q; \mathbf{A} ;\mathcal{D}^P)$ and $\mathcal{N}_{\mathbf{N}}^*(r,p,q; \mathbf{A} ;\mathcal{D}^P)$  given by reparametrizing each nonconstant disk $w_i$.
We denote the respective quotient spaces by 
\[
	\overline{\mathcal{N}}_{\mathbf{N}}(r,p,q;\mathbf{A};\mathcal{D}^P),\qquad 
	\overline{\mathcal{N}}^{*}_{\mathbf{N}}(r,p,q;\mathbf{A};\mathcal{D}^P).
\]
In the case $N=0$, where no such action exists, we set $\overline{\mathcal{N}}_{\mathbf{N}} = {\mathcal{N}}_{\mathbf{N}}$ and $\overline{\mathcal{N}}^{*}_{\mathbf{N}} = {\mathcal{N}}^{*}_{\mathbf{N}}$.

\begin{prop}(\cite[Section 5.2]{BC4} and \cite[Section 3.3]{BC6})\label{prop:transv_base_triangle}
	There is a residual set $\mathcal{J}_{(\Sigma,L)}^{\mathrm{reg},P} \subset \mathcal{J}_{\Sigma}$ such that every $J_{\Sigma} \in \mathcal{J}_{(\Sigma,L)}^{\mathrm{reg},P}$ 
	satisfies the following properties for a generic tuple  $(Z^1_{L},Z^2_{L},Z^3_{L})$ and for every $p,q,r \in \Crit f_L$, $\mathbf{N}$, and $\mathbf{A}$.
	\begin{enumerate}[(a)]
		\item 
		The moduli space $\overline{\mathcal{N}}^{*}_{\mathbf{N}} (r,p,q;{\bf A}; \mathcal{D}^P)$ is a smooth manifold of dimension
		\[ 
		\mathrm{ind}_{f_L}(r)- \mathrm{ind}_{f_L}(p)-\mathrm{ind}_{f_L}(q) +\mu_L(\mathbf{A}),\quad {where }\; \mu_L(\mathbf{A}):=\mu_L(A)+\sum\limits_{i=1}^{N}\mu_L(A_i).
		\] 
		\item If $\mathrm{ind}_{f_L}(r)- \mathrm{ind}_{f_L}(p)-\mathrm{ind}_{f_L}(q) +\mu_L(\mathbf{A}) \leq 1$, then every element in $\overline{\mathcal{N}}_{\mathbf{N}}(r,p,q;{\bf A}; \mathcal{D}^P)$ is simple, i.e., \[\overline{\mathcal{N}}_{\mathbf{N}}(r,p,q;{\bf A}; \mathcal{D}^P)=\overline{\mathcal{N}}^{*}_{\mathbf{N}} (r,p,q;{\bf A}; \mathcal{D}^P).\] 
		Moreover, if $\mathrm{ind}_{f_L}(r)- \mathrm{ind}_{f_L}(p)-\mathrm{ind}_{f_L}(q) +\mu_L(\mathbf{A})=0$,  it is compact.
	\end{enumerate}
\end{prop}

Let $J_\Sigma \in \mathcal{J}^{\mathrm{reg}}_{(\Sigma,L)} \cap \mathcal{J}^{\mathrm{reg},P}_{(\Sigma,L)}$. For generic $(Z^1_L, Z^2_L, Z^3_L)$, denote $\mathcal{D}_i=(f_L, Z^i_L, J_\Sigma)$ for $i=1,2,3$. An element  $[\mathbf{w}]\in \overline{\mathcal{N}}^*_{\mathbf{N}} (r,p,q;\mathbf{A};\mathcal{D}^P)$ with 
$\ind_{f_L}(r)-\ind_{f_L}(p)-\ind_{f_L}(q)+\mu_L(\mathbf{A})=0$
 induces an isomorphism 
\[
C([\mathbf{w}]) : \mathfrak{o}(p,k) \otimes \mathfrak{o}(q,k') \longrightarrow \mathfrak{o}(r,k''),
\]
via an argument analogous to that used for the map in \eqref{eq:pearl_bdry}. Here $k''=k+k'+\mu_L(\mathbf{A})/N_L$.
At the chain level, the quantum product is defined by
\[
	\star_{\mathcal{D}_1,\mathcal{D}_2,\mathcal{D}_3}:= \bigoplus \sum C([\mathbf{w}]): \QC_l(L^{\sigma};\mathcal{D}_2) \otimes \QC_m(L^{\sigma};\mathcal{D}_3) \to \QC_{l+m-\dim L}(L^{\sigma};\mathcal{D}_1)
\]
The direct sum runs over all $(p,q,r,k,k',k'') \in (\Crit f_L)^3\times\Z^3$ with $\deg\mathfrak{o}(p,k)+\deg\mathfrak{o}(q,k')-\deg\mathfrak{o}(r,k'')=\dim L$. 
The sum ranges over $[\mathbf{w}] \in \overline{\mathcal{N}}^*_{\mathbf{N}} (r,p,q;\mathbf{A};\mathcal{D}^P) $ for all  $\mathbf{N}$ and $\mathbf{A}$ satisfying $\mu_L(\mathbf{A})=N_L(k''-k-k')$. 
The map $\star_{\mathcal{D}_1,\mathcal{D}_2,\mathcal{D}_3}$ is a chain map and thus induces
\[ 
\star :  \QH_l(L^\sigma) \otimes \QH_m(L^\sigma) \longrightarrow \QH_{l+m-\dim L}(L^\sigma).
\]
This quantum product structure is associative and unital. Moreover, it is compatible with the $\Z[T,T^{-1}]$-module structure described in Remark \ref{rem:laurent}, and thus endows $\QH_*(L^\sigma)$ with the structure of a $\Z[T,T^{-1}]$-algebra. Note that, by our convention, $\QH_*(L^\sigma)\cong\H_{*}(L;\mathcal{O})\cong \H^{\dim L-*}(L)$ as rings in the absence of quantum contributions, where $\mathcal{O}$ denotes the local system of orientations.

	\section{Lagrangian Rabinowitz Floer homology}\label{sec:LRFH}
	Let $(\Sigma,\omega)$ and $L$ be as in the previous section. From now on, we assume that the cohomology class $[\omega]$ is in the image of $\H^2(\Sigma;\Z)\to \H^2(\Sigma;\R)$ induced by $\Z\hookrightarrow\R$. We fix an integral lift $[\omega]_\Z\in \H^2(\Sigma;\Z)$ of $[\omega]$ and consider a prequantization bundle $(Y,\alpha)$ of $(\Sigma, [\omega]_{\Z})$. This means that there is  a principal circle bundle $\pi: Y\longrightarrow \Sigma$
	with Euler class $e_Y=-[\omega]_\Z$, using the convention  $S^1=\R/\Z$ for the circle. 
	The 1-form $\alpha$ is a connection form satisfying $\pi^*\omega=d\alpha$. It follows that $\alpha$ is a contact form whose associated Reeb vector field $R$ is tangent to the fibers and generates a 1-periodic flow. The condition $\pi^*\omega=d\alpha$ determines $\alpha$ uniquely up to strict contactomorphism.
		
	\subsection{Legendrian lift}\label{sec:lift}
	
	Since $\omega$ vanishes on $L$, the parallel transport with respect to $\alpha$ gives rise to a holonomy representation
	\[
	\mathrm{hol}_L^\alpha : \H_1(L) \longrightarrow S^1=\R/\Z.
	\]
	From now on, we assume that the image of $\mathrm{hol}_L^\alpha$ is a finite cyclic subgroup, i.e.,~there exists $d\in\N$ such that
	\begin{equation}\label{eq:holonomy}
		\mathrm{hol}_L^\alpha  (\H_1(L)) =\Big(\frac{1}{d} \Z\Big) /\,\Z.
	\end{equation}
	Since $d\alpha=\pi^*\omega$, this yields $\omega(\pi_2(\Sigma,L))\subset \frac{1}{d}\Z$. Under this assumption, $L$ lifts to a Legendrian submanifold $\mathcal{L}$ in $(Y,\alpha)$ such that the covering map
	\[
	\pi|_{\mathcal{L}}: \mathcal{L} \longrightarrow L
	\]
	has deck transformation group $\Z_{d}$ generated by the time-$\frac{1}{d}$ Reeb flow $\phi_R^{1/{d}}$.
	
		\begin{lemma}
		\label{lem:degree}
		Assume that $\omega(\pi_2(\Sigma))\neq0$ and the composition $\pi_2(\Sigma, L)\to \pi_1(L) \to \H_1(L)$ is surjective. Then, the equality in \eqref{eq:holonomy} holds with 
		\[
		d=\frac{2c_\Sigma}{\mathrm{gcd}(2c_\Sigma, N_L  m_\Sigma)}
		\] 
		where $m_\Sigma$ is the positive integer satisfying $\omega(\pi_2(\Sigma))=m_\Sigma\Z$ and $c_\Sigma:= \tau_\Sigma m_\Sigma$ is the minimal Chern number of $\Sigma$.
	\end{lemma}
	\begin{proof}
		For any $a\in \H_1(L)$, there is $A\in \pi_2(\Sigma,L)$ such that the map in the statement sends $A$ to $a$. We compute
		$\mathrm{hol}_L^\alpha(a)=\omega(A)=\frac{\mu_L(A)}{2\tau_\Sigma}$, where the first equality follows from $d \alpha=\pi^{*}\omega$ and Stokes' theorem. 
		Therefore, $\mathrm{hol}_L^\alpha$ is surjective onto $\frac{N_L m_\Sigma}{2c_\Sigma}\Z$ modulo $\Z$, and the claim follows.
	\end{proof}
	The inclusion $L\subset \Sigma$ induces the long exact sequence 
	\[
	\cdots\longrightarrow\H_2(\Sigma,L)\stackrel{\partial}{\longrightarrow} \H_1(L)\stackrel{\iota_*}{\longrightarrow} \H_1(\Sigma)\longrightarrow\cdots.
	\]

	\begin{lemma}\label{lem:Leg_lift}
		Assume that $L$ is not necessarily monotone, and suppose that $\omega(\H_2(\Sigma,L))\subset\mathbb{Q}$, i.e.~$\omega(\H_2(\Sigma,L))=\frac{a}{b}\Z$ with $\gcd(a,b)=1$.  Then, there exists a connection 1-form $\alpha$ satisfying $\pi^*\omega=d\alpha$ such that $\mathrm{hol}_L^\alpha$ satisfies \eqref{eq:holonomy} with some $d\in\N$. Moreover, if in addition $\iota_*(\H_1(L))$ is torsion-free, then $d=b$. 
	\end{lemma}
	\begin{proof}
		Let $\alpha'$ be any connection 1-form satisfying $\pi^*\omega=d\alpha'$.  
	Then $\mathrm{hol}_L^{\alpha'}(\mathrm{im\,}\partial)=(\frac{1}{b}\Z)/\Z$ as in the proof of Lemma \ref{lem:degree}. 
	
	Let $\mathcal{T}$ be the torsion subgroup of $\mathrm{im\,}\iota_*$. We split $\mathrm{im\,}\iota_*\cong \mathcal{F}\oplus\mathcal{T}$, where $\mathcal{F}\cong \Z^r$. 
	We also choose a splitting $\H_1(\Sigma)\cong \mathcal{F}_\Sigma\oplus\mathcal{T}_\Sigma$, where $\mathcal{T}_\Sigma$ is the torsion subgroup and $\mathcal{F}_\Sigma\cong\Z^s$ for $s\geq r$.
	Let $\operatorname{pr}_{\mathcal{F}_{\Sigma}}:\H_1(\Sigma) \to \mathcal{F}_{\Sigma}$ be the projection associated with this splitting.
	Then the restriction $(\operatorname{pr}_{\mathcal{F}_{\Sigma}})|_{\mathcal{F}}$ is injective.
	By the Smith normal form applied to this injective homomorphism, there exist a basis $\{e_1,\dots,e_r\}$ of $\mathcal{F}$, a basis $\{f_1,\dots,f_s\}$ of $\mathcal{F}_{\Sigma}$, and positive integers $m_1,\dots,m_r$ such that $\operatorname{pr}_{\mathcal{F}_{\Sigma}}(e_i)=m_if_i$ for $i=1,\dots,r$.

	We choose $\tilde{e}_i\in \H_1(L)$ satisfying $\iota_*(\tilde{e}_i)=e_i$ for $i=1,\dots,r$, and define a homomorphism 
	\[
	\kappa_0: \mathrm{im\,}\iota_* \to \R,\qquad \kappa_0(e_i)=\mathrm{hol}_L^{\alpha'}(\tilde e_i),
	\]
	where we regard $\mathrm{hol}_L^{\alpha'}(\tilde e_i)$ as a real number in $[0,1)$. Note that $\kappa_0(\mathcal{T})=0$ as $\kappa_0$ is a homomorphism. We then extend $\kappa_0$ to a homomorphism $\kappa:\H_1(\Sigma)\to \R$. Explicitly, we define $\kappa$ by 
	\[
	\kappa(\mathcal{T}_\Sigma)=0,\quad \kappa(f_i)=\frac{1}{m_i}\kappa_0(e_i) \,\text{ for }\, 1\leq i\leq r,\quad \kappa(f_i)=0 \,\text{ for }\, r+1\leq i\leq s.
	\]
	Let $\beta$ be a closed 1-form on $\Sigma$ representing $\kappa$ via $\H^1_\mathrm{dR}(\Sigma)\cong \Hom(\H_1(\Sigma),\R)$.
	
	 For any $a\in\H_1(L)$, we write 
	\[
	a=c_1\tilde{e}_1+\cdots +c_r\tilde{e}_r+\mathfrak{t},\qquad c_1,\dots,c_r\in \Z,\quad \iota_*(\mathfrak{t})\in\mathcal{T}.
	\] 
	The holonomy of $a$ with respect to a new connection 1-form $\alpha:=\alpha'-\pi^*\beta$ equals 
	\[
	\mathrm{hol}_L^\alpha(a)=\sum_{i=1}^rc_i \big(\mathrm{hol}_L^{\alpha'}  (\tilde e_i) - \kappa(e_i)\big) + \mathrm{hol}_L^{\alpha'}  (\mathfrak{t})-\kappa(\iota_*(\mathfrak{t})) = \mathrm{hol}_L^{\alpha'}  (\mathfrak{t}) \quad \text{mod}\;\Z.
	\]
	Let $\exp(\mathcal{T})\in\N$ be the exponent of the torsion subgroup $\mathcal{T}$, i.e.,~it is the smallest natural number such that $\exp(\mathcal{T}) \eta=0$ for any $\eta\in\mathcal{T}$. Then, $\exp(\mathcal{T}) \mathfrak{t}\in\mathrm{im\,}\partial$ since $\iota_*(\exp(\mathcal{T}) \mathfrak{t})=\exp(\mathcal{T})\iota_*(\mathfrak{t})=0$. Therefore, 
	\[
	\exp(\mathcal{T})\mathrm{hol}_L^\alpha(a)=\exp(\mathcal{T})\mathrm{hol}_L^{\alpha'}(\mathfrak{t})=\mathrm{hol}_L^{\alpha'}  (\exp(\mathcal{T})\mathfrak{t}) \in \Big(\frac{1}{b}\Z\Big)/\Z
	\]
	This proves that $\mathrm{hol}_L^\alpha(a)\in (\frac{1}{\exp(\mathcal{T})b}\Z)/\Z$, and hence $\mathrm{hol}_L^\alpha(\H_1(L))=(\frac{1}{d}\Z)/\Z$ for some $d\in\N$ dividing $\exp(\mathcal{T})b$. 
	
	Suppose now that $\mathcal{T}=0$. Then, $\mathrm{hol}_L^\alpha(a)=\mathrm{hol}_L^{\alpha'}  (\mathfrak{t})$ where $\iota_*(\mathfrak{t})=0$, or equivalently $\mathfrak{t}\in\mathrm{im}\,\partial$. Therefore, $\mathrm{hol}_L^\alpha(\H_1(L))=\mathrm{hol}_L^{\alpha'}(\mathrm{im\,}\partial)=(\frac{1}{b}\Z)/\Z$, and this finishes the proof.
	\end{proof}

	\begin{cor}
		Assume $L$ is homologically monotone, i.e.,~$\mu_L =2 \tau_\Sigma \omega$ holds on $\H_2(\Sigma,L)$ for some $\tau_\Sigma>0$. Then, there exists a connection 1-form $\alpha$ satisfying $\pi^*\omega=d\alpha$ such that $\mathrm{hol}_L^\alpha$ satisfies \eqref{eq:holonomy} with some $d\in\N$. Moreover, if $\iota_*(\H_1(L))\subset \H_1(\Sigma)$ is torsion-free, then $d=\frac{2c_\Sigma^h}{\mathrm{gcd}(2c_\Sigma^h, N_L^h  m_\Sigma^h)}$, where the constants are the homological counterparts of those in Lemma \ref{lem:degree}, e.g.~$m_\Sigma^h\in\N$ is the positive integer satisfying $\omega(\H_2(\Sigma))=m_\Sigma^h\Z$.
	\end{cor}

	\subsection{Reeb chords and indices}\label{sec:Reeb}
	We denote the fibers of $\pi$ and $\pi|_{\mathcal{L}}$ over $p\in L\subset\Sigma$ by  
	\[
	Y_p:=\pi^{-1}(p),\qquad \mathcal{L}_p:= Y_p\cap \mathcal{L}.
	\]
	We define the relative winding number of a chord $c : ([0,1], \{0,1\}) \to (Y_p, \mathcal{L}_p)$ by
	\[
	\mathfrak{w}(c) := \deg \big(\tilde{c}:[0,1]/(0\sim 1)\to Y_p/ \Z_d \big),
	\]
	where $\tilde c$ is a continuous map $\tilde c$ between circles induced by $c$.
	\begin{prop}\label{prop:contractible}
		Let $m_L$ be the nonnegative integer satisfying $\omega(\pi_2(\Sigma,L))  = \frac{m_L}{d} \Z$, which by monotonicity is given by $m_L=\frac{d N_L}{2\tau_\Sigma}$.
		Then $c: ([0,1], \{0,1\}) \to (Y_p, \mathcal{L}_p)$ is trivial in $\pi_1(Y,\mathcal{L})$ if and only if $m_L$ divides $\mathfrak{w}(c)$.
	\end{prop}
	\begin{proof}
		The fibration 
		$
		(Y_p, \mathcal{L}_p) \stackrel{i}{\hookrightarrow}  (Y,\mathcal{L}) \to (\Sigma,L)	
		$
		induces the long exact sequence  
		\begin{equation}\label{eq:homotopy_les}
			\dots \to \pi_2(Y,\mathcal{L}) \to \pi_2(\Sigma,L) \xrightarrow{\delta} \pi_1( Y_p, \mathcal{L}_p) \xrightarrow{i_*} \pi_1(Y,\mathcal{L}) \to \cdots .
		\end{equation}
		Through the isomorphism $\pi_1(Y_p,\mathcal{L}_p) \cong \pi_1(\R/\Z, \frac{1}{d}\Z/\Z) \cong \pi_1(\R/\frac{1}{d}\Z) \cong \Z$,  the homotopy class $[c] \in \pi_1(Y_p,\mathcal{L}_p)$ corresponds to $\mathfrak{w}(c) \in \Z$. 
		On the other hand, for $A\in \pi_2(\Sigma,L)$, $\delta(A)$ corresponds to $-d\cdot e_Y (A) = d\cdot \omega(A)$, which yields  $\mathrm{im\,} \delta  = d\cdot\omega (\pi_2(\Sigma,L)) = m_L \Z$. See Remark \ref{rem:disk_lift} below.
		Therefore, $[c]  \in \ker i_*  = \mathrm{im\,} \delta$ if and only if $m_L$ divides $\mathfrak{w}(c)$.
	\end{proof}

	Consider the half-disk
	$\D_+ = \{ z \in \mathbb{C}\, |\, |z| \leq 1, \mathrm{Im}\ z > 0\}$.
	The boundary of the closure $\overline{\D}_+$ decomposes into $\partial \overline{\D}_+= \partial_1 \overline{\D}_+\cup \partial_2 \overline{\D}_+$, where
	\[
		\partial_1 \overline{\D}_+ := \{z \in \overline{\D}_+ \,|\,|z| =1\},\qquad 
		\partial_2 \overline{\D}_+ := \{ z\in \overline{\D}_+ \,|\,\mathrm{Im}\ z = 0\}=[-1,1].
	\]	
	
	\begin{remark}\label{rem:disk_lift}
		For a given $u:(\D,\partial\D,1)\to (\Sigma,L,p)$, the map $u$ restricted to $\D\setminus\{1\} \cong \D_+$ lifts to a map $(\D_+, \partial_1 \D_+)\to (Y,\mathcal{L})$. If we denote the closure of this lifted map  by $\tilde u: (\overline{\D}_+,\partial_1 \overline{\D}_+)\to (Y,\mathcal{L})$, then $\tilde{u}|_{\partial_2 \overline{\D}_+}:[-1,1]\to Y_p$. 
		If we choose our lift to satisfy $\tilde{u}|_{\partial_2 \overline{\D}_+}(-1)\in\mathcal{L}_p$, then $\tilde{u}|_{\partial_2 \overline{\D}_+}(1)\in\mathcal{L}_p$  with $\mathfrak{w}(\tilde{u}|_{\partial_2 \overline{\D}_+})=d\cdot\omega([u])$. This is how the map $\delta$ in \eqref{eq:homotopy_les} is defined.
	\end{remark}
	
	We now turn to the study of Reeb chords of $\mathcal{L}$. Let $c : ([0,1],\{0,1\}) \to (Y_p,\mathcal{L}_p)$ be a generalized Reeb chord, meaning that $\dot c(t)=T R(c(t))$ for some $T \in \frac{1}{d} \Z$.
	Up to reparametrization, such $c$ corresponds to an orbit of $R$, an orbit of $-R$, or a constant orbit according as $T>0$, $T<0$, or $T=0$, respectively.
	Assume that $c$ is trivial in $\pi_1(Y,\mathcal{L})$. There is a capping with the opposite orientation, denoted by 
	$\overline{c} : (\overline{\D}_+,\partial_1 \overline{\D}_+) \to (Y, \mathcal{L})$, i.e.,
	\begin{equation}\label{eq:capping}
	\overline{c}(\tau) = c\left(\frac{1-\tau}{2}\right),\qquad \tau\in [-1,1] = \partial_2 \overline{\D}_+.		
	\end{equation}
	Let us denote $n=\dim \mathcal{L}$ and $\xi=\ker \alpha$. We choose a symplectic trivialization of $\xi$ adapted to $\mathcal{L}$ over $\overline{c}$:
	\[
	\Phi : \overline{c}^*\xi \to \overline{\D}_+ \times \C^{n}, \qquad \Phi(T_{\overline{c}(r)}\mathcal{L})=\R^n\quad\forall r\in \partial_1 \overline{\D}_+,
	\]
	where in the latter we abuse notation by viewing $\Phi$ as a map $\xi_{\overline{c}(z)}\to \{z\}\times \C^{n}=\C^{n}$.
	Then, using the restriction of $\Phi$ to $\partial_2\overline{\D}_+$, we obtain a path $c^\Phi:[0,1]\to \mathrm{Lag}(\C^{n})$ of Lagrangian subspaces in $\C^{n}$ defined by
	\[
	c^{\Phi}(t) := \Phi\left(d\phi_{R}^{Tt} (T_{c(0)} \mathcal{L})\right),
	\]
	where $\Phi$ is evaluated over the point $\overline c(1-2t)$.
	We define $\mu_\mathrm{RS} (c,\overline{c})$ for a capped chord $(c,\bar c)$ by the Robbin--Salamon index of the path $c^{\Phi}$ with respect to the reference Lagrangian subspace $\R^{n}\subset\C^n$, see \cite[Section 2]{RS93}.
	This is independent of the choice of a symplectic trivialization of $\overline{c}^*\xi$ adapted to $\mathcal{L}$. 
	Note that $c^{\Phi}$ is actually a loop of Lagrangian subspaces since $d\phi_R^T (T_{c(0)} \mathcal{L}) = T_{c(1)} \mathcal{L}$, and both $T_{c(0)} \mathcal{L}$ and $T_{c(1)} \mathcal{L}$ are mapped to $\R^{n}$ by $\Phi$. Therefore, 
	the Robbin--Salamon index of $c^\Phi$ is simply the Maslov index $\mu_\mathrm{Mas}(c^\Phi)$ of $c^{\Phi}$, i.e.,~$\mu_\mathrm{RS}(c,\overline{c}) = \mu_\mathrm{Mas}(c^{\Phi})$. 
	We consider the Maslov class 
	\[
	\mu_{\mathcal{L}} :\pi_2(Y, \mathcal{L}) \longrightarrow \mathbb{Z},
	\]
	where $\mu_{\mathcal{L}}([u])$ is the Maslov index of $(u|_{\partial \D})^* T\mathcal{L}$ defined by a symplectic trivialization of $u^*\xi$. 
	Since $\pi^*T\Sigma\cong\xi$ and $(\pi|_{\mathcal{L}})^*TL\cong T\mathcal{L}$, we have $\pi^*\mu_L=\mu_{\mathcal{L}}$. By  Lemma \ref{lem:Maslov_vanish}, the index $\mu_\mathrm{RS}(c,\overline{c})$ is independent of the choice of a capping, and we simply write $\mu_\mathrm{RS}(c)=\mu_\mathrm{RS}(c,\overline{c})$.

	\begin{lemma}\label{lem:Maslov_vanish}
		The Maslov class $\mu_{\mathcal{L}}$ is zero.
	\end{lemma}
	\begin{proof}
		For a given smooth map $u :(\D,\partial\D) \to (Y, \mathcal{L})$,
		\begin{align*}
			\mu_{\mathcal{L}} ([u])&=\pi^* \mu_L ([u])
			= \pi^* ( 2\tau_\Sigma \omega)([u])
			= 2\tau_\Sigma  (d\alpha)([u])
			= 2\tau_\Sigma \int_{\partial \D} u^*\alpha
			=0.
		\end{align*}
		The last equality follows from the fact that $\mathcal{L}$ is a Legendrian submanifold of $(Y,\alpha)$. 
	\end{proof}

	The composition $\pi\circ \overline{c}$ maps $\partial_2 \overline{\D}_+$ to $p$. We interpret this composite map as 
	\[
	\pi \circ \overline{c} : (\D, \partial \D,-1) \longrightarrow (\Sigma, L,p).	
	\]

	\begin{prop}
		\label{prop:maslov index of projection}
		For any capped chord $(c,\overline{c})$, we have
		\[
		\mu_\mathrm{RS}(c) = -\mu_L([ \pi \circ \overline{c}]) = -2\tau_\Sigma\omega ([\pi \circ \overline{c}])=\frac{2\tau_\Sigma\mathfrak{w}(c)}{d}.
		\]
	\end{prop}
	
	\begin{proof}
		
		Let us consider another symplectic trivialization 
		$
		\Psi : \overline{c}^*\xi \to \overline{\D}_+ \times \C^{n}
		$
		such that 
		\[ 
		\Psi(d\phi_R^{Tt} (T_{c(0)} \mathcal{L}))= \R^n \qquad  \forall  t\in [0,1],
		\]
		where $\Psi$ is evaluated over the point $\overline c(1-2t)$.
		This gives rise to a loop of Lagrangian subspaces 
		\[ 
		c^\Psi:[0,1]/(0\sim 1) \to \mathrm{Lag}(\C^{n}),\qquad 
		c^{\Psi}(t) :=
		\Psi (T_{\overline{c}(e^{\pi i (1-t)})} \mathcal{L}).
		\]
		This is indeed a loop since $d\phi_R^{T}(T_{c(0)} \mathcal{L})= T_{c(1)}\mathcal{L}$. As observed in  
			Remark \ref{rem:equality of maslov indices} below, the Maslov index $\mu_\mathrm{Mas}(c^{\Psi})$ of $c^{\Psi}$ agrees with $\mu_\mathrm{Mas}(c^{\Phi})=\mu_\mathrm{RS}(c)$.
						
			It remains to verify $\mu_\mathrm{Mas}(c^{\Psi})=-\mu_L([\pi\circ \overline{c}])$. Since $\pi^*T\Sigma\cong \xi$, the trivialization $\Psi$ induces a symplectic trivialization of $(\pi \circ \overline{c})^* T\Sigma$, which we denote by $\psi$. We define 
			\[
			((\pi\circ \overline{c})|_{\partial \D})^\psi : \R/\Z \to \mathrm{Lag}(\C^{n}),\qquad ((\pi\circ \overline{c})|_{\partial \D})^\psi(t):= \psi(T_{(\pi\circ \overline{c})(e^{\pi i(1-2t)})}L).
			\]
			Since $(\pi|_{\mathcal{L}})^*TL\cong T\mathcal{L}$, we have $c^\Psi =((\pi\circ \overline{c})|_{\partial \D})^\psi $. Therefore, 
			\[
			\mu_\mathrm{Mas}(c^\Psi)=\mu_\mathrm{Mas}(((\pi\circ \overline{c})|_{\partial \D})^\psi)=-\mu_L([\pi\circ \overline{c}]),
			\]  
			where the minus sign appears as we use the opposite orientation of the capping, see \eqref{eq:capping}.
			The last equality in the statement follows from $\frac{1}{d}\mathfrak{w}(c)=-\omega([\pi\circ\overline{c}])$, see  Remark \ref{rem:disk_lift}.
		\end{proof}
		
		\begin{remark}\label{rem:equality of maslov indices}
			For two Lagrangian loops $c^\Phi$ and $c^\Psi$ in the proof of the preceding proposition,   
			$\mu_\mathrm{Mas}(c^{\Phi}) = \mu_\mathrm{Mas}(c^{\Psi})$ holds. Indeed, recall that if two Lagrangian loops $\ell_1(t)$ and $\ell_2(t)$ satisfy $\ell_2(t)=A(t)\ell_1(t)$ for some loop of symplectic matrices $A(t)$, then $\mu_\mathrm{Mas}(\ell_2)=2\mu_{\mathrm{CZ}}(A)+\mu_\mathrm{Mas}(\ell_1)$, where $\mu_{\mathrm{CZ}}$ denotes the Conley--Zehnder index. 
			In our situation, $c^\Phi$ and $c^\Psi$ are related by the loop of symplectic matrices $\Phi\circ\Psi^{-1}$ restricted to $\partial \overline{\D}_+$. Since $(\Phi\circ\Psi^{-1})|_{\partial \overline{\D}_+}$ is contractible as this admits an extension to $\overline{\D}_+$, its Conley--Zehnder index equals zero. 
		\end{remark}

		\subsection{Lagrangian Floer chain complex for $\vee$-shaped Hamiltonians}\label{sec:LFC}
		Rabinowitz Floer homology was first introduced in \cite{CF09}. We adopt the formulation from \cite{CFO10,CO18} using $\vee$-shaped Hamiltonians rather than the Rabinowitz action functional. 
		In this section, we define the Floer chain complex of the Lagrangian submanifold $\R\times \mathcal{L}$ in the symplectization $\R\times Y$ of $Y$, associated to a $\vee$-shaped Hamiltonian.
		From now on, we assume that the minimal Maslov number $N_L$ of $L$ is greater than $2$. 
		This assumption guarantees the necessary compactness properties of the relevant Floer moduli spaces, even in the absence of a symplectic filling of $Y$ or a Lagrangian filling of $\mathcal{L}$.

		\subsubsection{Hamiltonian chords for $\vee$-shaped Hamiltonians}\label{sec:Ham_chords}
		For a small $\epsilon>0$, we denote by $\mathcal{H}$ the set of smooth functions $H: \R \times Y \to \R$ such that $H(r,y)=h(e^r)$ for a smooth function $h : \R_{>0} \to \R$ satisfying
		\begin{equation}\label{eq:Hamiltonian}
			h(1)<0, \qquad   \begin{cases}
				h(\rho)=\text{constant}, &\rho \in (0,e^{-\epsilon}),\\[.5ex]
				h''(\rho)<0, & \rho \in (e^{-\epsilon},e^{-\epsilon+\eta}),\\[.5ex]
				h(\rho)=-a\rho+b^-,  &\rho \in (e^{-\epsilon+\eta}, e^{-\eta}),\\[.5ex]
				h''(\rho)>0, & \rho \in (e^{-\eta},e^{\eta}),\\[.5ex]
				h(\rho)=a\rho+b^+, & \rho \in (e^\eta, +\infty),
			\end{cases}
		\end{equation}
		for some $a \in (0,\infty) \setminus \frac{1}{d}\Z$, $b^\pm \in \R$, and $\eta \in (0,\frac{\epsilon}{2})$. 
		For $H \in \mathcal{H}$, consider the action functional
		\begin{equation}\label{eq:action_functional}
			\begin{split}
				&\mathcal{A}_H : C^{\infty} \big( ([0,1], \{0,1\}), (\R \times Y, \R \times \mathcal{L}) \big) \longrightarrow \R\\[1ex]
				&\mathcal{A}_H (x) := \int_{0}^{1}x^*(e^r \alpha) - \int_{0}^{1} H(x(t)) dt.	
			\end{split}
		\end{equation}
		Critical points of $\mathcal{A}_H$ are precisely the chords $x$ of the  Hamiltonian vector field $X_H$ defined by $\iota_{X_H}(d(e^r \alpha))=-dH$.
		If we write $x=(r_x, c_x)$, then $r_x$ is constant and $c_x$ satisfies
		\[
		c_x:([0,1], \{0,1\})\to (Y, \mathcal{L}),\qquad \partial_t c_x=h'(e^{r_x})R(c_x).
		\] 
		Since Reeb chords of $(Y,\mathcal{L})$ have actions in $\frac{1}{d}\N$, it holds that $h'(e^{r_x})\in \frac{1}{d}\Z$. 
		We are particularly interested in critical points appearing in the region $(-\eta,\eta) \times Y$. 
		For each $k\in \mathbb{Z}$, we define 
		\begin{equation}\label{eq:crit_L}
			\mathcal{L}_{k}^H:=\big\{x=(r_x,c_x)\in \Crit\mathcal{A}_H \mid r_x \in (-\eta,\eta),\,h'(e^{r_x})= k/d \big\}.
		\end{equation}
		It is a Morse-Bott critical manifold of $\mathcal{A}_H$. 
		The condition $h'(e^{r_x})= k/d$ is equivalent to the condition that $c_x$ has relative winding number $k$, i.e.,~$\mathfrak{w}(c_x)=k$. Since $h$ is strictly convex on $(e^{-\eta},e^\eta)$, there is a unique point $r^H_k\in(-\eta,\eta)$ such that 
		\begin{equation}\label{eq:beta}
			h' \big( e^{r^H_k} \big) = \frac{k}{d}
		\end{equation} 
		provided $-a<\frac{k}{d}<a$. The action value of $x\in \mathcal{L}_k^H$ is computed as 
		\begin{equation}\label{eq:action}
			\mathcal{A}_H(x)= e^{r_x}h'(e^{r_x})-h(e^{r_x}).
		\end{equation} 
		The space $\mathcal{L}_{k}^H$ for this $k$ is naturally identified with $\mathcal{L}$ via 
		\begin{equation}\label{eq:identify}
			\mathcal{L}_k^H \cong \mathcal{L},\qquad  x=(r^H_k,c_x) \mapsto c_x(0).
		\end{equation}

		Let $x=(r^H_k,c_x)\in \mathcal{L}_k^H$ for some $k =\mathfrak{w}(c_x) \in m_L \mathbb{Z}$. By the property of $m_L$ in Proposition \ref{prop:contractible}, $c_x$ admits a capping $\overline{c}_x$ in $Y$, and the index $\mu_\mathrm{RS} (c_x,\overline{c}_x)$ defined as in Section \ref{sec:Reeb} coincides with $-\mu_L([\pi \circ \overline{c}_x])$ by Proposition \ref{prop:maslov index of projection}. Then,
		\begin{equation*}
			\overline{x} = (r^H_k, \overline{c}_x): (\overline \D_+,\partial_1 \overline \D_+) \to (\R\times Y,\R\times \mathcal{L})
		\end{equation*}
		is a capping of $x$. To define index, we choose a symplectic trivialization 
		\[
		\Phi : \overline{x}^* T(\mathbb{R} \times Y)=\overline{x}^* ((\R\partial_r\oplus\R R)\oplus\xi)  \longrightarrow  \overline \D_+ \times (\C \oplus \C^{n}),
		\]
		where $n=\frac{1}{2}\dim\Sigma$ and $\partial_r$ denotes the coordinate vector field on the $\R$-factor. This induces 
		\[
		x^{\Phi}:[0,1]\to \mathrm{Lag}(\C^{n+1}),\qquad  x^{\Phi}(t)=\Phi (d\varphi_{X_H}^t( T_{x(0)} (\mathbb{R} \times \mathcal{L}))),
		\]
		where $\Phi$ is evaluated over the point $\overline x(1-2t)$ and $\varphi_{X_H}^{t}$ denotes the time-$t$ flow of $X_H$.
		We denote by $\mu_\mathrm{RS}(x,\overline{x})$ the Robbin--Salamon index of $x^{\Phi}$ with respect to the reference Lagrangian subspace $\R\oplus \R^n\subset\C\oplus\C^{n}$. Since $\mu_{\mathcal{L}}=0$ by Lemma \ref{lem:Maslov_vanish}, $\mu_\mathrm{RS}(x,\overline{x})$ is independent of the choice of a capping $\overline{x}$ of $x$, and we simply write $\mu_\mathrm{RS}(x)=\mu_\mathrm{RS}(x,\overline{x})$. In fact, since $H(r,p)=h(e^r)$ with convex $h|_{(e^{-\eta},e^\eta)}$, it follows that
		\begin{equation*}
			\mu_\mathrm{RS} (x) = \mu_\mathrm{RS}(c_x)+\frac{1}{2} = -\mu_L ([\pi \circ \overline{c}_x]) +\frac{1}{2}=\frac{2\tau_\Sigma k}{d}+\frac{1}{2}
		\end{equation*}
		where the second and third equalities are proved in Proposition \ref{prop:maslov index of projection}.
		
		\medskip

	Let $(f_L,Z_L)$ be a Morse-Smale pair as in Section \ref{sec:QC}. We consider its lift to $\mathcal{L}$ given by 
	\[
	f_{\mathcal{L}}:= (\pi|_\mathcal{L})^* f_L,\qquad Z_{\mathcal{L}}:= (\pi|_\mathcal{L})^* Z_L.
	\]
	For each $p\in \mathrm{Crit}f_L$, there exist exactly $d$ critical points 
	\[
	p^1,\;\;p^2=\phi_R^{\frac{1}{d}}(p^1),\;\;\dots,\;\; p^{d}=\phi_R^{\frac{d-1}{d}}(p^1) \in \Crit f_\mathcal{L}
	\]
	lying over $p$, and all of them have the same index, i.e.,~$\mathrm{ind}_{f_L} (p) = \mathrm{ind}_{f_{\mathcal{L}}} (p^i)$. 
	Note that the Morse complex of $(f_L,Z_L)$ is identical to the $\Z_{d}$-equivariant Morse complex of $(f_{\mathcal{L}},Z_{\mathcal{L}})$.

		For each $p \in \mathrm{Crit} f_L$, we write
		\begin{equation}\label{eq:d_many_chords}
			p_k^1,\, \dots \,,p_k^{d} \in \mathcal{L}_k^H	
		\end{equation}
		for the Hamiltonian chords  corresponding to $p^1,\dots , p^{d} \in \mathrm{Crit} f_{\mathcal{L}}$, respectively, via the identification $\mathcal{L}_{k}^H \cong \mathcal{L}$ in \eqref{eq:identify}. In other words, $p_k^i$ is the Hamiltonian chord with $p_k^i(0)=(r^H_k,p^i)$ whose $Y$-component has relative winding number $k$ for $i=1,\dots,d$.
		The Floer chain complex we will define has $p_k^1,\dots,p_k^{d}$ for $k \in m_L\Z$ as generators. Throughout the paper, when it is unnecessary to distinguish between $p^1,\dots, p^{d}$, we simply write $\tilde p$. Similarly, we use $\tilde p_k$ to represent $p_k^1,\dots,p_k^{d}$. 
		We define the index of $\tilde p_k\in \{p_k^1,\dots,p_k^{d}\}$ by
		\begin{align}
			\begin{split}\label{eq:index}
				\mu(\tilde p_k):= \dim \mathcal{L}+\mu_\mathrm{RS}(\tilde p_k)  -\frac{1}{2} -\ind_{f_{\mathcal{L}}}(\tilde{p})&=  \dim \mathcal{L}+\frac{ 2\tau_\Sigma k}{d} - \mathrm{ind}_{f_{\mathcal{L}}}(\tilde{p})\\
				&= \dim L+\frac{k N_L}{m_L}   - \mathrm{ind}_{f_L} (p).
			\end{split}
		\end{align}
		Our choice of $-\frac{1}{2}$ above ensures that the correspondence between the quantum homology of $L$ and the Rabinowitz Floer homology of $\mathcal{L}$ is degree-preserving.
		
		\subsubsection{Moduli space of Floer strips}
		Let $\mathcal{J}_Y$ denote the set of cylindrical almost complex structures $J_Y$ on $\R \times Y$, meaning that the following properties hold:
		\begin{itemize}
			\item $J_Y \partial_r = R$,	where $R$ is the Reeb vector field, and $r$ is the $\R$-coordinate of $\R \times Y$. 
			\item $J_Y$ is invariant under translations in the $r$-coordinate.
			\item $\xi=\ker\alpha$ is invariant under $J_Y$, and $J_Y|_{\xi}$ is compatible with $d\alpha|_{\xi}$.
		\end{itemize}
		The horizontal subbundle $T^\mathrm{h} Y \subset TY$ defined by the connection $1$-form $\alpha$ agrees with $\xi$ and with $\pi^*T\Sigma$. Therefore, for each $J_Y\in \mathcal{J}_Y$, there is $J_\Sigma\in \mathcal{J}_\Sigma$ such that $\pi_*(J_Y|_\xi)=J_\Sigma$. We call $J_\Sigma$ the horizontal part of $J_Y$ and, unless otherwise specified, use this notation whenever $J_Y$ is fixed.
			
		For $(H,J) \in \mathcal{H}\times \mathcal{J}_Y$, we consider Floer strips, namely smooth maps 
		\[
		\tilde{v}=(b,v):(\R \times [0,1], \R \times \{0,1\}) \to (\R \times Y, \R \times \mathcal{L})
		\] 
		that solve the Floer equation
		\begin{equation}\label{eq:Floer_equation_strip}
			\partial_s \tilde{v} + J_Y(\tilde{v}) (\partial_t \tilde{v} - X_H(\tilde{v}) ) =0
		\end{equation}
		on $\R\times (0,1)$, with the asymptotic condition 
		\[
		\ev_{\pm}(\tilde{v})= (\ev_{\pm}(b),\ev_{\pm}(v)) := \lim_{s\to \pm \infty} \tilde{v} (s, \cdot) \in \mathrm{Crit} \mathcal{A}_H.
		\]
		The projection $\pi \circ v$ of $v$ defined on $\R \times [0,1]\cong \D \setminus\{\pm 1\}$ has finite $\omega$-energy. By the removal of singularity theorem, it extends to a $J_\Sigma$-holomorphic map defined on the closed disk $\D$.
		By abuse of notation, we continue to write the map defined on $\D$ by $\pi \circ v$, i.e.,
		\begin{equation}\label{eq:extend_disk}
			\pi \circ v : (\D, \partial \D) \longrightarrow (\Sigma, L).
		\end{equation}
	
		Throughout this section, we only consider integers $k,l \in m_L\Z$ for which $r^H_{k},r^H_l \in (-\eta,\eta)$ exist as in \eqref{eq:beta}.
		For a Floer strip $\tilde{v}=(b,v)$, we set 
		\begin{equation}\label{eq:asymptotic_strip}
			\mathfrak a_{\pm}(\tilde v):=(\ev_\pm(b), \w(\ev_\pm(v))),
		\end{equation}
	    where we identify the constant path $\ev_\pm(b)$ with its value in $\R$.
		\begin{definition}
			Let $\mathbf{A}:=(A_1,\dots,A_N) \in (\pi_2(\Sigma,L)\setminus \{0\})^N$ for $N \in \N$, and let $k,l \in m_L\Z$.  
			We define the moduli space
			\begin{equation*}
				\mathcal{M}_{N,l,k}(\mathbf{A};H,J_Y) =\{\mathbf{v}=(\tilde{v}_1,\dots,\tilde{v}_N)\}
			\end{equation*}
			of $N$-tuples of Floer strips $\tilde{v}_i=(b_i,v_i)$ with respect to $(H, J_Y)$ such that 
			\begin{itemize}
				\item $\pi \circ v_i$ represents $A_i$ for $i=1,\dots,N$,
				\item $(r^H_l,l)=\mathfrak{a}_-(\tilde v_1)$, $(r^H_k,k)=\mathfrak{a}_+(\tilde v_N)$, and  $\mathfrak{a}_+(\tilde v_i)=\mathfrak{a}_-(\tilde v_{i+1})$ for $1 \leq i \leq N-1$.
			\end{itemize}

			We also define the subspace 
			\[
			\mathcal{M}^*_{N,l,k}(\mathbf{A};H,J_Y) \subset \mathcal{M}_{N,l,k}(\mathbf{A};H,J_Y)
			\]
			consisting of simple elements $\mathbf{v}$, namely $	(\pi\circ v_1,\dots,\pi\circ v_N) \in \mathcal{N}^{*}_N(\mathbf{A};J_\Sigma)$.
		\end{definition}

		We consider the evaluation map 
		\[
			\ev :\mathcal{M}_{N,l,k}(\mathbf{A};H,J_{Y}) \to \mathcal{L}^{2N} ,\quad 
			\ev(\mathbf{v}) :=\big(\ev_{-,0}(\tilde{v}_1),\ev_{+,0}(\tilde{v}_1),\dots, \ev_{-,0}(\tilde{v}_N),\ev_{+,0}(\tilde{v}_N)\big),
		\]
		where $\ev_{\pm,0}(\tilde{v}):=\ev_\pm(v)(0)$.

		\begin{definition}\label{def:strips_trajectories}
			Let us abbreviate $\mathcal{D}_H=(f_{\mathcal{L}},Z_{\mathcal{L}},H,J_Y)$. For $(\tilde p,k),(\tilde q,l)\in \Crit f_{\mathcal{L}}\times  m_L\Z$, we define 
			\[
			\mathcal{M}_{N}(\tilde{q}_{l},\tilde{p}_{k};\mathbf{A};\mathcal{D}_H):= (\ev)^{-1} \left(W^u_{Z_{\mathcal{L}}}(\tilde{q}) \times \Delta_{Z_\mathcal{L}}^{N-1} \times W^s_{Z_{\mathcal{L}}}(\tilde{p}) \right), 
			\]
			where \[\Delta_{Z_\mathcal{L}}:=\{(x,\varphi_{Z_{\mathcal{L}}}^t(x))\in \mathcal{L}\times \mathcal{L} \; |\; x\in \mathcal{L} \setminus \Crit f_{\mathcal{L}},\; t\in \R_{>0} \},\] see Figure \ref{fig:pearl_Leg}. For $N=1$, this means $\left(\ev\right)^{-1}(W^u_{Z_{\mathcal{L}}}(\tilde{q}) \times W^s_{Z_{\mathcal{L}}}(\tilde{p}))$.
			We also define 			\[
			\mathcal{M}^*_{N}(\tilde{q}_{l},\tilde{p}_{k};\mathbf{A};\mathcal{D}_H)
			:= \mathcal{M}_{N}(\tilde{q}_{l},\tilde{p}_{k};\mathbf{A};\mathcal{D}_H) \, \cap \, \mathcal{M}^*_{N,l,k}(\mathbf{A};H,J_Y).
			\]
		\end{definition}
		\begin{figure}[h]
			\centering
			\includegraphics[height = 3.5cm]{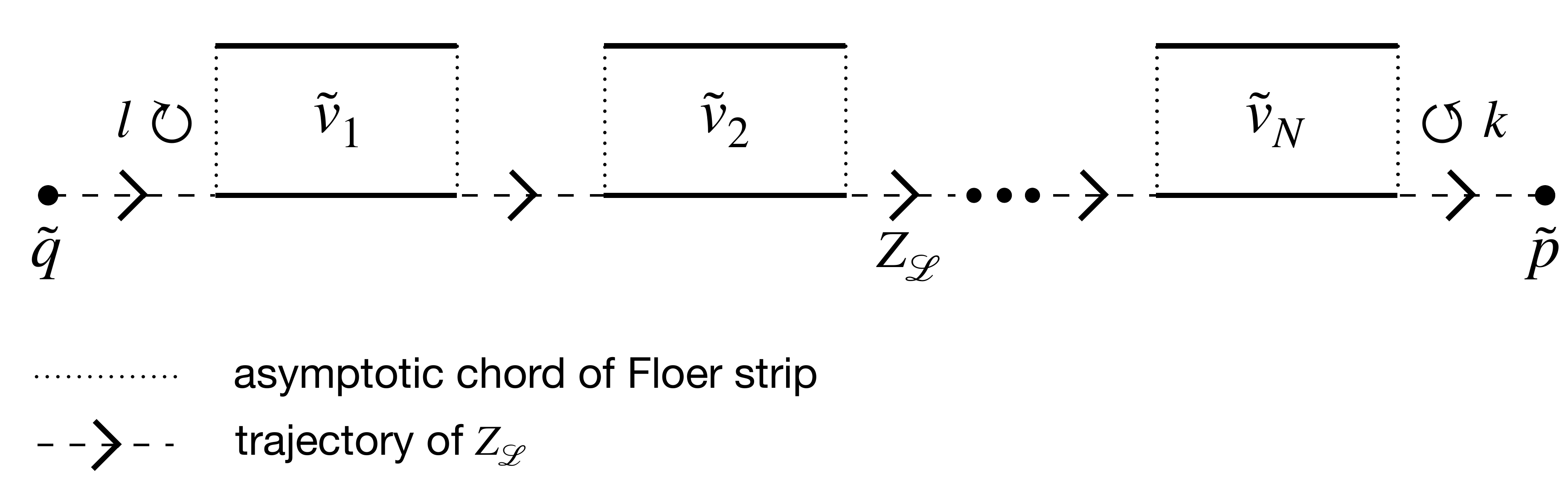}
			\caption{An element of $\mathcal{M}_{N}(\tilde{q}_{l},\tilde{p}_{k};\mathbf{A};\mathcal{D}_H)$.}
			\label{fig:pearl_Leg}
		\end{figure} 
		There is a canonical projection map
		\begin{equation}\label{eq:moduli_projection_strip}
			\Pi : \mathcal{M}_{N}(\tilde{q}_{l},\tilde{p}_{k};\mathbf{A}; \mathcal{D}_H) \to \mathcal{N}_{N} (q,p;{\bf A};\mathcal{D}),\quad
				\Pi(\mathbf{v}):=(\pi \circ v_1,\dots, \pi\circ v_N)
		\end{equation} 
		This map restricts to $\Pi:\mathcal{M}^{*}_N(\tilde{q}_{l},\tilde{p}_{k};\mathbf{A};\mathcal{D}_H)\to\mathcal{N}^{*}_N(q,p;{\bf A};\mathcal{D})$.
		
		\begin{prop}\label{prop:transv_strip}
			Let $\mathcal{J}_{(Y,\mathcal{L})}^{\mathrm{reg}} \subset \mathcal{J}_Y$ be the subset of $\mathcal{J}_Y$ whose horizontal part $J_\Sigma$ belongs to $\mathcal{J}_{(\Sigma, L)}^{\mathrm{reg}}$ defined in Proposition \ref{prop:base_transv}. 
			Then for every choice of 
			\[
			(\tilde p,k),(\tilde q,l)\in \Crit f_{\mathcal{L}}\times  m_L\Z,\quad N \in \mathbb{N}, \quad \text{and }\;\mathbf{A} = (A_1,\dots,A_N) \in (\pi_2(\Sigma, L))^N,
			\]
			and for every $J_Y \in \mathcal{J}_{(Y,\mathcal{L})}^{\mathrm{reg}}$, the moduli space
			$\mathcal{M}^*_{N}(\tilde{q}_{l},\tilde{p}_{k};\mathbf{A};\mathcal{D}_H) $
			is a smooth manifold of dimension $\mu(\tilde{p}_{k})-\mu(\tilde{q}_{l})+N-1$.
		\end{prop}
		\begin{proof}
			We first show that $\mathcal{M}^{*}_{N,l,k}(\mathbf{A};H,J_Y)$ of  $J_Y \in \mathcal{J}_{(Y,\mathcal{L})}^{\mathrm{reg}}$ is a smooth manifold of dimension $N \dim L+\mu_L(\mathbf{A})$ where $\mu_L(\mathbf{A})=\sum\limits_{i=1}^{N}\mu_L(A_i)$. We abbreviate $Z=\R\times[0,1]$ and $\partial Z=\R\times\{0,1\}$, and consider the linearized operator associated to a Floer strip $\tilde{v}=(b,v)$:
			\[
				D_{\tilde{v}}: W^{1,p,\delta}_{V_\pm}\big((Z, \partial Z),(\tilde{v}^*T(\R \times Y), \tilde{v}|_{\partial Z}^*T(\R\times \mathcal{L})\big) \to
				 L^{p,\delta}\big(Z,\Omega^{0,1}_{Z}\otimes (\tilde{v}^*T(\R \times Y))\big)		
			\] 
			for $p>2$ and sufficiently small $\delta>0$.  
			Here, $V_+= T_{\tilde x}\mathcal{L}$ and $V_-= T_{\tilde y}\mathcal{L}$ for $\tilde x=\ev_{+}(v)(0)$ and $\tilde y=\ev_{-}(v)(0)$. The kernels of the associated asymptotic operators are precisely given by $\{d\varphi_{X_H}^t(v_\pm)\}_{t\in[0,1]}$ for $v_\pm\in V_\pm$, where $\varphi_{X_H}^t$ denotes the flow of $X_H$. The notation $W^{1,p,\delta}_{V_{\pm}}$ means that it consists of sections $\sigma$ of regularity $W_{\mathrm{loc}}^{1,p}$ such that $e^{\pm \delta s}(\sigma(\pm s,t) - \hat{v}_\pm(\pm s,t))$ is of class $W^{1,p}$ near the corresponding ends for some $v_\pm\in V_\pm$. Here $\hat{v}_\pm$ is a smooth section satisfying $\hat{v}_\pm(\pm s,t)=d\varphi_{X_H}^t(v_\pm)$ for large $s$.
			Writing $w := \pi \circ v$, we consider the decomposition 
			\begin{align}\label{eq:isom_bundles}
				\begin{split}
					\tilde{v}^*T(\R \times Y) &\cong \tilde{v}^* (\R \partial_r \oplus \R R)\oplus w^*T\Sigma, \\[.5ex]
					\tilde{v}|_{\partial Z}^* T(\R \times \mathcal{L}) &\cong \tilde{v}|_{\partial Z}^*(\R \partial_r \oplus 0)\oplus w|_{\partial Z}^* TL.
				\end{split}
			\end{align}
			As in \cite[Lemma 5.22]{DL2} or in \cite[Equation (3.19)]{BKK24}, $D_{\tilde{v}}$ is written as
			\begin{equation}\label{eq:D_decomposition}
				D_{\tilde{v}}= 
				\begin{pmatrix}
					D_{\tilde{v}}^\mathrm{v} & L_\mathrm{cpt} \\[.5ex]
					0 &  D_{\tilde{v}}^\mathrm{h}
				\end{pmatrix}
			\end{equation}
			with respect to the splitting in \eqref{eq:isom_bundles}. Here $L_\mathrm{cpt}$ is a compact operator. The vertical part $D_{\tilde{v}}^\mathrm{v}$ has the form 
			\[
				D_{\tilde{v}}^\mathrm{v}: W^{1,p,\delta}((Z,\partial Z),(\C,\R))\to L^{p,\delta}(Z,\C),\quad 				D_{\tilde{v}}^\mathrm{v}\sigma:=\partial_s\sigma+J_0\partial_t\sigma+\begin{pmatrix}
					h''(e^b)e^b&0 \\
					0 &  0
				\end{pmatrix}	\sigma.
			\]
			The horizontal part $D_{\tilde v}^\mathrm{h}$ is the Cauchy--Riemann operator
			\[
				D_{\tilde{v}}^\mathrm{h}=\nabla_w^{0,1}: W^{1,p,\delta}_{V_\pm}\big((Z, \partial Z),(w^*T\Sigma, w|_{\partial Z}^*TL)\big) 
				\to L^{p,\delta}\big(Z,\Omega^{0,1}_{Z}\otimes w^*T\Sigma\big).
			\]
			The kernels of the asymptotic operators of $D_{\tilde{v}}^\mathrm{h}$ are $T_xL$ and $T_yL$ for $x=\pi(\tilde x)$ and $y=\pi(\tilde y)$, which justifies the notational abuse $V_+=T_{\tilde x}\mathcal{L}\cong T_x L$ and $V_-=T_{\tilde y}\mathcal{L}\cong T_y L$.
			As mentioned in \eqref{eq:extend_disk}, we could think of $w$ as a $J_\Sigma$-holomorphic map from the disk $(\D,\partial\D)$. Let $D_{w}$  be the associated  Cauchy--Riemann operator considered in \eqref{eq:CR_base}. Then the kernel and cokernel of the horizontal part $D_{\tilde{v}}^\mathrm{h}$ are canonically isomorphic to those of $D_{w}$, see \cite[page 29]{DL2} for the proof in the periodic case. Since $J_\Sigma\in \mathcal{J}_{(\Sigma, L)}^{\mathrm{reg}}$, $D_w$ and thus $D_{\tilde{v}}^\mathrm{h}$ are surjective. 
			
			Next, we show that $D_{\tilde{v}}^\mathrm{v}$ is surjective. The asymptotic operators of $D_{\tilde{v}}^\mathrm{v}$ are given by
			\[ 
			A_\pm = -J_0\frac{d}{dt}- \begin{pmatrix} h''(e^{b_\pm})e^{b_\pm} & 0 \\ 0 & 0\end{pmatrix} : W^{1,2}( ([0,1],\{0,1\}),(\C,\R)) \to L^2([0,1], \C),
			\]
			where $J_0$ is the complex structure on $\C$ and  $b_\pm:=\ev_\pm(b)$ for $\ev_\pm$ defined after \eqref{eq:Floer_equation_strip}.  
			Let $\alpha_-(A_{\pm})\in\frac{1}{2}\Z$ be the maximum relative winding number of eigenfunctions associated with the negative eigenvalues of $A_{\pm}$. By \cite{FK16}, we have  
			\begin{equation}\label{eq:RS_line_bundle}
				\mu_\RS(A_{\pm})=2\alpha_-(A_{\pm})+\frac{1}{2}.
			\end{equation}
			Since $h''(e^{b_{\pm}})>0$, Lemma \ref{lem:index} below implies that $\alpha_-(A_{\pm})=0$, and hence $\mu_\RS(A_{\pm})=\frac{1}{2}$.  
			As $\ind(D_{\tilde{v}}^\mathrm{v})=\mu_\RS(A_{+})-\mu_\RS(A_{-})=0$, we have $\mathrm{coker\,}  D_{\tilde{v}}^\mathrm{v}=0$ if and only if $\ker D_{\tilde{v}}^\mathrm{v}=0$. Assume that there is nonzero $\zeta\in \ker D_{\tilde{v}}^\mathrm{v}$. Due to the asymptotic formula of $\zeta$ established in \cite[Theorem 3.12]{Abb04}, the asymptotic relative winding numbers $\wind_{\pm\infty}(\zeta)$ satisfy
			\begin{equation}\label{eq:asymp_winding}
				\wind_{+\infty}(\zeta)\leq \alpha_-(A_{+})=0,\qquad \wind_{-\infty}(\zeta)\geq \alpha_+(A_{-})=\frac{1}{2},
			\end{equation}
			where $\alpha_+(A_{-})=\alpha_-(A_{-})+\frac{1}{2}$ is the minimum relative winding number of eigenfunctions associated with the positive eigenvalues of $A_{-}$. 
			The algebraic count of zeros of $\zeta$ equals $\wind_{+\infty}(\zeta)- \wind_{-\infty}(\zeta)<0$, which contradicts the fact that all zeros of $\zeta$ are positive. This proves that $\ker D_{\tilde{v}}^\mathrm{v}=0$ and $D_{\tilde{v}}^\mathrm{v}$ is surjective. 
			
			Therefore, $D_{\tilde{v}}$ is surjective and its index is computed as 
			\[
			\ind(D_{\tilde{v}})=\ind(D_{\tilde{v}}^\mathrm{v})+\ind({D}_w)=\dim L+\mu_L([w]).
			\]
			This proves that $\mathcal{M}^*_{N,l,k}(\mathbf{A};H,J_Y)$ is a smooth manifold of dimension $N\dim L+\mu_L(\mathbf{A})$.

			It remains to prove the following transversality result:
			\[
			\left({\ev} : \mathcal{M}^*_{N,l,k}(\mathbf{A};H,J_Y) \to \mathcal{L}^{2N}\right) \pitchfork \left(W^u_{Z_\mathcal{L}}(\tilde{q}) \times \Delta_{Z_\mathcal{L}}^{N-1} \times W^s_{Z_\mathcal{L}}(\tilde{p})\right).
			\]
			By \eqref{eq:D_decomposition} and the fact that $Z_\mathcal{L}$ is the horizontal lift of $Z_L$, it suffices to achieve this transversality in the horizontal direction, or equivalently in the projection to $\Sigma$. The latter is part of Proposition \ref{prop:base_transv}.(a). 	
			The dimension of $\mathcal{M}^*_{N}(\tilde{q}_{l},\tilde{p}_{k};\mathbf{A};\mathcal{D}_H) $ is computed as
			\[
			\begin{split}
				&\dim\mathcal{M}^*_N(\tilde{q}_{l},\tilde{p}_{k};\mathbf{A};{\mathcal{D}}_{Y}) \\[.5ex]
				&\quad =N \dim L +\mu_L(\mathbf{A}) - (\dim\mathcal{L}- \ind_{f_\mathcal{L}}(\tilde{q})+ (N-1)(\dim \mathcal{L}-1)+\ind_{f_\mathcal{L}}(\tilde{p}))\\[.5ex]
				&\quad =\ind_{f_\mathcal{L}}(\tilde{q})-\ind_{f_\mathcal{L}}(\tilde{p})+\mu_L(\mathbf{A})+N-1\\[.5ex]
				&\quad = \mu(\tilde{p}_{k}) -\mu(\tilde{q}_{l})+N-1,
			\end{split} 
			\]
			where the last line follows from Proposition \ref{prop:maslov index of projection}. 
			This completes the proof.
		\end{proof}
		
		Let $J_Y \in \mathcal{J}_{(Y,\mathcal{L})}^{\mathrm{reg}}$.
		For $N=0$, let
		\[
		\mathcal{M}_{N=0}(\tilde{q}_{l}, \tilde{p}_{k}):= \begin{cases*} W^u_{Z_{\mathcal{L}}}(\tilde{q})\cap W^s_{Z_{\mathcal{L}}}(\tilde{p}) & if $k = l$, \\  \emptyset & if $k \neq l$. \end{cases*}
		\]
		There are a free $\R$-action on $\mathcal{M}_{N=0}(\tilde{q}_{l}, \tilde{p}_{k})$ given by $Z_{\mathcal{L}}$ and a free $\R^N$-action on $\mathcal{M}_N(\tilde{q}_{l}, \tilde{p}_{k};\mathbf{A};\mathcal{D}_H)$ given by the translation in the $s$-direction on each strip component.
		We define
		\begin{equation}\label{eq:moduli_leg}
			\begin{split}
				&\overline{\mathcal{M}}(\tilde{q}_{l}, \tilde{p}_{k}; \mathcal{D}_H):=\mathcal{M}_{N=0}(\tilde{q}_{l}, \tilde{p}_{k})/\R \cup \bigcup_{N\in \N}\mathcal{M}_N(\tilde{q}_{l}, \tilde{p}_{k};\mathbf{A};\mathcal{D}_H)/\R^N, \\[.5ex]
				&\overline{\mathcal{M}}^*(\tilde{q}_{l}, \tilde{p}_{k};\mathcal{D}_H):=\mathcal{M}_{N=0}(\tilde{q}_{l}, \tilde{p}_{k})/\R \cup \bigcup_{N\in \N}\mathcal{M}^*_N(\tilde{q}_{l}, \tilde{p}_{k};\mathbf{A};\mathcal{D}_H)/\R^N,
			\end{split}
		\end{equation}
		where the latter space has dimension $\mu(\tilde{p}_{k})-\mu(\tilde{q}_{l})-1$ by Proposition \ref{prop:transv_strip}.

		\begin{prop}\label{prop:simple_no_escape}
			Assume that $\mu(\tilde{p}_{k})-\mu(\tilde{q}_{l})\leq 2$.
			\begin{enumerate}[(a)]
				\item If $N_L\geq2$, all elements of $\overline{\mathcal{M}} (\tilde{q}_{l},\tilde{p}_{k};\mathcal{D}_H)$ are simple, i.e.,
				$\overline{\mathcal{M}} (\tilde{q}_{l},\tilde{p}_{k};\mathcal{D}_H)=\overline{\mathcal{M}}^{*} (\tilde{q}_{l},\tilde{p}_{k};\mathcal{D}_H)$.
				
				\item Assume $N_L>2$. Then, there exist $r_\pm\in\R$ such that every element in this moduli space has its image in $(r_-,r_+)\times Y$. Consequently, this moduli space is compact up to breaking.
			\end{enumerate}  
		\end{prop}
		The proof of the above proposition is deferred  to Section \ref{sec:proof_proposition}.

		\subsubsection{Orientation lines}\label{sec:orientation_lines}
		
		Let $\dot{\D} := \D\setminus \{-1\}$ be the punctured disk and let
		\[
		\epsilon_- : (-\infty,0] \times [0,1] \longrightarrow \dot\D
		\]
		be negative holomorphic strip-like coordinates near the boundary puncture $-1\in\D$. We take any $\tilde p_k\in \mathcal{L}_k^H$ for $(\tilde{p},k)\in\Crit f_\mathcal{L}\times m_L\Z$. A capping of $\tilde p_k$ is a smooth map 
		\[
		\hat x:(\dot{\D}, \partial \dot{\D})\longrightarrow(\R \times Y, \R\times \mathcal{L})
		\]  
		such that $\hat{x}(\epsilon_-(s,\cdot))$ converges to $\tilde p_k$ in the $C^\infty$-topology as  $s\to-\infty$. Instead of providing a general construction of an orientation line $\mathfrak{o}(\tilde p_k,\hat x)$, which involves all relevant choices of an almost complex structure and a connection as in \cite[Section 3.7]{Z1}, we define $\mathfrak{o}(\tilde p_k,\hat x)$ with a specific choice of an almost complex structure and a connection. To this end, we decompose 
		\begin{equation} \label{eq:decomposition_tangent}
			\hat x^*T(\R\times Y) \cong \hat x^* (\R \partial_r\oplus \R R)\oplus \hat x^*\xi \cong (\dot{\D} \times \C) \oplus w^* T\Sigma,
		\end{equation}
		where $w:= \pi_{\R\times Y}\circ \hat x$, with $\pi_{\R\times Y}:\R\times Y \to \Sigma$ denoting the natural projection induced by $\pi$, and where $\R \partial_r$ and $\R R$ correspond to $\R$ and $i\R$, respectively.
		Note that $\hat{x}^*T\mathcal{L}$ corresponds to $w^*TL$ through the isomorphism.  
		For $p>2$ and $\delta>0$, we consider an operator
		\[
		D_{\hat x}: W^{1,p,\delta}_V\big((\dot{\D}, \partial\dot{\D}),(\hat x^*T(\R \times Y), \hat{x}|_{\partial\dot{\D}}^*T(\R\times \mathcal{L})\big) \to L^{p,\delta}\big(\dot{\D},\Omega^{0,1}_{\dot{\D}}\otimes (\hat x^*T(\R \times Y))\big)
		\]
		which is of the form { $D_{\hat x}= \begin{pmatrix}
				D_{\hat x}^\mathrm{v} & L_\mathrm{cpt} \\
				0 &  D_{\hat x}^\mathrm{h}
			\end{pmatrix}$} with respect to the above decomposition. The horizontal part $D_{\hat x}^\mathrm{h}$ is the Cauchy--Riemann operator
		\[
		D_{\hat x}^\mathrm{h}=\nabla_w^{0,1}: W^{1,p,\delta}_V\big((\dot{\D}, \partial\dot{\D}),(w^*T\Sigma, w|_{\partial\dot{\D}}^*TL)\big) \to L^p\big(\dot \D,\Omega^{0,1}_{\dot{\D}}\otimes w^*T\Sigma\big)
		\]
		for the connection $\nabla_w=w^*\nabla$ on $w^* T\Sigma$, cf.~\eqref{eq:CR_base}. Here, $V$ is the kernel of the asymptotic operator associated to the puncture $-1\in \D$, which is canonically isomorphic to $T_{\tilde p}\mathcal{L}\cong T_pL$. The notation $W^{1,p,\delta}_V$ is explained in the preceding section.  
		The vertical part $D_{\hat x}^\mathrm{v}$ is given by
		\[
		D_{\hat x}^\mathrm{v}: W^{1,p,\delta}((\dot\D,\partial{\dot\D}),(\C,\R))\to L^{p,\delta}(\dot\D,\C),\qquad \sigma\mapsto  \partial_s \sigma + J_0 \partial_t \sigma + B\sigma
		\]
		for some smooth map $B:\dot\D\to \operatorname{End}_{\R}(\C)$ such that $B(\epsilon_-(s,t))= \begin{pmatrix}
			h''(e^{r_x})e^{r_x} & 0 \\
			0 & 0
		\end{pmatrix}$ 
		for $s \ll 0$, $B(\epsilon_-(s,t))=0$ for $s$ sufficiently close to $0$, and $B=0$ on the complement of $\epsilon_-$. The off-diagonal term $L_\mathrm{cpt}$ denotes an arbitrary compact operator. We include it to account for the compact operator that appears in \eqref{eq:D_decomposition}. For a sufficiently small $\delta>0$, the operator $D_{\hat x}$ is Fredholm.
		
		Using the evaluation map
		\begin{equation*}
		\begin{split}
			&\ev_{-1} :  W^{1,p,\delta}_V\big((\dot{\D}, \partial\dot{\D}),(\hat x^*T(\R \times Y), \hat{x}|_{\partial\dot{\D}}^*T(\R\times \mathcal{L}))\big) \longrightarrow T_{\tilde{p}} \mathcal{L},\\[.5ex]
			&  \ev_{-1}(\zeta):= \lim_{s\to -\infty} \pi_\mathcal{L}\circ \zeta(\epsilon_{-}(s,0)),
		\end{split}			
		\end{equation*}
		where $\pi_\mathcal{L}:\R\times T_{\tilde{p}}\mathcal{L}\to T_{\tilde{p}}\mathcal{L}$ is the obvious projection,  
		we define the Fredholm operator
		\begin{align*}
			&D_{\hat{x}} \# T_{\tilde{p}}W^{u}_{Z_{\mathcal{L}}} (\tilde{p}) := 	D_{\hat{x}}|_{(\ev_{-1})^{-1} ( T_{\tilde{p}} W^{u}_{Z_{\mathcal{L}}} (\tilde{p}))}.
		\end{align*}
		We denote the orientation line of $\det(D_{\hat{x}} \# T_{\tilde{p}}W^{u}_{Z_{\mathcal{L}}} (\tilde{p}))$ by $\mathfrak{o}(\tilde{p}_k,\hat x)$. Note that $\R\times\mathcal{L}$ is relatively $\Pin$ in $\R\times Y$ since $L$ is so in $\Sigma$, see Remark \ref{rem:pin_lift}. Therefore, if $\hat x'$ is another capping of $\tilde{p}_k$ such that gluing  $\hat x$ and $\hat{x}'$ with the reversed orientation along $\tilde{p}_k$ yields a disk that is  trivial in $\pi_2(\R\times Y,\R\times \mathcal{L})$, then, by \cite[Lemma 3.6]{Z1}, there is a canonical isomorphism 
		\[
		\mathfrak{o}(\tilde{p}_k,\hat x)\cong \mathfrak{o}(\tilde{p}_k,\hat x').
		\] 
		
		\begin{remark}\label{rem:pin_lift}
			Let $\sigma$ be a relative $\Pin(n)$-structure on $L\subset\Sigma$ as in Remark \ref{rem:pin}. We define a relative $\Pin(n+1)$-structure $\tilde\sigma$ on $\R\times L\subset \R\times Y$ as follows.
			Let $(g,\beta)\in \check{\mathrm{C}}^1(\mathcal{U};\Pin(n)) \times \check{\mathrm{Z}}^2(\mathcal{V}; \Z_2) $ be a representative of $\sigma$. 
			The projection  $\pi : Y \to \Sigma$ induces a relative $\Pin(n)$-structure on $\mathcal{L}\subset Y$ defined by the pair $(\pi^*g,\pi^*\beta) \in \check{\mathrm{C}}^1(\tilde{\mathcal{U}};\Pin(n)) \times \check{\mathrm{Z}}^2(\tilde{\mathcal{V}};\Z_2)$ for some good covers $\tilde{\mathcal{V}}$ on $Y$ and $\tilde{\mathcal{U}}$ on $\mathcal{L}$ finer than $\{\pi^{-1}(V)\}_{V\in \mathcal{V}}$ and $\{\pi^{-1}(U)\}_{U \in \mathcal{U}}$, respectively. Consider the  diagram 
			\begin{equation*}
				\begin{tikzcd}
					\Z_2 & \Pin(n) & O(n) \\ \Z_2 & \Pin(n+1) & O(n+1),
					\arrow[from=1-1, to=1-2]
					\arrow[from=1-1, to=2-1, "="]
					\arrow[from=1-2, to=1-3]
					\arrow[from=1-2, to=2-2, "i"]
					\arrow[from=1-3, to=2-3, "i"]
					\arrow[from=2-1, to=2-2]
					\arrow[from=2-2, to=2-3]
				\end{tikzcd}
			\end{equation*}
			where the vertical maps $i$ are the obvious inclusions with respect to $\R^{n+1}=\R\oplus\R^n$. 
			Extending $(i_*\pi^*g,\pi^*\beta) \in \check{\mathrm{C}}^1(\tilde{\mathcal{U}};\Pin(n+1)) \times \check{\mathrm{Z}}^2(\tilde{\mathcal{V}};\Z_2)$ trivially in the $\R$-direction, we obtain a relative $\Pin(n+1)$-structure $\tilde\sigma$ on $\R\times\mathcal{L}\subset \R\times Y$. We call $\tilde \sigma$ the lift of $\sigma$.
		\end{remark}

		The projection $w=\pi_{\R\times Y}\circ \hat x$ smoothly extends over $-1$, and thus we may also regard $w$ as a smooth map $w:(\D,\partial\D)\to (\Sigma,L)$.
		Then, the kernel and cokernel of the horizontal part $D_{\hat x}^\mathrm{h}$ are canonically isomorphic to those of the Cauchy--Riemann operator $D_{w}$ in \eqref{eq:CR_base}, see \cite[page 29]{DL2} for the proof in the periodic case. This yields an isomorphism of determinant line bundles
		\[
		\det(D_{\hat{x}})\cong  \det(D_{\hat{x}}^\mathrm{v})\otimes \det(D_{\hat x}^\mathrm{h})\cong \det(D_{\hat{x}}^\mathrm{v})\otimes \det(D_{w}).
		\]
		Moreover, through the isomorphism \eqref{eq:decomposition_tangent}, we have
		$T_{\tilde{p}} W^{u}_{Z_{\mathcal{L}}}(\tilde{p}) \cong 0 \oplus T_p W^{u}_{Z_L}(p)$
		and 
		\begin{equation}
			\label{eq:det_decomp}
			\det(D_{\hat{x}} \# T_{\tilde{p}}W^{u}_{Z_{\mathcal{L}}} (\tilde{p}))\cong \det(D_{\hat{x}}^\mathrm{v}) \otimes \det(D_{w}\# T_pW^u_{Z_L}(p)).
		\end{equation}
			
		\medskip
		
		Next, we observe that  $\det(D_{\hat{x}}^\mathrm{v})$ admits a canonical orientation. 
		We consider positive holomorphic strip-like coordinates $\epsilon_+ : [0,\infty) \times [0,1] \to \dot\D$ near the puncture $-1\in\D$. 
		We define an operator $D_B^+: W^{1,p,\delta}((\dot\D,\partial{\dot\D}),(\C,\R))\to L^{p,\delta}(\dot\D,\C)$ analogously to $D_{\hat x}^\mathrm{v}$ with $\epsilon_-$ replaced by $\epsilon_+$.  
		Gluing $(\dot\D,\epsilon_+)$ and $(\dot\D,\epsilon_-)$ along the respective cylindrical ends, we obtain $\D$. We also glue the operators $D_B^+$ and $D_{\hat x}^\mathrm{v}$, which have the same asymptotic operator, and obtain an operator $D_B:W^{1,p}((\D,\partial \D),(\C,\R))|_{(\ev_{-1})^{-1}(0)}\to L^p(\D,\C)$. We therefore have an isomorphism of determinant lines 
		\begin{equation}
			\label{eq:gluing_iso}
			\det (D_B^+) \otimes \det (D_{\hat x}^\mathrm{v} ) \cong \det (D_B).
		\end{equation}
		 We homotope $D_B$ to the Cauchy--Riemann operator $\bar{\partial}_{(\C,\R)}\# 0:W^{1,p}((\D,\partial{\D}),(\C,\R))|_{(\ev_{-1})^{-1}(0)}\allowbreak\to L^p(\D,\C)$ through Fredholm operators. Since $\bar{\partial}_{(\C,\R)}\#0$ is an isomorphism, $\det(\bar{\partial}_{(\C,\R)}\#0)$ has a canonical orientation, and this yields an orientation on $\det (D_B)$. The computation in the proof of Proposition \ref{prop:transv_strip} shows that $D_B^+$ has index zero and is surjective, and hence, it is an isomorphism. Thus, $\det (D_B^+)$ admits a canonical orientation. Using \eqref{eq:gluing_iso}, we obtain a canonical orientation on $\det (D_{\hat x}^\mathrm{v})$.

		The isomorphism in \eqref{eq:det_decomp} and the  orientation on $\det (D_{\hat x}^\mathrm{v} \#0)$ induces  an isomorphism of orientation lines
		\begin{equation}\label{eq:projection_orientation}
			f_{\hat x} : \mathfrak{o}(\tilde{p}_k,\hat x) \longrightarrow \mathfrak{o}(p,w).
		\end{equation}
		
		Moreover, for any $(\tilde p,k),(\tilde q,l)\in \Crit f_{\mathcal{L}}\times m_L\Z$ satisfying $\mu(\tilde{p}_{k})-\mu(\tilde{q}_{l})=1$, every element $[\mathbf{v}] \in \overline{\mathcal{M}}^*(\tilde{q}_{l}, \tilde{p}_{k};\mathcal{D}_H)$ induces an isomorphism 
		\begin{equation}
			\label{eq:rfh_bdry}
			C([\mathbf{v}]) :  \mathfrak{o}(\tilde{p}_{k},\hat x) \longrightarrow \mathfrak{o}(\tilde{q}_{l},\hat y) 
		\end{equation}
		where $\hat y$ is a capping of $\tilde{q}_{l}$ obtained by gluing $\hat x$ with $\mathbf{v}$, see \cite[Section 4.2.2]{Z1} for details.

		\begin{prop}\label{prop:orientation_commute}
			With the notation above, let $\hat{x},\hat{x}'$ be two cappings of $\tilde{p}_k \in\mathcal{L}_k^H$. Then,
			the relative $\mathrm{Pin}^{\pm}(n+1)$-structure $\tilde{\sigma}$ on $\R \times \mathcal{L} \subset \R\times Y$ determines an isomorphism 
			\begin{equation}
				\label{eq:pin_iso_rfh}
				\phi^{\tilde{\sigma}}_{(\hat{x}, \hat{x}')} : \mathfrak{o}({\tilde{p}_k,\hat{x}})\longrightarrow \mathfrak{o}({\tilde{p}_k,\hat{x}'}).
			\end{equation}
			Moreover, this isomorphism fits into the following commutative diagram
			\[
			\begin{tikzcd}
				& \mathfrak{o}({\tilde{q}_{l},\hat{y}})  \arrow[rr, "\phi^{\tilde{\sigma}}_{(\hat{y}, \hat{y}')}", swap, near start] \arrow[ddd, "f_{\hat{y}}", dotted] &                                  & \mathfrak{o}({\tilde{q}_{l},\hat{y}'})  \arrow[ddd, "f_{\hat{y}'}"]  \\
				\mathfrak{o}({\tilde{p}_{k},\hat{x}}) \arrow[ru, "\Cv"] \arrow[rr, "\phi^{\tilde{\sigma}}_{(\hat{x}, \hat{x}')}", swap, near start] \arrow[ddd, "f_{\hat{x}}"] &                                                                                                                                         & \mathfrak{o}({\tilde{p}_{k},\hat{x}'}), \arrow[ru, "\Cv"] \arrow[ddd, "f_{\hat{x}'}"] &                                                                    \\
				& & &\\
				& \mathfrak{o}({q,v}) \arrow[rr, "\psi^{\sigma}_{(v, v')}",dotted, swap, near start]                                                                                     &                                  & \mathfrak{o}({q,v'})                                   \\
				\mathfrak{o}({p,w}) \arrow[ru, "\Cw",dotted] \arrow[rr, "\psi^{\sigma}_{(w, w')}",swap, near start]                           &                                                                                                                                         & \mathfrak{o}({p,w'})     \arrow[ru,, "\Cw"]                       &                                                                   
			\end{tikzcd}      
			\]
			where
			\begin{itemize}
				\item the capping $\hat{y}$ of $\tilde{q}_l$ is obtained by gluing $\hat{x}$ with $\mathbf{v}$, and likewise for $\hat{y}'$,
				\item $w= \pi_{\R\times Y}\circ \hat{x}$ and $v=\pi_{\R\times Y}\circ \hat{y}$, and similarly for $w'$ and $v'$,
				\item $\mathbf{w}=\Pi(\mathbf{v})$, with $\Pi$ as in \eqref{eq:moduli_projection_strip},
				\item the vertical maps are given by \eqref{eq:projection_orientation},
				\item the maps $c([\mathbf{w}])$ and $c([\mathbf{v}])$ are defined just above \eqref{eq:pearl_bdry} and in \eqref{eq:rfh_bdry}, respectively,
				\item the horizontal maps on the bottom face are given by \eqref{eq:pin_iso}.
			\end{itemize} 
		\end{prop}
		
		\begin{proof} 
			The maps $\psi^{\sigma}_{(w,w')}$ and $\phi^{\tilde{\sigma}}_{(\hat{x}, \hat{x}')}$ are defined as in \cite[Proposition 7.4]{Z1}. We recall the construction of these maps and  show the commutativity of the front face of the diagram.
			We define disk maps
			\[
			\tilde{v}:=(-\hat{x})\#\hat{x}',\qquad v:=(-w)\#w'
			\]
			where the minus sign indicates reversal of orientation.
			Since the Maslov class of $\mu_\mathcal{L}$ on $\pi_2(Y,\mathcal{L})\cong\pi_2(\R\times Y,\R\times \mathcal{L})$  vanishes by Lemma \ref{lem:Maslov_vanish}, we have $\mu_L(v)=\mu_{\mathcal{L}}(\tilde{v})=0$. We take any unitary trivialization of $v^*T\Sigma$. Then $v|_{\partial\D}^*TL$ induces a loop $F_v$ of Lagrangian subspaces in $\C^{n}$, which is contractible since $\mu_L(v)=0$. 
			
			Let $\Omega_0(\mathrm{Lag}(\C^{n}))$ be the space of contractible loops of Lagrangian subspaces in $\C^{n}$. On one hand, for each $F\in\Omega_0(\mathrm{Lag}(\C^{n}))$, we denote by $\mathfrak{o}_{F}$ the orientation line of $\bar{\partial}_{(\C^n,F)}\#0$, where $\bar{\partial}_{(\C^n,F)}:W^{1,p}((\D,\partial\D),(\C^n,F))\to L^p(\D,\C^n)$ is the standard Cauchy--Riemann operator and $\#0$ refers to the restriction to $(\mathrm{ev}_{-1})^{-1}(0)$ as before. 
			On the other hand, viewing $F\in\Omega_0(\mathrm{Lag}(\C^{n}))$ as a real vector bundle over $S^1$, we associate to $F$ two $\Pin(n)$-structures. Then the maps 
			\[
			\mathfrak{o}_{F}\mapsto F,\qquad \{\text{$\Pin(n)$-structures on $F$}\}\mapsto F
			\] 
			define isomorphic double covers over $\Omega_0(\mathrm{Lag}(\C^{n}))$. Over a constant loop $F_c \in \Omega_0(\mathrm{Lag}(\C^{n}))$, we have the canonical  orientation $o_{F_c}^+$ of $\mathfrak{o}_{F_c}$ as $\bar{\partial}_{(\C^n,F_c)}\#0$ is an isomorphism, and the trivial $\Pin(n)$-structure $\mathrm{p}_{F_c}^\mathrm{triv}$. 
			
			Let $o_{F_v}$ be one of the two orientations of $\mathfrak{o}_{F_v}$ that belongs to the same component as $o_{F_c}^+$. The relative $\Pin(n)$-structure $\sigma$ on $L\subset\Sigma$ determines a $\Pin(n)$-structure $\sigma_{F_v}$ of $F_v$. If $\sigma_{F_v}$ lies in the same component as $\mathrm{p}_{F_c}^\mathrm{triv}$, then we choose the orientation $o_{F_v}$, and otherwise   we choose $-o_{F_v}$.
			Finally, the isomorphism $\psi^{\sigma}_{(w,w')}$ is given by
			\begin{align*}
				\det(D_w \# T_p W^{u}
				_{Z_L}(p)) &\to\det(D_w \#  T_p W^{u}
				_{Z_L}(p))\otimes \det(D_v\#0) \cong \det(D_{w'} \# T_{p}W^{u}_{Z_L}(p)), \\[.5ex]
				o_{(p,w)} &\mapsto o_{(p,w)}\otimes (\pm o_{F_v})
			\end{align*} 
			where the sign $\pm$ is determined by $\sigma$ as mentioned above.

			Similarly, we choose a unitary trivialization of $\tilde{v}^*T(\R\times Y)$. Then $\tilde{v}|_{\partial{\D}}^* T(\R\times\mathcal{L})$ gives a contractible loop $F_{\tilde{v}}\in\Omega_0(\mathrm{Lag}(\C^{n+1}))$. The relative $\Pin(n+1)$-structure $\tilde{\sigma}$ determines an orientation $\pm o_{F_{\tilde v}}$ of $\det(D_{\tilde v}\#0)$, and this defines an isomorphism 
			\[
			\phi^{\tilde{\sigma}}_{(\hat{x},\hat{x}')}: \det (D_{\hat{x}} \# T_{\tilde{p}}W^{u}_{Z_{\mathcal{L}}} (\tilde{p}))\longrightarrow \det (D_{\hat{x}'} \# T_{\tilde{p}}W^{u}_{Z_{\mathcal{L}}} (\tilde{p})) 
			\]
			as above.	Due to \eqref{eq:det_decomp}, the front page of the diagram in the statement is isomorphic to 
			\[	 	
			\begin{tikzcd}
				\det(D_{\hat{x}}^\mathrm{v}) \otimes \det(D_{w}\# T_pW^u_{Z_L}(p))
				& {\det(D_{\hat{x}'}^\mathrm{v} )\otimes \det(D_{w'}\# T_{p}W^{u}_{Z_{L} } (p))}
				\\
				{\det(D_{w} \# T_{p}W^{u}_{Z_{L}} (p))} & {\det(D_{w'} \# T_{p}W^{u}_{Z_{L}} (p))}.
				\arrow["{\phi^{\tilde{\sigma}}_{(\hat{x},\hat{x}')}}", from=1-1, to=1-2]
				\arrow["{f_{\hat{x}}}"', from=1-1, to=2-1]
				\arrow["{f_{\hat{x}'}}", from=1-2, to=2-2]
				\arrow["{\psi^{\sigma}_{(w,w')}}"', from=2-1, to=2-2]
			\end{tikzcd}
			\]
			To show the commutativity of this diagram, we take a unitary trivialization of $\tilde{v}^*T(\R\times Y)$, which respects the splitting $\tilde{v}^*T(\R\times Y)  \cong ({\D} \times \C) \oplus v^* T\Sigma$ as in \eqref{eq:decomposition_tangent}. Then, $F_{\tilde{v}}\in\Omega_0(\mathrm{Lag}(\C^{n+1}))$ corresponding to $\tilde{v}|_{\partial{\D}}^* T(\R\times\mathcal{L})$ has the form $\R\oplus F_v$, where $F_{v}\in\Omega_0(\mathrm{Lag}(\C^{n}))$ corresponds to $v|_{\partial{\D}}^* TL$. Here $\R$ denotes the real part of the first $\C$-factor of $\C^{n+1}$. Furthermore, the relative $\Pin(n+1)$-structure $\tilde\sigma$ induces a $\Pin(n+1)$-structure $\mathrm{p}_1^\mathrm{triv}\otimes\sigma_{F_v}$ on $F_{\tilde{v}}=\R\oplus F_v$, where $\mathrm{p}_1^\mathrm{triv}$ denotes the trivial $\Pin(1)$-structure on the trivial line bundle over $S^1$. 
			
			The vertical part of $\phi^{\tilde\sigma}_{(\hat{x},\hat{x}')}$ is given by the gluing isomorphism
			\[
			\det (D_{\hat{x}}^\mathrm{v} ) \otimes \det (D_{\tilde{v}}^\mathrm{v} \#0) \longrightarrow  \det (D_{\hat{x}'}^\mathrm{v})
			\]
			together with the orientation on $\det (D_{\tilde{v}}^\mathrm{v} \#0) $ determined by the relative $\Pin$-structure $\tilde\sigma$. Here the vertical part $D_{\tilde{v}}^\mathrm{v}$ of $D_{\tilde{v}}$ is homotopic to the standard Cauchy--Riemann operator $\overline{\partial}_{(\C,\R)}$. Since the associated $\Pin(1)$-structure is trivial, $\det(D_{\tilde{v}}^\mathrm{v}\#0)$ is oriented by the canonical orientation. Then, one can readily see that  $\phi^{\tilde\sigma}_{(\hat{x},\hat{x}')}$ maps the canonical orientation on $\det(D_{\hat{x}}^\mathrm{v})$ to that on $\det(D_{\hat{x}'}^\mathrm{v})$. This proves that the above diagram commutes. 
			
			The commutativity of the top and bottom faces of the diagram in the statement is established in \cite{Z1}. 
			The commutativity of the remaining faces follows by arguments similar to those given above.
		\end{proof}
		
		Let $\tilde{p}_k\in \mathcal{L}_k^H$ with $(\tilde{p}, k) \in \Crit f_{\mathcal{L}} \times m_L\Z$. Using \eqref{eq:pin_iso_rfh}, we define
		\begin{equation}\label{eq:generator_leg}
			\mathfrak{o}(\tilde{p}_{k}):= \varinjlim_{\hat{x}} \mathfrak{o}(\tilde{p}_k,\hat{x}).
		\end{equation}
		In view of Proposition \ref{prop:orientation_commute}, and by abuse of notation, we write $C([\mathbf{v}])$ for the induced isomorphism
		\[
		C([\mathbf{v}]):\mathfrak{o}(\tilde{p}_{k})\longrightarrow \mathfrak{o}(\tilde{q}_{l}).
		\]
		Following \eqref{eq:index}, we endow $\mathfrak{o}(\tilde{p}_k)$ with the grading
		\begin{equation}\label{eq:generator_leg_grading}
			\deg \mathfrak{o}(\tilde{p}_k):= \mu(\tilde{p}_k)= \dim L+ \frac{ 2\tau_\Sigma k}{d}  - \mathrm{ind}_{f_L} (p).
		\end{equation}
		
		\subsubsection{Floer chain complex}\label{sec:floer_chain}
		For real numbers $a<b$ and $H \in \mathcal{H}$, we consider the set
		\[
		\w^{(a,b)}_\mathcal{L}(H):= \big\{ k \in m_L \Z \,|\, \exists \, r \in (-\eta,\eta) \text{ with } h'(e^r) = \tfrac{k}{d} \text{ and } a< e^r h'(e^r) - h(e^r) <b\big\},
		\]
		where $H(r,y)=h(e^r)$ for $(r,y)\in \R\times Y$, as in Section \ref{sec:Ham_chords}.
		In view of Proposition \ref{prop:contractible}, \eqref{eq:beta}, and \eqref{eq:action}, $(\tilde p,k)\in \Crit f_{\mathcal{L}}\times \m^{(a,b)}_\mathcal{L}(H)$ determines $\tilde p_k=(r^H_k,c_{\tilde{p}_k}) \in \Crit \mathcal{A}_H $ such that $c_{\tilde p_k}(0)=\tilde p$,  $\mathfrak{w}(c_{\tilde p_k})=k$,  $a<\mathcal{A}_H(\tilde p_k)<b$, and $\tilde{p}_k$ admits a capping. 
		We write $\mathcal{L}^{\tilde{\sigma}}$ for the Legendrian submanifold $\mathcal{L}$ equipped with the lifted $\Pin$-structure $\tilde{\sigma}$.
		The Floer complex associated to ${\mathcal{D}}_H=(f_\mathcal{L},Z_\mathcal{L},H,J_Y)$  on $\mathcal{L}^{\tilde{\sigma}}$ with action window $(a,b)\subset \R$ is defined by 
		\[
		\FC^{(a,b)}_*(\mathcal{L}^{\tilde{\sigma}};\mathcal{D}_H) := \bigoplus_{(\tilde{p},k)}\mathfrak{o}(\tilde{p}_k).
		\]
		where the direct sum ranges over $(\tilde{p},k)\in\mathrm{Crit} f_{\mathcal{L}}\times \m^{(a,b)}_\mathcal{L}(H)$. 
		The boundary map is defined by 
		\begin{equation}\label{eq:differential_Leg}
			\partial_{\mathcal{D}_H} : \FC^{(a,b)}_*(\mathcal{L}^{\tilde{\sigma}};\mathcal{D}_H) \to \FC^{(a,b)}_{*-1}(\mathcal{L}^{\tilde{\sigma}};\mathcal{D}_H), \quad \partial_{\mathcal{D}_H} := \bigoplus_{(\tilde{p},\tilde{q},k,l)} \sum_{ [\mathbf{v}] } C([\mathbf{v}]).
		\end{equation}
		The direct sum ranges over all $(\tilde{p},k),(\tilde{q},l)\in \Crit f_\mathcal{L}\times \m^{(a,b)}_\mathcal{L}(H)$ with $\mu(\tilde{p}_k)-\mu(\tilde{q}_l)=1$.
		The sum runs over $[\mathbf{v}]\in \overline{\mathcal{M}}^*(\tilde{q}_{l}, \tilde{p}_{k};\mathcal{D}_H)$, where each  $C([\mathbf{v}])$ is an isomorphism from $\mathfrak{o}(\tilde{p}_k)$ to $\mathfrak{o}(\tilde{q}_l)$.  Note that $X_{H}$ has chords other than those contained in $(-\eta,\eta)\times Y$. Nevertheless, by \cite[Lemmas 2.2 and 2.3]{CO18}, Floer cylinders asymptotic to chords in $(-\eta,\eta)\times Y$ do not break along chords outside this region. Together with Proposition \ref{prop:simple_no_escape}, this implies that  $\partial_{\mathcal{D}_H}$ is a well-defined boundary operator. We denote the resulting Floer homology by
		\[
		\FH_i^{(a,b)}(\mathcal{L}^{\tilde{\sigma}};\mathcal{D}_H) : = \H_i (\FC^{(a,b)}_*(\mathcal{L}^{\tilde{\sigma}};\mathcal{D}_H),\partial_{\mathcal{D}_H}).
		\]

		\subsection{Lagrangian Rabinowitz Floer homology}
		\subsubsection{Floer continuation strips}
		The set $\mathcal{H}$ of $\vee$-shaped Hamiltonians has a strict partial order $\prec$ : $H_+ \prec H_-$ if $H_+ < H_-$ pointwise. 
		A monotone homotopy for $H_+ \prec H_-$ refers to a family of smooth functions $H_s : \R \times Y \to \mathbb{R}$, smoothly parametrized by $s\in \mathbb{R}$, satisfying the following requirements:
		\begin{enumerate}
			\item[(i)] {$H_s$ satisfies the condition \eqref{eq:Hamiltonian} for $\mathcal{H}$ with $a \in (0,\infty)$ depending smoothly on $s$, }
			\item[(ii)] $H_s = H_-$ for $s \leq -1$ and $H_s = H_+$ for $s\geq 1$, and $\frac{\partial}{\partial s} H_s \leq 0$.
		\end{enumerate}
		By the definition of $\mathcal{H}$, $H_s$ has the form $H_s(r,y) = h_s(e^r)$ for some smooth family of smooth functions $h_s : \mathbb{R}_{>0} \to \mathbb{R}$ and $X_{H_s}$ vanishes on $(-\infty, -\epsilon) \times Y$.
		
		Let $H_s$ be a monotone homotopy for $H_+ \prec H_-$, and let $J_{Y} \in \mathcal{J}_Y$. 
		A smooth solution  $\tilde{v}=(b,v) : (\mathbb{R} \times [0,1], \R \times \{0,1\} ) \to (\R \times Y,\R \times \mathcal{L})$ of
		\begin{equation}\label{eq:Floer_equation_conti}
			\partial_s \tilde{v} + J_{Y} (\tilde{v}) (\partial_t \tilde{v} - X_{H_s} (\tilde{v}))=0
		\end{equation}
		with the asymptotic condition
		\begin{align*}
			\ev_\pm(\tilde{v})=(\ev_\pm(b),\ev_\pm(v)) 
			:=\lim_{s\to \pm\infty} \tilde{v}(s, \cdot) \in \mathrm{Crit}\mathcal{A}_{H_\pm}
		\end{align*}
		is referred to as a {\it Floer continuation strip} with respect to $(H_s,J_Y)$. 
		
		\medskip
		
		Let $\mathbf{N}:=(N_-, N_+)\in (\mathbb{N}\cup \{0\})^2$, $N:= N_- + N_+ +1$, and 
		\[
		\mathbf{A}=(A^-_1,\dots, A^-_{N_-},A,A^+_1,\dots, A^+_{N_+}) \in  \pi_2(\Sigma,L)^N,
		\]
		where $A^-_i\neq 0$ and $A^+_j\neq 0$ for all $1\leq i\leq N_-$ and $1\leq j\leq N_+$. 
		Let $k_\pm \in m_L \mathbb{Z}$ with $r^{H_{\pm}}_{k_{\pm}} \in (-\eta,\eta)$ as in \eqref{eq:beta}.
		For a Floer continuation strip $\tilde{v}=(b,v)$, we use the same notation as in \eqref{eq:asymptotic_strip} and set
		\[
			\mathfrak{a}_{\pm}(\tilde{v}):= (\ev_\pm(b), \w(\ev_\pm(v))).
		\]

		\begin{definition}
			 The moduli space
			\begin{equation*}
				\mathcal{M}_{\mathbf{N},k_-,k_+}(\mathbf{A};H_s,J_{Y})
			\end{equation*}
			consists of tuples 
			$
			\mathbf{v}=(\tilde{v}^{-}_1,\dots,\tilde{v}^{-}_{N_-}, \tilde{v}, \tilde{v}^{+}_1,\dots,\tilde{v}^{+}_{N_+})$
			such that, for every $i=1,\dots, N_-$ and for every $j =1,\dots,N_+$,
			\begin{itemize}
				\item $\tilde{v}^-_i=(b^-_i,v^-_i)$ is a Floer strip with respect to $(H_-,J_Y)$ with $[\pi \circ v^-_i]=A^-_i$,
				\item $\tilde{v}=(b,v)$ is a Floer continuation strip with respect to $(H_s,J_Y)$ with $[\pi \circ v]=A$,
				\item $\tilde{v}^+_j=(b^+_j,v^+_j)$ is a Floer strip with respect to $(H_+,J_Y)$ with $[\pi \circ v^+_j]=A^+_j$,
			\end{itemize}
			and for every $i=1,\dots,N_- -1$ and $j=1,\dots,N_+ -1$,
			\begin{align*}
				&(r^{H_-}_{k_-}, k_-)=\mathfrak{a}_-(\tilde{v}^-_1),\quad \mathfrak{a}_+(\tilde{v}^-_i)=\mathfrak{a}_-(\tilde{v}^-_{i+1}), \quad \mathfrak{a}_+(\tilde{v}^-_{N_-})=\mathfrak{a}_-(\tilde{v}), \\[0.5ex] 
				&\mathfrak{a}_+(\tilde{v})=\mathfrak{a}_-(\tilde{v}^+_1), \quad \mathfrak{a}_{+}(\tilde{v}^+_j)=\mathfrak{a}_-(\tilde{v}^+_{j+1}),\quad \mathfrak{a}_+(\tilde{v}^+_{N_+})=(r^{H_+}_{k_+},k_+).
			\end{align*}
			We define the subspace
			\[
			\mathcal{M}^*_{\mathbf{N},k_-,k_+}(\mathbf{A};H_s,J_{Y}) \subset \mathcal{M}_{\mathbf{N},k_-,k_+}(\mathbf{A};H_s,J_{Y})
			\]
			consisting of simple elements $\mathbf{v}$, i.e.,
			\[
			(\pi\circ v^-_1,\dots,\pi\circ v^-_{N_-}, \pi\circ v, \pi\circ v^+_1,\dots,\pi\circ v^+_{N_+}) \in \mathcal{N}^{*}_N(\mathbf{A};J_\Sigma).	
			\]
		\end{definition}
		We also consider the evaluation map
		\begin{align*}
			&{\ev}:\mathcal{M}_{\mathbf{N},k_-,k_+}(\mathbf{A};H_s,J_Y) \to \mathcal{L}^{2N_-} \times \mathcal{L}^{2} \times \mathcal{L}^{2N_+},\\[.5ex]
			&{\ev}(\mathbf{v}):=\big(\ev_{-,0}(\tilde{v}^-_1),\ev_{+,0}(\tilde{v}^-_1),\dots, \ev_{-,0}(\tilde{v}^-_{N_-}),\ev_{+,0}(\tilde{v}^-_{N_-}),\ev_{-,0}(\tilde{v}), \ev_{+,0} (\tilde{v}),\\
			&\hspace{19mm} \ev_{-,0}(\tilde{v}^+_1),\ev_{+,0}(\tilde{v}^+_1) ,\dots,\ev_{-,0}(\tilde{v}^{+}_{N_+}),\ev_{+,0}(\tilde{v}^{+}_{N_+})\big),
		\end{align*}
		where $\ev_{\pm,0}(\tilde{v}):=\ev_\pm(v)(0)$. 
		Let us abbreviate ${\mathcal{D}}_{H_s}=(f_{\mathcal{L}}, Z_{\mathcal{L}}, H_s,J_Y)$. Let $\tilde{p}^{H_+}_{k_+} \in \mathrm{Crit} \mathcal{A}_{H_+}$ and $\tilde{q}^{H_-}_{k_-} \in \mathrm{Crit}\mathcal{A}_{H_-}$ be the Hamiltonian chords of $H_{\pm}$ in $(-\eta,\eta)\times Y$ with relative winding number $k_\pm$ corresponding to $\tilde{p},\tilde{q}\in\Crit f_{\mathcal{L}}$ via the correspondence \eqref{eq:identify}, respectively.

		\begin{definition}\label{def:conti_trajectories}
			We define the moduli space
			\[\mathcal{M}_{\mathbf{N}}(\tilde{q}^{H_-}_{k_-},\tilde{p}^{H_+}_{k_+};\mathbf{A};{\mathcal{D}}_{H_s}):= \ev^{-1} \left(W^u_{Z_{\mathcal{L}}}(\tilde{p}) \times \Delta_{Z_\mathcal{L}}^{N-1} \times W^s_{Z_{\mathcal{L}}}(\tilde{q}) \right). \]
			The subspace of simple elements is defined by
			\[
			\begin{split}
				&\mathcal{M}^*_{\mathbf{N}}(\tilde{q}^{H_-}_{k_-},\tilde{p}^{H_+}_{k_+};\mathbf{A};{\mathcal{D}}_{H_s}):=	\mathcal{M}_{\mathbf{N}}(\tilde{q}^{H_-}_{k_-},\tilde{p}^{H_+}_{k_+};\mathbf{A};{\mathcal{D}}_{H_s}) \, \cap \, \mathcal{M}^*_{\mathbf{N},k_-,k_+}(\mathbf{A};H_s,J_Y) .
			\end{split}
			\]
		\end{definition}
		\begin{prop}\label{prop:transv_conti}
			Let $\mathcal{J}_{(Y,\mathcal{L})}^{\mathrm{reg}} \subset \mathcal{J}_Y$ be as in Proposition \ref{prop:transv_strip}. 
			Then, for every choice of $
			\tilde{p}^{H_+}_{k_+} \in \mathrm{Crit} \mathcal{A}_{H_+}$,  $\tilde{q}^{H_-}_{k_-} \in \mathrm{Crit}\mathcal{A}_{H_-}$, $\mathbf{N}$, $\mathbf{A}$	as above 
			and for every $J_Y \in \mathcal{J}_{(Y,\mathcal{L})}^{\mathrm{reg}}$, the moduli space
			$\mathcal{M}^*_{\mathbf{N}}(\tilde{q}^{H_-}_{k_-},\tilde{p}^{H_+}_{k_+};\mathbf{A};\mathcal{D}_{H_s}) $ 
			is a smooth manifold of dimension
			$\mu(\tilde{p}^{H_+}_{k_+})-\mu(\tilde{q}^{H_-}_{k_-})+N-1$.
		\end{prop}
		\begin{proof}
			The proof is almost identical to that of Proposition \ref{prop:transv_strip}. For the corresponding proof in the periodic case, we refer to \cite[Proposition 4.5.(a)]{BKK24}.
		\end{proof}
		
		There is a free $\R^{N-1}$-action on $\mathcal{M}_{\mathbf{N}}^*(\tilde{q}_{k_-}^{H_-},\tilde{p}_{k_+}^{H_+};\mathbf{A};\mathcal{D}_{H_s})$ translating each $\tilde{v}_i^-$ and $\tilde{v}_j^+$. The quotient space 
		\[
		\overline{\mathcal{M}}^*(\tilde{q}_{k_-}^{H_-}, \tilde{p}_{k_+}^{H_+};\mathcal{D}_{H_s} ):=\bigcup_{\mathbf{N}, \mathbf{A}} \mathcal{M}^*_{\mathbf{N}}(\tilde{q}_{k_-}^{H_-},\tilde{p}_{k_+}^{H_+};\mathbf{A};\mathcal{D}_{H_s})/\R^{N-1},
		\]
		is a smooth manifold of dimension $\mu(\tilde{p}_{k_+}^{H_+})-\mu(\tilde{q}_{k_-}^{H_-})$ by Proposition \ref{prop:transv_conti}.

		\begin{prop}\label{prop:conti_unique}
			Assume that we are in the same setting as in Proposition \ref{prop:transv_conti}. 
			\begin{enumerate}[(a)]
				\item Assume condition (C2), i.e.,~$N_L>2$. If $\mu(\tilde{p}^{H_{+}}_{k_+})-\mu(\tilde{q}^{H_{-}}_{k_-})\leq1$, all elements of the moduli space ${\mathcal{M}}_{\mathbf{N}}(\tilde{q}^{H_{-}}_{k_-},\tilde{p}^{H_{+}}_{k_+};\mathbf{A};\mathcal{D}_{H_s})$ are simple, i.e.,
				\[
				{\mathcal{M}}_{\mathbf{N}}(\tilde{q}^{H_{-}}_{k_-},\tilde{p}^{H_{+}}_{k_+};\mathbf{A};\mathcal{D}_{H_s})={\mathcal{M}}^{*}_{\mathbf{N}}(\tilde{q}^{H_{-}}_{k_-},\tilde{p}^{H_{+}}_{k_+};\mathbf{A};\mathcal{D}_{H_s}).
				\]
				Moreover, there exist $r_\pm\in\R$ such that every element in this moduli space has its image in $(r_-,r_+)\times Y$. Consequently, this moduli space is compact up to breaking.
				\item Assume $N_L\geq 2$ and $\mu(\tilde{p}^{H_{+}}_{k_+})=\mu(\tilde{q}^{H_{-}}_{k_-})$. Then 
				$
				{\mathcal{M}}_{\mathbf{N}}(\tilde{q}^{H_{-}}_{k_-},\tilde{p}^{H_{+}}_{k_+};\mathbf{A};\mathcal{D}_{H_s})
				$
				is nonempty if and only if 
				\[
				N_{+}=N_{-}=0,\quad \tilde{p}=\tilde{q}, \quad k_+=k_-.
				\]
				In this case, it consists of a single element $\mathbf{v}=\tilde{v}$ whose projection to $\Sigma$ is a constant map at $p=q$.
			\end{enumerate}
		\end{prop}
		We postpone the proof of the proposition to Section \ref{sec:proof_proposition}.

		\subsubsection{Lagrangian Rabinowitz Floer homology of $\mathcal{L}$}\label{sec:lag_rfh}
		
		Let $H_s$ be a monotone homotopy for  $H_+ \prec H_-$, and let $J_Y \in \mathcal{J}^{\mathrm{reg}}_{(Y,\mathcal{L})}$.
		Proposition \ref{prop:conti_unique}.(b) implies that the only relevant moduli spaces for the filtered continuation map with action window $(a,b)$ are $\overline{\mathcal{M}}^*(\tilde{p}_{k}^{H_-}, \tilde{p}_{k}^{H_+};\mathcal{D}_{H_s})=\{\mathbf{v}=\tilde{v}\}$, where $(\tilde{p},k)$ ranges over $\Crit f_{\mathcal{L}}\times \w^{(a,b)}_{\mathcal{L}}(H_+)$. 
		Gluing a capping disk $\hat{x}$ of $\tilde{p}^{H_+}_{k}$ and $\mathbf{v}$, we obtain a capping disk $\hat{y}$ of $\tilde{p}^{H_-}_{k}$. As in the boundary map case discussed earlier, see also \cite[Section 3.9.1]{Z1}, we obtain an isomorphism of determinant lines 
		\[
		\det (D_{\hat{x}} \# T_{\tilde{p}} W^{u}_{Z_{\mathcal{L}}} (\tilde{p}) ) \longrightarrow \det (D_{\hat{y}} \# T_{\tilde{p}} W^{u}_{Z_{\mathcal{L}}} (\tilde{p}))
		\]
		and in turn an isomorphism of orientation lines 
		$C([\mathbf{v}]) : \mathfrak{o}(\tilde{p}^{H_+}_{k}, \hat{x}) \to \mathfrak{o}(\tilde{p}^{H_-}_{k},\hat{y}).
		$
		As in Proposition \ref{prop:orientation_commute}, one can see that this isomorphism is independent of the choice of cappings. Thus we simply write 
		\[
		C([\mathbf{v}]) : \mathfrak{o}(\tilde{p}^{H_+}_{k}) \longrightarrow \mathfrak{o}(\tilde{p}^{H_-}_{k}). 
		\]
		Moreover, the following diagram commutes:
		\[
		\begin{tikzcd}
			\mathfrak{o}(\tilde{p}^{H_+}_{k},\hat{x}) \arrow[r,"C({[\mathbf{v}]})"] \arrow[d, "f_{\hat{x}}",swap]
			& \mathfrak{o}(\tilde{p}^{H_-}_{k},\hat{y}) \arrow[d,"f_{\hat{y}}"] \\
			\mathfrak{o}(p,w) \arrow[r, equal]           
			&  \mathfrak{o}(p,w),
		\end{tikzcd}
		\]
		where $w=\pi_{\R\times Y}\circ\hat{x}=\pi_{\R\times Y}\circ\hat{y}$. 
		We define the chain-level continuation map by
		\begin{equation*}
			\mathfrak{c} = \mathfrak{c}_{H_+,H_-}: \FC_*^{(a,b)} (\mathcal{L}^{\tilde\sigma};{\mathcal{D}}_{H_+}) \to \FC_*^{(a,b)} (\mathcal{L}^{\tilde\sigma};{\mathcal{D}}_{H_-}),\qquad 	
			\mathfrak{c} := \bigoplus_{(\tilde{p},k)}  C([\mathbf{v}]),
		\end{equation*}
		where the direct sum runs over all $(\tilde{p},k) \in \mathrm{Crit} f_{\mathcal{L}}\times \mathfrak{w}^{(a,b)}_\mathcal{L}(H_{+})$, and $\mathbf{v}$ is the unique element of $\overline{\mathcal{M}}(\tilde{p}^{H_-}_{k}, \tilde{p}^{H_+}_{k};\mathcal{D}_{H_s})$.
		The map $\mathfrak{c}$ is a chain map by Proposition \ref{prop:conti_unique}.(a), and the induced continuation homomorphism on homology is independent of the choice of $H_s$.
		The Floer homology modules $\FH_*^{(a,b)}(\mathcal{L}^{\tilde\sigma};{\mathcal{D}}_H)$ and the continuation homomorphisms form a direct system over $(\mathcal{H},\prec)$. 
		We define the Rabinowitz Floer homology of $\mathcal{L}^{\tilde\sigma}$ for the finite action window $(a,b)\subset \mathbb{R}$ by
		\begin{equation*}
			\RFH_*^{(a,b)} (\mathcal{L}^{\tilde\sigma}) := \varinjlim_{H\in \mathcal{H}} \FH_*^{(a,b)}(\mathcal{L}^{\tilde\sigma};{\mathcal{D}}_H).
		\end{equation*}
		On the other hand, the natural action filtration chain homomorphisms
		\begin{equation*}
			\FC_*^{(a,b)}(\mathcal{L}^{\tilde\sigma};\mathcal{D}_H) \twoheadrightarrow \FC_*^{(a',b)}(\mathcal{L}^{\tilde\sigma};\mathcal{D}_H),\quad \FC_*^{(a,b)}(\mathcal{L}^{\tilde\sigma};\mathcal{D}_H)\hookrightarrow \FC_*^{(a,b')}(\mathcal{L}^{\tilde\sigma};\mathcal{D}_H)
		\end{equation*}
		for $a<a'<b<b'$	induce a bidirect system on $\RFH_*^{(a,b)} (\mathcal{L}^{\tilde\sigma})$.
		The (full) Rabinowitz Floer homology of $\mathcal{L}^{\tilde\sigma}$ is defined by
		\begin{equation*} 
			\RFH_{*}(\mathcal{L}^{\tilde\sigma}) := \varinjlim_{b \uparrow +\infty} \varprojlim_{a \downarrow -\infty} \RFH_{*}^{(a,b)} (\mathcal{L}^{\tilde\sigma}).
		\end{equation*}
		
		Note that the limits on action window and Hamiltonians stabilize in each degree since there are only finitely many generators in each degree according to \eqref{eq:index} and $\mathfrak{c} (\mathfrak{o}(\tilde{p}^{H_+}_k))=\mathfrak{o}(\tilde{p}^{H_-}_k)$.
		Therefore, it is reasonable to think of  
		$\RFH_j(\mathcal{L}^{\tilde\sigma})$ as the homology of 
		\[
		\RFC_j(\mathcal{L}^{\tilde\sigma}):= \varinjlim_{H\in \mathcal{H}}\FC_j^{(a,b)}(\mathcal{L}^{\tilde\sigma};\mathcal{D}_H)
		\] 
		for a sufficiently large action window $(a,b)$ depending on $j$. We write 
		\[
		\mathfrak{o}(\tilde p_k):= \varinjlim_{H\in \mathcal{H}}\mathfrak{o}(\tilde{p}^{H}_k)
		\]
		for the rank-one free submodule of $\RFC_*(\mathcal{L}^{\tilde\sigma})$ associated with $(\tilde p,k)$.
		It also follows that there is a canonical isomorphism
		$
		\RFH_j(\mathcal{L}^{\tilde\sigma})\cong \FH_j^{(a,b)}(\mathcal{L}^{\tilde\sigma};\mathcal{D}_H) 
		$
		for sufficiently large $H \in \mathcal{H}$ and  $(a,b)$ depending on $j$. 
		
		\begin{remark}
				\label{rem:RFH_laurent}
				When $N_L<\infty$, Remark \ref{rem:laurent} shows that $\QC_*(L^{\sigma})$ admits a $\Z[T,T^{-1}]$-module structure. By the proof of Theorem \ref{thm:isom_transfer}, see  \eqref{eq:projection_orientation2}, this induces a $\Z[T,T^{-1}]$-module structure on $\RFC_*(\mathcal{L}^{\tilde{\sigma}})$, where the action of $T$ is  given by $T:\mathfrak{o}(\tilde{p}_k)\to\mathfrak{o}(\tilde{p}_{k+m_L})$ when $N_L$ is even, and by $T:\mathfrak{o}(\tilde{p}_k)\to\mathfrak{o}(\tilde{p}_{k+2m_L})$ otherwise. Thus,  
				 $\RFH_*(\mathcal{L}^{\tilde{\sigma}})$ admits a $\Z[T,T^{-1}]$-module structure.	The multiplication by $T$ should coincide with the Seidel morphism constructed using the $S^1$-action on $Y$ given by the Reeb flow, see \cite{Sei97,HL09,Ueb19}.
		\end{remark}
		
		\begin{lemma}
			\label{lem:RFH_fg}
			If $N_L<\infty$, then $\RFH_*(\mathcal{L}^{\tilde{\sigma}})$ is a finitely generated $\mathbb{Z}[T,T^{-1}]$-module. 
		\end{lemma}
		\begin{proof}
			From the index formula in \eqref{eq:index}, we know that $\RFC_j(\mathcal{L}^{\tilde{\sigma}})$ has finitely many generators  for each $j\in \Z$. Since $T : \RFH_{j}(\mathcal{L}^{\tilde{\sigma}} ) \to \RFH_{j+N_L}(\mathcal{L}^{\tilde{\sigma}})$ is an isomorphism for every $j\in \Z$, we conclude that $\RFH_*(\mathcal{L}^{\tilde{\sigma}})$ is finitely generated  over $\Z[T,T^{-1}]$.
		\end{proof}

		\subsubsection{Proofs of Proposition \ref{prop:simple_no_escape} and Proposition \ref{prop:conti_unique}}\label{sec:proof_proposition}
		Before embarking on the proofs, we briefly describe the setup for Floer strips with punctures. We refer to \cite{DL2,BKK24} for details. 
		Let $\Gamma := \Gamma_{\partial} \cup \Gamma_{\mathrm{int}}$ where $\Gamma_{\partial}$ and $\Gamma_{\mathrm{int}}$ are (possibly empty) finite subsets of $\R \times \{0,1\}$ and $\R \times (0,1)$, respectively.
		We consider smooth maps 
		\[
		\tilde{v} = (b,v) :\big((\R \times [0,1]) \setminus \Gamma, (\R \times \{0,1\}) \setminus \Gamma_\partial\big) \longrightarrow (\R \times Y, \R \times \mathcal{L})
		\] 
		solving the Floer equation in \eqref{eq:Floer_equation_strip} 
		and converging to critical points of $\mathcal{A}_H$ at the positive and negative ends. 
		Furthermore, we require the map $\tilde{v}$ to satisfy the following asymptotic behavior near punctures. For $z\in \Gamma_{\partial}$ and holomorphic strip-like coordinates $\epsilon_z : (-\infty,0] \times [0,1] \to (\R\times[0,1]) \setminus \{z\}$ near $z$, there is a Reeb chord $c_z:([0,T],\{0,T\}) \to (Y,\mathcal{L})$ such that
		\[
		\lim_{s\to -\infty}(b\circ\epsilon_z(s,t), v\circ\epsilon_z(s,t))=(-\infty, c_z(Tt)).
		\]
		For $z\in \Gamma_{\mathrm{int}}$, a similar asymptotic behavior   holds with $c_z$ replaced by a closed Reeb orbit $\gamma_z : \R/T\mathbb{Z} \to Y$.
		We refer to such solutions as {\emph{punctured Floer strips}}.
		If $\tilde{v}$ is a punctured Floer strip, then the projection $\pi \circ v$ has finite energy. 
		By the removal of singularity theorem, $\pi \circ v$ extends to a $J_\Sigma$-holomorphic map defined on $\D$, identifying $\D\setminus\{\pm 1\}$ with $\R \times [0,1]$. 
		We abuse notation and write $\pi \circ v:(\D,\partial \D) \to (\Sigma, L)$ for  the extended map.  
		
		\medskip
		
		Let $\mathcal{M}_{N,m}(\tilde{q}_{l},\tilde{p}_{k};\mathbf{A};\mathcal{D}_H)$ be the moduli space of punctured Floer strips with trajectories, defined by replacing Floer strips with punctured ones in Definition \ref{def:strips_trajectories}. 
		Here, $m$ denotes the total number of punctures, which are allowed to move freely within $\R \times \{0,1\}$ or $\R \times (0,1)$, depending on whether they belong to $\Gamma_{\partial}$ or $\Gamma_{\mathrm{int}}$, respectively. We again have
		\[ 
		\begin{split}
			\Pi : \mathcal{M}_{N,m}(\tilde{q}_{l},\tilde{p}_{k};\mathbf{A};\mathcal{D}_H) &\longrightarrow \mathcal{N}_{N} (q,p;{\bf A};\mathcal{D}) \\[.5ex]
			(\tilde{v}_1,\dots,\tilde{v}_N)	&\longmapsto (\pi\circ v_1,\dots,\pi\circ v_N).
		\end{split}
		\]
		For a monotone homotopy $H_s$ for $H_{+}\prec H_{-}$, one can analogously define the punctured Floer continuation strip and the moduli space $\mathcal{M}_{\mathbf{N},m}(\tilde{q}^{H_{-}}_{k_-},\tilde{p}^{H_{+}}_{k_+};\mathbf{A};\mathcal{D}_{H_s})$.
		
		\begin{proof}[Proof of Proposition \ref{prop:simple_no_escape}]
			Statement (a) follows from \cite[Proposition 3.1.3]{BC4}.
			We therefore focus on the proof of (b).
			The proof of (b) proceeds along similar lines to those in \cite[Theorem 9.1]{DL1} and \cite[Proposition 4.5]{BKK24}. 
			
			Assume for contradiction that there exists a sequence $\{\mathbf{v}_\nu\}_{\nu\in \N}$ in $\overline{\mathcal{M}} (\tilde{q}_{l},\tilde{p}_{k};\mathcal{D}_H)$ with $\mu(\tilde{p}_{k})-\mu(\tilde{q}_{l})\leq 2$ such that Floer cylinders in $\mathbf{v}_\nu$ are not contained in a bounded subset of $\R \times Y$. Then,  
			by the maximum principle, they escape to the negative end of $\R \times Y$.
			By the SFT compactness theorem in \cite{BEHWZ03,BO2}, $\mathbf{v}_\nu$ converges to a Floer-holomorphic building whose top component lies in $\mathcal{M}_{N,m}(\tilde{q}_{l},\tilde{p}_{k};\mathbf{A};\mathcal{D}_H)$ with $m \geq 1$. We call this component $\mathbf{v}_\infty$. 
			
			Let $\{k_i\}^{m_{\mathrm{int}}}_{i=1}$ and $\{k'_j\}_{j=1}^{m_\partial}$ be the (relative) winding numbers of the asymptotic closed Reeb orbits $\{ \gamma_i \}_{i=1}^{m_\mathrm{int}}$ and Reeb chords $\{c_j \}_{j=1}^{m_\partial} $ of $\mathbf{v}_\infty$ at punctures, respectively, where $m_\mathrm{int}+m_\partial=m$. 
			A closed Reeb orbit $\gamma_i$ can be viewed as a map between circles, $\R/ T \Z\to Y_p$ for some $p\in\Sigma$, and the degree of this map is $k_i$. 
			From Proposition \ref{prop:maslov index of projection} and a similar computation for $\gamma_i$ in \cite[Equation (2.11)]{BKK24}, we have
			\begin{equation}\label{eq:index_CZ_RS}
				\mu_{\mathrm{CZ}}(\gamma_i)= 2\tau_\Sigma k_i, \qquad \mu_{\mathrm{RS}}(c_j)=\frac{2\tau_\Sigma k_j'}{d},
			\end{equation}
			where $\mu_{\mathrm{CZ}}$ denotes the Conley--Zehnder index. 
			We compute
			\begin{align}\label{eq:index_0}
				\begin{split}
					\mu(\tilde{p}_k)-\mu(\tilde{q}_l)
					&= \ind_{f_L}({q})-\ind_{f_{L}}({p})+2\frac{\tau_\Sigma}{d}(k-l)\\
					&=\ind_{f_L}(q)-\ind_{f_L}(p)+2\frac{\tau_\Sigma}{d}\Big(k-l -d\sum_{i=1}^{m_\mathrm{int}} k_i-\sum_{j=1}^{m_\partial}k'_j\Big)\\
					&\quad +2\frac{\tau_\Sigma}{d}\Big(d\sum_{i=1}^{m_\mathrm{int}} k_i+\sum_{j=1}^{m_\partial}k'_j\Big)\\
					&=\ind_{f_L}(q)-\ind_{f_L}(p)+\mu_L(\mathbf{A})+\sum_{i=1}^{m_\mathrm{int}} \mu_\CZ(\gamma_i)+\sum_{j=1}^{m_\partial}\mu_\RS(c_j).
				\end{split}
			\end{align}
			The last equality follows from \eqref{eq:index_CZ_RS} and a similar computation for $\mu_L(\mathbf{A})$. 
			
			We reach a contradiction by showing that $\Pi(\mathbf{v}_\infty)\in \mathcal{N}_{N}(q,p;\mathbf{A};\mathcal{D})$ cannot exist. 
			We first treat the case of $\dim L \geq 3$. Following the argument of \cite[Proposition 3.1.3]{BC4}, which also assumes $\dim L\geq 3$, we take an underlying simple chain of pearls of $\Pi(\mathbf{v}_{\infty})$, and denote by $\mathcal{N}^*_{N^*}(q,p;\mathbf{A}^*;\mathcal{D})$ the moduli space containing it, where $N^* \leq N$.
			From the proof of Proposition \ref{prop:base_transv}.(a), it follows that
			\[
				\dim \mathcal{N}^{*}_{N^*}(q,p;\mathbf{A}^*;\mathcal{D})=\ind_{f_L}(q)-\ind_{f_L}(p)+\mu_L(\mathbf{A}^*)+N^*-1.
			\]
			This formula still holds when some entries of $\mathbf{A}^*$ are zero. 
			Let $N^*_0$ denote the number of zero entries in $\mathbf{A}^*$. 
			Due to a free $\R$-action on each non-constant disk, we have 
			\[
				\dim \mathcal{N}^{*}_{N^*}(q,p;\mathbf{A}^*;\mathcal{D}) \geq N^*_1:=N^*-N^*_0.
			\]
			Furthermore, a Floer strip in $\mathbf{v}_\infty$ which projects to a constant disk in $\Sigma$ has at least one puncture.
			Hence, we have
			\begin{equation}\label{eq:puncture_bound}
				m=m_{\mathrm{int}}+m_\partial \geq N^*_0.	
			\end{equation}
			Combining the above estimates with the assumption $\mu(\tilde{p}_k)-\mu(\tilde{q}_l)\leq 2$, we obtain  
			\begin{align}\label{eq:ineq_index}
				\begin{split}
					2&\geq \dim \mathcal{N}^{*}_{N^*}(q,p;\mathbf{A}^*;\mathcal{D})-N^*+1+\mu_L(\mathbf{A})-\mu_L(\mathbf{A}^*) +\sum_{i=1}^{m_\mathrm{int}} \mu_\CZ(\gamma_i)+\sum_{j=1}^{m_\partial}\mu_\RS(c_j) \\ 
					&\geq -N^*_0+1+\mu_L(\mathbf{A})-\mu_L(\mathbf{A}^*)+ \sum_{i=1}^{m_\mathrm{int}} \mu_\CZ(\gamma_i)+\sum_{j=1}^{m_\partial}\mu_\RS(c_j) \\
					&\geq -m_\mathrm{int}-m_\partial+1+\mu_L(\mathbf{A})-\mu_L(\mathbf{A}^*)+\sum_{i=1}^{m_\mathrm{int}} \mu_\CZ(\gamma_i)+\sum_{j=1}^{m_\partial}\mu_\RS(c_j)\\
					&\geq 1+(\mu_L(\mathbf{A})-\mu_L(\mathbf{A}^*))+\sum_{i=1}^{m_\mathrm{int}}(\mu_\CZ(\gamma_i)-1)+\sum_{j=1}^{m_\partial}(\mu_\RS(c_j)-1).
				\end{split}
			\end{align}
				Without loss of generality, we may assume that $\omega(\pi_2(\Sigma))\neq 0$ since otherwise there is no contractible periodic Reeb orbit and  $m_\mathrm{int}=0$. Thus, the minimal Chern number $c_\Sigma$ of $\Sigma$ is finite, and $N_L$ divides $2c_\Sigma$.
			By condition (C2), i.e.,~$N_L>2$, we have  $c_\Sigma \geq 2$.
			Since $\gamma_i$ and $c_j$ are contractible, \eqref{eq:index_CZ_RS} yields
			\begin{equation}\label{eq:index_lb}
				\mu_{\mathrm{CZ}}(\gamma_i)\geq 2c_\Sigma \geq 4,\qquad \mu_\RS(c_j)\geq N_L\geq 3.
			\end{equation}
			These inequalities together with \eqref{eq:ineq_index} force $m_{\mathrm{int}}=m_{\partial}=0$.
			This contradicts the assumption $m\geq 1$, completing the proof of Proposition \ref{prop:simple_no_escape}.(b) for the case $\dim L \geq 3$.
			
			For the case that $\dim L \leq 2$, we compute from \eqref{eq:index_0}
			\begin{align*}
				0 \leq \mu_L(\mathbf{A})&\leq 2+\ind_{f_L}(p)-\ind_{f_L}(q)-\sum_{i=1}^{m_\mathrm{int}} \mu_\CZ(\gamma_i)-\sum_{j=1}^{m_\partial}\mu_\RS(c_j)\\
				&\leq 4-\sum_{i=1}^{m_{\mathrm{int}}} \mu_\CZ(\gamma_i)-\sum_{j=1}^{m_{\partial}}\mu_\RS(c_j).
			\end{align*}
			Using the inequalities in \eqref{eq:index_lb}, we deduce $\mathbf{A}=0$. This implies that there is a flow line of $Z_L$ from $q$ to $p$, and thus $\ind_{f_L}(q)\geq \ind_{f_L}(p)$. The above estimate together with \eqref{eq:index_lb} yields the contradiction $\mu_L(\mathbf{A})<0$. 
			This completes the proof of Proposition \ref{prop:simple_no_escape}.(b).	
		\end{proof}

		\begin{proof}[Proof of Proposition \ref{prop:conti_unique}.(a)]			
			For $\mathbf{v}\in {\mathcal{M}}_{\mathbf{N}}(\tilde{q}^{H_{-}}_{k_-},\tilde{p}^{H_{+}}_{k_+};\mathbf{A};\mathcal{D}_{H_s})$ with $\mu(\tilde{p}^{H_{+}}_{k_+})-\mu(\tilde{q}^{H_{-}}_{k_-})\leq1$, we show that  $\Pi(\mathbf{v})\in \mathcal{N}_{N}(q,p;\mathbf{A};\mathcal{D})$ is simple. Let $\mathcal{N}^{*}_{N^*}(q,p;\mathbf{A}^*;\mathcal{D})$ be the moduli space containing the underlying simple chain of pearls of $\Pi(\mathbf{v})$. Note that at most one entry of $\mathbf{A}^*$, namely the one corresponding to a Floer continuation strip, can be zero. If all entries of $\mathbf{A}^*$ are nonzero, there is a free $\R^{N^*}$-action given by the translations of the nonconstant disks. Similarly, in the case that one entry of $\mathbf{A}^*$ is zero, there is instead a free $\R^{N^*-1}$-action. However, the constant disk corresponding to the zero entry of $\mathbf{A}^*$ is free to move along an integral curve of $Z_L$. This gives an additional free $\R$-action.
			Therefore, in either case, we have $\dim \mathcal{N}^{*}_{N^*}(q,p;\mathbf{A}^*;\mathcal{D})\geq N^*$, and deduce
			\[
			\begin{split}
				1\geq \mu(\tilde{p}^{H_{+}}_{k_+})-\mu(\tilde{q}^{H_{-}}_{k_-})
				&= \ind_{f_L}(q)-\ind_{f_L}(p)+\mu_L(\mathbf{A})\\
				&=\dim \mathcal{N}^{*}_{N^*}(q,p;\mathbf{A}^*;\mathcal{D}) - N^*+1-\mu_L(\mathbf{A}^*)  +\mu_L(\mathbf{A})\\
				&\geq 1-\mu_L(\mathbf{A}^*)  +\mu_L(\mathbf{A}).
			\end{split}
			\]
			This implies $\mathbf{A}=\mathbf{A}^*$, and we conclude $\Pi(\mathbf{v})$ is simple. 
			
			For compactness, we argue as in the proof of Proposition \ref{prop:simple_no_escape}.(b). The only difference is that \eqref{eq:puncture_bound} is replaced by $m_{\mathrm{int}}+m_\partial \geq N^*_0-1$. 
			The term $-1$ arises from the fact that, in contrast to a nontrivial Floer strip, a nontrivial Floer continuation strip may project to a constant disk in $\Sigma$. This does not pose a problem as the left-hand side of the estimate in \eqref{eq:ineq_index} is replaced by 1 instead of 2 in this case.
			This completes the proof.
		\end{proof}
		
		\begin{proof}[Proof of Proposition \ref{prop:conti_unique}.(b)]
			The proof is similar to that of \cite[Proposition 4.5.(b)]{BKK24}, but simpler. 
			Although \eqref{eq:index_lb} is weakened to 
			\[
			\mu_{\mathrm{CZ}}(\gamma_i)\geq 2c_\Sigma \geq 2,\qquad \mu_\RS(c_j)\geq N_L\geq 2
			\]
			under our relaxed assumption $N_L\geq 2$, this remains  sufficient to deduce the simpleness of the moduli space in the case where $\mu(\tilde{p}^{H_{+}}_{k_+})=\mu(\tilde{q}^{H_{-}}_{k_-})$.

			Next, we show that, for every $\mathbf{v}\in \mathcal{M}_{\mathbf{N}}(\tilde{q}^{H_{-}}_{k_-},\tilde{p}^{H_{+}}_{k_+};\mathbf{A};\mathcal{D}_{H_s})$, $\Pi(\mathbf{v})\in \mathcal{N}^{*}_{N}(q,p;\mathbf{A};\mathcal{D})$ is a constant disk. Recall that there is a free $\R^{N-1}$-action on the shadow of Floer cylinders in $\mathbf{v}$. If the projection of the Floer continuation cylinder $\tilde{v}$ in $\mathbf{v}$ is nonconstant, we have an additional free $\R$-action. This contradicts the fact that  $\dim \mathcal{N}^{*}_{N}(q,p;\mathbf{A};\mathcal{D})=N-1$, which follows from $\mu(\tilde{p}^{H_{+}}_{k_+})=\mu(\tilde{q}^{H_{-}}_{k_-})$. Thus, the projection of $\tilde{v}$ is constant. Suppose $N>1$, i.e.~there is a Floer cylinder in $\mathbf{v}$, or $p\neq q$. In either case, there is a free $\R$-action induced by  translating the constant disk $\pi_{\R\times Y}\circ\tilde v$ along the integral curves of $Z_L$, which again leads to a contradiction. Therefore, we conclude $N=1$, $p=q$, and $k_-=k_+$ by the index formula.

			It remains to show that such $\tilde v$ uniquely exists. Since $p=q=\pi_{\R\times Y}\circ \tilde{v}$ is constant, we write
			\[
			\tilde{v}=(b,v):(\R \times [0,1], \R \times \{0,1\}) \longrightarrow (\R\times Y_p, \R\times\mathcal{L}_p).
			\]
			Since $\mathcal{L}_p=Y_p\cap \mathcal{L}$ is a finite set, $v(s,0)$ and $v(s,1)$ are constant in $s$. This proves $\tilde p=\tilde q$. 
			 The Floer equation \eqref{eq:Floer_equation_conti} reads 
			\begin{equation*}
				\begin{cases}
					\partial_s b-\alpha|_{Y_p}(\partial_t v) + h'_s(e^b) = 0, \\[0.5ex]
					\partial_t b+\alpha|_{Y_p}(\partial_sv)=0,
				\end{cases}	
			\end{equation*}
			with the boundary condition, namely $v(\R,0)$ and $v(\R,1)$ are points in $\mathcal{L}_p$. 
			Identifying $(Y_p, \mathcal{L}_p)$ with $(\R/\Z, \frac{1}{d}\Z/\Z)$, we have $
			v :(\R \times [0,1], \R \times \{0,1\}) \to (\R/\Z, \{0,\tfrac{k}{d} \})$, where $k:=k_+=k_-$.
			We define $g(s,t):=v(s,t)-\frac{k}{d}t$. Then, we have 
			\[
			(b,g): (\R \times [0,1], \R \times \{0,1\})\longrightarrow (\R \times (\R/\Z),\R\times \{0\})
			\] 
			and
			\begin{equation}\label{eq:Floer_fiber}
				\begin{cases}
					\partial_s b-\partial_t g + h'_s(e^{b})-\frac{k}{d} = 0, \\[0.5ex]
					\partial_t b+\partial_s g=0, \\[0.5ex]
					\lim\limits_{s\to \pm \infty} g(s,\cdot)=0.
				\end{cases}	
			\end{equation}	
			Differentiating \eqref{eq:Floer_fiber}, we obtain $\Delta g +  h_s''(e^b)e^{b} \partial_s g=0$.
			Since $g$ vanishes on the boundary $\R \times \{0,1\}$ and converges to zero as $s\to\pm\infty$, the maximum principle implies that $g$ is identically zero.
			Then the equations in \eqref{eq:Floer_fiber} reduce to an ODE for $b$, which is the same as   the one in \cite[Equation (4.7)]{BKK24}. Therefore, the proof of \cite[Proposition 4.7]{BKK24} shows that there exists a unique solution $b(s,t)$, which is independent of $t$. This completes the proof. 
		\end{proof}
		
		\subsection{Product structure on Lagrangian Rabinowitz Floer homology}\label{sec:product}
		In contrast to the assumption (C2), namely $N_L>2$, we assume the following stronger condition for necessary compactness results for a product structure on $\RFH_*(\mathcal{L}^{\tilde\sigma})$. 
		\begin{itemize}
			\item[($\rm{C^P}$)] {\it The minimal Maslov number $N_L$ of $L$ is greater than $\max\{\frac{1}{2}(\dim L+1),2\}$.}
		\end{itemize}
		
		We consider a strip with a slit: 
		\[
		S^{P}:=\big((\R \times [-1,0])\sqcup  (\R \times [0,1]) \big)/ \sim.
		\] 
		The equivalence relation is given by $(s,0^{+}) \sim (s,0^{-})$ for $s \leq 0$, where $0^{+}\in [0,1]$ and $0^{-}\in  [-1,0]$. 
		We endow $S^{P}$ with a natural complex structure. Then, the global coordinate $z=s+it$ is holomorphic everywhere except at the point $(0,0^-)=(0,0^+)$, see \cite[Section 3.2]{AS10a} for details.
 		For $(H,J_Y) \in \mathcal{H}\times \mathcal{J}_Y$, we consider a smooth map $\tilde{v}=(b,v): (S^{P},\partial S^{P}) \to (\R \times Y, \R \times \mathcal{L})$ that solves the Floer equation 
 		$\partial_s\tilde v+J_Y(\tilde{v})(\partial_t\tilde v-X_H(\tilde v))=0$
		in the interior of $S^P$, and is
		asymptotic to a critical point of $\mathcal{A}_H$ at the positive ends and to a critical point of $\mathcal{A}_{2H}$ at the negative end, i.e.,
		\begin{align*}
			(\ev^{m}_{+}(\tilde{v}))(t)&=(\ev^{m}_{+}(b),\ev^{m}_{+}(v))(t):=\lim_{s \to +\infty} \tilde{v}(s,t-2+m) \in \Crit \mathcal{A}_{H} \;\text{ for } m\in \{1,2\}, \\[.5ex]
			(\ev_-(\tilde{v}))(t)&=({\ev}_-(b),{\ev}_-(v))(t):= \lim_{s \to - \infty} \tilde{v}(s,2t-1) 
			\in \Crit \mathcal{A}_{2H}.
		\end{align*}
		We refer to such maps as \textit{Floer triangles} with respect to the pair $(H,J_Y)$. 
		As before, the projection $\pi \circ v$ extends to a $J_\Sigma$-holomorphic disk with boundary in $L$, passing through the points $\pi(\ev_+^1(v))$, $\pi(\ev_+^2(v))$, and $\pi(\ev_-(v))$.
		
		\medskip
		
		Let $\mathbf{k}=(k,k',k'')\in (m_L \Z)^3$. Throughout this section, whenever $\mathbf{k}$  is used, we assume that $r^{H}_k, r^{H}_{k'}, r^{2H}_{k''} \in (-\eta,\eta)$ exist as in \eqref{eq:beta}. Let $\mathbf{N}, N$, and $\mathbf{A}$ be as in \eqref{eq:auxiliary_Lag}, i.e.,
		\[
		\begin{split}
			&\mathbf{N}:=(N_1,N_2,N_3)\in (\N\cup\{0\})^3,\qquad N:=N_1+N_2+N_3,  \\  
			&\mathbf{A}:=(A_1,\dots, A_{N_1}, A, A_{N_1+1},\dots, A_N) \in (\pi_2(\Sigma,L))^{N+1},\quad A_i\neq 0 \quad \;\; i\in\{1,\dots,N\}.
		\end{split}
		\]
		For a Floer triangle $\tilde v=(b,v)$, we use the analogous notation
		\[
			\mathfrak a_-(\tilde v):=(\ev_-(b),\w(\ev_-(v))),\quad \mathfrak a_+^m(\tilde v):=(\ev_+^m(b),\w(\ev_+^m(v)))\;\text{ for } m\in\{1,2\}.
		\]
		\begin{definition}
			The moduli space 
			\[
			\mathcal{M}_{\mathbf{N}, \mathbf{k}}(\mathbf{A}; H,J_Y)=\big\{\mathbf{v}=(\tilde{v}^{2H}_1,\dots, \tilde{v}^{2H}_{N_1}, \tilde{v}, \tilde{v}^{H}_1, \dots, \tilde{v}^{H}_{N_2+N_3})\big\}
			\]
			consists of tuples of curves satisfying the following conditions. For every $i=1,\dots, N_1$ and $j=1,\dots, N_2+N_3$,
			\begin{itemize}
				\item  $\tilde{v}^{2H}_i=(b^{2H}_i,v^{2H}_i)$ is a Floer strip with respect to $(2H,J_Y)$ satisfying $[\pi \circ v^{2H}_i]=A_i$
				\item  $\tilde{v}=(b,v)$ is a Floer triangle with respect to $(H,J_Y)$ satisfying $[\pi \circ v]=A$,
				\item  $\tilde{v}^{H}_j=(b^{H}_j,v^{H}_j)$ is a Floer strip with respect to $(H,J_Y)$ satisfying $[\pi \circ v^{H}_j]=A_{N_1+j}$,
			\end{itemize}
				and for every $i=1,\dots,N_1-1$, $j_1=1,\dots, N_2-1$, and $j_2=N_2+1,\dots, N_2+N_3-1$,
				\[
				\begin{split}
					&(r^{2H}_{k''},k'')=\mathfrak{a}_-(\tilde{v}^{2H}_1),\quad \mathfrak{a}_+(\tilde{v}^{2H}_i)=\mathfrak{a}_-(\tilde{v}^{2H}_{i+1}), \quad \mathfrak{a}_+(\tilde{v}^{2H}_{N_1})=\mathfrak{a}_-(\tilde{v}),\\[0.5ex]
					&\mathfrak{a}^{1}_+(\tilde{v})=\mathfrak{a}_-(\tilde{v}^{H}_{1}),\quad \mathfrak{a}_+(\tilde{v}^{H}_{j_1}) = \mathfrak{a}_-(\tilde{v}^{H}_{{j_1}+1}),\quad \mathfrak{a}_+(\tilde{v}^{H}_{N_2}) = (r^{H}_{k},k),\\[0.5ex]
					&\mathfrak{a}^{2}_+(\tilde{v}) = \mathfrak{a}_-(\tilde{v}^{H}_{N_2+1}), \quad \mathfrak{a}_+(\tilde{v}^{H}_{j_2}) =\mathfrak{a}_-(\tilde{v}^{H}_{j_2+1}),\quad \mathfrak{a}_+(\tilde{v}^{H}_{N_2+N_3}) = (r^{H}_{k'},k'). 
				\end{split}
				\]
			We also define the subspace 
			\[
			\mathcal{M}^{*}_{\mathbf{N}, \mathbf{k}}(\mathbf{A}; H,J_Y) \subset \mathcal{M}_{\mathbf{N}, \mathbf{k}}(\mathbf{A}; H,J_Y)
			\] 
			consisting of simple elements $\mathbf{v}$, that is 
			\[
			\left(\pi\circ {v}^{2H}_1,\dots,\pi\circ {v}^{2H}_{N_1},\pi\circ {v},\pi\circ {v}^{H}_1,\dots,\pi\circ {v}^{H}_{N_2+N_3}\right) \in \mathcal{N}^{*}_{\mathbf{N}}(\mathbf{A};J_\Sigma).
			\]
		\end{definition}
		We consider the evaluation map
		\begin{align*}
			&\ev : \mathcal{M}_{\mathbf{N}, \mathbf{k}}(\mathbf{A}; H,J_Y)  \to \mathcal{L}^{2N+3} \\[.5ex]
			&\ev(\mathbf{v}) := \big( \ev_{-,0}(\tilde{v}^{2H}_1), \ev_{+,0}(\tilde{v}^{2H}_1), \dots, \ev_{-,0} (\tilde{v}^{2H}_{N_1}), \ev_{+,0}(\tilde{v}^{2H}_{N_1}), \ev_{-,0}(\tilde{v}),\\[.5ex]
			& \hspace{1.9cm}  \ev^1_{+,0}(\tilde{v}), \ev_{-,0}(\tilde{v}^H_{1}), \ev_{+,0}(\tilde{v}^H_{1}), \dots,\ev_{-,0}(\tilde{v}^H_{N_2}), \ev_{+,0}(\tilde{v}^H_{N_2}),  \\[.5ex]
			& \hspace{1.9cm}  \ev^2_{+,0}(\tilde{v}), \ev_{-,0}(\tilde{v}^{H}_{N_2+1}), \ev_{+,0}(\tilde{v}^{H}_{N_2+1}), \dots ,\ev_{-,0}(\tilde{v}^{H}_{N_2+N_3}), \ev_{+,0}(\tilde{v}^{H}_{N_2+N_3})  \big),
		\end{align*}
		where $\ev_{-,0}(\tilde{v}):=\ev_{-}(v)(0)$ and $\ev^{m}_{+,0}(\tilde{v}):=\ev^{m}_{+}(v)(0)$ for $m\in \{1,2\}$.
		
		Let $(f_L,Z^i_{L})$ for $i=1,2,3$ be Morse-Smale pairs on $L$ as in Section \ref{sec:product_QH}. Let $(f_\mathcal{L},Z^i_\mathcal{L})$  for $i=1,2,3$ be the lifted Morse-Smale pairs on $\mathcal{L}$. We denote $\mathcal{D}^{P}_{H}:=(f_{\mathcal{L}}, Z^1_{\mathcal{L}}, Z^2_{\mathcal{L}}, Z^3_{\mathcal{L}},H,J_Y)$.
		\begin{definition}
			
			For $(\tilde{p},k), (\tilde{q},k'),(\tilde{r},k'')\in \Crit f_{\mathcal{L}}\times m_L \Z$, we define the moduli space 
			\[
			\mathcal{M}_{\mathbf{N}} (\tilde{r}_{k''},\tilde{p}_{k}, \tilde{q}_{k'}; \mathbf{A}; \mathcal{D}^P_{H}) := \ev^{-1} \Big ( W^u_{Z^1_\mathcal{L}}(\tilde{r}) \times \Delta_{Z^1_\mathcal{L}}^{N_1} \times \Delta_{Z^2_\mathcal{L}}^{N_2} \times W^s_{Z^2_\mathcal{L}}(\tilde{p}) \times \Delta_{Z^3_\mathcal{L}}^{N_3} \times W^s_{Z^3_\mathcal{L}}(\tilde{q}) \Big ),
			\]
			where
			\[
			\Delta_{Z^i_{\mathcal{L}}}:= \{(x,\varphi^{t}_{Z^i_{\mathcal{L}}}(x))\in \mathcal{L} \times \mathcal{L} \;\mid\; x \in (\mathcal{L} \setminus \Crit f_{\mathcal{L}}) ,\; t \in \R_{>0} \} \quad \text{for } i=1,2,3,
			\]   
			see Figure \ref{fig:triangle_Leg}.
			The subspace of simple elements is defined by
			\begin{align*}
				\mathcal{M}^{*}_{\mathbf{N}} (\tilde{r}_{k''}, \tilde{p}_{k}, \tilde{q}_{k'}; \mathbf{A}; \mathcal{D}^P_{H}):= \mathcal{M}_{\mathbf{N}} (\tilde{r}_{k''},\tilde{p}_{k}, \tilde{q}_{k'}; \mathbf{A}; \mathcal{D}^P_{H}) \, \cap \, \mathcal{M}^{*}_{\mathbf{N},\mathbf{k}} (\mathbf{A}; H, J_Y).    
			\end{align*}
		\end{definition}
		\begin{figure}[h]
			\centering
			\includegraphics[height = 6cm]{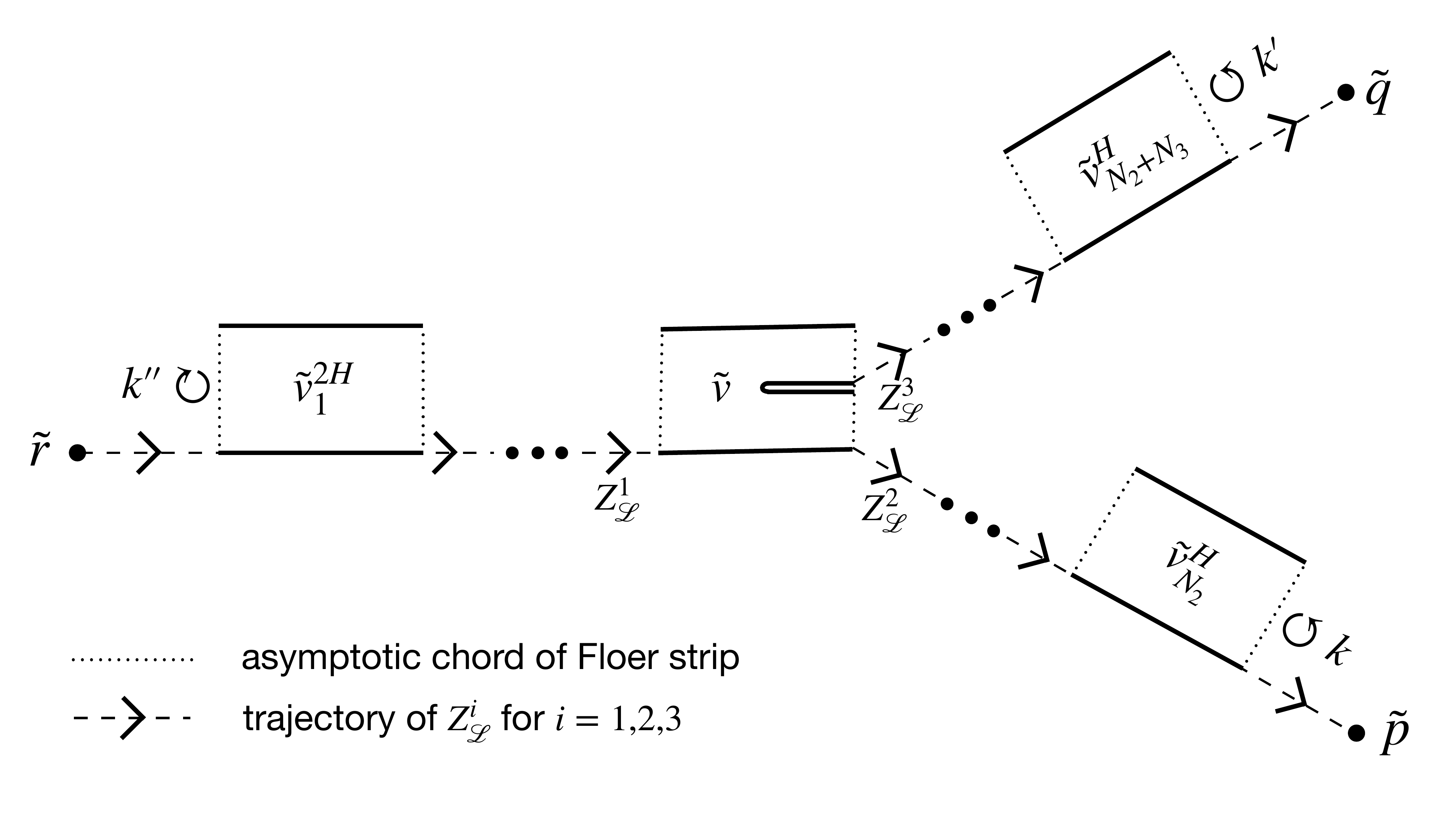}
			\caption{An element of $\mathcal{M}_{\mathbf{N}} (\tilde{r}_{k''},\tilde{p}_{k}, \tilde{q}_{k'}; \mathbf{A}; \mathcal{D}^{P}_{H})$}
			\label{fig:triangle_Leg}
		\end{figure} 
		
		There is a canonical projection map
		\begin{equation}\label{eq:projection_product}
			\Pi:\mathcal{M}_{\mathbf{N}} (\tilde{r}_{k''},\tilde{p}_{k}, \tilde{q}_{k'}; \mathbf{A}; \mathcal{D}^{P}_{H}) \longrightarrow \mathcal{N}_{\mathbf{N}}(r,p,q; \mathbf{A} ; \mathcal{D}^{P}), 
		\end{equation}
		defined by $\Pi (\mathbf{v}) : = \mathbf{w}$, where $\mathbf{w} := (\pi\circ {v}^{2H}_1,\dots,\pi\circ {v}^{H}_{N_2+N_3})$. 
		We have the same projection map with $\mathcal{M}_{\mathbf{N}}$ and $\mathcal{N}_{\mathbf{N}}$ replaced by $\mathcal{M}_{\mathbf{N}}^*$ and $\mathcal{N}_{\mathbf{N}}^*$, respectively.
		
		\begin{prop}\label{prop:transv_triangle}
			Let $\mathcal{J}^{\mathrm{reg},P}_{(Y,\mathcal{L})}$ be the subset of $\mathcal{J}_Y$ consisting of $J_Y$ whose horizontal part $J_\Sigma$ is an element of $\mathcal{J}^{\mathrm{reg},P}_{(\Sigma,L)}$ given in Proposition \ref{prop:transv_base_triangle}.
			Let $J_Y\in \mathcal{J}^{\mathrm{reg},P}_{(Y,\mathcal{L})}$, and let $(Z^1_{\mathcal{L}},Z^2_{\mathcal{L}},Z^3_{\mathcal{L}})$ be generic. Then, for every $(\tilde{p},k), (\tilde{q},k'),(\tilde{r},k'')\in \Crit f_{\mathcal{L}}\times m_L \Z$, $\mathbf{N}$, and $\mathbf{A}$, the moduli space $\mathcal{M}^{*}_{\mathbf{N}} (\tilde{r}_{k''}, \tilde{p}_{k}, \tilde{q}_{k'}; \mathbf{A}; \mathcal{D}^{P}_{H})$ is a smooth manifold of dimension 
			$\mu(\tilde{p}_{k}) + \mu (\tilde{q}_{k'})-\mu(\tilde{r}_{k''})+N-\dim \mathcal{L}$.
		\end{prop}
		\begin{proof}
			The proof proceeds along similar lines to that of Proposition \ref{prop:transv_strip}. As in that proof, the linearized operator for the Floer triangle  $\tilde{v}=(b,v)$ decomposes into 
			$
			D_{\tilde{v}}= 
			\begin{pmatrix}
				D_{\tilde{v}}^\mathrm{v} & L_\mathrm{cpt} \\[.5ex]
				0 &  D_{\tilde{v}}^\mathrm{h}
			\end{pmatrix}.
			$
			The horizontal part $D_{\tilde{v}}^\mathrm{h}$ is surjective due to our choice of $J_Y$. The vertical part has the form 
			\[
			\begin{split}
				&D_{\tilde{v}}^\mathrm{v}: W^{1,p,\delta}((S^P, \partial S^P), (\C, \R)) \longrightarrow L^{p,\delta}(S^P, \C),  \nonumber \\
				&D_{\tilde{v}}^\mathrm{v}(\sigma) = \partial_s \sigma + J_0 \partial_t \sigma + \left(\begin{matrix} h^{''}(e^{b})e^{b} & 0 \\ 0 & 0 \end{matrix} \right) \sigma,	
			\end{split}
			\]
			with the asymptotic operators  $A_{b_i}:= -J_0\frac{d}{dt} - \begin{pmatrix} h^{''}(e^{b_i})e^{b_i} & 0 \\ 0 & 0 \end{pmatrix}$ for $b_1=\ev_+^1(b)$, $b_2=\ev_+^2(b)$, and $b_3=\ev_-(b)$. Since $h''(e^{b_i})>0$ for $i=1,2,3$, Lemma \ref{lem:index} and \eqref{eq:RS_line_bundle} imply that $\mu_\mathrm{RS}(A_{b_i})=\frac{1}{2}$ for $i=1,2,3$, and thus
				\[
				\ind D_{\tilde{v}}^\mathrm{v}=\chi(S^P)-\frac{3}{2}+\mu_\mathrm{RS}(A_{b_1})+\mu_\mathrm{RS}(A_{b_2})-\mu_\mathrm{RS}(A_{b_3})=0.\\
				\]
				Assume that there exists a nonzero element $\zeta \in \ker D^{\mathrm{v}}_{\tilde{v}}$.
				Then, the asymptotic relative winding numbers $\wind_{+\infty,m}(\zeta)$ for $m=1,2$ and $\wind_{-\infty}(\zeta)$ satisfy
				\[
				\wind_{+\infty,m}(\zeta) \leq \alpha_{-}(A_{b_m})=0, \qquad \wind_{-\infty}(\zeta) \geq \alpha_{+}(A_{b_3})=\frac{1}{2}, 
				\]
				as in \eqref{eq:asymp_winding}.
				Hence, the algebraic count of zeros of $\zeta$ satisfies \[\wind_{+\infty,1}(\zeta)+\wind_{+\infty,2}(\zeta)-\wind_{-\infty}(\zeta)<0,\] which contradicts the positivity of zeros of $\zeta$. 
				Therefore, $D^{\mathrm{v}}_{\tilde{v}}$ is injective. Since $\ind D^{\mathrm{v}}_{\tilde{v}}=0$, it follows that $D^{\mathrm{v}}_{\tilde{v}}$ and hence $D_{\tilde{v}}$ are surjective.
			
			The dimension of $\mathcal{M}^{*}_{\mathbf{N}} (\tilde{r}_{k''}, \tilde{p}_{k}, \tilde{q}_{k'}; \mathbf{A}; \mathcal{D}^{P}_{H})$ is computed in the same way as in the proof of Proposition \ref{prop:transv_strip}.
			This completes the proof of Proposition \ref{prop:transv_triangle}.
		\end{proof}
		Let $J_Y \in \mathcal{J}_{(Y,\mathcal{L})}^{\mathrm{reg}} \cap \mathcal{J}_{(Y,\mathcal{L})}^{\mathrm{reg},P}$.
		We denote
			\begin{align*}
				\mathcal{M}_{\mathbf{N}} (\tilde{r}_{k''},\tilde{p}_{k}, \tilde{q}_{k'}; \mathcal{D}^P_{H}):= \bigcup_{\mathbf{A}} \mathcal{M}_{\mathbf{N}} (\tilde{r}_{k''},\tilde{p}_{k}, \tilde{q}_{k'}; \mathbf{A};\mathcal{D}^P_{H}).
			\end{align*}
		There is a free $\mathbb{R}^N$-action on $\mathcal{M}_{\mathbf{N}} (\tilde{r}_{k''},\tilde{p}_{k}, \tilde{q}_{k'}; \mathcal{D}^P_{H})$ given by translating the domain $\mathbb{R}\times [0,1]$ of each $\tilde{v}^{2H}_i$ and $\tilde{v}^{H}_j$ in the $\mathbb{R}$-direction. 
		We denote 
		\[			
		\overbar{\mathcal{M}}(\tilde{r}_{k''}, \tilde{p}_{k}, \tilde{q}_{k'}; \mathcal{D}^P_{H}) := \bigcup_{\mathbf{N}} \mathcal{M}_{\mathbf{N}} (\tilde{r}_{k''},\tilde{p}_{k}, \tilde{q}_{k'}; \mathcal{D}^P_{H}) /\mathbb{R}^N, \\
		\]
		As usual, we write $\overbar{\mathcal{M}}^{*}$ to denote the subspace consisting of simple elements.
		
		\begin{prop}\label{prop:no_escape_triangle}
			If $\mu(\tilde{p}_{k})+\mu(\tilde{q}_{k'})-\mu(\tilde{r}_{k''})-\dim \mathcal{L} \leq 1$, the following hold.
		\begin{enumerate}[(a)]
				\item Assume $N_L\geq2$. All elements of $\overbar{\mathcal{M}}(\tilde{r}_{k''}, \tilde{p}_{k}, \tilde{q}_{k'}; \mathcal{D}^P_{H})$ are simple.
				\item Assume $\rm(C^P)$, i.e.,~$N_L>\max\{\frac{1}{2}(\dim L+1),2\}$. Then there exist real numbers $r_- < r_+$ such that the image of every element of this moduli space lies in $(r_-,r_+)\times Y$. Consequently, the moduli space is compact up to breaking.
			\end{enumerate}
		\end{prop}
		
		\begin{proof}
			The statement (a) follows directly from \cite[Lemma 5.2.2]{BC4}.
			
			For (b), we proceed along similar lines as in the proof of Proposition \ref{prop:simple_no_escape}.
			Assume that there exists a sequence $\{\mathbf{v}_{\nu}\}_{\nu \in \N}$ in $\overbar{\mathcal{M}}(\tilde{r}_{k''}, \tilde{p}_{k}, \tilde{q}_{k'}; \mathcal{D}^P_{H})$ such that Floer strips or Floer triangles in $\mathbf{v}_{\nu}$ are not contained in a bounded subset of $\R\times Y$. We analyze the top component $\mathbf{v}_{\infty}$ of the limit of $\mathbf{v}_{\nu}$. There are two possible scenarios. First, $\mathbf{v}_{\infty}$ may develop a puncture. However, this is ruled out by the assumption $N_L>2$ as in the proof of Proposition \ref{prop:simple_no_escape}. 
			
			For simplicity, we describe the alternative scenario in the case of $\mathbf{N}=0$. In this situation, $\mathbf{v}_{\infty}=v_\infty$ consists of two connected curves ${v}_{\infty,\mathrm{I}}$ and ${v}_{\infty,\mathrm{II}}$. Here 
			 ${v}_{\infty,\mathrm{I}}$ is a Floer-holomorphic strip with a $Z_\mathcal{L}$-trajectory from a Reeb chord to $\tilde{q}_{k'}$, while  ${v}_{\infty,\mathrm{II}}$ is a punctured Floer strip with $Z_\mathcal{L}$-trajectories from $\tilde{r}_{k''}$ to $\tilde{p}_{k}$. In the next level of the limit, there appears a U-shaped $J_Y$-holomorphic strip connecting the asymptotic Reeb chord $c_1$ of ${v}_{\infty,\mathrm{I}}$ and the asymptotic Reeb chord $c_2$ at the puncture of ${v}_{\infty,\mathrm{II}}$; see Figure \ref{fig:u_curve}. In both ${v}_{\infty,\mathrm{I}}$ and ${v}_{\infty,\mathrm{II}}$, there is no additional puncture by the assumption $N_L >2$.
			 \begin{figure}[h]
			\centering
			\includegraphics[height = 4.5cm]{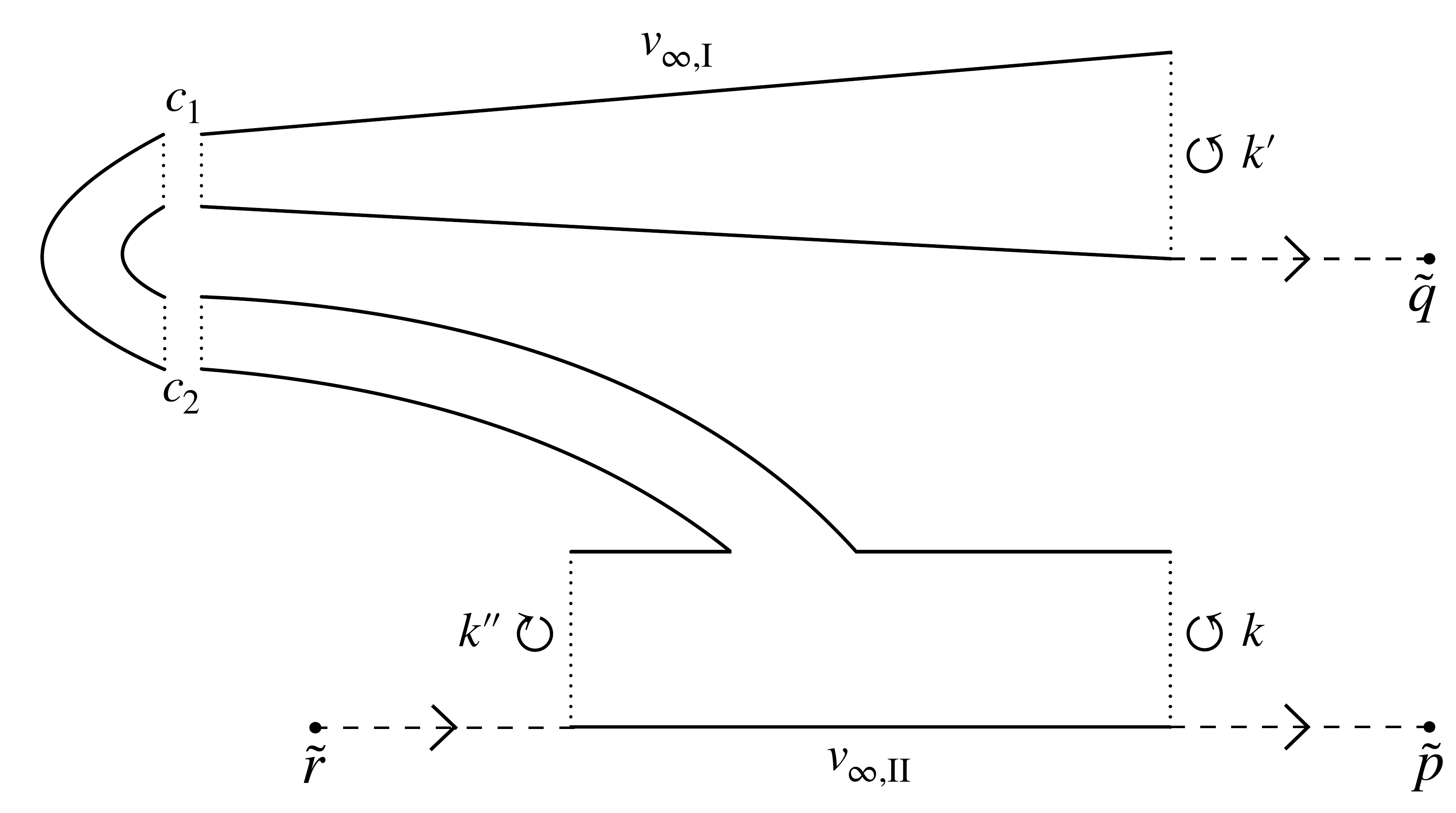}
			\caption{The limit configuration of $\mathbf{v}_{\nu}$ when $\mathbf{N}=0$}
			\label{fig:u_curve}
			\end{figure} 
			
			We first assume that $\dim L \geq 3$.
			Following the notation used in the proof of Proposition \ref{prop:simple_no_escape}, let $\mathcal{N}_{1}(\emptyset,q;{A}_\mathrm{I};\mathcal{D})$ be the moduli space of chains of pearls positively asymptotic to $q$,  with no asymptotic condition at the negative end, containing $\Pi({v}_{\infty,\mathrm{I}})$. Similarly, let  $\mathcal{N}_{1}(r,p;{A}_\mathrm{II};\mathcal{D})$ be moduli space containing $\Pi({v}_{\infty,\mathrm{II}})$. As observed in \eqref{eq:index_0}, the asymptotic Reeb chords $c_1$ and $c_2$ satisfy 
			\begin{equation}\label{eq:index_pop_punctures}
				\frac{2\tau_\Sigma}{d}(k+k'-k'')=\mu_L({A}_\mathrm{I})+\mu_L({A}_\mathrm{II})+\mu_{\RS}(c_1)+\mu_{\RS}(c_2).
			\end{equation}
			Since $\dim L \geq 3$, we may take the underlying simple chains of pearls of $\Pi({v}_{\infty,\mathrm{I}})$ and $\Pi({v}_{\infty,\mathrm{II}})$. We denote by $\mathcal{N}^*_{1}(\emptyset,q;{A}_\mathrm{I}^*;\mathcal{D})$ and $\mathcal{N}^*_{1}(r,p;{A}_\mathrm{II}^*;\mathcal{D})$ the moduli spaces containing these underlying simple chains of pearls, respectively. 
			They have
			\begin{equation}\label{eq:dim_limit}
			\begin{split}
				&\dim \mathcal{N}^*_{1}(\emptyset,q;{A}_\mathrm{I}^*;\mathcal{D}) = \dim L -\ind_{f_L}(q)+\mu_L({A}_\mathrm{I}^*) \geq 0,\\[.5ex]
				&\dim  \mathcal{N}^*_{1}(r,p;{A}_\mathrm{II}^*;\mathcal{D}) = \ind_{f_L}(r)-\ind_{f_L}(p)+\mu_L({A}_\mathrm{II}^*) \geq 0.
			\end{split}				
			\end{equation}
			Using the index condition on $\tilde{p}_{k}$, $\tilde{q}_{k'}$, and $\tilde{r}_{k''}$, we obtain
			\[
			\begin{split}
				1&\geq \mu(\tilde{p}_{k})+\mu(\tilde{q}_{k'})-\mu(\tilde{r}_{k''})-\dim \mathcal{L} \\[.5ex]
				&= \ind_{f_L}(r)-\ind_{f_L}(p)-\ind_{f_L}(q)+\mu_L({A}_\mathrm{I})+\mu_L({A}_\mathrm{II})+\mu_{\RS}(c_1)+\mu_{\RS}(c_2)\\[.5ex]
				&\geq -\dim L + \mu_{\RS}(c_1)+\mu_{\RS}(c_2).
			\end{split}
			\]
			The second line follows from \eqref{eq:index} and \eqref{eq:index_pop_punctures}. The last line follows from $\mu_L({A}_\mathrm{I})\geq \mu_L({A}_\mathrm{I}^*)$, $\mu_L({A}_\mathrm{II})\geq \mu_L({A}_\mathrm{II}^*)$, and \eqref{eq:dim_limit}. Both $c_1$ and $c_2$ are trivial in $\pi_1(Y,\mathcal{L})$ as $\tilde{p}_{k}$, $\tilde{q}_{k'}$, and $\tilde{r}_{k''}$ are so. Thus $\mu_{\RS}(c_1)$ and $\mu_{\RS}(c_2)$ are at least $N_L$. By our assumption $N_L> \frac{1}{2}(\dim L + 1)$, we arrive at a contradiction. In the case of $\mathbf{N}\neq0$, an analogous argument, incorporating  translations of the additional disks, yields a contradiction. 
			
			It remains to consider the case $\dim L \leq 2$. As before, we have
			\[
			\begin{split}
				1&\geq \mu(\tilde{p}_{k})+\mu(\tilde{q}_{k'})-\mu(\tilde{r}_{k''})-\dim \mathcal{L} \\[.5ex]
				&= \ind_{f_L}(r)-\ind_{f_L}(p)-\ind_{f_L}(q)+\mu_L({A}_\mathrm{I})+\mu_L({A}_\mathrm{II})+\mu_{\RS}(c_1)+\mu_{\RS}(c_2)\\[.5ex]
				&\geq 0-2-2+0+0+3+3=2,
			\end{split}
			\]
			which is a contradiction.
			This completes the proof of Proposition \ref{prop:no_escape_triangle}.(b).
		\end{proof}
		
		Each $\mathbf{v}\in \overbar{\mathcal{M}}^{*}(\tilde{r}_{k''}, \tilde{p}_{k}, \tilde{q}_{k'};\mathcal{D}^P_{H})$ with $\mu(\tilde{p}_{k})+\mu(\tilde{q}_{k'})-\mu(\tilde{r}_{k''})=\dim \mathcal{L}$ induces an isomorphism
		\[
		C([\mathbf{v}]) : \mathfrak{o}(\tilde{p}_{k})\otimes \mathfrak{o}(\tilde{q}_{k'}) \longrightarrow \mathfrak{o}(\tilde{r}_{k''}).
		\]
		The chain-level triangle product is defined by  
		\[
			\begin{split}
				&\star_H:= \bigoplus_{(\tilde{p},\tilde{q},\tilde{r},k,k',k'')}\sum_{[\mathbf{v}]} C([\mathbf{v}]):\FC^{(a,b)}(\mathcal{L}^{\tilde{\sigma}};\mathcal{D}_H) \otimes \FC^{(a,b)}(\mathcal{L}^{\tilde{\sigma}};\mathcal{D}_H) \to \FC^{(a+b, 2b)}(\mathcal{L}^{\tilde{\sigma}};\mathcal{D}_{2H})
			\end{split}
		\]
		The direct sum runs over $(\tilde{p},k), (\tilde{q},k')\in \Crit f_{\mathcal{L}} \times \m_{\mathcal{L}}^{(a,b)}(H)$ and $(\tilde{r},k'')\in \Crit f_{\mathcal{L}}\times \m_{\mathcal{L}}^{(a+b,2b)}(2H)$ satisfying $\mu(\tilde{p}_{k})+\mu(\tilde{q}_{k'})-\mu(\tilde{r}_{k''})=\dim \mathcal{L}$. 
		The sum runs over $[\mathbf{v}] \in \overbar{\mathcal{M}}^{*}(\tilde{r}_{k''}, \tilde{p}_{k}, \tilde{q}_{k'};\mathcal{D}^P_{H})$.
		It is a chain map by Proposition \ref{prop:no_escape_triangle} and induces 
		\[ \FH^{(a,b)}_l(\mathcal{L}^{\tilde{\sigma}};\mathcal{D}_H) \otimes \FH^{(a,b)}_m(\mathcal{L}^{\tilde{\sigma}};\mathcal{D}_H) \longrightarrow \FH^{(a+b ,2b)}_{l+m-\dim \mathcal{L}}(\mathcal{L}^{\tilde{\sigma}};\mathcal{D}_{2H}).\]
		By taking limits $ \varinjlim\limits_{b\uparrow+\infty}\varprojlim\limits_{a\downarrow-\infty}  \varinjlim\limits_{H\in\mathcal{H}}$, we define the triangle product
		\[ \star :  \RFH_l(\mathcal{L}^{\tilde{\sigma}}) \otimes \RFH_m(\mathcal{L}^{\tilde{\sigma}}) \longrightarrow \RFH_{l+m-\dim \mathcal{L}}(\mathcal{L}^{\tilde{\sigma}}).\]
		Finally, we remark that the $\Z[T,T^{-1}]$-module structure in Remark \ref{rem:RFH_laurent} is compatible with this product, and therefore $\RFH_*(\mathcal{L}^{\tilde{\sigma}})$ is an algebra over $\Z[T,T^{-1}]$.

		\section{Proof of the main results}
		In this section, we prove Theorem \ref{thm:isom_transfer} by establishing a correspondence between the moduli spaces used to define $\QH(L^{\sigma})$ and those used in the definition of $\RFH(\mathcal{L}^{\tilde{\sigma}})$.
		We first show that a $J_\Sigma$-holomorphic disk in $(\Sigma,L)$ lifts to a $J_Y$-holomorphic strip in $(\R\times Y,\R\times \mathcal{L})$. See \cite{Riz16,KPS24} for related works.
 		We then establish a bijection between $J_Y$-holomorphic strips and Floer strips in $(\R\times Y,\R\times \mathcal{L})$.  
		
		\subsection{Lifting of holomorphic disks}
		
		We fix $J_Y \in \mathcal{J}_{Y}$ and denote its horizontal part by $J_\Sigma$. We denote
		\[
		\mathbb{R}_0:=\{x+iy\in \partial \D \mid y<0 \},\qquad \mathbb{R}_1:=\{x+iy\in \partial \D \mid y>0\}.
		\]
		We tacitly use a biholomorphism $(\D\setminus\{\pm 1\},\R_0,\R_1)\cong (\R\times [0,1],\R\times \{0\},\R\times \{1\})$.
		
		\begin{lemma}\label{lem:hol_lift}
			Let $w:(\D,\partial \D)\to (\Sigma, L)$ be a $J_\Sigma$-holomorphic disk with $q:=w(-1)$ and $p:=w(1)$. Let $c_q:([0,1],\{0,1\})\to(Y_q,\mathcal{L}_q)$ be any generalized Reeb chord. Then, there exists a $J_Y$-holomorphic strip $\tilde{u}=(a,u):\R\times[0,1]\to (\R\times Y,\R\times\mathcal{L})$ such that $\pi\circ u=w$ and
			\[
			\lim_{s\to-\infty} u(s,t) = c_q(t),\qquad \lim_{s\to\infty} u(s,t) = c_p(t),
			\]
			where $c_p:([0,1],\{0,1\})\to(Y_p,\mathcal{L}_p)$ is a generalized Reeb chord with $\w(c_p)-\w(c_q)= d\cdot\omega([w])$. 
			Moreover, if there is another such strip $\tilde u'=(a',u')$, then $u=u'$ and $a-a'$ is constant.
		\end{lemma}
		\begin{proof}
			The proof is inspired by the argument in \cite[Lemma 7.1]{BK13}. 
			We consider the bundle $\pi_{\mathbb{R}\times Y}:\mathbb{R}\times Y \to \Sigma$ with fibers $\R\times S^1\cong \C^*:=\C\setminus\{0\}$. We 
			choose a unitary trivialization $\Phi: (w^*(\mathbb{R}\times Y), w^*J_{Y}) \to (\D\times \mathbb{C}^*, J_0)$ of the pullback  bundle by $w$ and denote
			\[
			\mathcal{L}_{\mathbb{R}_i}:=\Phi(w|_{\mathbb{R}_i}^*(\mathbb{R}\times \mathcal{L}))\subset \mathbb{R}_i \times \mathbb{C}^*, \quad i=0,1
			.
			\] 
			Note that there are smooth functions $f:\mathbb{R}_0\to \mathbb{R}$ and $g:\mathbb{R}_1\to\mathbb{R}$ such that 
			\begin{align*}
				\mathcal{L}_{\mathbb{R}_0}=\{(\xi,re^{i (f(\xi)+\frac{2\pi\ell}{d}) })\in \mathbb{R}_0\times \mathbb{C}^* \mid  r\in\mathbb{R}_{>0},\, \ell=1,\dots,d\},\\[.5ex]
				\mathcal{L}_{\mathbb{R}_1}=\{(\xi,re^{i (g(\xi)+\frac{2\pi\ell}{d}) } )\in \mathbb{R}_1\times \mathbb{C}^* \mid  r\in\mathbb{R}_{>0},\, \ell=1,\dots,d\},
			\end{align*}
			and
			\[
			\lim_{\sigma\to-\infty} (f-g)(\sigma)=0 ,\qquad 
			\lim_{\sigma\to\infty} (f-g)(\sigma)\in \frac{1}{d}\Z.
			\]
			Viewing $f$ and $g+\w(c_q)$ together as a function on $\partial\D$ of class $L^1$, we know that there is a harmonic function $h:\D\setminus\partial\D \to \mathbb{R}$ whose  radial limit agrees with $f$ on $\R_0$ and with $g+\w(c_q)$ on $\R_1$. Let $\bar{h}$ be a harmonic conjugate of $h$. We define a holomorphic function 
			\[
			\mu:= e^{\bar{h}+ih}:\D\setminus \{\pm 1\}\cong\R\times [0,1]\longrightarrow \mathbb{C}^*.
			\] 
			Through the trivialization $\Phi$, the pair $(w|_{\D\setminus \{\pm 1\}},\mu)$ gives rise to a $J_Y$-holomorphic lift $\tilde u$  of $w|_{\D\setminus \{\pm 1\}}$ such that the negative asymptotic orbit is a generalized Reeb chord over $q$  with relative winding number $\w(c_q)$. After multiplying $\mu$ by a suitable factor $e^{\frac{2\pi i\ell}{d}}$ with $1\leq \ell\leq d$, we can arrange that the negative asymptotic orbit of $\tilde{u}$ coincides with $c_q$. The positive asymptotic orbit of $\tilde u$ is a generalized Reeb chord $c_p$ over $p$, and the identity 
			\[
			\omega([w])=\int_{\R\times [0,1]} \tilde{u}^*d\alpha = \frac{1}{d}(\w(c_p)-\w(c_q))
			\] 
			follows from $\pi^*\omega = d\alpha$ and Stokes' theorem. 
			The described correspondence between $\tilde u$ and $\mu$ is bijective up to the addition of a constant to  $\bar h$. This finishes the proof.
		\end{proof}
		
		\begin{remark}
			In the previous lemma, if $\w(c_q)=0$ or $\w(c_p)=0$, then $\displaystyle \lim_{s\to-\infty}a$ or $\displaystyle\lim_{s\to \infty}a$ is finite, respectively.
			Otherwise, we have
			\[
			\w(c_q)\cdot \big(\lim_{s\to-\infty}a\big)=-\infty,\qquad \w(c_p) \cdot \big(\lim_{s\to\infty}a\big)=+\infty.
			\]	
		\end{remark}
		
		In fact, Lemma \ref{lem:hol_lift} extends to $J_\Sigma$-holomorphic disks with an arbitrary number of boundary punctures. The following lemma addresses the case with three punctures. Its proof is a straightforward modification of that of Lemma \ref{lem:hol_lift}. 
		
		\begin{lemma}\label{lem:hol_lift2}
			Let $w : (\D , \partial \D) \to (\Sigma,L)$ be a $J_\Sigma$-holomorphic disk with $p:=w(e^{\frac{-\pi i}{3}})$, $q:=w(e^{\frac{\pi i}{3}})$, and $r:=w(-1)$.
			Let $c_p:([0,1],\{0,1\})\to (Y_p,\mathcal{L}_p)$ and $c_q:([0,1],\{0,1\})\to (Y_q,\mathcal{L}_q)$ be any generalized Reeb chords.
			Then, there exists a $J_Y$-holomorphic triangle $\tilde{u}=(a,u):(S^P,\partial S^P) \to (\R\times Y, \R \times \mathcal{L})$, together with an integer $\ell \in \{1,\dots, d\}$, such that $\pi \circ u =w$ and 
			\[
			\lim_{s\to \infty} u(s,t-1)=c_p(t), \quad \lim_{s\to \infty} u(s,t)=\phi_R^{\frac{\ell}{d}}(c_q(t)), \quad \lim_{s\to -\infty} u(s,2t-1)= c_r(t),
			\]
			where $c_r:([0,1],\{0,1\})\to (Y_r,\mathcal{L}_r)$ is a generalized Reeb chord with $\w(c_p)+\w(c_q)-\w(c_r)=d\cdot\omega([w])$.
			Moreover, if there exists another such triangle $\tilde{u}'=(a',u')$, together with an integer $\ell'$, satisfying the same properties, then $u=u'$, $\ell=\ell'$, and $a-a'$ is constant.
		\end{lemma}
		
		\subsection{Correspondence between Floer and $J_Y$-holomorphic strips}\label{sec:correspond_strip}
		
		In this section, we establish a correspondence between Floer strips and $J_Y$-holomorphic strips by adapting the proof of \cite[Proposition 6.1]{DL1}, which treats the periodic case.  Let $(H,J_Y) \in \mathcal{H}\times\mathcal{J}_{Y}$ with $H(r,y)=h(e^r)$.
		
		\begin{prop}\label{prop:correspondence_strip}
			Let $\tilde{u}=(a,u):(\R\times[0,1],\R\times\{0,1\}) \to (\R \times Y,\R \times \mathcal{L})$ be a $J_Y$-holomorphic strip. Assume that $\ev_\pm(u)$ exist and there are $r_\pm\in(-\eta,\eta)$ such that $d\cdot h'(e^{r_\pm})=\w(\ev_\pm(u))$, i.e.,~there are  chords of $X_H$ in $(-\eta,\eta)\times Y$ corresponding to $\ev_\pm(u)$.
			Then, there exists a unique pair of a Floer strip $\tilde{v}=(b,v): (\R\times[0,1],\R\times\{0,1\}) \to (\R \times Y,\R \times \mathcal{L})$ and a smooth map $\mathbf{e}=(e_1,e_2):(\R\times[0,1],\R\times\{0,1\}) \to (\C,\R)$ satisfying
			\begin{equation}\label{eq:hol_Floer_strip}
				de_1-de_2\circ i +h'(e^{a+e_1})ds=0,
			\end{equation}
			and 
			\[
			\ev_\pm(\tilde{v})=(r_\pm, \ev_\pm(u)),\qquad (a,u)=(b-e_1, \phi_R^{-e_2}\circ v).
			\]			 
		\end{prop}

		\begin{proof} 
			We take a smooth function $\nu : \R \times [0,1] \to \R$ satisfying $\partial_t \nu=0$ and 
			\[
			\nu (s,t) = \begin{cases*}
				h'(e^{r_+})s, & $s>1$, \\[.5ex]
				h'(e^{r_-})s, & $s<-1$.
			\end{cases*}
			\]
			We consider the following family of equations parametrized by $\tau \in [0,1]$,
			\begin{equation}\label{eq:interpolate_tau_strip}
				d(a+e_1)-de_2\circ i +h'(e^{a+e_1})ds -\tau da-(1-\tau)d\nu=0, 
			\end{equation}
			which interpolates  \eqref{eq:hol_Floer_strip} and
			\begin{equation}\label{eq:interpolate_0_strip}
				d(a+e_1)-de_2\circ i +h'(e^{a+e_1})ds -d\nu=0.
			\end{equation}
			Substituting $(f_1,f_2):=(a+e_1,e_2)$, we are interested in the zero set of  
			\[
			\begin{split}
				\mathcal{F}_\tau: 	\mathcal{Y}&\longrightarrow L^{p}(\R\times [0,1], \C)\\[.5ex]
				(f_1,f_2)&\longmapsto	df_1-df_2\circ i +h'(e^{f_1})ds -\tau da-(1-\tau)d\nu,
			\end{split}
			\]
			where $\mathcal{Y}$ consists of $(f_1,f_2)$ satisfying $(f_1-\partial_s \nu,f_2) \in W^{1,p}((\R\times [0,1],\R\times \{0,1\}), (\C,\R))$. Let $(f_1,f_2)$ be a zero of $\mathcal{F}_\tau$. By linearizing $\mathcal{F}_\tau$ at $(f_1,f_2)$, we obtain a linear  operator 
			\begin{align*}
				&D\mathcal{F}_\tau: T_{(f_1,f_2)}\mathcal{Y}= W^{1,p}((\R\times [0,1],\R\times \{0,1\}), (\C,\R)) \longrightarrow L^{p}(\R\times [0,1], \C),  \nonumber \\[.5ex]
				&D\mathcal{F}_\tau(\sigma) = \partial_s \sigma + J_0 \partial_t \sigma + \left(\begin{matrix} h^{''}(e^{f_1})e^{f_1} & 0 \\ 0 & 0 \end{matrix} \right) \sigma.
			\end{align*}
			The asymptotic operators of $D\mathcal{F}_\tau$ are given by 
			\[
			A_\pm:= -J_0\frac{d}{dt} - \begin{pmatrix} h^{''}(e^{r_\pm})e^{r_\pm} & 0 \\ 0 & 0 \end{pmatrix} : W^{1,2}(([0,1], \{0,1\}),(\C,\R)) \to L^{2}([0,1], \C).
			\]
			The operator $D\mathcal{F}_\tau$ is Fredholm, and since $h''(e^{r_\pm})>0$, Lemma \ref{lem:index} and \eqref{eq:RS_line_bundle} imply that $\ind D\mathcal{F}_\tau= \mu_\RS(A_+)-\mu_\RS(A_-)=0$.
			The argument in Proposition \ref{prop:transv_strip} shows that $D\mathcal{F}_\tau$ is an isomorphism for every $\tau\in [0,1]$.
			This yields a bijection between solutions of \eqref{eq:hol_Floer_strip} and those of \eqref{eq:interpolate_0_strip}.

			Suppose now that $(f_1,f_2)=(a+e_1,e_2)$ solves \eqref{eq:interpolate_0_strip}. Then,
			\[
			\Delta f_2= \partial_s(-\partial_t f_1) + \partial_t(\partial_sf_1+h'(e^{f_1})-\partial_s\nu)= h''(e^{f_1})e^{f_1}\partial_t f_1 = - h''(e^{f_1})e^{f_1}\partial_s f_2.
			\]
			Since $f_2$ vanishes on the boundary $\R\times \{0,1\}$ and converges to $0$ as $s\to \pm\infty$, $f_2$ is identically zero by the maximum principle.
			Thus, equation \eqref{eq:interpolate_0_strip} reduces to the ODE:
			\[
				\begin{cases}
					\; \partial_s f_1 + h'(e^{f_1})-\partial_s\nu=0, \\
					\; \partial_t f_1 = 0. 
				\end{cases}
			\]
			An argument in \cite[Lemma 6.16]{DL1} (see also \cite[Proof of Proposition 4.7]{BKK24}) shows that this equation admits a unique solution.
			Therefore, \eqref{eq:hol_Floer_strip} also has a unique solution $\mathbf{e}$. Moreover, a straightforward computation shows that $(b,v)$ induced by $\tilde u$ and $\mathbf{e}$ as in the statement is indeed a solution to the Floer equation. 
			This completes the proof.	
		\end{proof}
		
		\begin{lemma}\label{lem:index}
			For $C \geq 0$, the spectrum $\sigma(A_C)$ of the operator
			\[ A_C := -J_0\frac{d}{dt}- \begin{pmatrix} C & 0 \\ 0 & 0\end{pmatrix} : W^{1,2}( ([0,1],\{0,1\}),(\C,\R)) \to L^2([0,1], \C) \]
			is given by
			\[ \left\{\frac{-C+\sqrt{C^2+4\pi^2k^2}}{2} \; | \; k \in \N \right\} \cup \left\{\frac{-C-\sqrt{C^2+4\pi^2k^2}}{2}\; | \; k\in \N \cup \{0\} \right\}.  \]
			Moreover, eigenfunctions associated with eigenvalues $\frac{-C\pm\sqrt{C^2+4\pi^2k^2}}{2}$ have relative winding numbers $\pm\frac{k}{2}$, respectively.
			Note also that $0 \notin \sigma(A_C)$ for $C >0$, while $\sigma(A_0)=\pi \Z$.
		\end{lemma}
		\begin{proof}
			Let $\lambda\in \sigma(A_C)$. 
			Then, there exists  $v:([0,1],\{0,1\})\to(\C,\R)$ satisfying
			\begin{equation}\label{eq:spectrum}
				- \begin{pmatrix} 0 & -1 \\ 1 & 0 \end{pmatrix} \dot{v}- \begin{pmatrix} C & 0 \\ 0 & 0 \end{pmatrix} v= \lambda v \; \iff \;\dot{v}=\begin{pmatrix} 0 & -\lambda \\ C+\lambda & 0 \end{pmatrix} v. 
			\end{equation}
			Due to the boundary condition, the eigenvalues of the matrix $\begin{pmatrix} 0 & -\lambda \\ C+\lambda & 0 \end{pmatrix}$ must be of the form $ i \pi k$ for $k \in \Z$.
			Solving the characteristic equation yields $\lambda=\frac{-C\pm\sqrt{C^2+4\pi^2k^2}}{2}$.
			For $C>0$, $\lambda=0$ is not an eigenvalue since the only solution to \eqref{eq:spectrum} with $\lambda=0$ and the boundary condition is the zero function. 
			For other $\lambda=\frac{-C\pm\sqrt{C^2+4\pi^2k^2}}{2}$, $v(t)=e^{\pm i \pi k t} v(0)$ with $v(0)\in \R \setminus\{0\}$ is an eigenfunction with relative winding number $ \pm \frac{k}{2}$.
		\end{proof}

		The converse of Proposition \ref{prop:correspondence_strip} also holds. We include it here for completeness, although it will not be used. 
		\begin{prop}\label{prop:correspondence_strip_2}
			Let $\tilde{v}=(b,v): (\R\times[0,1],\R\times\{0,1\}) \to (\R \times Y,\R \times \mathcal{L})$ be a Floer strip with $\ev_\pm(b)\subset (-\eta,\eta)$. Then, there exists a pair of a $J_Y$-holomorphic strip $\tilde{u}=(a,u):(\R\times[0,1],\R\times\{0,1\})  \to (\R \times Y,\R \times \mathcal{L})$ and a smooth map $\mathbf{e}=(e_1,e_2):(\R\times[0,1],\R\times\{0,1\}) \to (\C,\R)$ satisfying
			\[
			de_1-de_2\circ i +h'(e^b)ds=0,\quad \ev_\pm(u)=\ev_\pm(v),\quad (b,v)=(a+e_1, \phi_R^{e_2}\circ u).
			\]
			Moreover, such a pair $(\tilde{u},\mathbf{e})$ is determined uniquely up to the translation:
			\[c \cdot ((a,u),(e_1,e_2))=((a+c,u),(e_1-c,e_2))\quad \forall c \in \R.\]
		\end{prop}
		\begin{proof}
			We consider the Cauchy--Riemann operators
			\[
			\begin{split}
				\overline{\partial}&=\partial_s  + J_0 \partial_t : W^{1,p,\delta}_{\R}((\R\times [0,1],\R\times \{0,1\}), (\C,\R)) \to L^{p,\delta}(\R\times [0,1], \C),\\[.5ex]
				\overline{\partial}'&=\partial_s  + J_0 \partial_t : W^{1,p,-\delta}((\R\times [0,1],\R\times \{0,1\}), (\C,\R)) \to L^{p,-\delta}(\R\times [0,1], \C),
			\end{split}
			\]
			for $0<\delta<\pi$. 
			The notation $W^{1,p,\delta}_{\R}$ means that it consists of sections $\sigma$ of $W^{1,p}_{\mathrm{loc}}$ such that $e^{\delta |s|}(\sigma(s,t)-v_{\pm})\in W^{1,p}$ near the corresponding ends for some $v_{\pm} \in \R$, while $W^{1,p,-\delta}$ means that it consists of sections $\sigma$ such that  $e^{-\delta |s|}\sigma(s,t)\in W^{1,p}$. These two operators have isomorphic kernels and cokernels, see \cite[Lemma 5.20]{DL2}. The relevant asymptotic operators for $\overline{\partial}'$ are $A-\delta$ at the positive end and $A+\delta$ at the negative end, where $A:=-J_0\frac{d}{dt}$. In particular, we compute 
				$\ind  \overline{\partial}=\ind  \overline{\partial}' = \mu_\mathrm{RS}(A-\delta)-\mu_\mathrm{RS}(A+\delta)=1$.
				
				To show the surjectivity of $\overline{\partial}'$,
				we consider the formal adjoint operator $(\overline{\partial}')^*$ of $\overline{\partial}'$. 
				The relevant asymptotic operators for $(\overline{\partial}')^*$ are $A+\delta$ and $A-\delta$ at the positive and negative ends, respectively.
				Assume that there is a nonzero element $\zeta \in \ker (\overline{\partial}')^*$.
				Then, as in \eqref{eq:asymp_winding}, its asymptotic winding numbers satisfy
				\[
				\wind_{+\infty}(\zeta)\leq \alpha_-(A+\delta)=0, \qquad \wind_{-\infty}(\zeta) \geq \alpha_+(A-\delta)=\frac{1}{2}.
				\]
				This leads to a contradiction that the algebraic count of zeros of $\zeta$, $\wind_{+\infty}(\zeta)-\wind_{-\infty}(\zeta)$, is negative.
				Therefore, $\ker(\overline{\partial}')^*=0$.
				It follows that $\overline{\partial}'$ and hence $\overline{\partial}$ are surjective.
				Since $\ind\overline{\partial}=\ind\overline{\partial}'=1$, $\ker\overline{\partial}$ is 1-dimensional and consists precisely of real-valued constant functions.
			
			Next, let $\nu : \R \times [0,1] \to \R$ be a smooth function satisfying $\partial_t \nu=0$ and 
			\[
			\nu (s,t) = \begin{cases*}
				h'(e^{\ev_{+}(b)})s, & $s>1$, \\[.5ex]
				h'(e^{\ev_{-}(b)})s, & $s<-1$.
			\end{cases*}
			\]
			By the exponential convergence of $b$,  we have $\partial_s \nu-h'(e^b)\in L^{p, \delta}(\R \times [0,1], \C)$ for small $\delta>0$. Since $\overline{\partial}$ is surjective, there exists $f_1+ i f_2\in W^{1,p,\delta}_{\R}((\R\times [0,1],\R\times \{0,1\}), (\C,\R)) $ such that 
			\[
			\overline{\partial}(f_1+ i f_2)=\partial_s \nu-h'(e^b).
			\]
			A direct computation as in \cite[p.998]{DL1} shows that $(e_1, e_2):=(f_1-\nu,  f_2)$ satisfies the claimed properties. 
			This completes the proof.
		\end{proof}

		An analogue of Proposition \ref{prop:correspondence_strip} for Floer triangles also holds.
		
		\begin{prop}\label{prop:correspondence_triangle}
			Let $\tilde{u}=(a,u):(S^P, \partial S^P) \to (\R \times Y,\R \times \mathcal{L})$ be a $J_Y$-holomorphic triangle. Assume that $\ev_-(u)$, $\ev_+^{1}(u)$, $\ev_+^{2}(u)$ exist and there are $r_0,r_1,r_2\in(-\eta,\eta)$ such that $d h'(e^{r_0})=\w(\ev_-(u))$, $d h'(e^{r_1})=\w(\ev_+^1(u))$, and $d h'(e^{r_2})=\w(\ev_+^2(u))$, i.e.,~there are  chords of $X_H$ in $(-\eta,\eta)\times Y$ corresponding to $\ev_-(u)$, $\ev_+^{1}(u)$, $\ev_+^{2}(u)$.
			Then there exists a unique pair of a Floer triangle $\tilde{v}=(b,v): (S^P, \partial S^P) \to (\R \times Y,\R \times \mathcal{L})$ and a smooth map $\mathbf{e}=(e_1,e_2):(S^P, \partial S^P) \to (\C,\R)$ such that
			\begin{equation*}
				de_1-de_2\circ i +h'(e^{a+e_1})ds=0,
			\end{equation*}	
			in the interior of $S^P$, $(a,u)=(b-e_1, \phi_R^{-e_2}\circ v)$, and 
			\[
			\ev_-(\tilde{v})=(r_0,\ev_-(u)),\quad \ev_+^1(\tilde{v})=(r_1,\ev_+^1(u)),\quad\ev_+^2(\tilde{v})=(r_2,\ev_+^2(u)).
			\]
		\end{prop}
		
		\begin{proof}	   
			The proof is essentially identical to that of Proposition \ref{prop:correspondence_strip}. We only highlight the points where it differs. We again consider the family of equations in \eqref{eq:interpolate_tau_strip}, this time with a smooth function $\nu : S^P \to \R$ satisfying $\partial_t \nu=0$ and 
			\[
			\nu (s,t) = \begin{cases*}
				h'(e^{r_0})s, & $s<-1$ and $t \in [-1,1]$, \\[.5ex]
				h'(e^{r_1})s, & $s>1$ and $t \in [-1,0]$, \\[.5ex]
				h'(e^{r_2})s, & $s>1$ and $t \in [0,1]$. 		
			\end{cases*}
			\]
			We then analyze the linearized operator $
			D\mathcal{F}_\tau:  W^{1,p}((S^P,\partial{S}^P), (\C,\R)) \to L^{p}(S^P, \C)$, 
			whose asymptotic operators are
			\[
			A_{\ell}:= -J_0\frac{d}{dt} - \begin{pmatrix} h^{''}(e^{r_\ell})e^{r_\ell} & 0 \\ 0 & 0 \end{pmatrix} : W^{1,2}(([0,1], \{0,1\}),(\C,\R)) \to L^{2}([0,1], \C)
			\]
			for $\ell=0,1,2$.  By Lemma \ref{lem:index} and the standard Fredholm index formula, $D\mathcal{F}_\tau$ is Fredholm with index 
			\[
			\ind D\mathcal{F}_\tau= \chi(S^P)-\frac{3}{2}+ \mu_\RS(A_1)+\mu_\RS(A_2)-\mu_\RS(A_0)=0,
			\]
			as $\mu_\RS(A_\ell)=\frac{1}{2}$ for all $\ell=0,1,2$.
			 The argument in Proposition \ref{prop:transv_triangle} shows that $D\mathcal{F}_\tau$ is an isomorphism for every $\tau\in [0,1]$. The rest is the same as the proof of Proposition \ref{prop:correspondence_strip}.
		\end{proof}
		
		The converse of Proposition \ref{prop:correspondence_triangle} is also true, and its proof is parallel to that of Proposition \ref{prop:correspondence_strip_2}. We omit this, as it is not needed.

		\subsection{Proof of Theorem \ref{thm:isom_transfer}}\label{sec:proof_thm}
		
		The cyclic subgroup $\Z_{d}\subset S^1=\R/\Z$ acts freely on $Y$ and $\mathcal{L}$. Due to our choice of $(H,J_Y)\in\mathcal{H}\times \mathcal{J}_Y$, both of which are invariant under the $S^1$-action, the chain complex $\FC_*^{(a,b)}(\mathcal{L}^{\tilde\sigma};\mathcal{D}_H)$ and hence $\RFC_*(\mathcal{L}^{\tilde\sigma})$ admit a free $\Z_{d}$-action, see \eqref{eq:d_many_chords}. Thus, the $\Z_{d}$-equivariant complex  $\RFC_*^{\Z_{d}}(\mathcal{L}^{\tilde\sigma}):=\RFC_*(\mathcal{L}^{\tilde\sigma})/\Z_d$ is well-defined. 
		Furthermore, the chain-level triangle product on $\RFC_*(\mathcal{L}^{\tilde\sigma})$ defined in Section \ref{sec:product} descends to $\RFC_*^{\Z_{d}}(\mathcal{L}^{\tilde\sigma})$. Thus, $\RFH_*^{\Z_{d}}(\mathcal{L}^{\tilde{\sigma}})$ is defined as a $\Z$-module provided $N_L>2$, and as a ring under the assumption $N_L>\max\{\frac{1}{2}(\dim L+1),2\}$. 
		
		The quotient space $Y/\Z_{d}$ is a prequantization bundle over $\Sigma$ with Euler class $-d[\omega]_\Z$, and $\mathcal{L}/\Z_{d}$ is a Legendrian lift of $L$ to $Y/\Z_{d}$. Note that $(Y,\mathcal{L})\to (Y/\Z_{d},\mathcal{L}/\Z_{d})$ is a $d$-fold covering, and that $\RFC_*^{\Z_{d}}(\mathcal{L}^{\tilde\sigma})$ is canonically isomorphic to the ordinary Rabinowitz Floer complex $\RFC_*(\mathcal{L}^{\tilde\sigma}/\Z_{d})$ for $\R\times \mathcal{L}/\Z_{d}\subset \R\times Y/\Z_{d}$. 
		Hence, it suffices to assume $d=1$, namely $\pi|_{\mathcal{L}}:\mathcal{L}\to L$ is a diffeomorphism, and construct the isomorphism $\RFH_*(\mathcal{L}^{\tilde\sigma})\cong \QH_*(\mathcal{L}^{\sigma})$ as stated in Theorem \ref{thm:isom_transfer}.(b) and (c). 
		To this end, we fix a degree $j\in\N$ and recall from Section \ref{sec:lag_rfh} that 
		\[
		\RFC_*(\mathcal{L}^{\tilde\sigma})=\FC_*^{(a,b)}(\mathcal{L}^{\tilde\sigma};\mathcal{D}_H),\qquad  *\in\{j-1,j,j+1\}
		\] 
		for sufficiently large $H\in\mathcal{H}$ and $(a,b)$. We claim that $\FC_*^{(a,b)}(\mathcal{L}^{\tilde\sigma};\mathcal{D}_H)$ is isomorphic to $\QC_*(L;\mathcal{D})$ in such degrees. By the commutativity of the front face of the diagram in Proposition \ref{prop:orientation_commute}, we have a degree-preserving isomorphism
		\begin{equation}\label{eq:projection_orientation2}
		f_{\tilde{p}_k} : \mathfrak{o}(\tilde{p}_k) \longrightarrow \mathfrak{o}(p,\tfrac{-k}{m_L}).
		\end{equation}
		The definitions and gradings of $\mathfrak{o}(\tilde{p}_k)$ and $\mathfrak{o}(p,\frac{-k}{m_L})$ are given in \eqref{eq:generator_base}, \eqref{eq:generator_leg}, and \eqref{eq:generator_leg_grading}. 
		This together with Remark \ref{rem:disk_lift} yields that the map
		\begin{equation}\label{eq:chain_isom}
			\Theta:= \bigoplus_{(\tilde{p},k)} f_{\tilde{p}_k}: \FC_*^{(a,b)}(\mathcal{L}^{\tilde\sigma};\mathcal{D}_H) \longrightarrow \QC
			_*(L^{\sigma};\mathcal{D})
		\end{equation}
		for $*\in \{j-1,j,j+1\}$, where the direct sum ranges over all $(\tilde{p},k)\in \Crit f_{\mathcal{L}}\times \mathfrak{w}_{\mathcal{L}}^{(a,b)}(H)$ with $\mu(\tilde{p}_k)=*$, is an isomorphism between the two chain modules.
		
		It remains to verify that $\Theta$ is a chain map. Recall from \eqref{eq:differential_Leg} and \eqref{eq:differential_base} that the boundary maps $\partial_{\mathcal{D}_H}$ on $\FC_*^{(a,b)}(\mathcal{L}^{\tilde{\sigma}};\mathcal{D}_H)$ and $\partial_{\mathcal{D}}$ on $\QC_*(L^{\sigma};\mathcal{D})$ restricted to $\mathfrak{o}(\tilde{p}_k)$ and $\mathfrak{o}(p,\frac{-k}{m_L})$, respectively, are given by
		\[
		\partial_{\mathcal{D}_H} |_{\mathfrak{o}(\tilde{p}_k)} = \bigoplus_{(\tilde{q},{l})} \sum_{ [\mathbf{v}] } C([\mathbf{v}]),\qquad \partial_{\mathcal{D}} |_{\mathfrak{o}(p,-k/m_L)}=\bigoplus_{(q,\frac{-l}{m_L})} \sum_{[\mathbf{w}]} C([\mathbf{w}]).
		\]
		Here, the direct sums run over all $(\tilde{q},l)\in \Crit f_{\mathcal{L}}\times \w_{\mathcal{L}}^{(a,b)}(H)$ with $\deg\mathfrak{o}(\tilde{q}_l)=\deg\mathfrak{o}(\tilde{p}_k)-1$ and $(q,\frac{-l}{m_L})\in \Crit f_{L} \times \Z$ with $\deg \mathfrak{o}(q,\frac{-l}{m_L})=\deg \mathfrak{o}(p,\frac{-k}{m_L})-1$.
		The sums run over all $[\mathbf{v}] \in \overline{\mathcal{M}}^{*}(\tilde{q}_{l}, \tilde{p}_{k}; \mathcal{D}_H)$ and $[\mathbf{w}]\in \overline{\mathcal{N}}^{*}((q,\frac{-l}{m_L}),(p,\frac{-k}{m_L});\mathcal{D})$. 
		See \eqref{eq:moduli_base} and \eqref{eq:moduli_leg} for the definitions of the moduli spaces. The projection map $\Pi$ given in \eqref{eq:moduli_projection_strip} yields 
		\[
		\Pi: \overline{\mathcal{M}}^*(\tilde{q}_{l}, \tilde{p}_{k};\mathcal{D}_H)\longrightarrow \overline{\mathcal{N}}^{*}\big((q,\tfrac{-l}{m_L}),(p,\tfrac{-k}{m_L});\mathcal{D}\big),
		\]
		which is bijective since it admits an inverse by Lemma \ref{lem:hol_lift} and Proposition \ref{prop:correspondence_strip}. Therefore, the commutativity of the side faces of the diagram in Proposition \ref{prop:orientation_commute} yields that $\Theta$ is a chain map, i.e.,~$\Theta\circ \partial_{\mathcal{D}_H}= \partial_{\mathcal{D}}\circ \Theta$. This finishes the proof of Theorem \ref{thm:isom_transfer}.(b).
		
		\medskip
		
		Similarly, Lemma \ref{lem:hol_lift2} and Proposition \ref{prop:correspondence_triangle} imply that the projection map in \eqref{eq:projection_product} induces  a bijection between moduli spaces involved in the definitions of the product structures on $\RFH_*(\mathcal{L}^{\tilde{\sigma}})$ and $\QH_*(L^\sigma)$. This yields that, under the assumption $N_L > \max\{\frac{1}{2}(\dim L+1),2\}$, the module isomorphism $\RFH_*(\mathcal{L}^{\tilde{\sigma}})\cong \QH_*(L^\sigma)$ given by \eqref{eq:chain_isom} is a ring homomorphism. This completes the proof. \qed

		\section{Lagrangian RFH and wrapped Floer homology}\label{sec:RFH_filling}
		Let $\mathcal{L}\subset Y$ be a Legendrian lift of $L\subset \Sigma$ as before. Throughout this section, we assume that there exist an exact (or Liouville) filling $(W,\lambda)$ of $Y$ and an exact Lagrangian filling $\mathcal{L}_W\subset W$ of  $\mathcal{L}$.
		By attaching the cylindrical ends using the Liouville flow, we obtain 
		\[
		\widehat{W} := W \cup ([0,\infty) \times Y),\qquad \widehat{\mathcal{L}}_W:=\mathcal{L}_W \cup ([0,\infty) \times \mathcal{L}).
		\]
		The entire symplectization $(\R\times Y,\R\times \mathcal{L})$ is embedded in $(\widehat W,\mathcal{L}_W)$.
		By abuse of notation, we still denote the extension of the Liouville 1-form $\lambda$ to the completion $\widehat{W}$ by $\lambda$.
		In this section, we analyze the interplay between the wrapped Floer homology for $\mathcal{L}_W$ and the Rabinowitz Floer homology for $\mathcal{L}$ or $(\mathcal{L}_W,\mathcal{L})$  under the following conditions:
		\begin{enumerate}
			\item[(B1)] The exact filling $W$ is topologically simple, meaning that the inclusion-induced map $\pi_1(Y)\to \pi_1(W)$ is injective and the first Chern class $c_1^{TW}$ of $W$ vanishes on $\pi_2(W)$. 
			\item[(B2)] The exact Lagrangian filling $\mathcal{L}_W$ is topologically simple, i.e.,~the inclusion-induced map $\pi_1(Y,\mathcal{L})\to \pi_1(W,\mathcal{L}_W)$ is injective and the Maslov class $\mu_{\mathcal{L}_W}$ vanishes on $\pi_2(W,\mathcal{L}_W)$. Moreover, the minimal Maslov number $N_L$ of $L$ satisfies $N_L\geq 2$. 
			\item[(B3)]  There is a relative $\Pin$-structure $\hat{\sigma}$ on $\widehat{\mathcal{L}}_W$ that restricts to the relative $\Pin$-structure $\tilde\sigma$ on $\R\times\mathcal{L}$. More precisely, with respect to the diagram
			\begin{equation}
				\label{eq:B3}
				\begin{tikzcd}[arrows={->},column sep=small]
					\check{\mathrm{C}}^1(\widehat{\mathcal{L}}_W;\Pin(n+1)) \times \check{\mathrm{Z}}^{2}(\widehat{W};\Z_2) \arrow[r,"\iota^*"]	&   \check{\mathrm{C}}^1(\R\times \mathcal{L};\Pin(n+1)) \times \check{\mathrm{Z}}^2(\R\times Y;\Z_2)  \\
					&   \check{\mathrm{C}}^1(L;\Pin(n)) \times \check{\mathrm{Z}}^{2}(\Sigma;\Z_2) \arrow[u],
				\end{tikzcd}
			\end{equation}	
			of \v{C}ech cochains and cocycles with respect to suitable good covers, $\iota^*\hat{\sigma} = \tilde{\sigma}$ holds, where $\tilde{\sigma}$ is obtained from $\sigma$ by the lifting described in Remark \ref{rem:pin_lift} and indicated by the vertical arrow.
			Here, $\iota : (\R\times Y,\R \times \mathcal{L}) \to (\widehat{W},\widehat{\mathcal{L}}_W)$ denotes the embedding.
		\end{enumerate}

		Here, our standing assumption (C2), i.e.,~$N_L>2$, is weakened to $N_L\geq2$.

		\subsection{Lagrangian RFH in the presence of fillings}\label{sec:Floer_complex_general}
		We briefly recall the construction of $\RFH_*(\mathcal{L}_W^{\hat{\sigma}},\mathcal{L})$, where $\hat{\sigma}$ is a relative $\Pin$-structure on $\widehat{\mathcal{L}}_W$. Without this structure, the homology is understood to be defined over $\Z_2$. We denote by $H_W:\widehat{W}\to\R$ the Hamiltonian obtained by extending $H\in\mathcal{H}$ defined in \eqref{eq:Hamiltonian}  to be constant over the rest of $\widehat{W}$.  Since $\mathcal{L}_W$ is exact, there is a smooth function $g$ on ${\widehat{\mathcal{L}}_W}$ such that $dg = \lambda|_{\widehat{\mathcal{L}}_W}$. We define 
		\begin{equation}\label{eq:action_filling}
			\begin{split}
				&\mathcal{A}_{H_{W}} : C^{\infty}(([0,1],\{0,1\}),(\widehat{W},\widehat{\mathcal{L}}_W)) \longrightarrow \R\\
				&\mathcal{A}_{H_{W}}(x):= \int_0^1 x^*\lambda - \int_{0}^{1} H_{W}(x(t))dt-g(x(1))+g(x(0)).
			\end{split}
		\end{equation}
		Critical points of $\mathcal{A}_{H_{W}}$ on $(\R\times Y,\R\times \mathcal{L})\subset (\widehat W,\mathcal{L}_W)$ coincide with those of $\mathcal{A}_{H}$ defined in \eqref{eq:action_functional}. We use the same notation $\tilde{p}_k\in\Crit \mathcal{A}_{H_W}$ for a chord in $(-\eta,\eta)\times Y\subset \widehat{W}$ associated to $(\tilde p,k)\in\Crit f_{\mathcal{L}}\times\Z$. Throughout this section, the same convention applies to all chords denoted by $\tilde{p}_k$, $\tilde{q}_l$, $\tilde{q}_{k'}$, and $\tilde{r}_{k''}$. For any action window $(a,b)\subset\R$, the chain-level generators for $\RFH_*^{(a,b)}(\mathcal{L}_W^{\hat{\sigma}},\mathcal{L})$ consist of the orientation lines associated with $\tilde{p}_k$ that are contractible in $W$ relative to $\mathcal{L}_W$, equipped with equivalence classes of cappings whose actions lie in $(a,b)$. Here two cappings are equivalent if the disk obtained by gluing them has vanishing Maslov index. We write $\RFH_*(\mathcal{L}_W^{\hat{\sigma}},\mathcal{L})=\RFH_*^{(-\infty,\infty)}(\mathcal{L}_W^{\hat{\sigma}},\mathcal{L})$. In general, it may possess a larger set of chain-level generators than the chain complex for $\RFH_*(\mathcal{L})$. 
		
		For our purpose, we use a specific choice of almost complex structures on $\widehat{W}$. Let $J_W$ be a $d\lambda$-compatible almost complex structure on $\widehat{W}$ such that, on $(-2,\infty)\times Y$, $J_W$ agrees with some $J_Y\in \mathcal{J}_{(Y,\mathcal{L})}^{\mathrm{reg}}$. Let $\{ J_W^{\kappa} \}_{\kappa\in\N}$ be a family of almost complex structures on $\widehat{W}$ obtained by stretching $J_W$ along $\{-1\}\times Y\subset\widehat{W}$, e.g.~see \cite[Lemma 2.7]{DL1}. Here $\kappa$ is the stretching parameter. 
		We abbreviate $\mathcal{D}_{H_W}^\kappa:=(f_{\mathcal{L}},Z_{\mathcal{L}},H_W,J_W^\kappa)$ 
		and define the moduli space 
		\begin{equation*}
			\overline{\mathcal{M}}^*(\tilde{q}_{l}, \tilde{p}_{k};\mathcal{D}_{H_W}^\kappa)=\big\{[\mathbf{v}]=[(\tilde{v}_1,\dots,\tilde{v}_N)]\big\},
		\end{equation*}
		where $\tilde{v}_1,\dots,\tilde{v}_N$ are simple Floer strips in $(\widehat{W},\widehat{\mathcal{L}}_W)$, joined by trajectories of $Z_{\mathcal{L}}$, which connect $\tilde{q}_{l}$ and $\tilde{p}_{k}$. 
		The only difference from $\overline{\mathcal{M}}^*(\tilde{q}_{l}, \tilde{p}_{k};\mathcal{D}_H)$ is that $\tilde{v}_1,\dots,\tilde{v}_N$ are smooth maps from $(\R \times [0,1], \R \times \{0,1\})$ to $(\widehat{W},\widehat{\mathcal{L}}_W)$, rather than to $(\R \times Y, \R \times \mathcal{L})$, solving
		\[
		\partial_s \tilde{v} + J_W^\kappa(\tilde{v}) (\partial_t \tilde{v} - X_{H_W}(\tilde{v}) ) =0.
		\]
		Floer continuation strips in $(\widehat{W},\widehat{\mathcal{L}}_W)$ are defined analogously, and the Rabinowitz Floer homology $\RFH_*(\mathcal{L}_W^{\hat{\sigma}},\mathcal{L})$ is constructed in the same manner as $\RFH_*(\mathcal{L}^{\tilde{\sigma}})$. Similarly, the moduli space $\overline{\mathcal{M}}^* (\tilde{r}_{k''}, \tilde{p}_{k}, \tilde{q}_{k'}; \mathcal{D}^{P,\kappa}_{H_W})$ of simple Floer triangles with trajectories in $(\widehat{W},\widehat{\mathcal{L}}_{W})$  associated to $\mathcal{D}^{P,\kappa}_{H_W}= (f_{\mathcal{L}}, Z^1_{\mathcal{L}}, Z^2_{\mathcal{L}}, Z^3_{\mathcal{L}},H_W,J^{\kappa}_W)$  defines a ring structure on $\RFH_*(\mathcal{L}_W^{\hat{\sigma}},\mathcal{L})$.

		\begin{prop}
			\label{prop:index positivity B} 
			Assume that the stretching parameter $\kappa\in\N$ is  sufficiently large. Assume also conditions (B1) and (B2). Then the following hold. 
			\begin{enumerate}[(a)]
				\item For any $(\tilde{p},k),(\tilde{q},l)\in\Crit f_{\mathcal{L}}\times m_L\Z$ with $\mu(\tilde{p}_{k})-\mu(\tilde{q}_{l})=1$, 
			every element in $\overline{\mathcal{M}}^*(\tilde{q}_{l}, \tilde{p}_{k};\mathcal{D}_{H_W}^\kappa)$ lies in $(-1,\infty) \times Y$. The corresponding statement for moduli spaces of Floer continuation strips with trajectories holds when the asymptotic chords have the same index.
				\item Assume  $N_L \geq \mathrm{max} \{ \frac{1}{2} (\dim L +1),2\}$. Then, for any $(\tilde{p},k), (\tilde{q},k'), (\tilde{r},k'')\in \mathrm{Crit} f_{\mathcal{L}} \times m_L\Z$ with $\mu(\tilde{p}_{k}) + \mu(\tilde{q}_{k'}) -  \mu(\tilde{r}_{k''}) = \dim L$, every element in $\overline{\mathcal{M}}^* (\tilde{r}_{k''}, \tilde{p}_{k}, \tilde{q}_{k'}; \mathcal{D}^{P,\kappa}_{H_W})$
				lies in $(-1,\infty) \times Y$.
			\end{enumerate}
		\end{prop}
		This proposition follows from the proofs of Proposition \ref{prop:simple_no_escape} and Proposition \ref{prop:no_escape_triangle} with a minor modification. See \cite[Corollary 3.7.(ii)]{Ueb19} for the corresponding statement regarding Floer cylinders.  Since  we only consider the cases $\mu(\tilde{p}_{k})-\mu(\tilde{q}_{l})=1$ and $\mu(\tilde{p}_{k}) + \mu(\tilde{q}_{k'})- \mu(\tilde{r}_{k''})=\dim L$, the weaker assumption $N_L\geq 2$ is sufficient.

		\begin{prop}
			\label{prop:RFH_well_defined}
			Assume conditions (B1) and (B2). Then the following hold.
			\begin{enumerate}
				\item[(a)] The homology $\RFH_*(\mathcal{L};\Z_2)$ is well-defined and  isomorphic to $\RFH_*(\mathcal{L}_W,\mathcal{L};\Z_2)$.  
				\item[(b)] If further $N_{L} \geq \mathrm{max}\{\frac{1}{2}(\dim L+1), 2\}$, then $\RFH_*(\mathcal{L};\Z_2)$ admits the triangle product structure and the isomorphism in statement (a) is a ring isomorphism.	
				\item[(c)] If condition (B3) is also satisfied, $\RFH_*(\mathcal{L}^{\tilde{\sigma}})$ is well-defined, and statements (a) and (b) hold for $\RFH_*(\mathcal{L})$ and $\RFH_*(\mathcal{L}_W^{\hat{\sigma}},\mathcal{L})$.
			\end{enumerate}
		\end{prop}
		\begin{proof}
		This proposition is the Lagrangian counterpart of \cite[Corollaries 3.7 and 3.16]{Ueb19}. Recall that the chain-level generators for $\RFH_*(\mathcal{L};\Z_2)$ are $\tilde{p}_k$ for $(\tilde{p},k) \in \mathrm{Crit} f_{\mathcal{L}} \times m_L\Z$, where $k\in m_L\Z$ is equivalent to $\tilde{p}_k$ being contractible in $Y$ relative to $\mathcal{L}$ by Proposition \ref{prop:contractible}. By condition (B2), $\RFH_*(\mathcal{L};\Z_2)$ and $\RFH_*(\mathcal{L}_W,\mathcal{L};\Z_2)$ have the same chain-level generators. Moreover, by Proposition \ref{prop:index positivity B}, the boundary operators, continuation homomorphisms, and triangle products defining $\RFH_*(\mathcal{L};\Z_2)$ and $\RFH_*(\mathcal{L}_W,\mathcal{L};\Z_2)$ coincide via neck-stretching. This proves statements (a) and (b).
			Statement (c) can be proved analogously. 
		\end{proof}

		\begin{remark}
			In the language of augmentations, Proposition \ref{prop:index positivity B}.(a) shows that conditions (B1) and (B2) ensure that the augmentation induced by the filling $(W,\mathcal{L}_W)$ is trivial. 
			If these conditions are dropped, $\RFH_*(\mathcal{L}_W,\mathcal{L};\Z_2)$ can be viewed as the Rabinowitz Floer homology of $\mathcal{L}$ twisted by a potentially nontrivial augmentation induced by $(W,\mathcal{L}_W)$. It is tempting to conjecture that a variant of Theorem \ref{thm:isom_transfer} holds in this setting, namely, $\RFH_*(\mathcal{L}_W,\mathcal{L};\Z_2)$ corresponds to the quantum homology of $L$ equipped with a twisted differential.
		\end{remark}

		\subsection{Wrapped Floer homology and Viterbo transfer map}\label{sec:sh}
		In the rest of Section \ref{sec:RFH_filling}, we assume conditions (B1), (B2), and (B3). We recall the wrapped Floer homology $\HW_*(\mathcal{L}^{\hat{\sigma}}_{W})$ of $\mathcal{L}_{W}^{\hat{\sigma}}$ and the Viterbo transfer map $\mathcal{V} : \HW_*(\mathcal{L}^{\hat{\sigma}}_{W}) \to \RFH_*(\mathcal{L}^{\hat{\sigma}}_{W},\mathcal{L})$.

		To describe the  map $\mathcal{V}$, we define $\HW_*(\mathcal{L}^{\hat{\sigma}}_W)$ using $H_W:\widehat{W}\to\R$ appeared in the previous section. There are three types of  critical points for $\mathcal{A}_{H_W}$ given in \eqref{eq:action_filling}:
		\begin{itemize}
			\item[(I)] constant chords on $\widehat{W}\setminus((-\epsilon,0]\times Y)$,
			\item[(II)] non-constant chords in the region $(-\epsilon,-\epsilon+\eta) \times Y$, and
			\item[(III)] constant and non-constant chords in the region $(-\eta,\eta) \times Y$.
		\end{itemize}
		
		When defining the Rabinowitz Floer homology of $\mathcal{L}$ or $(\mathcal{L}_W,\mathcal{L})$, we exclude type I and II chords. In contrast,  the wrapped Floer homology of $\mathcal{L}_W$ incorporates all three types of chords.
		For any $b\in\R$, we consider the set $\m^{(a,b)}_\mathcal{L}(H)$ defined in Section \ref{sec:floer_chain} with $a=-\infty$,
		\[
		\m^{(-\infty,b)}_\mathcal{L}(H)= \{ k \in m_L \Z \,|\, \exists \, r \in (-\eta,\eta) \text{ with } h'(e^r) = \tfrac{k}{d} \text{ and } e^r h'(e^r) - h(e^r) <b\}.	
		\]
		Type II chords also appear in families similar to type III chords. For each $k\in \Z$,  we set
		\[
		\mathcal{L}_{k}^{H,\mathrm{II}}:=\big\{x=(r_x,c_x)\in \Crit\mathcal{A}_{H_W} \mid r_x \in (-\epsilon,-\epsilon+\eta),\,h'(e^{r_x})= k/d \big\},
		\]
		cf.~\eqref{eq:crit_L}. Since $h'<0$ on $(e^{-\epsilon},e^{-\epsilon+\eta}) $, see \eqref{eq:Hamiltonian}, $\mathcal{L}_{k}^{H,\mathrm{II}}$ is nonempty only for $k<0$. If $\mathcal{L}_{k}^{H,\mathrm{II}}$ is nonempty, then it is naturally identified with $\mathcal{L}$ by $(r_x,c_x)\mapsto c_x(0)$. We set
		\[
		\m^\mathrm{II}_\mathcal{L}(H):= \big\{ k \in m_L \Z \,|\, \exists \, r \in (-\epsilon,-\epsilon+\eta) \text{ with } h'(e^r) = k/d \big\}.	
		\]
		Each $(\tilde{p}, k)\in \Crit f_{\mathcal{L}}\times  \m^\mathrm{II}_\mathcal{L}(H)$ determines $\tilde{p}_k^\mathrm{II}\in \mathcal{L}_{k}^{H,\mathrm{II}}$ by  $\mathfrak{w}(\tilde{p}_k^\mathrm{II})=k$ and $\tilde{p}_k^\mathrm{II}(0)=\tilde p$. The orientation line $\mathfrak{o}(\tilde{p}_{k}^\mathrm{II})$ is defined in the same way as for type III chords and graded by 
		\begin{equation}\label{eq:index_type2}
			\mu(\tilde{p}_k^\mathrm{II}) := \dim \mathcal{L} -  \ind_{f_\mathcal{L}} (\tilde{p})+ \frac{2\tau_{\Sigma}k}{d}  - 1,
		\end{equation}
		where $-1$ is due to the concavity of $h$ on $(e^{-\epsilon},e^{-\epsilon+\eta})$, cf.~\eqref{eq:index}.

		The space of type I chords is canonically isomorphic to $\mathcal{L}_{W}\setminus((-\epsilon,0]\times \mathcal{L})$. We also use a Morse-Bott approach to type I chords following \cite[Section 3]{DL1}. We choose a Morse-Smale pair $(f_{\mathcal{L}_W},Z_{\mathcal{L}_W})$ on $\mathcal{L}_W$ such that $Z_{\mathcal{L}_W}$ equals $-\partial_r$  on $(-\epsilon,0]\times \mathcal{L}$. Let $x_q$ denote the type I chord at  $q \in \mathrm{Crit} f_{\mathcal{L}_W}$.
		As before, we have the associated orientation line $\mathfrak{o}(x_q)$ graded by \begin{equation*}
			\mu(x_q) := \dim \mathcal{L}_W -1 - \ind_{f_{\mathcal{L}_W} }(q) = \dim \mathcal{L} - \ind_{f_{\mathcal{L}_W}} (q).
		\end{equation*}

		We define 
		\begin{equation*}
			\FC_*^{(-\infty,b)}(\mathcal{L}^{\hat{\sigma}}_W;H_W) := \bigoplus_{q \in  \mathrm{Crit} f_{\mathcal{L}_W}} \mathfrak{o}(x_q) \oplus \bigoplus_{(\tilde{p},k)}  \mathfrak{o}(\tilde{p}_k^\mathrm{II})\oplus  \bigoplus_{(\tilde{p},k)}\mathfrak{o}(\tilde{p}_k), 
		\end{equation*}
		where the second and third big direct sums are over all  $(\tilde{p},k)\in \mathrm{Crit} f_{\mathcal{L}}\times \m^\mathrm{II}_\mathcal{L}(H) $ and $(\tilde{p},k)\in \mathrm{Crit} f_{\mathcal{L}}\times \m^{(-\infty,b)}_\mathcal{L}(H)$, respectively. 		Note that if $H_W$ is sufficiently large in $(\mathcal{H},\prec)$, then the Hamiltonian actions of type I and II chords are  lower than any chosen $b$. 
		The boundary map, as well as the continuation maps, are defined in a similar way as before. 
		Let us denote by $\FH^{(-\infty,b)}_*(\mathcal{L}^{\hat{\sigma}}_W;H_W)$ the resulting homology. The filtered and full wrapped Floer homologies of $\mathcal{L}_W^{\hat{\sigma}}$ are defined by
		\begin{equation*}
			\HW_*^{(-\infty,b)}(\mathcal{L}_W^{\hat{\sigma}}) :=   \varinjlim_{H_{W}} \FH^{(-\infty,b)}_*(\mathcal{L}^{\hat{\sigma}}_W;H_{W}),\qquad 	\HW_*(\mathcal{L}_W^{\hat{\sigma}}) := \varinjlim_{b \to \infty} \HW_*^{(-\infty,b)}(\mathcal{L}_W^{\hat{\sigma}}).
		\end{equation*}
		The triangle product  induces a ring structure on these homologies.

		Let $n=\dim L$. 
		For any $a<0<b$, the short exact sequence of chain complexes 
		\[
		0\to \FC_*^{(-\infty,a)}(\mathcal{L}^{\hat{\sigma}}_W; {H_W}) \to \FC_*^{(-\infty,b)}(\mathcal{L}^{\hat{\sigma}}_W; {H_W})\to \FC_*^{(a,b)}(\mathcal{L}^{\hat{\sigma}}_W; {H_W})\to 0,
		\]
		where the last one is the quotient complex, induces a long exact sequence 
		\[
			\cdots \to \HW^{n-i}_{(-\infty,-a)}(\mathcal{L}^{\hat{\sigma}}_W) \to  \HW_i^{(-\infty,b)} (\mathcal{L}^{\hat{\sigma}}_W)  \to \RFH_{i}^{(a,b)} ( \mathcal{L}^{\hat{\sigma}}_{W},\mathcal{L} ) \to  \HW^{n-i+1}_{(-\infty,-a)} (\mathcal{L}^{\hat{\sigma}}_W) \to \cdots
		\] 
		by taking the limit over $H_W$, see \cite{CFO10} for details. Here and below $\HW^{*}$ refers to wrapped Floer cohomology. 
		Passing to the limits $ \varinjlim\limits_{b\uparrow+\infty}\varprojlim\limits_{a\downarrow-\infty}$, we obtain 
		\begin{equation}\label{eq:les_SH}
			\cdots \to \HW^{n-i}(\mathcal{L}^{\hat{\sigma}}_W) \to  \HW_i (\mathcal{L}^{\hat{\sigma}}_W)  \xrightarrow{\mathcal{V}} \RFH_{i} ( \mathcal{L}^{\hat{\sigma}}_{W},\mathcal{L}) \to  \HW^{n-i+1} (\mathcal{L}^{\hat{\sigma}}_W) \to \cdots.
		\end{equation}   
		The Viterbo transfer map $\mathcal{V}$ is a ring homomorphism.
		
		Following \cite{Ueb19,CO26}, we define the reduced wrapped Floer homology 
		\[
		\HW^{\mathrm{red}}_*(\mathcal{L}^{\hat{\sigma}}_W):= \HW_* (\mathcal{L}^{\hat{\sigma}}_W) / \ker \mathcal{V}.
		\] 
		Since $\mathcal{V}$ is a ring homomorphism, $\ker \mathcal{V}$ is an ideal and $\HW^{\mathrm{red}}_*(\mathcal{L}^{\hat{\sigma}}_W)$ is a ring with unit. Thus $\mathcal{V}$ induces an injective ring homomorphism 
		\begin{equation*}
			\mathcal{V}^\mathrm{red}:\HW^{\mathrm{red}}_*(\mathcal{L}^{\hat{\sigma}}_{W})\longrightarrow\RFH_{*}(\mathcal{L}^{\hat{\sigma}}_{W},\mathcal{L}).
		\end{equation*}
		We observe from the construction of \eqref{eq:les_SH} that $\HW^*(\mathcal{L}^{\hat{\sigma}}_W)$ is generated by type I and II chords. Thus, by \eqref{eq:index_type2}, $\HW^*(\mathcal{L}^{\hat{\sigma}}_W)$ is supported in degrees  bounded from below. This implies 
		\begin{equation}\label{eq:sh_red}
			\HW_i(\mathcal{L}^{\hat{\sigma}}_W)= \HW_i^\mathrm{red}(\mathcal{L}^{\hat{\sigma}}_W), \qquad \mathcal{V} =\mathcal{V}^{\mathrm{red}}: \HW_i(\mathcal{L}^{\hat{\sigma}}_W) \stackrel{\cong}{\to} \RFH_i(\mathcal{L}^{\hat{\sigma}}_{W},\mathcal{L}),  \quad \forall i\gg0.
		\end{equation}   
		In fact, the map $\HW^{n-i}(\mathcal{L}^{\hat{\sigma}}_W) \to  \HW_i(\mathcal{L}^{\hat{\sigma}}_W)$ in \eqref{eq:les_SH} factors through the map $\H_{i+1}(\mathcal{L}_W;\mathcal{O}) \to \H_{i+1}(\mathcal{L}_W,\mathcal{L};\mathcal{O})$ induced by the inclusion, where $\mathcal{O}$ denotes the relevant orientation local system. Thus, $\HW_i(\mathcal{L}^{\hat{\sigma}}_W)=\HW^{\mathrm{red}}_i(\mathcal{L}^{\hat{\sigma}}_W)$ for all $i>n=\dim L$.

		\subsection{Localization of reduced wrapped Floer homology}
		Throughout this section, we assume that $N_L$ is finite and satisfies $N_L \geq \mathrm{max}\{ \frac{1}{2}(\dim L+1), 2\}$. By Proposition \ref{prop:RFH_well_defined}, $\RFH_*(\mathcal{L}^{\tilde{\sigma}})$ is isomorphic to $\RFH_*(\mathcal{L}_W^{\hat{\sigma}},\mathcal{L})$ as a ring, and we denote the multiplicative unit by $1_{\RFH}$. We choose a Morse function $f_L$ on $L$ with exactly one local minimum point, which we denote by $q \in \mathrm{Crit}f_L$.  
		As before, for $i=1,\dots, d$, let $q^i_0 \in \mathcal{L}^{H}_0$ be the Hamiltonian chord with $\w(q^i_0)=0$ and $q^i_0(0)=q^i \in \mathrm{Crit} f_{\mathcal{L}}$ that projects to $q$, i.e.,~$q_0^i$ is the constant chord at $q^i$. The orientation line $\mathfrak{o}(q^i_0)$ has a canonical orientation. 
		To see this, let $\hat{x}$ and $w$ be the constant cappings of $q^i_0$ and $q$, respectively. The horizontal part of  $D_{\hat{x}}\# 0$ and $D_{w}\#0$ have isomorphic kernel and cokernel as discussed in Section \ref{sec:orientation_lines}. Thus, the horizontal part has a canonical orientation. The vertical part of $D_{\hat{x}}\# 0$ has a canonical orientation as discussed below \eqref{eq:gluing_iso}. As a consequence, $D_{\hat{x}}\# 0$ has a canonical orientation, which determines a generator $o_{q^i_0}\in \mathfrak{o}(q^{i}_{0})$. The unit $1_{\RFH}$ is represented by 
		\begin{equation}
			\label{unit_RFH}
			o_{q_0}:=\sum_{i=1}^{d} o_{q^i_0}  \in \RFC_{n}(\mathcal{L}^{\tilde{\sigma}}).
		\end{equation}
		As discussed in Remark \ref{rem:RFH_laurent}, $\RFC_*(\mathcal{L}^{\tilde{\sigma}})$ admits a $\Z[T,T^{-1}]$-algebra structure which is compatible with the one in $\QC_*(L^\sigma)$. In the case where $\deg T=N_L$,  $T(1_\RFH)\in \RFH_{n+N_L}(\mathcal{L}^{\tilde{\sigma}})$ is represented by 
		\[
		o_{q_{m_L}}:= \sum_{i=1}^{d} o_{q^i_{m_L}} \in \RFC_{n+N_L}(\mathcal{L}^{\tilde{\sigma}}),
		\]
		where $o_{q^i_{m_L}}:=T o_{q^i_0} \in \mathfrak{o}(q^{i}_{m_L})$ for  $i=1,\dots,d$. When $\deg T=2N_L$, we have the same formula with $N_L$ and $m_L$ replaced by $2N_L$ and $2m_L$, respectively.

		\medskip
		
		\begin{lemma}\label{lem:T_image}
			The homology class $T(1_{\RFH})$ is in the image of the map $\mathcal{V}$.
		\end{lemma}
		\begin{proof}
			We present the proof for the case where $\deg T=N_L$. The same proof applies to the other case.
			Recall from Section \ref{sec:sh} that $\mathcal{V}$ is induced by the quotient map 
			\[
			\FC_*^{(-\infty,b)}(\mathcal{L}^{\hat{\sigma}}_W; {H_W}) \longrightarrow  \FC_*^{(a,b)}(\mathcal{L}^{\hat{\sigma}}_W; {H_W}).
			\]
			Let $a\ll 0\ll b$ and let $H_W$ be sufficiently large. The chain $o_{q_{m_L}}\in \FC_*^{(a,b)}(\mathcal{L}^{\hat{\sigma}}_W; {H_W})$ can be viewed as a chain in $\FC_*^{(-\infty,b)}(\mathcal{L}^{\hat{\sigma}}_W; {H_W})$, where these two chains are in correspondence via the above quotient map.   
			Therefore, it suffices to show that $o_{q_{m_L}}$ is a cycle in $\FC_*^{(-\infty,b)}(\mathcal{L}^{\hat{\sigma}}_W; {H_W})$, i.e.,~$\partial_{H_W}(o_{q_{m_L}})=0$, where $\partial_{H_{W}}$ is the boundary map. 
			To this end, we write 
			\begin{equation*}
				\partial_{H_W}(o_{q_{m_L}})= \sum_{z \in \Crit f_{\mathcal{L}_W} } a_z o_{x_z}  + \sum_{(\tilde p,k)} b_{\tilde{p},k} o_{\tilde{p}_k^\mathrm{II}}  + \sum_{(\tilde p,k)}   c_{\tilde{p},k} o_{\tilde{p}_k} 
			\end{equation*}
			where the last two sums range over  $(\tilde p,k)\in \Crit f_\mathcal{L}\times \m^\mathrm{II}_\mathcal{L}(H)$ and  $(\tilde p,k)\in \Crit f_\mathcal{L}\times \m^{(-\infty,b)}_\mathcal{L}(H)$. 
			Here, $o_{x_z} \in \mathfrak{o}(x_z)$, $o_{\tilde{p}_k^\mathrm{II}} \in \mathfrak{o}(\tilde{p}_k^\mathrm{II})$, $o_{\tilde{p}_k} \in \mathfrak{o}({\tilde{p}_k})$ are generators, and $a_z, b_{\tilde{p},k},c_{\tilde{p},k}\in\Z$.
			
			We first show $c_{\tilde{p},k}=0$. Since the Hamiltonian action decreases along $\partial_{H_W}$, we have $c_{\tilde{p},k}=0$ provided $k > m_L$. 
			For $k <m_L$, if $c_{\tilde{p},k} \neq 0$, we arrive at a contradiction:
			\[
			1= \mu(\tilde{q}_{m_L}) - \mu(\tilde{p}_{k}) = \frac{2\tau_{\Sigma} m_L}{d} -\frac{2\tau_{\Sigma}k}{d} -\ind_{f_{\mathcal{L}}}(\tilde{q}) + \ind_{f_{\mathcal{L}}}(\tilde{p}) \geq N_L \geq 2.
			\]
			Here we used that $N_L=\frac{2\tau_\Sigma m_L}{d}$ and $\ind_{f_{\mathcal{L}}}(\tilde{q})=0$. In the case of $k=m_L$, the boundary map contributing to $c_{\tilde{p},m_L}$ is solely given by flow lines of $Z_{\mathcal{L}}$. However, this contribution also vanishes since $q$ is the unique minimum point of $f_L$. This proves $c_{\tilde{p},m_L}=0$. 
			
			An analogous degree argument also shows $b_{\tilde{p}, k}=0$. Finally, $a_z = 0$ since otherwise
			\[
			1=\mu(\tilde q_{m_L})- \mu(z)=  n+\frac{2\tau_{\Sigma} m_L}{d} - \ind_{f_{\mathcal{L}}}(\tilde{q}) - (n - \ind_{f_{\mathcal{L}_W}}(z))= N_L + \ind_{f_{\mathcal{L}_W}}(z)  \geq 2.
			\]
			Therefore, $\partial_{H_W}(o_{q_{m_L}})=0$ holds, and $T(1_{\RFH})$ lies in the image of $\mathcal{V}$. 
		\end{proof}

		By Lemma \ref{lem:T_image}, there is a unique  $\beta\in \HW_*^{\mathrm{red}}(\mathcal{L}^{\hat{\sigma}}_W)$ satisfying $\mathcal{V}^\mathrm{red}(\beta)=T(1_\RFH)$. 	Since $\mathrm{\mathcal{V}}^{\mathrm{red}}$ is injective, $\beta$ lies in the center of the ring $\HW_*^{\mathrm{red}}(\mathcal{L}_{W}^{\hat{\sigma}})$.
		Using the triangle product  $\star$ on $\HW_*^{\mathrm{red}}(\mathcal{L}^{\hat{\sigma}}_W)$,
		we define a $\mathbb{Z}[T]$-algebra structure on $\HW_*^{\mathrm{red}}(\mathcal{L}^{\hat{\sigma}}_W)$ by
		\[
		\mathbb{Z}[T] \times \HW_*^{\mathrm{red}}(\mathcal{L}^{\hat{\sigma}}_W) \to \HW_*^{\mathrm{red}}(\mathcal{L}^{\hat{\sigma}}_W), \qquad (T^k,a) \mapsto \beta^k\star a.
		\]
		One can readily see that $\mathcal{V}^{\mathrm{red}}$ is $\Z[T]$-linear and hence a $\Z[T]$-algebra homomorphism. 
		Let us consider the localization 
		\[
		T^{-1}\HW_*^{\mathrm{red}}(\mathcal{L}^{\hat{\sigma}}_W) :=\big\{\beta^k\big\}_{k \in \N\cup\{0\}}^{-1}\, \HW_*^{\mathrm{red}}(\mathcal{L}^{\hat{\sigma}}_W),
		\]
		which has a natural $\mathbb{Z}[T,T^{-1}]$-algebra structure.

		\begin{prop}\label{prop:localization}
			The map $\mathcal{V}^\mathrm{red}$ induces an $\mathbb{Z}[T,T^{-1}]$-algebra isomorphism
			\[
			\widehat{\mathcal{V}}^\mathrm{red}:T^{-1}\HW_*^{\mathrm{red}}(\mathcal{L}^{\hat{\sigma}}_W) \to \RFH_*(\mathcal{L}_W^{\hat{\sigma}},\mathcal{L}) \cong \RFH_*(\mathcal{L}^{\tilde{\sigma}}),\qquad 
				\frac{A}{\beta^k} \mapsto T^{-k} (\mathcal{V}^\mathrm{red} (A)),
			\]
			where $A\in \HW_*^{\mathrm{red}}(\mathcal{L}^{\hat{\sigma}}_W)$ and $k\in\N\cup\{0\}$.
		\end{prop}
		\begin{proof}
				By Lemma \ref{lem:RFH_fg}, $\RFH_*(\mathcal{L}^{\tilde{\sigma}})$ is generated by some $A_{1},\dots, A_{l} \in \RFH_*(\mathcal{L}^{\tilde{\sigma}})$ over $\Z[T,T^{-1}]$.
				As pointed out in \eqref{eq:sh_red},  $\mathcal{V}^\mathrm{red}:\HW_i^\mathrm{red}(\mathcal{L}^{\hat{\sigma}}_W) \to \RFH_i(\mathcal{L}^{\tilde{\sigma}})$ is an isomorphism for  large $i$. Thus, for large $N \in \mathbb{N}$, $T^N(A_j)$ lies in the image of $\mathcal{V}^\mathrm{red}$ for every $j=1,\dots,l$. This proves that $\widehat{\mathcal{V}}^\mathrm{red}$ is surjective. Moreover, $\widehat{\mathcal{V}}^\mathrm{red}$ is injective since $\mathcal{V}^\mathrm{red}$ is so.  This completes the proof. 
		\end{proof}

	\section{Applications}

	\subsection{Computations for the unit cotangent fiber of Zoll manifolds}\label{sec:CROSS}

	Let $T^*N$ be the cotangent bundle of a closed, connected, $n$-dimensional Riemannian manifold $N$ with $n \geq 3$, equipped with its canonical exact symplectic form.
	We denote the unit disk bundle by $D^*N$ and the unit circle bundle by $S^*N=\partial D^*N$ with respect to the dual metric on $T^*N$.
	The unit cotangent fiber $S^*_qN$ over $q\in N$ is a Legendrian sphere, which comes with an exact Lagrangian filling $D^*_qN$.
	Let $w_2^{TN}\in \H^2(N;\Z_2)$ denote the second Stiefel--Whitney class of the tangent bundle $TN$. For the projection $\mathrm{pr}:T^*N\to N$, we write 
	\begin{equation}
		\label{eq:background_class_cotangent_bdl}
		\hat{b} := \mathrm{pr}^* w_2^{TN}.
	\end{equation}
	Since $T^*_qN$ is contractible, there is a unique relative $\mathrm{Spin}$-structure $\hat\sigma$ on $T^*_qN$ with respect to the background class $\hat{b}$. Let $\tilde\sigma$ be the restriction of $\hat\sigma$ to $\R\times S^*_qN$. 
	We also use the notation $(D_q^*N;\hat{\sigma})=(D_q^*N)^{\hat{\sigma}}$ and $(S_q^*N;\tilde{\sigma})=(S_q^*N)^{\tilde\sigma}$.
	
	In general, the Rabinowitz Floer homology for the symplectization of $S^*N$ may not be well-defined. However, as in Section \ref{sec:Floer_complex_general}, one can define the Rabinowitz Floer homology $\RFH_*(D^*_qN,S^*_qN;\hat\sigma)$  for the pair $(D^*_qN,S^*_qN)$ with $\hat\sigma$ inside $T^*N$.

		\begin{lemma}
		\label{lem:viterbo_injective}
		The transfer map $\mathcal{V} : \HW_*(D^*_qN;\hat\sigma) \to \RFH_*(D^*_qN,S^*_qN;\hat\sigma)$ is injective. In particular, $\HW_*^{\mathrm{red}}(D_q^*N;\hat\sigma)=\HW_*(D_q^*N;\hat\sigma)$. 
	\end{lemma}
	\begin{proof}
		We recall from \eqref{eq:les_SH} the long exact sequence
		\[\dots \to \HW^{n-1-i}(D^*_qN;\hat\sigma ) \to  \HW_{i} (D^*_qN;\hat\sigma) \xrightarrow{\mathcal{V}} \RFH_{i}(D^*_qN,S^*_qN;\hat\sigma)  \to  \cdots .\]
		As established in \cite{APS08, AS10a,Abo12,AS14}, we have
		\begin{equation}
			\label{eq:basedloop}
			\HW_*(D^*_qN;\hat\sigma) \cong \H_{*-n+1}(\Omega^0 N),\qquad \HW^{n-1-*}(D^*_qN;\hat\sigma) \cong \H^{-*}(\Omega^0 N),
		\end{equation}
		where $\Omega^0 N$ is the contractible component of the based loop space of $N$. Recall that, under our grading convention, the unit of $\HW_*(D^*_qN;\hat\sigma)$ lies in degree $\dim S_q^*N=n-1$. 
		This yields that $\HW_{i} (D^*_qN;\hat\sigma) =0$ for $i <n-1$, and $\HW^{n-1-i}(D^*_qN;\hat\sigma) =0$ for $i> 0$. Therefore, the map $\HW^{n-1-i}(D^*_qN;\hat\sigma) \to \HW_{i}(D^*_qN;\hat\sigma)$ vanishes, and $\mathcal{V}$ is injective. The last assertion follows from the definition of the reduced wrapped Floer homology.
	\end{proof}

	\begin{lemma}\label{lem:homotopy_cross}
		For any $k\in\N$, the map $\pi_k(S^*N,S^*_qN)\to \pi_k(D^*N,D^*_qN)$ induced by the inclusion $\iota:(S^*N,S^*_qN)\hookrightarrow (D^*N,D^*_qN)$ is an isomorphism. 
	\end{lemma}
	\begin{proof}
		The projections $\wp:(D^*N,D^*_qN) \to(N,q)$ and $\wp|_{S^*N}:(S^*N,S^*_qN)\to (N,q)$ induce isomorphisms on homotopy groups. 
		Since $\wp|_{S^*N}=\wp\circ \iota$, $\iota$ also induces isomorphisms. 
	\end{proof}

	This lemma and the well-known fact $\mu_{D_q^*N}(\pi_2(D^*N,D_q^*N))=0$ verify condition (B2) (except the condition on $N_L$) introduced in Section \ref{sec:RFH_filling}. Condition (B1)  holds automatically in the current setting.
	
	For condition (B3) (as well as $N_L$) to be well-defined, we now assume that the metric on $N$ is Zoll, i.e.,~all geodesics are closed and have the same minimal length. 
	After rescaling, we normalize this minimal length to be $1$.
	The associated cogeodesic flow induces a free $S^1$-action on  $S^*N$, preserving the Liouville 1-form, and hence defining a prequantization bundle $\pi_{S^*N}:S^*N\to \Sigma:=S^*N/S^1$.
	For each $q\in N$, the unit cotangent fiber $S^*_qN$ projects to a possibly immersed Lagrangian sphere in $\Sigma$. 
	We further assume that all prime geodesics on $N$ are simple, i.e.,~have no self-intersections, so that 
	\begin{equation*}
		L_N:=\pi_{S^*N}(S^*_qN)
	\end{equation*}
	is an embedded Lagrangian sphere in $\Sigma$.

	\begin{remark}[Cf. Remark \ref{rem:pin}]\label{rem:spin}
		Since $S^*_q N$ and $D^*_q N$ are orientable, we consider relative $\mathrm{Spin}$-structures rather than relative $\Pin$-structures. 		
		Recall that $\mathrm{Spin}(n)$ is the unique central extension of $\mathrm{SO}(n)$ by $\Z_2$.
		An $n$-dimensional orientable submanifold $Q$ of $M$ is said to be relatively $\mathrm{Spin}$ if
		\[ 
		w_2^{TQ} \in \mathrm{im}\Big(\iota^*: \H^2(M;\Z_2) \to \H^2(Q;\Z_2)\Big).
		\]
		Let $\mathcal{U}$ and $\mathcal{V}$ be good covers of $Q$ and $M$, respectively, as in Remark \ref{rem:pin}. A relative $\mathrm{Spin}$-structure on $Q$ is an equivalence class of a pair 
		\[
		\sigma=(g,\beta) \in \check{\mathrm{C}}^1(\mathcal{U};\mathrm{Spin}(n)) \times \check{\mathrm{Z}}^2(\mathcal{V};\Z_2)
		\] 
		where $g$ is a lift of the element of $\check{\mathrm{C}}^1(\mathcal{U};\mathrm{SO}(n))$ representing the frame bundle of $TQ$ and $\beta \in \check{\mathrm{Z}}^2(\mathcal{V};\Z_2)$ satisfies $\beta|_L = \delta g$.
		The equivalence relation is given as in Remark \ref{rem:pin}.
	\end{remark}
	
	\begin{lemma}
		\label{lem:pin_cross}
		Assume $\dim N\geq 3$. Let $\hat\sigma=[(\hat{g},\hat{\beta})]$ be the unique relative $\mathrm{Spin}$-structure with background class $\hat{b}=\mathrm{pr}^*w_2^{TN}$ as above. If there exists $b\in  \H^2(\Sigma;\Z_2)$ satisfying 
		\begin{equation*}
			\iota^*\hat{b} =\pi_{S^*N}^*b \in \H^2(S^*N;\Z_2)
		\end{equation*}
		for $\iota:S^*N \hookrightarrow T^*N$, then, for each such $b$, there is a unique relative $\mathrm{Spin}$-structure $\sigma = [(g,\beta)]$ with $[\beta]=b$ on $L_N$ whose lift $\tilde{\sigma}$ to $\R\times S^*_q N$ coincides with $\hat{\sigma}$ in the sense of (B3). 
	\end{lemma}
	\begin{proof}
		We consider the diagram in \eqref{eq:B3} with $\Pin$ replaced by $\mathrm{Spin}$. Let $[(\iota^*\hat g,\iota^*\hat\beta)]$ be a relative $\mathrm{Spin}$-structure on $\R\times  S^*_qN$ obtained from $\hat\sigma$. Since $[\iota^*\hat\beta]=\iota^*\hat{b}$, we have $w_2^{TS_q^*N}=(\iota^*\hat b)|_{S_q^*N}=(\pi_{S^*N}^*b)|_{S_q^*N}$. Since $\pi^*_{S^*N}$ restricts to a diffeomorphism from $S^*_qN$ to $L_N$, we have $w_2^{TL_N}=b|_{L_N}$ and in particular $w_2^{TL_N}$ is in the image of the map $\H^2(\Sigma;\Z_2)\to \H^2(L_N;\Z_2)$. Therefore, there exists a relative $\mathrm{Spin}$-structure $[(g,\beta)]$ on $L_N$ with $[\beta]=b$. 
		
		It remains to show that the lift $\tilde\sigma$ of $\sigma$ with background class $[\pi_{S^*N}^*\beta]$ agrees with $[(\iota^*\hat g,\iota^*\hat\beta)]$. 
		Since $\dim N\geq 3$, the map $\H^2(S^*N,S^*_qN;\Z_2)\to \H^2(S^*N;\Z_2)$ is injective. Hence, there exists a unique relative $\mathrm{Spin}$-structure with respect to the background class $\iota^*\hat{b}$. Both $\pi^*_{S^*N}\beta$ and $\iota^*\hat\beta$ represent $\iota^*\hat{b}$, and this completes the proof. 
	\end{proof}

	\begin{cor}\label{cor:zoll}
		Let $N$ be a closed, connected, Riemannian Zoll manifold with $\dim N \geq 3$, all of whose prime geodesics are simple. Assume that there exists a relative $\mathrm{Spin}$-structure $\sigma$ on $L_N$ described in Lemma \ref{lem:pin_cross} and that the minimal Maslov number of $L_{N}$ is at least $\mathrm{max}\{\frac{1}{2}(\dim L_{N}+1),2\}$. Then, there exist $\Z[T,T^{-1}]$-algebra isomorphisms
		\[
		\QH_*(L_N^\sigma) \cong \RFH_*(S^*_qN;\tilde\sigma) \cong T^{-1}\HW_*(D_q^*N;\hat\sigma),
		\]
		where the $\Z[T]$-algebra structure on $\HW_*(D_q^*N;\hat\sigma)\cong \HW_*^\mathrm{red}(D_q^*N;\hat\sigma)$ is defined before Proposition \ref{prop:localization}.
	\end{cor}
	\begin{proof}
		By Theorem \ref{thm:isom_transfer}, Proposition \ref{prop:RFH_well_defined}, and Proposition \ref{prop:localization}, we have isomorphisms 
		\[
		\QH_*(L_N^\sigma) \cong\RFH_*(S^*_{q}N;\tilde{\sigma}) \cong \RFH_*(D^*_qN,S^*_qN;\hat\sigma) \cong T^{-1} \HW_*^{\mathrm{red}}(D^*_qN;\hat\sigma).
		\]
		By Lemma \ref{lem:viterbo_injective}, $\HW_*^{\mathrm{red}}(D^*_qN;\hat\sigma) \allowbreak = \HW_*(D^*_qN;\hat\sigma)$, and this finishes the proof.
	\end{proof}

	In the rest of this section, we examine the Lagrangian sphere $L_N$ for the cases $N\in\{S^k,\mathbb{RP}^k,\mathbb{CP}^k\}$ equipped with the standard metric. In these cases,  well-known computations of $\H_*(\Omega^0N)$ enable one to compute $\QH_*(L_N^\sigma)$ as a ring.

	\subsubsection{$N=S^k$ for $k\geq 3$}
	\label{sec:sphere}
	We consider the unit cotangent bundle of the unit sphere $S^k\subset \R^{k+1}$,
	\begin{equation*}
		D^*S^{k} = \{(\mathbf{a},\mathbf{b}) \in \R^{k+1} \times \R^{k+1} \mid |\mathbf{a}|^2\leq 1,\; |\mathbf{b}|^2 = 1,\; \mathbf{a} \cdot \mathbf{b} =0\}
	\end{equation*} 
	and the smooth quadric 
	\[
	Q^{k-1} = \{ [z_0:\dots :z_{k}] \in \mathbb{CP}^{k} \mid z_0^2 +\dots +z_{k}^2=0\}
	\]
	equipped with the restriction of the Fubini-Study form $\omega_{\mathrm{FS}}$. 
	The cogeodesic flow on $S^*S^k$ induces the $S^1$-action on $S^*S^k$			
	\[t \cdot (\mathbf{a},\mathbf{b}) = \big(\cos(2\pi t) \mathbf{a} - \sin(2\pi t) \mathbf{b}, \sin (2\pi t)\mathbf{a} + \cos (2\pi t) \mathbf{b}\big),\qquad  t\in \R/\Z=S^1.
	\]
	The associated prequantization bundle is given by 
	\[
	\begin{split}
		\pi_{S^*S^k}:S^*S^k\to Q^{k-1},\qquad (\mathbf{a},\mathbf{b}) &\mapsto [a_{0}+i b_{0}: \dots : a_{k} + i b_{k}]
	\end{split}
	\]
	with the Euler class $-[\omega_\mathrm{FS}|_{Q^{k-1}}]$.
	The unit cotangent fiber $S^*_qS^k$ projects to \\ 
	\[
		L_{S^k} =\left\{ [i : x_1:\dots :x_k] \in \mathbb{CP}^{k} \mid  \sum_{j=1}^{k} x_j^2 =1, x_j\in \R\right\} \subset Q^{k-1}.
	\]
	It is a monotone Lagrangian sphere with minimal Maslov number $2(k-1)$ as $(Q^{k-1},\omega_{\mathrm{FS}}|_{Q^{k-1}})$ is monotone with minimal Chern number $k-1$.

	Next, we investigate relative $\mathrm{Spin}$-structures on $L_{S^k}$. Let $[\omega_\mathrm{FS}|_{Q^{k-1}}]_{\Z_2}\in\H^2(Q^{k-1};\Z_2)$ be the class corresponding to $[\omega_\mathrm{FS}|_{Q^{k-1}}]$. Since the second Stiefel--Whitney class $w_2^{TS^k}  \in \H^2(S^k;\Z_2)$ vanishes, the background class $\hat{b} \in \H^2(T^*S^k;\Z_2)$ defined in \eqref{eq:background_class_cotangent_bdl} is zero. Thus $\hat{\sigma}$ is actually an absolute $\mathrm{Spin}$-structure.
	The Gysin exact sequence
	\begin{equation*}
		\dots \to \H^0(Q^{k-1};\Z_2) \xrightarrow{\cup\, [\omega_\mathrm{FS}|_{Q^{k-1}}]_{\Z_2}} \H^2(Q^{k-1};\Z_2) \xrightarrow{\pi_{S^*S^k}^*} \H^2(S^*S^k;\Z_2) \to 0,
	\end{equation*}
	where we used the fact $\H^1(Q^{k-1};\Z_2) =0$, implies that $\pi_{S^*S^k}^*$ is surjective and thus an (absolute) $\mathrm{Spin}$-structure $\sigma$ on $L_{S^k}$ satisfying $\tilde\sigma=\iota^*\hat\sigma$  exists due to  Lemma \ref{lem:pin_cross}.
	Furthermore, there are precisely two such relative $\mathrm{Spin}$-structures since, by Lemma \ref{lem:pin_cross}, there is a unique one for each element in $(\pi_{S^*S^k}^*)^{-1}(\iota^*\mathrm{pr}^*w_2^{TS^{k}})=\ker(\pi_{S^*S^k}^*)\cong\Z_2\langle[\omega_\mathrm{FS}|_{Q^{k-1}}]\rangle$.
	Let $\sigma$ be one of these two relative $\mathrm{Spin}$-structures on $L_{S^k}$. 
	
	\begin{prop}\label{prop:algebra_iso_for_sphere}(Cf.~\cite[Remark 36]{KS21a})
		There exists a ring isomorphism 
		\[
		\QH_*(L_{S^k}^{\sigma})[k-1]  \cong \mathbb{Z}[x,x^{-1}],
		\]
		where $x$ has degree $k-1$. Here $[k-1]$ denotes a degree shift by $k-1$.
	\end{prop}
	\begin{proof}
		The minimal Maslov number $2(k-1)$ of ${L_{S^k}}$ is greater than $\mathrm{max}\{ \frac{1}{2}(\dim L_{S^{k}}\allowbreak+1), 2\}$. It follows from Corollary \ref{cor:zoll} that
		$\QH_*(L_{S^k}^{\sigma})  \cong T^{-1}\HW_*(D_q^*S^k;\hat\sigma)$
		as $\mathbb{Z}[T,T^{-1}]$-algebras. 	
		By \cite[Theorem III.1.A]{BS53} and \eqref{eq:basedloop}, we have $\HW_*(D_{q}^* S^k;\hat\sigma)[k-1] \cong \H_*(\Omega S^{k})\cong \mathbb{Z}[x]$ where the formal variable $x$ has degree $k-1$. Since $T(1_{\HW})$ is a generator of $\HW_{3(k-1)}({D_{q}^*S^k};\hat\sigma) $, it corresponds to $\mu x^2\in \H_{2(k-1)}(\Omega S^k)$ for $\mu\in\{-1,1\}$. Therefore, we have
		\[
		T^{-1}\HW_*(D_q^*S^k;\hat\sigma)[k-1]\cong T^{-1}\Z[x]/(T-\mu x^2)\cong \Z[x,x^{-1}],
		\]
		where $x^{-1}=\mu T^{-1}x$. 
		This completes the proof.
	\end{proof}

	\subsubsection{$N = \mathbb{RP}^k$ for $k\geq 3$}
	\label{sec:real_proj_space}
	Let $\hat \sigma '$ and $\hat \sigma$ be relative $\mathrm{Spin}$-structures on $T^*_{[q]}\mathbb{RP}^k$ and $T^*_qS^k$, respectively, with the background classes described in \eqref{eq:background_class_cotangent_bdl}.
	Since $\RFH_*(D^*_{[q]}\mathbb{RP}^k ,S^*_{[q]}\mathbb{RP}^k;\hat{\sigma}')$ is defined using only contractible chords and the homology and cohomology groups of $\Omega^0 \mathbb{RP}^{k}$ are naturally isomorphic to those of $\Omega S^{k}$, we obtain 
	\begin{equation}\label{eq:isom_projective_sphere}
	\RFH_*(D^*_{[q]}\mathbb{RP}^k ,S^*_{[q]}\mathbb{RP}^k;\hat{\sigma}')\cong \RFH_*(D^*_qS^k,S^*_qS^k;\hat\sigma).	
	\end{equation}
	In what follows, we provide an alternative perspective of this isomorphism.
	We equip $\mathbb{RP}^k$ with the metric induced from $S^k$ via the antipodal quotient. 
	The cogeodesic flow induces the prequantization bundle 
	\[
	\pi_{S^*\mathbb{RP}^k}:S^*\mathbb{RP}^k\longrightarrow Q^{k-1}
		\] 
	with the Euler class $-2[\omega_\mathrm{FS}|_{Q^{k-1}}]$. Moreover, 
	\[
	L_{\mathbb{RP}^k}=\pi_{S^*\mathbb{RP}^k}(S^*_{[q]}\mathbb{RP}^k)=\pi_{S^*S^k}(S^*_{q} S^k)=L_{S^k}.
	\] 
	The second Stiefel--Whitney class of $T\mathbb{RP}^k$ is computed as 
	\[
	w_2^{T\mathbb{RP}^k} = \binom{k+1}{2} \in  \Z_2 \cong \H^2(\mathbb{RP}^k;\Z_2).
	\]
	By the Gysin exact sequence 
	\[
	\dots \to \H^0(Q^{k-1};\Z_2) \xrightarrow{\cup\, 2[\omega_\mathrm{FS}|_{Q^{k-1}}]_{\Z_2}}\H^2(Q^{k-1};\Z_2) \xrightarrow{\pi_{S^*\mathbb{RP}^k}^*} \H^2(S^*\mathbb{RP}^k;\Z_2)\to 0,
	\]
	we know that $\pi_{S^*\mathbb{RP}^k}^*$ is an isomorphism. 
	Thus, there exists a unique  class $b\in\H^2(Q^{k-1};\Z_2)$ satisfying $\pi^*_{S^*\mathbb{RP}^k}b=\iota^*\mathrm{pr}^*w_2^{T\mathbb{RP}^k}$. 
	By Lemma \ref{lem:pin_cross}, there is a unique relative $\mathrm{Spin}$-structure $\sigma'$ on $L_{\mathbb{RP}^k}$ which lifts to $\tilde\sigma'$, the restriction of $\hat\sigma'$ to $\R \times S^*_{[q]}\mathbb{RP}^k$. 
	The relative $\mathrm{Spin}$-structure $\sigma'$ lifts to $\tilde \sigma$, the restriction of $\hat\sigma$ to $\R \times S^*_qS^k$.
	For such $\sigma'$, we have								
	\[
	\RFH_*(S^*_{[q]}\mathbb{RP}^k;\tilde\sigma')\cong \QH_*(L_{\mathbb{RP}^k}^{\sigma'})= \QH_*(L_{S^k}^{\sigma'})\cong \RFH_*(S^*_qS^k;\tilde\sigma)
	\]
	where the isomorphisms follow from Corollary \ref{cor:zoll}. By Proposition \ref{prop:RFH_well_defined}, the first and last terms are identified with $\RFH_*(D^*_{[q]}\mathbb{RP}^k ,S^*_{[q]}\mathbb{RP}^k;\hat{\sigma}')$ and $\RFH_*(D^*_qS^k,S^*_qS^k;\hat\sigma)$, respectively. Thus, we recover the isomorphism in \eqref{eq:isom_projective_sphere}.
	
	\subsubsection{$N=\mathbb{CP}^k$ for $k\geq2$}\label{sec:CP^k}
	
	We consider the metric on $\mathbb{CP}^k$ induced by the round metric on $S^{2k+1}$ via  
	the Hopf fibration $\pi_{S^{2k+1}}:S^{2k+1}\to \mathbb{CP}^k$. We identify the unit tangent bundle $S\mathbb{CP}^k$ with the Hilbert 1-form and $S^*\mathbb{CP}^k$ with the Liouville 1-form via the musical isomorphism. The set of horizontal unit vectors of $T_qS^{2k+1}$ is written as
	\[
	\{v\in T_qS^{2k+1}\mid v\cdot iq=0 \}=\{v\in \R^{2k+2}\mid \|v\|=1,\;v\cdot q=0,\; v\cdot iq=0\},
	\] 
 where $i$ is the complex structure on $\R^{2k+2}$ and $\cdot$ is the dot product. 
	The map $\pi_{S^{2k+1}}$ gives a bijection between oriented closed geodesics on $\mathbb{CP}^k$ and oriented horizontal great circles in $S^{2k+1}$ modulo the $S^1$-action on $S^{2k+1}$. The latter corresponds to oriented horizontal real 2-planes in $\C^{k+1}$ modulo the $\C^\times$-action. Here, a 2-plane spanned by $x,y\in\R^{2k+2}$ is said to be horizontal if $x\cdot iy=0$. 
	This correspondence implies the following.
	For $(\bar{q},\bar{v})\in S\mathbb{CP}^k$, we  choose $(q,v)\in TS^{2k+1}$ such that $v$ is horizontal and $(q,v)$ projects to $(\bar{q},\bar{v})$ by $\pi_{S^{2k+1}}$. The closed geodesic on $\mathbb{CP}^k$ with initial velocity $\bar{v}$ at $\bar{q}$ lifts to the oriented great circle given by the intersection of $S^{2k+1}$ and the oriented horizontal real 2-plane in $\C^{k+1}$ spanned by $\{q,v\}$. The $S^1$-action on $S\mathbb{CP}^k$ given by the geodesic flow induces the prequantization bundle 
	\[
	\begin{split}
		\pi_{S\mathbb{CP}^k}:S\mathbb{CP}^k&\longrightarrow F:=\{ ([z],[w]) \in \mathbb{CP}^k \times \mathbb{CP}^k \mid  \langle z,w\rangle=0\},\\[1ex]
		(\bar{q},\bar{v})&\longmapsto ([q+iv],[q-iv])
	\end{split}
	\]
	where $\langle z,w\rangle$ denotes the standard Hermitian inner product. 
	Indeed, 
	\[
	\begin{split}
		\pi_{S{\mathbb{CP}^k}}(t\cdot (\bar{q},\bar{v})) &= \big([\cos(2\pi t)q - \sin(2\pi t) v+i( \sin (2\pi t)q + \cos (2\pi t) v)],\\
		&\qquad [\cos(2\pi t)q - \sin(2\pi t) v-i( \sin (2\pi t)q + \cos (2\pi t) v)]\big) \\[.5ex]
		&= \big([e^{2\pi it}(q+iv)], [e^{-2\pi it}(q-iv)]\big)\\[.5ex]
		&=\pi_{S{\mathbb{CP}^k}}(\bar{q},\bar{v}),
	\end{split}
	\]
	and 
	\[
	\langle q+iv, q-iv\rangle= \|q\|^2-\|v\|^2+ i (\langle q,v\rangle +\overline{\langle q,v\rangle})=0
	\]
	since $q$ and $v$ are orthogonal unit vectors. The base $F$ is equipped with the symplectic form  $\omega_F:=(\omega_{\mathrm{FS}} \oplus -\omega_{\mathrm{FS}})|_F$. A straightforward computation shows that $-\pi_{S{\mathbb{CP}^k}}^*\omega_F$ equals the exterior derivative of the Hilbert 1-form. The unit tangent fiber $S_q\mathbb{CP}^k$ at $q=[1:0:\cdots:0]$ projects to the Lagrangian sphere 
	\[
	L_{\mathbb{CP}^k} = \{([1:w_1:\dots:w_k],[1:-w_1:\dots:-w_k]) \in F\}.
	\]
	Its minimal Maslov number is $2k$ as the minimal Chern number of $F$ is $k$.

	To investigate relative $\mathrm{Spin}$-structures, we observe	
	\begin{equation}\label{eq:SW_class}
	w_2^{T\mathbb{CP}^k} = k+1   \in \Z_2 \cong \H^2(\mathbb{CP}^k;\Z_2).		
	\end{equation}
	The Gysin exact sequence for $S^{2k-1}\hookrightarrow S\mathbb{CP}^k \stackrel{\mathrm{pr}}{\to} \mathbb{CP}^k$ shows that
	\[ 
	\mathrm{pr}^*:\H^2(\mathbb{CP}^k;\Z_2) \longrightarrow\H^2(S\mathbb{CP}^k;\Z_2)
	\]
	for $k\geq2$ is an isomorphism. By the Gysin exact sequence for $S^1\hookrightarrow S\mathbb{CP}^k \to F$,
	\[\dots \to  \H^0(F;\Z_2) \xrightarrow{\cup [\omega_F]_{\Z_2}} \H^2(F;\Z_2) \xrightarrow{\pi_{S\mathbb{CP}^k}^*} \H^2(S\mathbb{CP}^k;\Z_2) \to 0 ,
	\]
	where we used $\H^{1}(F;\Z_2) =0$,  $\pi_{S\mathbb{CP}^k}^*$ is surjective with kernel generated by $[\omega_F]_{\Z_2} \in  \H^2(F;\Z_2)$, the class corresponding to $[\omega_F]$. By the Lefschetz hyperplane theorem, $\H^2(F;\Z_2)\cong \H^2(\mathbb{CP}^k\times \mathbb{CP}^k;\Z_2)$. By \eqref{eq:SW_class}, $\pi^*_{S\mathbb{CP}^k}(b)=\mathrm{pr}^*w_2^{T\mathbb{CP}^k}$ holds for $b\in\{0,[\omega_F]_{\Z_2}\}$ when $k$ is odd, and for $b\in \H^2(F;\Z_2) \setminus \{0,[\omega_F]_{\Z_2}\}$ when $k$ is even. For each such background class $b$, there exists a unique $\mathrm{Spin}$-structure on $L_{\mathbb{CP}^k}$ by Lemma \ref{lem:pin_cross}. Let $\sigma$ be any such relative $\mathrm{Spin}$-structure.

	\begin{prop}\label{prop:algebra_iso_for_complex_proj_space}
		
		There exists a ring isomorphism
		\[
		\QH_*(L_{\mathbb{CP}^k}^{\sigma})[2k-1] \cong \Z[x,y,y^{-1}]/(x^2),
		\]
		where $x$ has degree $1$ and $y$ has degree $2k$.
	\end{prop}
	
	\begin{proof}
		We argue as in the proof of Proposition \ref{prop:algebra_iso_for_sphere}. 
		\
		Since the minimal Maslov number of ${L_{\mathbb{CP}^k}}$ equals $2k$, Corollary \ref{cor:zoll} yields 
		$
		\QH_*(L_{\mathbb{CP}^k}^{\sigma}) \cong T^{-1}\HW_*(D_{[q]}^*\mathbb{CP}^k;\hat\sigma)
		$
		as rings. 
		It is computed in \cite[Section 3]{CJY04} that 
		\[
		\HW_*(D_{[q]}^*\mathbb{CP}^k;\hat\sigma)[2k-1]\cong \H_*(\Omega \mathbb{CP}^k)\cong\Z[x,y]/(x^2),
		\] 
		where the formal variables $x$ and $y$ have degrees $1$ and $2k$, respectively. For a degree reason, $T(1_{\HW})\in \HW_{4k-1}(D_{[q]}^*\mathbb{CP}^k;\hat\sigma)$ corresponds to $\mu y\in \H_{2k}(\Omega \mathbb{CP}^k)$ where $\mu\in\{-1, 1\}$. This completes the proof. 				
	\end{proof}
	
	\begin{prop}
		For the Lagrangian $3$-sphere $L_{\mathbb{CP}^2}\subset F\subset\mathbb{CP}^2\times\mathbb{CP}^2$ with its unique absolute $\mathrm{Spin}$-structure $\sigma_0$, we have
		\[
		\QH_*(L_{\mathbb{CP}^2}^{\sigma_0};\mathbb{K})=0,
		\] 
		where $\mathbb{K}$ is any field of characteristic different from $2$.
	\end{prop}
	\begin{proof}
		Let $\hat\sigma_0$ be the absolute $\mathrm{Spin}$-structure on $T_{[q]}^*\mathbb{CP}^2$ whose restriction is the lift $\tilde\sigma_0$ of $\sigma_0$.
		It follows from \cite{Sei10}, see also \cite[Section 6.3]{BKK24}, that $\SH_*(D^*\mathbb{CP}^2;\mathbb{K})=0$. Since $\HW_*(D_{[q]}^*\mathbb{CP}^2;\hat{\sigma}_0,\mathbb{K})$ is a module over $\SH_*(D^*\mathbb{CP}^2;\mathbb{K})$, it also vanishes. 
		Arguing as in the proof of Corollary \ref{cor:zoll}, we conclude $\QH_*(L_{\mathbb{CP}^2}^{\sigma_0};\mathbb{K}) \cong T^{-1}\HW_*(D_{[q]}^*\mathbb{CP}^2;\hat{\sigma}_0,\mathbb{K})=0$.			
	\end{proof}

	\subsection{Some consequences of quantum invertibility of $\omega$}\label{subsection:invertibilityofeulerclass}
	Let $(\Sigma,\omega)$ be a closed, connected symplectic manifold with an integral lift $[\omega]_\Z\in\H^2(\Sigma)$ as before. 
	In this section, we present some consequences of the invertibility of $[\omega]_{\Z}$ in the quantum cohomology $\QH^*(\Sigma)$ of $\Sigma$. 
	\begin{lemma}\label{lem:invertibility}
		Assume that  $[\omega]_\Z$  is invertible in $\QH^*(\Sigma)$. Then any other integral lift of $[\omega]$ is also invertible.
	\end{lemma}
	\begin{proof}
		Let $\beta$ be another integral lift of $[\omega]$.  Then, $\mu:=[\omega]_\Z-\beta$ is a torsion class.  By the divisor axiom for the Gromov-Witten invariants, $\mathrm{GW}_{A,3}(\cdot,\cdot,\mu)=\mathrm{GW}_{A,2}(\cdot,\cdot)\mu(A)$ for every nonzero $A\in\H_2(\Sigma)$. Since $\mu$ is torsion, these Gromov-Witten invariants vanish, and hence the quantum cup product with $\mu$ agrees with the ordinary cup product with $\mu$. In particular, $\mu$ is nilpotent in $\QH^*(\Sigma)$. If we denote by $[\omega]_\Z^{-1}$ the inverse of $[\omega]_\Z$, then $[\omega]_\Z^{-1}\star \mu$ is also nilpotent, where $\star$ denotes the quantum cup product. Since $\beta= [\omega]_\Z-\mu=[\omega]_\Z\star(1_{\QH}-[\omega]_\Z^{-1}\star\mu)$, we have  
		\[
		\beta\star \Big(\big(1_\QH+ ([\omega]_\Z^{-1}\star \mu)+\cdots +([\omega]_\Z^{-1}\star \mu)^m\big)\star[\omega]_\Z^{-1}\Big)= 1_{\QH}
		\]
		for sufficiently large $m$, where $1_\QH$ denotes the unit.  This proves the lemma.
	\end{proof}

	\begin{prop}\label{prop:torsion_QH_base}
		Let $L\subset (\Sigma,\omega)$ be a closed, connected monotone Lagrangian submanifold such that $N_L \geq 3$ and the map $\pi_2(\Sigma,L)\to\H_1(L)$ is surjective. If $[\omega]_\Z\in \QH^*(\Sigma)$ is invertible, then, for any $\Pin$-structure $\sigma$ on $L$, $\QH_*(L^{\sigma})$ is torsion of order $d=\frac{2c_\Sigma}{\gcd(2c_\Sigma, m_\Sigma N_L)}$.
	\end{prop}
	\begin{proof} 

		We consider the long exact sequence 
		\begin{equation*}
			\dots \longrightarrow \H^2(\Sigma, L) \stackrel{\jmath^*}{\longrightarrow} \H^2(\Sigma) \stackrel{\iota^*}
			{\longrightarrow}\H^2(L) \longrightarrow \cdots					
		\end{equation*}
		and claim that $d[\omega]_\Z\in \mathrm{im\,} \jmath^*$. To see this, we choose any $A \in \H_2(\Sigma, L)$. By the surjectivity of $\pi_2(\Sigma,L)\to \H_2(\Sigma, L)\to \H_1(L)$, there exists $B\in \mathrm{im}(\pi_2(\Sigma,L)\to\H_2(\Sigma, L))$ such that $A+B\in \mathrm{im}(\H_2(\Sigma) \to \H_2(\Sigma,L))$. Since $\omega$ is integral, we have $\omega(A+B)\in \Z$.
		Using the monotonicity of $L$ and $c_\Sigma=\tau_\Sigma m_\Sigma$, we have
		\[
		d\cdot\omega(B) = \frac{2\tau_\Sigma m_\Sigma\omega(B) }{\gcd (2c_\Sigma, m_\Sigma N_L)} 
		=  \frac{m_\Sigma \mu_L(B)}{\gcd (2c_\Sigma, m_\Sigma N_L)}\in \frac{m_\Sigma N_L}{\gcd (2c_\Sigma, m_\Sigma N_L)} \Z \subset \Z.
		\]
		It follows that $d\cdot \omega(A)\in \Z$.
		Hence there exists a class $\eta \in \H^2(\Sigma,L)$ such that $\jmath^*\eta-d[\omega]_\Z$ is torsion, i.e.,~
		\[\jmath^* \eta-d[\omega]_{\Z} \in \mathrm{im}(\mathrm{Ext}(\H_1(\Sigma),\Z)\to \H^2(\Sigma)),\]
		where the map $\mathrm{Ext}(\H_1(\Sigma),\Z)\to \H^2(\Sigma)$ is the natural homomorphism arising from the universal coefficient theorem. 
		By our assumption, the map $\H_2(\Sigma,L)\to \H_1(L)$ is  surjective, and thus the map $\H_1(\Sigma) \to \H_1(\Sigma,L)$ is injective.
		Since $\Z$ is a principal ideal domain, this injectivity implies that the induced map $\mathrm{Ext}(\H_1(\Sigma,L),\Z) \to \mathrm{Ext}(\H_1(\Sigma),\Z)$ is surjective.
		Consequently, \[
		d[\omega]_\Z \in \jmath^*\big(\eta + \mathrm{im}(\mathrm{Ext}(\H_1(\Sigma,L),\Z)\to \H^2(\Sigma,L))\big),
		\]
		and therefore $d[\omega]_\Z \in \mathrm{im\,} \jmath^*$. This proves the claim.
		
		Next, we recall from \cite[Section 5.2]{Bir06} and \cite[Proposition 15]{KS21b} the spectral sequence converging to $\QH_*(L^{\sigma})$ whose first page is given by
		\[
		E^1= \H_*(L;\Z[T,T^{-1}]\otimes \mathcal{O}) \oplus \H_*(L;\mathcal{T}\otimes \mathcal{O})[N_L],
		\]
		where $\mathcal{O}$ is the local system of orientations and $\mathcal{T}$ is a local system of $\Z[T,T^{-1}]$-module with $\deg T=2N_L$. Here $[N_L]$ denotes a grading shift by $N_L$. 
		The image of $\PD_\Sigma([\omega]_\Z)$ under the closed-open map $\mathrm{CO} : \QH_*(\Sigma) \to \QH_{*-\dim L}(L^\sigma)$ is represented by the class 
		\[\iota_!(\PD_\Sigma([\omega]_{\Z}))\otimes 1 \in \H_*(L;\mathcal{O})\otimes \Z[T,T^{-1}] \cong \H_*(L;\Z[T,T^{-1}]\otimes \mathcal{O}) \]
		in $E^1$ for degree reasons by the assumption $N_L \geq 3$. Here, $\iota_! = \PD_L \circ \iota^* \circ \PD_\Sigma^{-1} : \H_*(\Sigma) \to \H_{*-\dim L}(L;\mathcal{O})$.
		The invertibility of $[\omega]_\Z$ implies that $\mathrm{CO}(\PD_\Sigma([\omega]_\Z))\in \QH_{\dim L -2}(L^\sigma)$ is invertible. 
		Combining this with $\iota^*(d[\omega]_\Z)=0$, we conclude that $\QH_*(L^\sigma)$ is $d$-torsion.
	\end{proof}
	
	\begin{remark}\label{rem:field_invertible}
		Using field coefficients $\mathbb{F}$ makes it more likely  that the homology class corresponding to $\omega$ is invertible. 
		In this case, the proof of Proposition \ref{prop:torsion_QH_base} yields $\QH_*(L^\sigma;\mathbb{F})=0$, provided the characteristic of $\mathbb{F}$ does not divide $d$. 
	\end{remark}
	
	Recall from Lemma \ref{lem:degree} that the Lagrangian submanifold $L$ in Proposition \ref{prop:torsion_QH_base} admits a Legendrian lift to the prequantization bundle $Y$ over $(\Sigma,[\omega]_\Z)$. 
	The covering degree of this lift is $\frac{2c_\Sigma}{\gcd(2c_\Sigma, m_\Sigma N_L)}$, which coincides with the torsion degree of $\QH_*(L^{\sigma})$. 
	In fact, Proposition \ref{prop:torsion_QH_base} admits a proof that relies on the prequantization bundle $Y$, as demonstrated by the following theorem.
			
	\begin{theorem}\label{thm:torsion_QH}
		Let $L\subset (\Sigma,\omega)$ be a closed, connected monotone Lagrangian submanifold with $N_L \geq 3$. Assume that $L$ admits a $d$-fold Legendrian lift $\mathcal{L}$ to the prequantization bundle $Y$ over $(\Sigma,[\omega]_\Z)$. If $[\omega]_\Z\in \QH^*(\Sigma)$ is invertible,
		then, for any $\Pin$-structure $\sigma$ on $L$, $\QH_*(L^{\sigma})$ is $d$-torsion.
	\end{theorem}
	
	\begin{proof}
		We consider the Rabinowitz Floer homology $\RFH_*(Y)$, whose chain-level generators consist of constant orbits in $Y$ and periodic orbits of the vector fields $\pm R$, where $R$ denotes the Reeb vector field on $Y$. Since $L$ is monotone and $N_L\geq3$, $(\Sigma,\omega)$ is also monotone and the minimal Chern number is at least 2 or $\omega(\pi_2(\Sigma))=0$. Therefore, $\RFH_*(Y)$ is well-defined. 
		By \cite[Corollary 1.8]{BKK24}, the invertibility of $[\omega]_\Z$ implies $\RFH_*(Y)=0$. 
		Provided that a unital product structure on $\RFH_*(Y)$ is well-defined and $\RFH_*(\mathcal{L}^{\tilde{\sigma}})$ admits a module structure over $\RFH_*(Y)$, it follows that $\RFH_*(\mathcal{L}^{\tilde{\sigma}})=0$, and consequently $\QH_*(L^{\sigma})$ is $d$-torsion by Corollary \ref{cor:transfer}.
		This holds under the stronger condition $N_L>\dim L+3$.

		Now consider the case $N_L\geq 3$ as in the statement. 
		Let $1_{\RFH}\in \RFH_*(\mathcal{L}^{\tilde{\sigma}})$ be the homology class generated by $o_{q_0}$, see \eqref{unit_RFH}. By Proposition \ref{prop:vanishing_1_RFH}, we  have $1_\RFH=0$.  
		The unit class $1_{\QH} \in \QH_*(L^{\sigma})$ is represented by the canonical positive generator
		$o_{q,0} \in \mathfrak{o}(q,0) \subset \QC_{\dim L} (L^{\sigma})$
		for the unique local minimum $q$ of $f_L$. 
		The transfer map $\tau : \QC_* (L^{\sigma}) \to \RFC_*(\mathcal{L}^{\tilde{\sigma}})$ mentioned in Corollary \ref{cor:transfer} is explicitly given by
		$\tau(o_{q,0}) =  o_{q^1_0} +\cdots +o_{q^d_0}=o_{q_0}$. Hence, by Corollary \ref{cor:transfer}, 
		\begin{equation}\label{eq:transfer}
		d\cdot 1_{\QH} = d [o_{q,0}] = \pi\circ \tau([o_{q,0}]) = \pi ([o_{q_0}]) = \pi (1_{\RFH}) =0.
		\end{equation}	
			Since $1_{\QH}$ is the unit, this proves  that $\QH_*(L^{\sigma})$ is $d$-torsion.
	\end{proof}
	
	\begin{cor}
		\label{cor:torsion_QH_projectivespace}
		Let $L \subset \mathbb{CP}^n$ be a closed, connected monotone Lagrangian submanifold with $N_L \geq 3$ and equipped with a $\Pin$-structure $\sigma$. 
		Then $\QH_*(L^{\sigma})$ is $\frac{2(n+1)}{N_L}$-torsion.
	\end{cor}
	
	\begin{remark}\label{rem:generalize_rel_pin}
		Proposition \ref{prop:torsion_QH_base} and Theorem \ref{thm:torsion_QH} extend to relative $\Pin$-structures by working with the quantum homology of $\Sigma$ twisted by the background cohomology class, see \cite[Proposition 15.(c)]{KS21b}. 
		The quantum homology ring of $\mathbb{CP}^n$ remains unchanged under such twisting, and thus Corollary \ref{cor:torsion_QH_projectivespace} extends to relative $\Pin$-structures. This was proved in \cite[Lemma 11]{KS21b} under the additional assumption $\H^1(L;\mathbb{Z}) =0$.
				
		To relate the twisted  quantum homology of $\Sigma$ with the Rabinowitz Floer homology of $Y$ along the lines of \cite{BKK24}, it is necessary to twist the boundary operator of the latter as well. 
	\end{remark}

	\begin{prop}
		Let $(\Sigma,\omega)$ and $L^\sigma$ be as in Theorem \ref{thm:torsion_QH}. Let $p$ be any prime number. Then, $\dim \QH_{i}(L^{\sigma};\Z_p) = \dim \QH_{i+1}(L^{\sigma};\Z_p)$ for all $i\in \Z$.
	\end{prop}
	\begin{proof}
		By Theorem \ref{thm:torsion_QH}, there are positive integers $m_{q,i,1}\leq\dots\leq m_{q,i,\ell(i)}$   such that
		\[
		\QH_{i}(L^{\sigma}) \cong \bigoplus_{q|d}\,(\Z_{q^{m_{q,i,1}}} \oplus \cdots \oplus \Z_{q^{m_{q,i,\ell(i)}}}),
		\]
		where the big direct sum runs over all prime numbers $q$ dividing $d$.
		Moreover, by the universal coefficient theorem, we have for any prime number $p$ dividing $d$
		\begin{equation}\label{primary_decomposition_QH}
			\begin{split}
				\QH_{i}(L^{\sigma};\Z_p)&\cong \big(\QH_{i}(L^{\sigma}) \otimes_{\Z} \Z_p \big)\oplus \mathrm{Tor}(\QH_{i-1}(L^{\sigma}),\Z_p)\\[.5ex]
				&\cong (\Z_{p^{m_{p,i,1}}} \oplus \cdots \oplus \Z_{p^{m_{p,i,\ell(i)}}})\oplus (\Z_{p^{m_{p,i-1,1}}} \oplus \cdots \oplus \Z_{p^{m_{p,i-1,\ell(i-1)}}}).
			\end{split}
		\end{equation}
		Since $\mathrm{CO}(\PD_\Sigma([\omega]_\Z))\in \QH_*(L^\sigma)$ is invertible as mentioned in the proof of Proposition \ref{prop:torsion_QH_base}, we have $\QH_i(L^{\sigma})\cong \QH_{i+2}(L^{\sigma})$ for all $i\in\Z$. This together with \eqref{primary_decomposition_QH} proves the claim for $p$ dividing $d$. If $p$ does not divide $d$,  we obtain $\QH_i(L^\sigma;\Z_p)=0$ for all $i\in\Z$. This finishes the proof. 
	\end{proof}

	In the rest of this section, we revisit several results from \cite{Bir06}. 
	An interesting question raised in \cite[Remark in p.283]{Bir06} (also implicitly in \cite{BC01}) asks whether the following is true: for a symplectic manifold $(X,\omega_X)$, if $(\mathbb{CP}^n\times X, \omega_{\mathrm{FS}}\oplus \omega_X)$ admits a Lagrangian sphere, then $\dim X=2n+2$. If indeed $\dim X=2n+2$ and $\omega_X$ is sufficiently large, then $\mathbb{CP}^n\times X$ admits a Lagrangian sphere, see \cite[Section 4.3]{BC01}. 
	
	Now we assume that $X$ is closed (or geometrically bounded) and symplectically aspherical, i.e.,~both $c_1^{TX}$ and $\omega_X$ vanish on $\pi_2(X)$. Then, by  \cite[Theorem A]{BC01}, we know that there is no Lagrangian sphere in  $\mathbb{CP}^n\times X$ with $\dim X \leq 2n$. 
	Moreover, it is proved in \cite[Theorem D]{Bir06} that if $\mathbb{CP}^n\times X$ has a Lagrangian sphere, then $\dim X+2n+2$ is divisible by $4n+4$. See the same reference for related results. 
	As a consequence of Proposition \ref{prop:torsion_QH_base}, we obtain the following slight generalization.
	
	\begin{prop}\label{prop:lag_sphere}(Cf.~\cite[Theorem E]{Bir06})
		Let $(X,\omega_X)$ be a closed symplectically aspherical symplectic manifold, possibly of dimension $0$. Let $(M,\omega_M)$ be a closed monotone symplectic manifold such that $c_M\geq2$ and $[\omega_M]\in\QH^*(M;\R)$ is invertible. 
		If there exists a Lagrangian sphere $L$ of $(M \times X,\omega_{M} \oplus \omega_X)$ with $\dim L\geq2$, then $ \dim M + \dim X+2$ is divisible by $4c_M$. 
	\end{prop}
	
	The proof of Proposition \ref{prop:lag_sphere} relies on the following simple lemma. 	
	\begin{lemma}\label{lem:auto_nilpotent}
		Let $V$ and $W$ be two modules over the same ring. Let $f$ be an automorphism of $V$, and let $g$ be a nilpotent endomorphism of $W$. Then $f\otimes \mathrm{id}_W+\mathrm{id}_V\otimes g$ is an automorphism of $V \otimes W$.
	\end{lemma}
	\begin{proof} 
		Let $m\in\N$ be such that $g^m=0$. Then, the lemma follows from 
		\[
		(f\otimes \mathrm{id}_{W}+\mathrm{id}_{V}\otimes g) \circ \Big(\sum_{i=1}^m f^{-i}\otimes  g^{i-1}\Big)= \mathrm{id}_{V} \otimes \mathrm{id}_{W}.
	\]
	\end{proof}

	\begin{cor}\label{cor:product_invertible}
		Let $(X,\omega_X)$ and $(M,\omega_M)$ be as in Proposition \ref{prop:lag_sphere}.
		Then, $[\omega_M \oplus \omega_X]$ is invertible in $\QH^*(M\times X;\mathbb{R})$.
	\end{cor}
	\begin{proof}
		Since $(X,\omega_X)$ is symplectically aspherical, $M \times X$ is monotone and the map $\star[\omega_X] :\QH^*(X;\R) \to \QH^*(X;\R)$ is nilpotent. By \cite[Exercise 11.1.19]{MS12}, we have 
		\[
		\star[\omega_M\oplus \omega_X]= \star[\omega_M]\otimes \mathrm{id}_{\QH(X)}+\mathrm{id}_{\QH(M)}\otimes \star[\omega_X]
		\] 
		through the isomorphism $
		\QH^*(M \times X;\R) \cong \QH^*(M;\R) \otimes \QH^*(X;\R)$. It is an isomorphism by Lemma \ref{lem:auto_nilpotent}. This finishes the proof.
	\end{proof}
	\begin{proof}[Proof of Proposition \ref{prop:lag_sphere}]
		Note that the minimal Chern number of $M\times X$ equals $c_M$. Since $\pi_1(L)=0$, we have $N_L= 2c_M\geq3$. Let $\sigma$ be a unique $\mathrm{Spin}$-structure on $L$.  
		On the one hand, the chain module $\QC(L^\sigma;\R)$ has generators in degrees $\{0,\dim L\}+N_L\Z $. On the other hand, Proposition \ref{prop:torsion_QH_base} and Remark \ref{rem:field_invertible} imply  $\QH_*(L^\sigma;\R)=0$. Hence, for degree reasons, $N_L=2c_M$ divides $\dim L+1=\frac{1}{2}(\dim M+\dim X)+1$. This completes the proof.
	\end{proof}
	
	\begin{remark}(Cf.~\cite[Theorem F]{Bir06})
		Proposition \ref{prop:lag_sphere} and Corollary \ref{cor:product_invertible} also hold when $(X,\omega_X)$ is monotone with the same monotonicity constant as $(M,\omega_M)$ and $\star([\omega_X])$ is nilpotent. 
		In this case, if there is a Lagrangian sphere $L$ in $(M \times X,\omega_{M} \oplus \omega_X)$, then $ \dim M + \dim X+2$ is divisible by $4c_{M\times X}$.
	\end{remark}
	
	\begin{cor}(Cf.~\cite[Theorem B]{Bir06})
		Let $(X,\omega_X)$ be a closed symplectically aspherical integral symplectic manifold of dimension $2n+2$. If there is a closed simply connected Lagrangian submanifold $L\subset (\mathbb{CP}^n\times X,\omega_{\mathrm{FS}}\oplus\omega_X)$, then $\H_*(L;\Z_2)\cong \H_*(S^{2n+1};\Z_2)$.
		If, in addition, $L$ admits a $\mathrm{Spin}$-structure, then the isomorphism holds with $\Z$-coefficients.
	\end{cor} 
	\begin{proof}
		Since $(X,\omega_X)$ is symplectically aspherical, $(\mathbb{CP}^n\times X,\omega_\mathrm{FS}\oplus\omega_X)$ is monotone with monotonicity constant $2n+2$. 
		Since $\pi_1(L)$ is trivial, we have $N_L=2n+2$. 
		By Proposition \ref{prop:torsion_QH_base}, Remark \ref{rem:field_invertible}, and Corollary \ref{cor:product_invertible} applied to $\Z_2$-coefficients, we have $\QH_*(L;\Z_2)=0$. 
		Therefore, for degree reasons, $L$ is a $\Z_2$-homology sphere. The same argument proves the assertion for $\Z$-coefficients.	
		\end{proof}

		\subsection{Polarized symplectic manifold and Lagrangian trace}
	\label{sec:polarized}
	
	In this section, we prove Theorem \ref{thm:lag_trace}. We begin by recalling the statement of the theorem. Let $(X,\Omega,\Sigma)$ be a polarized symplectic manifold. The complement $X\setminus\Sigma$ is compactified into a Weinstein domain $W$, whose boundary $Y=\partial W$ is a prequantization bundle over $\Sigma$ with projection  $\pi:Y\to \Sigma$. Let $\varphi$ be the associated Morse function on $W$. Let $\Delta(p)\subset W$ denote the Lagrangian cocore disk associated to $p\in\Crit\, \varphi$ with $\ind_{\varphi}(p)=\frac{1}{2}\dim W$. 
	Let $L$ be a closed monotone Lagrangian submanifold of $\Sigma$ with $N_L\geq3$. We assume that $L$ admits a $d$-fold Legendrian lift $\mathcal{L}\subset Y$ and $L\cap \pi(\partial\Delta(p)) =\emptyset$ for all $p\in\Crit\, \varphi$ with $\ind_{\varphi}(p)=\frac{1}{2}\dim W$. 
		Theorem \ref{thm:lag_trace} asserts that, for any $\Pin$-structure $\sigma$ on $L$, $\QH_*(L^\sigma)$ is torsion of order $d$.
	\medskip
	
	\begin{proof}[Proof of Theorem \ref{thm:lag_trace}]
	By assumption, there is an open neighborhood $\mathcal{U}$ of $\pi^{-1}(L)$ such that $\mathcal{U}\cap\partial\Delta(p)=\emptyset$ for all $p\in\mathrm{Crit}\,\varphi$ satisfying $\ind_\varphi(p)=\frac{1}{2}\dim W$. Then, there is a subcritical Weinstein subdomain $W_{(+1)}\subset \mathrm{int}(W)$ such that $\phi_Z^{-\epsilon}(\mathcal{U})\subset \partial W_{(+1)}$ for small $\epsilon>0$. Here $\phi_Z$ denotes the flow of the Liouville vector field $Z$. The boundary $\partial W_{(+1)}$ is the result of performing contact $(+1)$-surgery on $\partial W$ along $\partial\Delta(p)$. The Reeb flow starting from $\mathcal{L}_{-\epsilon}:=\phi_Z^{-\epsilon}(\mathcal{L})$ is periodic and its image is exactly $\phi_Z^{-\epsilon}(\pi^{-1}(L))$. The Weinstein cobordism $V:=W\setminus \mathrm{int}(W_{(+1)})$ contains the trivial cobordisms $\widehat{\mathcal{L}}:=\displaystyle\bigcup_{0\leq \tau\leq \epsilon}\phi_Z^{-\tau}(\mathcal{L})$ and $\widehat{\mathcal U}:=\displaystyle\bigcup_{0\leq \tau\leq \epsilon}\phi_Z^{-\tau}(\mathcal{U})$.

	We denote by $\RFH_*(\mathcal{L};\partial W)$ and $\RFH_*(\mathcal{L}_{-\epsilon};\partial W_{(+1)})$ the Rabinowitz Floer homology of $\mathcal{L}$ in $\partial W$ and of $\mathcal{L}_{-\epsilon}$ in $\partial W_{(+1)}$, respectively. We employ the lift of a $\Pin$-structure $\sigma$ for orientations and suppress this choice from the notation. Both homologies are well-defined for sufficiently stretched almost complex structures as in Section \ref{sec:Floer_complex_general}   since the indices of the Reeb chords are sufficiently large due to the assumption $N_L\geq3$, see Remark \ref{rem:inde_nondeg}.  Arguing as in \cite[Proposition 9.19]{CO18}, we obtain 
	\begin{equation}\label{eq:rfh_correspondence}
	\RFH_*(\mathcal{L};\partial W)\stackrel{\cong}{\longleftarrow}\HW_*(\widehat{\mathcal{L}};V)\stackrel{\cong}{\longrightarrow} \RFH_*(\mathcal{L}_{-\epsilon};\partial W_{(+1)}).
	\end{equation}
	Here, $\HW_*(\widehat{\mathcal{L}};V)$ is the wrapped Floer homology for the Lagrangian cobordism $\widehat{\mathcal{L}}\subset V$, whose chain-level generators are Morse critical points in $\widehat{\mathcal{L}}$, the Reeb chords of $\mathcal{L}$, and the negative Reeb chords of $\mathcal{L}_{-\epsilon}$; see \cite[Section 8.3]{CO18}.  Indeed, there is an exact sequence 
	\[
	\cdots\to\HW_*(\widehat{\mathcal{L}},\mathcal{L}_{-\epsilon} ; V) \to \HW_*(\widehat{\mathcal{L}};V)\to \RFH_*(\mathcal{L}_{-\epsilon};\partial W_{(+1)})\to \HW_{*-1}(\widehat{\mathcal{L}},\mathcal{L}_{-\epsilon}; V)\to\cdots
	\]
	where $\HW_*(\widehat{\mathcal{L}},\mathcal{L}_{-\epsilon}; V)$ is the homology whose chain-level generators are Morse critical points in $\widehat{\mathcal{L}}$, the Reeb chords of $\mathcal{L}$, and the Reeb chords of $\mathcal{L}_{-\epsilon}$. By employing a sequence of Hamiltonians which have linear slope in $\widehat{\mathcal{U}}$, where the slopes are chosen to avoid the periods of the Reeb chords of $\mathcal{L}$, one can show that $\HW_*(\widehat{\mathcal{L}},\mathcal{L}_{-\epsilon}; V)=0$. This proves $\HW_*(\widehat{\mathcal{L}};V)\cong \RFH_*(\mathcal{L}_{-\epsilon};\partial W_{(+1)})$. By a similar argument, $\RFH_*(\mathcal{L};\partial W)\cong\HW_*(\widehat{\mathcal{L}};V)$ also holds, and thus \eqref{eq:rfh_correspondence} follows. See \cite[Proposition 9.19]{CO18} for the corresponding proof in the periodic case. 
	
	There are maps $\H^*(\mathcal{L})\to \RFH_*(\mathcal{L};\partial W)$ and $\H^*(\mathcal{L})\to \RFH_*(\mathcal{L}_{-\epsilon};\partial W_{(+1)})$ induced by continuation homomorphisms, and we write $1_{\RFH(\mathcal{L};\partial W)}$ and $1_{\RFH(\mathcal{L}_{-\epsilon}; \partial W_{(+1)})}$ for the image of the unit of $\H^*(\mathcal{L})$ under these maps, respectively.  Since $W_{(+1)}$ is subcritical, $1_{\RFH(\mathcal{L}_{-\epsilon}; \partial W_{(+1)})}$ is zero by Proposition \ref{prop:vanishing_1_RFH}. Note that in Appendix \ref{sec:appendix}, we consider contact manifolds that do not necessarily admit exact fillings, provided they satisfy a certain index condition for periodic orbits, see the first condition in \eqref{eq:ind_nondeg}. The discussion therein extends straightforwardly to the current setting, where $Y$ is equipped with a Weinstein filling $W$. Under the isomorphism $\RFH_*(\mathcal{L};\partial W)\cong \RFH_*(\mathcal{L}_{-\epsilon};\partial W_{(+1)})$, $1_{\RFH(\mathcal{L}_{-\epsilon}; \partial W_{(+1)})}$ corresponds to $1_{\RFH(\mathcal{L};\partial W)}$, and thus $1_{\RFH(\mathcal{L};\partial W)}$ is also zero. Arguing as in \eqref{eq:transfer}, this yields that the unit $1_\QH$ of $\QH_*(L^\sigma)$ is $d$-torsion. This completes the proof. 
\end{proof}

	\subsection{Topologically simple exact Lagrangian fillings}\label{sec:filling_proof}
	We again consider a closed, connected symplectic manifold $(\Sigma,\omega)$  with an integral lift $[\omega]_{\mathbb{Z}} \in \H^2(\Sigma;\Z)$. In this section, we assume that 
	\begin{itemize}
		\item the minimal Chern number $c_{\Sigma}$ of $\Sigma$ satisfies $c_\Sigma\geq 2$;
		\item the prequantization bundle $(Y,\alpha)$ over $(\Sigma,[\omega]_\Z)$ admits a topologically simple exact filling $(W,\lambda)$;
		\item a closed monotone Lagrangian $L\subset \Sigma$ admits a $d$-fold covering Legendrian lift $\mathcal{L}\subset Y$.
	\end{itemize}  
		We denote by $[\omega]_{\Z_2}$ the image of $[\omega]_{\mathbb{Z}}$ under the map $\H^2(\Sigma;\Z)\to \H^2(\Sigma;\Z_2)$ induced by the nontrivial ring homomorphism $\Z\to \Z_2$.
	\begin{theorem}
		\label{thm:nonexistence_filling_Z2}
		 Assume that $N_L=2$, $\QH_*(L;\Z_2)\neq 0$, and $[\omega]_{\mathbb{Z}_2}$ is invertible in $\QH_*(\Sigma;\mathbb{Z}_2)$. If $d$ is odd, then			
		$\mathcal{L}$ does not admit a topologically simple exact Lagrangian filling in $(W,\lambda)$.
	\end{theorem}
	
	\begin{proof}
		Let $\RFH_*(W,Y;\Z_2)$ be the Rabinowitz Floer homology generated by generalized periodic Reeb orbits on $Y$ contractible in $W$; this was denoted by $\SH_*(\partial W)$ in \cite{BKK24}. 
		By \cite[Theorem 1.1 and Remark 1.2]{BKK24}, the invertibility of $[\omega]_{\Z_2}$ with the assumption $c_\Sigma\geq2$ implies $\RFH_*(W,Y;\Z_2)=0$. 
		
		Assume that $\mathcal{L}$ admits a topologically simple exact Lagrangian filling $\mathcal{L}_W$ in $(W,\lambda)$. Then, $\RFH_*(\mathcal{L}_W,\mathcal{L};\Z_2)$ defined in Section \ref{sec:Floer_complex_general} is a module over $\RFH_*(W,Y;\Z_2)$, see \cite{Rit13,CO18}. Therefore, $\RFH_*(\mathcal{L}_W,\mathcal{L};\Z_2)=0$  and hence, by Corollary \ref{cor:transfer} and Proposition \ref{prop:RFH_well_defined}, $\QH_*(L;\Z_2)$ is $d$-torsion. 
		Since $d$ is assumed to be odd, this leads to a contradiction.
	\end{proof}
	
	\begin{remark}
		The only place where $c_\Sigma\geq 2$ is used is in deducing $\RFH_*(W,Y;\Z_2)=0$ from  the invertibility of $[\omega]_{\Z_2}$. We refer to \cite[Remark 1.2]{BKK24} for the necessary and sufficient condition for $\RFH_*(W,Y;\Z_2) =0$ when $c_\Sigma=1$. 
	\end{remark}

	Recall from Lemma \ref{lem:degree} that if the map $\pi_2(\Sigma,L) \to \H_1(L)$ is surjective, then $L$ admits a Legendrian lift with the covering degree $d= \frac{2c_\Sigma}{\gcd(2c_\Sigma,m_\Sigma N_L)}$. Thus if $N_L=2$ and $c_\Sigma$ is odd, then $d$ is odd. 
	An example of this is an even-dimensional complex projective space. 
	
	\begin{cor}
		\label{cor:nonexistence_filling_CP^2n}
		Let $L$ be a closed monotone Lagrangian submanifold of $\mathbb{CP}^{2n}$ for some $n\geq 1$ such that $N_L=2$ and $\QH_*(L;\Z_2)\neq 0$. Then the Legendrian lift $\mathcal{L}$ of $L$ in  $S^{4n+1}=\partial{B}^{4n+2}$ does not admit a topologically simple exact Lagrangian filling in $B^{4n+2}$.
	\end{cor}

	By working with orientations, we can extend Theorem \ref{thm:nonexistence_filling_Z2} to the case where $d$ is even, and thereby obtain a version of Corollary \ref{cor:nonexistence_filling_CP^2n} for odd-dimensional complex projective spaces.

	\begin{prop}
		\label{prop:nonexistence_filling_char_zero}
		Let $\sigma$ be a relative $\Pin$-structure on $L$, and let $\tilde\sigma$ be the lift of $\sigma$ on $\R \times\mathcal{L}$. Assume that $N_L=2$ and that one of the following conditions holds.
		\begin{itemize}
			\item[(i)] The class $[\omega]_\Z\in\QH^*(\Sigma)$ is invertible and $\QH_*(L^{\sigma})$ is not $d$-torsion.
			\item[(ii)] The class $[\omega]_\mathbb{Q}\in\QH^*(\Sigma;\mathbb{Q})$, obtained from $[\omega]_\Z$ via the inclusion $\Z\hookrightarrow\mathbb{Q}$, is invertible, and $\QH_*(L^{\sigma};\mathbb{Q}) \neq 0$.
		\end{itemize}
		Then, $\mathcal{L}$ does not admit a topologically simple exact Lagrangian filling in $(W,\lambda)$ equipped with a relative $\Pin$-structure that restricts to $\tilde{\sigma}$ on $\R\times \mathcal{L}$, in the sense of  (B3) in Section \ref{sec:RFH_filling}.
	\end{prop}
	
	\begin{proof}
		The proof follows along similar lines to that of Theorem \ref{thm:nonexistence_filling_Z2}. We outline the proof for case (ii). 
		Assume that $\mathcal{L}$ admits a topologically simple exact Lagrangian filling $\mathcal{L}_W$ in $(W,\lambda)$ equipped with a relative $\Pin$-structure $\hat{\sigma}$ satisfying  (B3). 
		Then $\RFH_*(\mathcal{L}_W^{\hat{\sigma}},\mathcal{L};\mathbb{Q})=0$, and consequently $\QH_*(L^{\sigma};\mathbb{Q}) =0$. This contradiction completes the proof.
	\end{proof}

	\begin{cor}
		\label{cor:nonexistence_filling_CP^2n+1}
		Let $L^\sigma$ be a closed monotone Lagrangian submanifold of $\mathbb{CP}^{2n+1}$ equipped with a relative $\Pin$-structure. Assume that $N_L=2$ and that either $\QH_*(L^{\sigma})$ is not $(2n+2)$-torsion or $\QH_*(L^{\sigma};\mathbb{Q})\neq 0$. Then, the Legendrian lift $\mathcal{L}$ of $L$ in  $S^{4n+3}=\partial{B}^{4n+4}$ does not admit a topologically simple exact Lagrangian filling equipped with a relative $\Pin$-structure that restricts to $\tilde\sigma$ on $\R\times \mathcal{L}$.
	\end{cor}

	\begin{remark}\label{remark equator in two sphere_short}
		As proved in \cite{Cho04}, the Floer cohomology of the Clifford torus $T_\mathrm{Cl}\subset \mathbb{CP}^n$ equipped with the standard $\mathrm{Spin}$-structure $\sigma_{\mathrm{std}}$ is isomorphic to its singular cohomology, i.e.,~in our terminology,  
		\begin{equation}
			\label{eq:clifford torus}
			\QH_*(T_\mathrm{Cl}^{\sigma_{\mathrm{std}}};\mathbb{Q}) \cong \H^*(T_\mathrm{Cl};\mathbb{Q})\otimes \mathbb{Q}[T,T^{-1}].
		\end{equation}
		This isomorphism also holds with $\mathbb{Q}$ replaced by $\Z_2$. It is also shown in \cite{Cho04} that, for even $n$, $T_\mathrm{Cl}$ has vanishing quantum homology for any nonstandard $\mathrm{Spin}$-structure, whereas, for odd $n$, it admits exactly one nonstandard $\mathrm{Spin}$-structure $\sigma_\mathrm{nstd}$ for which \eqref{eq:clifford torus} holds.  
	
		The Legendrian knot $\mathcal{L}\subset S^3$ that is a double cover of an equator $S^1_\mathrm{eq}$ of $\mathbb{CP}^1$ admits an exact Lagrangian disk filling $\mathcal{L}_{B^4}\subset B^4$. 
		According to Corollary \ref{cor:nonexistence_filling_CP^2n+1} and \eqref{eq:clifford torus}, the $\mathrm{Spin}$-structure $\tilde\sigma$ on $\R \times \mathcal{L}$ does not extend to $\mathcal{L}_{B^4}$, where $\tilde\sigma$ is the lift of either $\mathrm{Spin}$-structure on $S^1_\mathrm{eq}$. Indeed, $\tilde\sigma$ is standard for any choice of $\sigma$ due to the two-fold covering $\mathcal{L}\to S^1_\mathrm{eq}$. In contrast, the unique $\mathrm{Spin}$-structure $\hat{\sigma}$ on $\mathcal{L}_{B^4}$ restricts to a nonstandard one on $\R\times \mathcal{L}$.

		The situation is different for the Legendrian knot $S^*_qS^2\subset S^*S^2$ that is a bijective lift of the equator $S^1_\mathrm{eq}\subset\mathbb{CP}^1$. The unit disk cotangent fiber $D^*_qS^2\subset D^*S^2$ gives an exact Lagrangian filling of $S^*_qS^2$. Let $\hat{\sigma}$ be the unique $\mathrm{Spin}$-structure on $D^*_qS^2$. The restriction of $\hat{\sigma}$ to $\R \times S_q^*S^2$ is the lift $\tilde{\sigma}_{\mathrm{nstd}}$ of the nonstandard $\mathrm{Spin}$-structure ${\sigma}_{\mathrm{nstd}}$ on $S^1_\mathrm{eq}$. Theorem \ref{thm:isom_transfer} and Proposition \ref{prop:RFH_well_defined} yield the following isomorphisms
		\[
		\QH_*((S^1_\mathrm{eq})^
		{\sigma_{\mathrm{nstd}}};\mathbb{Q})\cong \RFH_*(S_q^*S^2;{\tilde{\sigma}_{\mathrm{nstd}}},\mathbb{Q}) \cong \RFH_*(D_q^*S^2,S^*_{q} S^2;{\hat{\sigma}},\mathbb{Q}).
		\]
		Moreover, by \cite[Theorem 1.5]{CHO25}, we have
		\[
		\RFH_j(D_q^*S^2,S^*_{q} S^2;{\hat{\sigma}},\mathbb{Q})\cong \H_{j-1}(\Omega S^2;\mathbb{Q}) \oplus \H^{-j}(\Omega  S^2;\mathbb{Q})\cong \mathbb{Q}\qquad \forall j\in\Z.
		\]
		Note that the second Stiefel--Whitney class $w_2^{TS^2}$ of $TS^2$  vanishes. 
		This computation agrees with the known result for $\QH_*((S^1_\mathrm{eq})^
		{\sigma_{\mathrm{nstd}}};\mathbb{Q})$ mentioned after \eqref{eq:clifford torus}.
	\end{remark}

	\appendix		
	\section{Appendix:~Proof of $1_\RFH=0$}\label{sec:appendix}

	In this section, let $(Y,\alpha)$ be a $(2n+1)$-dimensional closed contact manifold, which is not necessarily a prequantization space. Let $\mathcal{L}$ be a closed $n$-dimensional Legendrian submanifold of $Y$. Let $\tilde\sigma$ be a $\Pin$-structure on $\R\times \mathcal{L}$. For simplicity, we assume that all contractible periodic orbits and all contractible chords of the Reeb vector field $R$ on $(Y,\mathcal{L})$ are nondegenerate; however, our arguments below can be adapted to the Morse-Bott setting. Here contractibility refers to being trivial in $\pi_1(Y)$ and $\pi_1(Y,\mathcal{L})$, respectively. We assume that every contractible periodic Reeb orbit $\gamma$ and every contractible Reeb chord $c$ satisfy
		\begin{equation}\label{eq:ind_nondeg}
			\mu_\CZ(\gamma)\geq 4-n,\qquad \mu_\RS(c)\geq 3-\frac{n}{2}.
		\end{equation}
		
		\begin{remark}\label{rem:inde_nondeg}
		These index conditions correspond to $c_\Sigma\geq 2$ and $N_L\geq 3$ when $Y$ is a prequantization bundle over $\Sigma$ with a connection 1-form $\alpha$ and $\mathcal{L}$ is a lift of a Lagrangian $L$. In this case, Reeb chords appear in $\mathcal{L}$-families. We consider one such family containing a contractible Reeb chord $c$. If we perturb $\alpha$ using a small Morse function $f_\mathcal{L}$ so that this $\mathcal{L}$-family of Reeb chords bifurcates into $\Crit f_\mathcal{L}$ many nondegenerate Reeb chords with indices $\mu_\RS$ in the range $[\mu_\RS(c)-\frac{n}{2},\mu_\RS(c)+\frac{n}{2}]$. Since $\mu_\RS(c)\geq N_L$, the latter condition \eqref{eq:ind_nondeg} agrees with our standing assumption $N_L\geq3$.
		
			Our convention on $\mu_\CZ(\gamma)$ and $\mu_\RS(c)$ here is that the moduli space of unparametrized finite energy planes resp.~half-planes in $(\R\times Y,\R\times \mathcal{L})$ asymptotic to $\gamma$ resp.~$c$ has virtual dimension $\mu_\CZ(\gamma)+n-2$ resp.~$\mu_\RS(c)+\frac{n-2}{2}$.  
			Let $\mathcal{M}_k(x_+,x_-,c_1,\dots,c_k,\gamma_1,\dots,\gamma_l;H,J)$ be the moduli space of Floer strips $\tilde{v}$ satisfying $\ev_\pm(\tilde v)=x_\pm$, with $k$ negative boundary punctures asymptotic to $c_1,\dots,c_k$ and $l$ negative interior punctures asymptotic to $\gamma_1,\dots,\gamma_l$. 
			Then, the virtual dimension of this moduli space equals 
			\[
			\mu_\RS(x_+)-\mu_\RS(x_-)-\sum_{j=1}^k\Big(\mu_\RS(c_j)+\frac{n-2}{2}\Big)-\sum_{j=1}^l\Big(\mu_\CZ(\gamma_j)+{n-2}\Big).
			\] 
			Since there is an $\R$-action translating the domain, if $\mu_\RS(x_+)-\mu_\RS(x_-)\leq 2$, then this moduli space is empty for a generic choice of $J$ by the assumption \eqref{eq:ind_nondeg}, cf.~\cite[Corollary 3.7]{Ueb19}.
		\end{remark}
		
		Let $J$ be an $S^1$-family of almost complex structures on $\R\times Y$ which is cylindrical outside a neighborhood of $\{0\}\times Y$. We choose a generic $J$ such that the moduli spaces in this section are cut out transversely. Let $H_a \in \mathcal{H}$, where the subscript $a$ indicates the slope in   \eqref{eq:Hamiltonian}. Nonconstant chords of $X_{H_a}$ with endpoints on $\R\times\mathcal{L}$ are nondegenerate, but we still have an $\mathcal{L}$-family of constant chords. Thus we choose an auxiliary Morse function $f_\mathcal{L}$ on $\mathcal{L}$.
		
		For real numbers $b_1<b_2$, we define $\FH_*^{(b_1,b_2)}(\mathcal{L}^{\tilde{\sigma}};H_a)$ in the same manner as in Sections \ref{sec:floer_chain} and \ref{sec:lag_rfh}. Its chain module is the direct sum of the orientation lines of the critical points of $f_\mathcal{L}$ and the nonconstant chords of $X_{H_a}$ appearing in $(-\eta,\eta)\times Y$ whose $\mathcal{A}_{H_a}$-values lie in $(b_1,b_2)$. We define the grading $\mu$ by $\mu(x)=\mu_\RS(x)+\frac{n-1}{2}$ for nonconstant  chords $x=(r_x,c_x)$ and $\mu(c)=n-\ind_{f_{\mathcal{L}}}(c)$ for $c\in \Crit f_\mathcal{L}$. This convention matches \eqref{eq:index}. Recall from Section \ref{sec:Ham_chords} that  $\mu_\RS(x)=\mu_\RS(c_x)+\frac{1}{2}$ as $H_a$ is convex in the $\R$-direction along $x$. Note that $X_{H_a}$ has chords other than those contained in $(-\eta,\eta)\times Y$. Nevertheless, by \cite[Lemmas 2.2 and 2.3]{CO18}, this chain module generated by chords in $(-\eta,\eta)\times Y$ carries a well-defined boundary operator. The index conditions in \eqref{eq:ind_nondeg} exclude any potential bubbling-off of finite energy (half-)planes as explained in Remark \ref{rem:inde_nondeg}. 
		As before $\RFH_*(\mathcal{L}^{\tilde{\sigma}})$ is defined as $\varinjlim\limits_{b_2\uparrow+\infty}\varprojlim\limits_{b_1\downarrow-\infty}  \varinjlim\limits_{H_a\in\mathcal{H}}\FH_*^{(b_1,b_2)}(\mathcal{L}^{\tilde{\sigma}};H_a)$.

		\medskip
		
		We briefly recall the construction of the Rabinowitz Floer homology $\RFH_*(Y)$ based on periodic orbits, and refer to \cite{BKK24} for details, where $\RFH_*(Y)$ is denoted by $\SH_*(Y)$.
		We only consider 1-periodic orbits of $X_{H_a}$ appearing near $\{0\}\times Y$. These correspond to constant orbits on $Y$ and periodic orbits of $R$ or $-R$ with period less than $a$. We choose a Morse function $f_{Y}$ on the space of constant orbits, which is diffeomorphic to $Y$. Each nonconstant orbits appear in $S^1$-families. We then perturb $H_a$ in a time-dependent manner near each such family so that it gives rise to exactly two 1-periodic orbits. For brevity, the perturbed Hamiltonian is again denoted by $H_a$. 
		The Floer chain module $\FC_*^{(b_1,b_2)}(H_a)$ is the direct sum of the orientation lines associated to nonconstant orbits and critical points of $f_{Y}$ whose action values lie in $(b_1,b_2)$. Its homology $\FH_*^{(b_1,b_2)}(H_a)$ is well-defined due to \cite[Lemmas 2.2 and 2.3]{CO18} and the first index condition in \eqref{eq:ind_nondeg}. We define the grading $\mu_Y$ by $\mu_Y(q)=\mu_\CZ(q)$ if $q$ is nonconstant and $\mu_Y(q)=n+1-\ind_{f_Y}(q)$ if $q$ is constant. Then $\RFH_*(Y)$ is defined as  $\varinjlim\limits_{b_2\uparrow+\infty}\varprojlim\limits_{b_1\downarrow-\infty}  \varinjlim\limits_{H_a\in\mathcal{H}}\FH_*^{(b_1,b_2)}(H_a)$. 
		
		\medskip
		
		Now we describe the moduli space used to define $\RFH_*(Y)$-module structure on $\RFH_*(\mathcal{L}^{\tilde{\sigma}})$. 
		Note that this module structure is not well-defined in full generality due to the lack of a filling of $\mathcal{L}^{\tilde\sigma}$. We consider the surface
		\[
		S^{M}:=\big((\R \times [-1,-\tfrac{1}{2}]) \sqcup (\R \times [-\tfrac{1}{2},\tfrac{1}{2}]) \sqcup  (\R \times [\tfrac{1}{2},1]) \big)/ \sim^{M}.
		\] 
		The equivalence relation $\sim^{M}$ is given by 
		\begin{itemize}
			\item $(s,-\frac{1}{2}^{+}) \sim (s,-\frac{1}{2}^{-})$ and $(s,\frac{1}{2}^{+}) \sim (s,\frac{1}{2}^{-})$ for $s \leq 0$,
			\item $(s,-\frac{1}{2}^{+}) \sim (s,\frac{1}{2}^{-})$ and $(s,\frac{1}{2}^{+})\sim (s,-\frac{1}{2}^{-})$ for $s \geq 0$,
		\end{itemize}
		where $-\frac{1}{2}^{-}\in  [-1,-\frac{1}{2}]$, $-\frac{1}{2}^{+}\in [-\frac{1}{2},\frac{1}{2}]$, $\frac{1}{2}^{-}\in [-\frac{1}{2},\frac{1}{2}]$, and $\frac{1}{2}^{+}\in [\frac{1}{2},1]$.
		As in the case of $S^{P}$  in Section \ref{sec:product}, $S^{M}$ is endowed with a natural complex structure, and the global coordinate $z=s+it$ is holomorphic everywhere except at the branching point $(0,-\frac{1}{2}^\pm)=(0,\frac{1}{2}^\pm)$. 

		Let $H_{a_1},H_{a_2},H_{a_1+a_2}\in\mathcal H$ satisfy 
		$H_{a_1}+H_{a_2}\leq H_{a_1+a_2}$.
		Let $q$ be a 1-periodic orbit of $X_{H_{a_1}}$, which is either a critical point of $f_{Y}$ or a  nonconstant orbit. Let $c_+$ and $c_-$ be chords of $X_{H_{a_2}}$ and $X_{H_{a_1+a_2}}$, respectively. If any of $c_\pm$ is constant, then we assume that it is a critical point of $f_{\mathcal{L}}$. 
		Let $(H_z,J_z)_{z\in S^M}$ be a domain-dependent Floer datum which agrees with $(H_{a_1},J)$ on the positive cylindrical end, and with $(H_{a_2},J)$ and $(H_{a_1+a_2},J)$ on the positive and negative strip-like ends, respectively.	We require the condition $\partial_s H_z\leq 0$ to ensure that the maximum principle holds.
		We choose a generic $J_z$ so that the moduli spaces in \eqref{eq:moduli_for_module_str} are cut out transversely.
		We consider the moduli space
		\begin{equation}\label{eq:moduli_for_module_str}
		\mathcal{M}(q,c_+,c_-;H_{a_1},H_{a_2},H_{a_1+a_2})	
		\end{equation}
		consisting of smooth maps $\tilde{v}:(S^{M},\partial S^{M}) \to (\R\times Y,\R\times \mathcal{L})$ satisfying
		$\partial_s\tilde v+J_z(\tilde v)(\partial_t\tilde v-X_{H_z}(\tilde v))=0$ on the interior 
		and the following asymptotic conditions:
		\begin{itemize}
			\item $\displaystyle \lim_{s\to \infty} \tilde{v} (s, t-\tfrac{1}{2}) $ is $q$ or in the stable manifold of $-\nabla f_Y$ at $q$ if $q$ is constant,
			\item $\displaystyle\lim_{s\to \infty} \tilde{v} (s, t-1+\lfloor 2t \rfloor)$ is $c_+$ or in the stable manifold of $-\nabla f_\mathcal{L}$ at  $c_+$ if $c_+$ is constant,
			\item $\displaystyle\lim_{s\to -\infty} \tilde{v} (s, 2t-1)$ is $c_-$ or in the unstable manifold of $-\nabla f_\mathcal{L}$ at $c_-$ if $c_-$ is constant.
		\end{itemize}
		Here $\lfloor 2t \rfloor$ denotes the greatest integer less than or equal to $2t$.  
		This moduli space has dimension $\mu_Y(q)+\mu(c_+)-\mu(c_-)-(n+1)$.
		We are particularly interested in the case where $\mu_Y(q)=n+2$, $\mu(c_+)=\mu(c_-)=n$. Then, this moduli space is 1-dimensional, and its possible Floer-type breaking configurations are as in Figure \ref{fig:breaking_floer}. 
		\begin{figure}[h]
			\centering
			\includegraphics[height = 4.7cm]{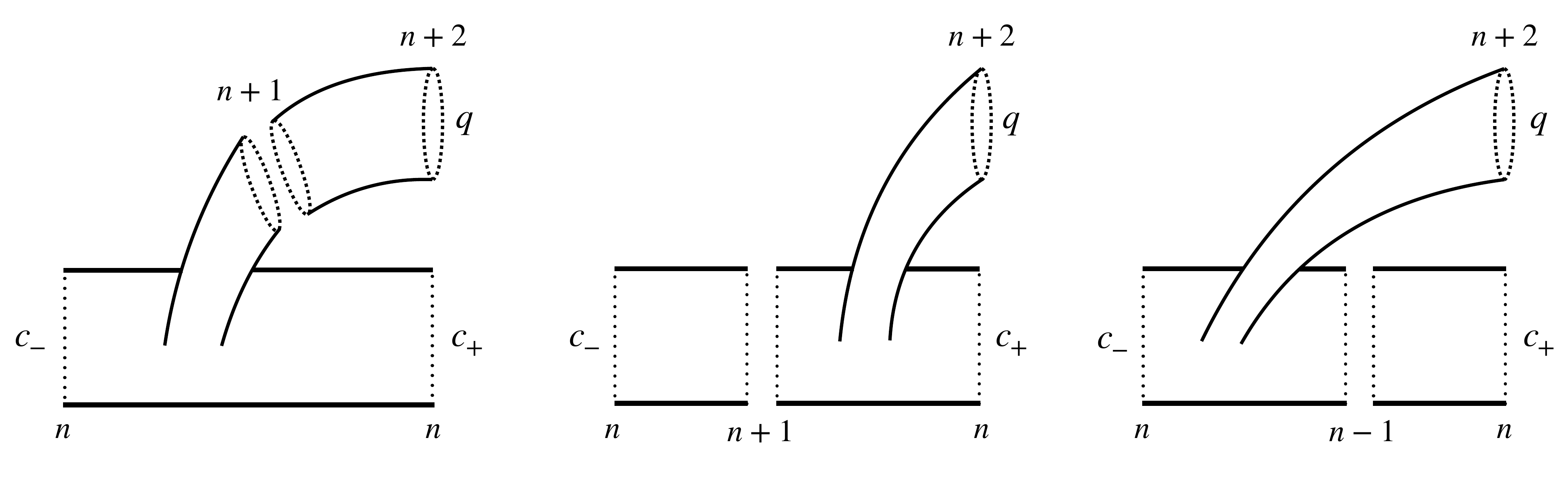}
			\caption{Possible configurations of broken curves}
			\label{fig:breaking_floer}
		\end{figure} 
		
		The last type of breaking contributes trivially when $c_+$ is a cycle.
		Even though $\mathcal{L}$ might not admit an exact Lagrangian filling and the condition in \eqref{eq:ind_nondeg} is not sufficiently strong, we can establish compactness of the above moduli spaces under suitable restrictions on the indices.

		\begin{prop}
			\label{prop:compactness_module}
			We assume that
			\[
			\big(\mu_Y(q),\mu(c_+),\mu(c_-)\big)\in \big\{(n+2,n,n),(n+1,n,n),(n+2,n,n+1)\big\}.
			\]
			Then, there exist real numbers $r_-<r_+$ such that the image of every element of the moduli space $\mathcal{M}(q,c_+,c_-;H_{a_1},H_{a_2},H_{a_1+a_2})$ is contained in $(r_-,r_+)\times Y$. 
		\end{prop}
		\begin{proof}
			We argue as in the proof of Proposition \ref{prop:no_escape_triangle}. Let $\{\tilde{v}_\nu\}_{\nu\in\N}$ be a sequence in the moduli space $\mathcal{M}(q,c_+,c_-;H_{a_1},H_{a_2},H_{a_1+a_2})$. In view of \eqref{eq:ind_nondeg} and Remark \ref{rem:inde_nondeg}, no $J$-holomorphic (half-)plane bubbles off in a limit of $\{\tilde{v}_\nu\}_{\nu\in\N}$. The only noncompact scenario is illustrated in Figure \ref{fig:breaking}. While the SFT-limit of $\{\tilde{v}_\nu\}_{\nu\in\N}$ could consist of more than two levels, our argument remains valid as we derive a contradiction by analyzing the topmost component. 
			
			\begin{figure}[h]
			\centering
			\includegraphics[height = 5.5cm]{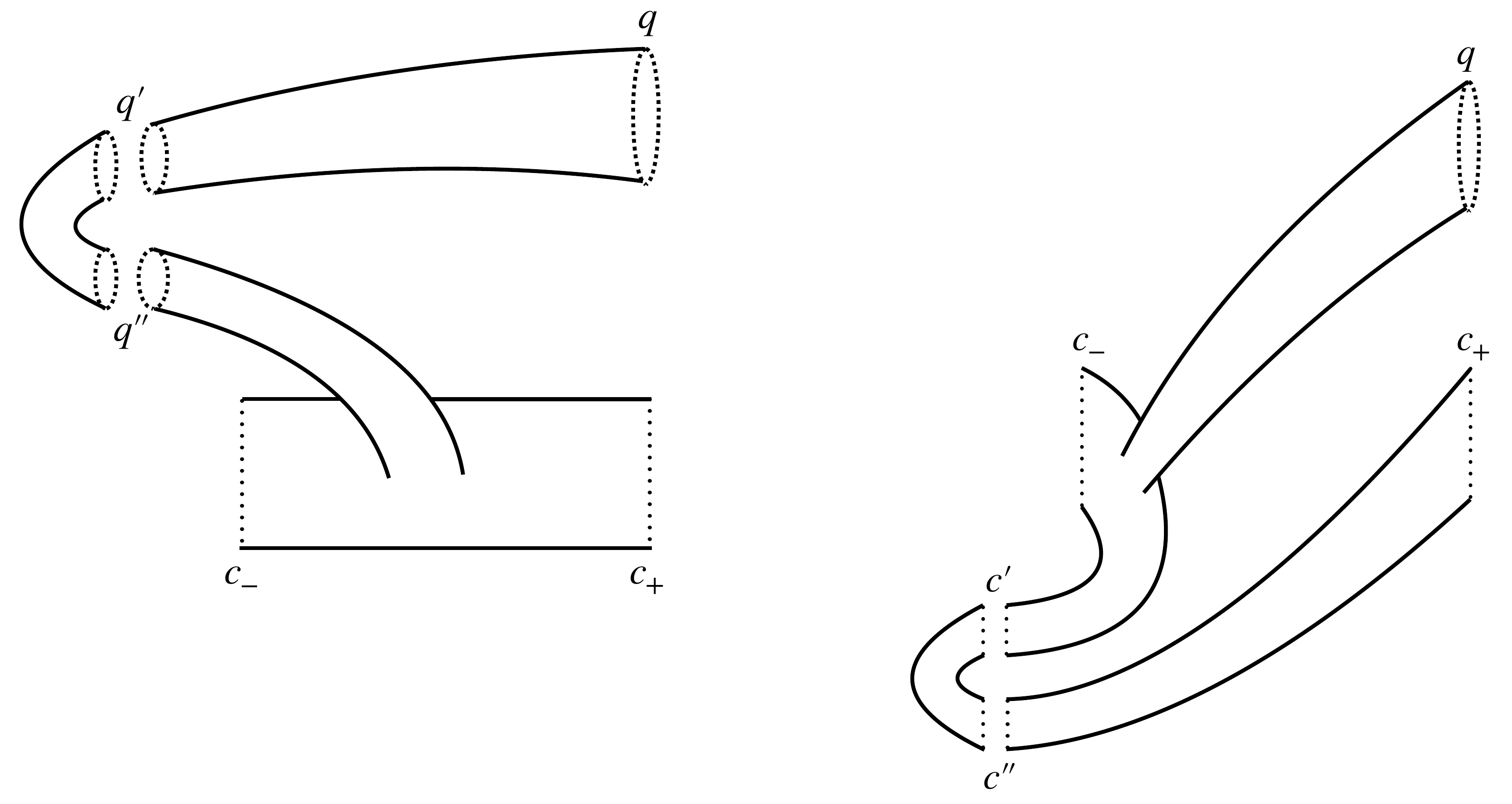}
			\caption{Possible configurations of SFT-broken curves}
			\label{fig:breaking}
		\end{figure} 

			In the first case of Figure \ref{fig:breaking}, the lower-right curve has index $\mu(c_+)-\mu(c_-)-(\mu_\CZ(q'')+n-2)$, which is negative by \eqref{eq:ind_nondeg}. In the second case, the upper-right curve has index $\mu_Y(q)-\mu(c_-)-1-(\mu_\RS(c')+\frac{n-2}{2})$, which is negative again by \eqref{eq:ind_nondeg}. Thus neither case can occur, and this finishes the proof.
		\end{proof}

		We may assume that $f_\mathcal{L}$ has a unique minimum point $c$. Then, for a sufficiently small $\varepsilon>0$, the canonical orientation $o_c\in \FC_n^{(b_1,b_2)}(\mathcal{L}^{\tilde{\sigma}};H_\varepsilon)$, where $b_1<0<b_2$,  associated to $c$ is a cycle. We write $1_\RFH \in\RFH_n(\mathcal{L}^{\tilde{\sigma}})$ for the limit of $[o_c]$ under continuation homomorphisms. In other words, the unit of $\H^*(\mathcal{L})$ maps to $1_\RFH$ in the limit. It serves as the unit for the triangle product on $\RFH_*(\mathcal{L}^{\tilde{\sigma}})$ when this product structure is well-defined.

		\begin{prop}
			\label{prop:vanishing_1_RFH}
			We assume the index conditions in \eqref{eq:ind_nondeg}. If $\RFH_*(Y)=0$, then, for any $\Pin$-structure $\tilde\sigma$ on $\R\times \mathcal{L}$, $1_\RFH \in \RFH_n(\mathcal{L}^{\tilde{\sigma}})$ is zero.
		\end{prop}
		\begin{proof}
			Let $a>0$ be sufficiently large. Let $\mathfrak{c}:\FC_*(\mathcal{L}^{\tilde{\sigma}};H_\varepsilon)\to \FC_*(\mathcal{L}^{\tilde{\sigma}};H_{a+\varepsilon})$ be a continuation homomorphism. Here and below, $\FC_*$ denotes $\FC_*^{(-\infty,\infty)}$. We show that $\mathfrak{c}(o_c)$ is a boundary. 
			
			 We may assume that $f_Y$ has a unique minimum point $p\in Y$. Then the canonical orientation $o_p\in \FC_{n+1}(H_\varepsilon)$ associated to $p$ is a cycle for a sufficiently small $\varepsilon>0$. Note that $o_p$ represents the unit when the pair of pants product is well-defined. Let $\mathfrak{c}_Y: \FC_*(H_\varepsilon)\to \FC_*(H_a)$ be a continuation homomorphism. Since $\RFH_*(Y)=0$, $\mathfrak{c}_Y(o_p)$ is a boundary in $\FC_*(H_a)$. Thus, there exists $\xi\in \FC_{n+2}(H_a)$ such that $\partial_{Y}(\xi)=\mathfrak{c}_Y(o_p)$, where $\partial_{Y}$ is the boundary map in $\FC_*(H_a)$. We consider the following diagram.
			  \begin{equation*}
				\begin{tikzcd}
					\FC_*(H_a)\otimes \FC_*(\mathcal{L}^{\tilde{\sigma}};H_\varepsilon)  & \FC_*(\mathcal{L}^{\tilde{\sigma}};H_{a+\varepsilon}) \\ \FC_*(H_a)\otimes \FC_*(\mathcal{L}^{\tilde{\sigma}};H_\varepsilon)& \FC_*(\mathcal{L}^{\tilde{\sigma}};H_{a+\varepsilon})
					\arrow[from=1-1, to=1-2, "\mathfrak{m}_2"]
					\arrow[from=1-1, to=2-1, "\partial_Y \otimes\mathrm{id} \pm \mathrm{id}\otimes \partial "]
					\arrow[from=1-2, to=2-2, "\partial"]
					\arrow[from=2-1, to=2-2, "\mathfrak{m}_2"]
				\end{tikzcd}
			\end{equation*}
			Here, $\mathfrak{m}_2$ is defined via the moduli spaces in \eqref{eq:moduli_for_module_str}. It induces a module structure on homology when $\mathcal{L}^{\tilde\sigma}$ has an exact Lagrangian filling.  By Proposition \ref{prop:compactness_module}, this diagram commutes when restricted to $\xi\otimes o_c\in \FC_{n+2}(H_a)\otimes \FC_n(\mathcal{L}^{\tilde{\sigma}};H_\varepsilon)$:
			\begin{equation}\label{eq:commute}
			\begin{split}
			\partial\circ\mathfrak{m}_2(\xi\otimes o_c) = \mathfrak{m}_2(\partial_Y(\xi)\otimes o_c)=\mathfrak{m}_2(\mathfrak{c}_Y(o_p)\otimes o_c).
			\end{split}				
			\end{equation}
			Similarly, $\mathfrak{m}_1$ denotes the corresponding chain-level map defined using the Hamiltonians $(H_\varepsilon,H_\varepsilon,H_{2\varepsilon})$.
			Then, the diagram 
			 \begin{equation*}
				\begin{tikzcd}
					\FC_*(H_\varepsilon)\otimes \FC_*(\mathcal{L}^{\tilde{\sigma}};H_\varepsilon)  & \FC_*(\mathcal{L}^{\tilde{\sigma}};H_{2\varepsilon}) \\ \FC_*(H_a)\otimes \FC_*(\mathcal{L}^{\tilde{\sigma}};H_\varepsilon)& \FC_*(\mathcal{L}^{\tilde{\sigma}};H_{a+\varepsilon}) 
					\arrow[from=1-1, to=1-2, "\mathfrak{m}_1"]
					\arrow[from=1-1, to=2-1, "\mathfrak{c}_Y\otimes\mathrm{id} "]
					\arrow[from=1-2, to=2-2, "\mathfrak{c}"]
					\arrow[from=2-1, to=2-2, "\mathfrak{m}_2"]
				\end{tikzcd}
			\end{equation*}
			commutes up to homotopy when restricted to $o_p\otimes o_c$ by an argument analogous to that of Proposition \ref{prop:compactness_module}. Let $\mathcal{H}$ denote the corresponding chain homotopy. Since both $o_p$ and $o_c$ are cycles, we have 
			\begin{equation}\label{eq:commute_homotopy}
			 \mathfrak{c}\circ \mathfrak{m}_1(o_p\otimes o_c)
			=  \mathfrak{m}_2 (\mathfrak{c}_Y(o_p)\otimes o_c) + \partial\circ \mathcal{H}(o_p\otimes o_c).
			\end{equation}
			By energy considerations, $\mathcal{M}(p,c,c';H_{\varepsilon},H_{\varepsilon},H_{2\varepsilon})=\emptyset$ for $c'\neq c$, and $\mathcal{M}(p,c,c;H_{\varepsilon},H_{\varepsilon},H_{2\varepsilon})$ consists only of the constant curve along $c$, and thus $\mathfrak{m}_1(o_p\otimes o_c)=o_c$. This together with \eqref{eq:commute} and  \eqref{eq:commute_homotopy} yields
			\[
			\mathfrak{c}(o_c)=  \mathfrak{c}\circ \mathfrak{m}_1(o_p\otimes o_c) 
			= \partial (\mathfrak{m}_2(\xi \otimes o_c)+ \mathcal{H}(o_p\otimes o_c)).
			\]			
			Hence, $\mathfrak{c}(o_c)$ is a boundary in $\FC_*(\mathcal{L}^{\tilde{\sigma}};H_{a+\varepsilon})$ and also in $\FC_*^{(b_1,b_2)}(\mathcal{L}^{\tilde{\sigma}};H_{a+\varepsilon})$ for $b_1<0 \ll b_2$. This proves $1_\RFH=0$ as desired. 
		\end{proof}

	\bibliographystyle{amsalpha}
	\bibliography{./ref.bib}

\end{document}